\documentclass[12pt,titlepage,twoside,a4paper,BCOR2cm]{scrartcl}
\usepackage{amsmath, amsthm}
\usepackage{mathrsfs}
\usepackage{amssymb}
\usepackage{amsfonts}
\usepackage{makeidx}
\usepackage{eucal}
\usepackage{scrpage2}
\usepackage[ngerman,english]{babel}

\usepackage[T1]{fontenc}
\usepackage[ansinew]{inputenc}

\usepackage{graphicx,color}
\usepackage[arrow, matrix, curve]{xy}
\usepackage{hyperref}

\clearscrheadfoot
\ohead{\headmark}
\ofoot[\pagemark]{\pagemark}
\automark[section]{section}

\hypersetup{colorlinks=true, linkcolor=red, citecolor=green, menucolor=blue,
pdfauthor={Michael Schr\"oder}, pdftitle={Patterson-Sullivan distributions for symmetric spaces of the noncompact type}, pdfsubject={Rank one symmetric spaces},
pdfkeywords={Analysis, Lie-Theory}}

\newcounter{vol2}
\newcommand {\rr}{{\mathbb{R}}}
\newcommand {\cc}{{\mathbb{C}}}
\newcommand {\hh}{{\mathbb{H}}}
\newcommand {\ff}{{\mathbb{F}}}
\newcommand {\zz}{{\mathbb{Z}}}
\newcommand {\nn}{{\mathbb{N}}}
\newcommand {\dd}{{\mathbb{D}}}
\newcommand {\g}{{\mathfrak{g}}}

\newcommand {\la}{{\mathfrak{a}}}

\newcommand {\lp}{{\mathfrak{p}}}

\newcommand {\lk}{{\mathfrak{k}}}

\newcommand {\lnn}{{\mathfrak{n}}}

\newcommand {\lh}{{\mathfrak{h}}}
\newcommand {\cl}{{\textnormal{cl}}}

\renewcommand{\Im}{\textnormal{Im}}
\renewcommand{\phi}{\varphi}

\DeclareMathOperator{\Ind}{Ind}
\DeclareMathOperator{\Ad}{Ad}
\DeclareMathOperator{\ad}{ad}

\DeclareMathOperator{\diverg}{div}
\DeclareMathOperator{\grad}{grad}
\DeclareMathOperator{\id}{id}
\DeclareMathOperator{\supp}{supp}
\DeclareMathOperator{\sign}{sign}
\DeclareMathOperator{\vol}{vol}

\def\dbar{{\ \mathchar'26\mkern-12mu d}}

\theoremstyle{plain}
    \newtheorem{thm}{Theorem}[section]
    \newtheorem{lem}[thm]{Lemma}
    \newtheorem{conj}[thm]{Conjecture}
    \newtheorem{prop}[thm]{Proposition}
    \newtheorem{cor}[thm]{Corollary}
\theoremstyle{definition}
    \newtheorem{defn}[thm]{Definition}
    \newtheorem{rem}[thm]{Remark}
    \newtheorem{problem}[thm]{Problem}

    \newtheorem{ex}[thm]{Example}

\numberwithin{equation}{section}

\makeindex

\begin{document}

\begin{titlepage}
\subject{
\begin{large}Universit\"{a}t Paderborn\\
Fakult\"{a}t f\"{u}r Elektrotechnik, Informatik und Mathematik \\
\vspace{5mm}Universit\'{e} Paul Verlaine Metz \\
\'{E}cole Doctorale, IAEM Lorraine\end{large}\\[18mm]}
\title{Patterson-Sullivan distributions\\for symmetric spaces\\of the noncompact type\\[20mm]}
\publishers{
\begin{large}
\selectlanguage{ngerman}\vspace{4mm}
27. Mai 2010 % \today
\selectlanguage{english}
\end{large}\\
[10mm]
%\begin{normalsize}
%\begin{eqnarray*}
%\textnormal{Gutachter:} \hspace{-4mm} && \textnormal{Prof. Dr. Joachim Hilgert}\\
%&& \textnormal{Prof. Dr. Martin Olbrich}\\
%&& \textnormal{Prof. Nalini Anantharaman}
%\end{eqnarray*}
%\end{normalsize}
}
%\author{Promotionsprojekt}\\
%\author{Dissertation zur Erlangung des akademischen Grades\\ Doktor der Naturwissenschaften (Dr. rer. nat.)\\
%[10mm]
\author{Dissertation\\ Th\`{e}se de doctorat\\
[12mm]
\begin{large}vorgelegt von\end{large}\\
Michael Schr\"{o}der\\[0cm]}
\date{}
\end{titlepage}

%\extratitle{\vspace{1.5cm}
%\begin{center}
%\textbf{\Huge Patterson-Sullivan distributions\\for symmetric spaces\\of the noncompact type}
%\end{center}}

\maketitle

\newpage
\pagebreak
\thispagestyle{empty} 
\section*{Abstract}
There is a curious relation between two kinds of phase space distributions associated to Laplace-eigenfunctions $\phi_{\lambda_k}$ on a compact hyperbolic manifold $Y$.

Given a pseudodifferential operator quantization $Op: C^{\infty}(S^*Y)\rightarrow B(L^2(Y))$, that is an assignment of bounded operators to smooth zero order symbols $a$ on the unit (co-)tangent bundle $S^*Y$, the functionals $\rho_{\lambda_j,\lambda_k}(A)=\langle A\phi_{\lambda_j},\phi_{\lambda_k}\rangle_{L^2(Y)}$ on the space of zero-order pseudodifferential operators give rise to Wigner distributions $W_{\lambda_j,\lambda_k}(a)=\rho_{\lambda_j,\lambda_k}(Op(a))$ on $S^*Y$, which are the key objects in quantum ergodicity. One studies the oscillation and concentration properties of the eigenfunctions through the so-called large energy limits of the distributions $W_{\lambda_j,\lambda_k}$, that is one investigates their behaviour when the eigenvalues tend to infinity.

If $Y$ is a symmetric space of the noncompact type, the Laplace operator is replaced by the corresponding algebra of translation invariant differential operators. Given moderate eigenfunctions $\phi$ and $\psi$, their distributional boudary values in the sense of Helgason give rise to the Patterson-Sullivan distribution $PS_{\phi,\psi}$ on $S^*Y$.

In the case of compact hyperbolic surfaces $Y=\Gamma\backslash\mathbb{H}$ it was observed by N. Anantharaman an S. Zelditch that there is an exact and an asymptotic relation between these phase space distributions.

We generalize parts of a special non-Euclidean calculus of pseudodifferential operators, which was invented by S. Zelditch for hyperbolic surfaces, to symmetric spaces $X=G/K$ of the noncompact type and their compact quotients $Y=\Gamma\backslash G/K$. We sometimes restrict our results to the case of rank one symmetric spcaes. The non-Euclidean setting extends the defintion of Patterson-Sullivan distributions in a natural way to arbitrary symmetric spaces of the noncompact type. Generalizing the exact formula given by Zelditch and Anantharaman, we find an explicit intertwining operator mapping Patterson-Sullivan distributions into Wigner distributions. We study the important invariance and equivariance properties of these distributions. Finally, we describe asymptotic properties of these distributions.

\newpage
\pagebreak
\thispagestyle{empty} 
\selectlanguage{ngerman}
\section*{Zusammenfassung}
Es gibt eine interessante Beziehung zwischen zwei Familien von Distributionen, welche zu Eigenfunktionen $\phi_{\lambda_k}$ des Laplace-Operators einer kompakten hyperbolischen Mannigfaltigkeit $Y$ assoziiert werden:

Gegeben eine Pseudodifferentialoperatoren-Quantisierung, d. h. eine Vorschrift $Op:C^{\infty}(S^*Y)\rightarrow B(L^2 (Y))$, die Symbolen $a$ der Ordnung $0$ auf dem Kosphärenbündel $L^2$-beschränkte Operatoren auf $Y$ zuweist, so erhält man aus den Funktionalen $\rho_{\lambda_j,\lambda_k}(A)=\langle A\phi_{\lambda_j},\phi_{\lambda_k}\rangle_{L^2}$ auf den Raum der Pseudodifferentialoperatoren nullter Ordnung die Wigner-Distributionen $W_{\lambda_j,\lambda_k}(a)=\rho_{\lambda_j,\lambda_k}(Op(a))$ auf dem Kosphärenb\"{u}ndel $S^*Y$. Diese sind die Schlüsselobjekte der Quanten-Ergodizität: Man studiert die Schwingungs- und Konzentrationseigenschaften der Eigenfunktionen, indem man das Hochfrequenzverhalten der Distributionen $W_{\lambda_j,\lambda_k}$ untersucht, d.h. wenn die Eigenwerte gegen unendlich streben.
%Grenzwerte von Wigner-Distributionen sind invariant under dem geodätischen Fluss und unter Zeitumkehrung. Das wirft die Frage auf, ob man aus den Eigenfunktionen Distributionen konstruieren kann, welche zu den Wigner-Distributionen asymptotisch äquivalent sind und welche bereits diese Invarianzeigenschaften besitzen.

Falls $Y$ ein symmetrischer Raum nichtkompakten Typs ist, so wird der Laplace-Operator durch die gesamte Algebra der invarianten Differentialoperatoren ersetzt. Gegeben moderate Eigenfunktionen $\phi$ und $\psi$ auf $Y$, so liefern ihre Helgason-Randwerte sogenannte Patterson-Sullivan Distributionen $PS_{\phi,\psi}$ auf $S^*Y$.
%Es stellt sich tatsächlich heraus, dass diese Distributionen bereits die oben genannten Invarianz-Eigenschaften besitzen.

Im Falle kompakter hyperbolischer Flächen $Y=\Gamma\backslash\mathbb{H}$ beobachteten N. Anantharaman und S. Zelditch eine exakte und eine asymptotische Beziehung zwischen diesen Distributionen.

Wir verallgemeinern Teile eines speziellen nicht-euklidischen Kalküls von Pseudodifferentialoperatoren, welcher zuerst von S. Zelditch für hyperbolische Flächen eingeführt wurde, auf symmetrische Räume $X=G/K$ nichtkompakten Typs und ihre kompakten Quotienten $Y=\Gamma\backslash G/K$. Wir werden uns bei einigen Resultaten auf den Fall von Räumen vom Rang eins beschränken. Das nicht-euklidische Setting erweitert die Definitionen der Patterson-Sullivan Distributionen auf natürliche Weise auf symmetrische Räume nichtkompakten Typs. Wir verallgemeinern die exakte Beziehung zwischen diesen und den Wigner-Distributionen und studieren die wichtigen Eigenschaften der Patterson-Sullivan Distributionen. Schließlich beschreiben wir asymptotische Verbindungen zwischen verschiedenen Arten von Distributionen.

\newpage
\pagebreak
\thispagestyle{empty} 
\selectlanguage{english}
\section*{Acknowledgements}
First of all I wish to express my sincere thanks to my supervisor Joachim Hilgert. I thank him for leaving to me this interesting and challenging topic, and for giving me the chance to be a part of the IRTG, as which I was able to look beyond the scenes of mathematics. I thank him for his enduring help, his patience, and for making me a member of his friendly working group. I express my thanks to Angela Pasquale for her help and for being a kind and helpful advisor, and for a very pleasant stay in Metz. Many thanks go to my longstanding mentor\selectlanguage{ngerman} Sönke Hansen\selectlanguage{english} for always having sympathetic ear and an open door for me. Special thanks go to Martin Olbrich for his expertise and experience in very difficult questions. In particular, I want to thank Nalini Anantharaman and Steve Zelditch for helpful discussions on further issues, problems, and research projects. I thank my colleagues from the IRTG and the Department of Mathematics of the University of Paderborn and last but not least, I thank my friends in Paderborn for their insight and for their support.

\newpage
\pagebreak
\thispagestyle{empty} 

\mbox{}
\newpage
\pagebreak
\thispagestyle{empty} 

\tableofcontents
\newpage
\pagebreak
\thispagestyle{empty} 

\section{Introduction}

Quantum ergodicity is a subfield of mathematics combining dynamical systems and microlocal analysis to investigate the global topography of eigenfunctions of the Laplace-Beltrami operator on Riemannian manifolds.

We begin by describing how the questions of quantum ergodicity are integrated in the greater picture of science. Then we give a brief summary of the basic definitions which are important in quantum ergodicity, and we list a couple of simple properties of the objects we want to investigate. It is important to collect these things in this introduction to motivate the concrete results of this work. We use definitions from the overview articles \cite{Z2005}, \cite{Z2009a} and \cite{Z2009b}, and we also follow the descriptions in \cite{Z87}, \cite{BO} and \cite{SV}.

\subsection*{Background}
Let $(M,g)$ denote a (compact) Riemannian manifold with metric $g$. We denote by $\Delta=\Delta_g$ the corresponding positive Laplace-Beltrami operator
\begin{eqnarray*}
\Delta = - \frac{1}{\sqrt{|\det g_{ij}|}} \sum_{i,j=1}^{n} \frac{\partial}{\partial x_i} \left( g^{ij} {\sqrt{|\det g_{ij}|}} \frac{\partial}{\partial x_j}\right),
\end{eqnarray*}
where $g_{ij}=g(\frac{\partial}{\partial x_i},\frac{\partial}{\partial x_j})$ and where $g^{ij}$ is the inverse matrix to $g_{ij}$. The starting point is the eigenvalue problem
\begin{eqnarray}\label{starting point}
\Delta \phi_{\lambda} = \lambda^2 \phi_{\lambda}, \,\,\,\,\,\,\, \lambda\in\rr.
\end{eqnarray}
In the compact case, the spectrum of $\Delta$ is discrete and we arrange the eigenvalues in non-decreasing order $\lambda_0\leq\lambda_1\leq\lambda_2\leq\ldots\rightarrow\infty$. We denote by $\phi_{\lambda_j}$ an orthonormal basis of real-valued eigenfunctions with respect to the inner product $\langle \phi_{\lambda_j},\phi_{\lambda_k}\rangle = \int_M \phi_{\lambda_j}(x)\phi_{\lambda_k}(x) \, dx$, where $dx$ denotes the volume density. The eigenvalue problem on $M$ is dual under the Fourier transform to the wave equation. We denote the eigenvalues by $\lambda^2$, which saves us from writing a few square root signs. We will later (in the other chapters) often consider the usual defintion of the Laplace-Beltrami operator, that is we will consider $-\Delta$ instead of $\Delta$.

Eigenfunctions of Laplace operators arise in physics as modes of periodic vibration of drums and membranes. They can also represent stationary states of a free quantum particle on a Riemannian manifold. More generally, eigenfunctions of Schr\"{o}dinger operators represent stationary energy states of atoms and molecules in quantum mechanics.

In mathematics, studies of eigenfunctions tend to fall into two categories:
\begin{itemize}
\item The analysis of ground states, i.e. $\phi_0$ or $\phi_1$. An eigenfunction is always the ground state Dirichlet eigenfunction in any of its nodal domains. Other questions in the spectral theory of the Laplacian concern estimates for the least positive eigenvalue (for example, see \cite{U}).
\item The analysis of high frequency limits (semi-classical limits) of eigenfunctions, i.e. the limit as the eigenvalue tends to infinity.
\end{itemize}
Our emphasis is on the high frequency behavior of eigenfunctions. Studies of high frequency behavior eigenfunctions also fall into two categories:
\begin{itemize}
\item Local results, which often hold for any solution of \eqref{starting point} on a (small) ball $B_r(x)$, irrespective of whether the eigenfunction extends to a global eigenfunction.
\item Global results for eigenfunctions that extend to $M$. A typical global assumption is that the eigenfunctions are also eigenfunctions of the wave group.
\end{itemize}

We are interested in global properties of eigenfunctions. These generally reflect the relation of the wave group and geodesic flow, particularly the long time behavior of waves and geodesics on the manifold.

The general approach to understand the global behavior of eigenfunctions is to do a phase space analysis, where the phase space is the co-tangent bundle $T^*M$ or an energy surface $S^*M$. We often identify $T^*M$ and $TM$ using the metric. For example, one often wishes to construct highly localized eigenfunctions or approximate eigenfunctions (quasi-modes) of $\Delta$ or to prove that they do not exist. To obtain global phase space results relating the behavior of eigenfunctions to the behavior of geodesics, it is necessary to use microlocal analysis, i.e. the calculus of pseudo-differential operators. Microlocal analysis is a mathematically precise formulation of the semi-classical limit in quantum mechanics. Pseudo-differential operators are quantizations $Op(a)$ of functions on the phase space $T^*M$: The classical pseudodifferential operators $Op(a)$ on $\rr^n$ are defined by the action on exponentials:
\begin{eqnarray*}
Op(a) e^{i\langle x,\xi\rangle} = a(x,\xi) e^{i\langle x,\xi\rangle}.\index{$Op(a)$, psudodifferential operator}
\end{eqnarray*}
The symbol $a(x,\xi)$ has order $m\in\rr$ if $\sup_K(1+|\xi|)^{j-m}|D_x^{\alpha}D_{\xi}^{\beta}a(x,\xi)|<\infty$ for all compact sets $K$ and all $\alpha,\beta,j$. Symbol classes can also be defined locally and the definition of pseudodifferential operators can be extended to manifolds. A symbol is called polyhomogeneous if it admits a classical asymptotic expansion
\begin{eqnarray*}
a(x,\xi) \sim \sum_{j=0}^{\infty} a_{m-j}(x,\xi),
\end{eqnarray*}
where the $a_{l}$ are homogeneous in $|\xi|\geq 1$ of order $l$. We call the leading term $\sigma_{Op(a)}:=a_m$\index{$\sigma_{Op(a)}$, principal symbol of a pseudodifferential operator $Op(a)$} the principal symbol of $Op(a)$. By $\Psi^m$\index{$\Psi^m(M)$, space of classical pseudodifferential operators on $M$ of order $m$} we denote the space of classical pseudodifferential operators on $M$ of order $m$. We have the exact sequence of algebras $0\rightarrow\Psi^{-1}\rightarrow\Psi^0 \overset{\sigma}{\rightarrow} C^{\infty}(SM)\rightarrow 0$, where $\sigma$ is the principal symbol map. A right-inverse of $\sigma$ mapping homogeneous symbols of order $0$ into $L^2$-bounded operators is called quantization or operator convention. Functions on $S^*M$ are also called observables.

The cotangent bundle is equipped with the symplectic form $\sum_i dx_i \wedge d\xi_i$. The metric defines the Hamiltonian vector field
\begin{eqnarray*}
H(x,\xi) = |\xi|_g = \sqrt{\sum_{i,j=1}^{n} g^{ij}(x) \xi_i \xi_j}
\end{eqnarray*}
on $T^*M$. The classical evolution is given by the geodesic flow of $(M, g)$, i.e. the Hamiltonian flow $g^t$ of $H$ on $T^*M$: By definition, $g^t(x,\xi) = (x_t,\xi_t)$, where $(x_t,\xi_t)$ is the terminal tangent vector at time $t$ of the unit speed geodesic starting at $x$ in the direction $\xi$. The Liouville measure $\mu_L$\index{$\mu_L$, Liouville measure} on $S^*M$ is by definition the measure $d\mu_L = \frac{dxd\xi}{dH}$ induced by the Hamitonian and the symplectic volume element $dx\,d\xi$ on $T^*M$. The geodesic flow preserves the Liouville measure. We can thus define a unitary operator $V^t$ on $L^2(S*M, d\mu_L)$ by
\begin{eqnarray*}
V^t(a) := a \circ g^t.
\end{eqnarray*}
The operator $V^ t$\index{$V^t$, translation operator associated to the geodesic flow} is called the translation operator associated to the geodesic flow. The geodesic flow is called ergodic, if $V^t$ has no invariant $L^2$-functions besides the constants. Equivalently, the geodesic flow is called ergodic, if any invariant set $E\subseteq S^*M$ has either zero measure or full measure.

The quantization of the Hamiltonian is the square root $\sqrt{\Delta}$ of the positive Laplacian. Quantum evolution is given by the wave group
\begin{eqnarray*}
U^t = e^{it\sqrt{\Delta}}.\index{$U^t$, wave group}
\end{eqnarray*}
It is generated by the pseudodifferential operator $\sqrt{\Delta}$ as defined by the spectral theorem: It has the same eigenfunctions as $\Delta$, but to the eigenvalues $\lambda$.

Evolution of observables is known in physics as the 'Heisenberg picture'. It is defined by
\begin{eqnarray*}
\alpha_t(A) = U^t A U^{-t}, \,\,\,\,\,\,\, A\in\Psi^m.
\end{eqnarray*}
Egorov's theorem yields a correspondence to the classical evolution $V^t(a)=a\circ g^t$. It says that $\alpha_t$ is an order preserving automorphism on the space of pseudodifferential operators, that is $\alpha_t(A)\in\Psi^m$ for all $A\in\Psi^m$ and that
\begin{eqnarray*}
\sigma_{\alpha_t(A)}(x,\xi) = \sigma_A(g^t(x,\xi)) = V^t(\sigma_A).
\end{eqnarray*}
This formula is almost universally taken to be the definition of quantization of a flow or map in the physics literature.

In quantum ergodicity, one studies the concentration and oscillation properties of eigenfunctions through the linear functionals
\begin{eqnarray*}
\rho_{\lambda_j}(A) = \langle A\phi_{\lambda_j},\phi_{\lambda_j}\rangle
\end{eqnarray*}
on the space of zeroth order pseudo-differential operators $A$. The possible limits of the family $\left\{\rho_{\lambda_j}\right\}$ are called quantum limits\index{quantum limit} or microlocal defect measures\index{microlocal defect measure}. The diagonal elements $\rho_{\lambda_j}(A)$ are interpreted in quantum mechanics as the expected value of the observable $Op(a)$ in the energy state $\lambda_j$. The off-diagonal matrix elements
\begin{eqnarray*}
\rho_{\lambda_j,\lambda_k}(A) = \langle A\phi_{\lambda_j},\phi_{\lambda_k}\rangle
\end{eqnarray*}
are interpreted as transition amplitudes between states. We fix a quantization $a\rightarrow Op(a)$. The matrix elements are then also called Wigner distributions:
\begin{eqnarray*}
W_{\lambda_j,\lambda_k}(a)=\rho_{\lambda_j,\lambda_k}(Op(a)).\index{$W_{\lambda_j}$, Wigner distribution}
\end{eqnarray*}
We first observe that $\rho_{\lambda_j,\lambda_k}(I)=\delta_{j,k}$ (Kronecker-Delta), since the eigenfunctions are orthonormal in $L^2(M)$. In the diagonal case, the functionals $\rho_{\lambda_k}$ are positive in the sense that for any operator $A$ we have $\rho_{\lambda_k}(A^*A)\geq0$. This can be seen by moving $A^*$ to the right side in the $L^2$-inner product. Writing out $\rho_{\lambda_j,\lambda_k}(U^tAU^{-t})$ and moving $U^{t}$ to the right side we find
\begin{eqnarray*}
\rho_{\lambda_j,\lambda_k}(U^tAU^{-t}) = e^{it(\lambda_k-\lambda_j)}\rho_{\lambda_j,\lambda_k}(A).
\end{eqnarray*}
These properties are summarized by saying that $\rho_{\lambda_j}$ is an invariant state\index{invariant state} on (the closure in the operator norm of) the algebra $\Psi^0$.

Let $Q$ denote the set of possible quantum limits. Any orthonormal basis such as $\left\{\phi_{\lambda_k}\right\}$ tends to $0$ weakly in $L^2$. Hence $\left\{K\phi_{\lambda_k}\right\}$ tends to $0$ weakly in $L^2$ for each compact operator $K$. Then, the diagonal elements $\rho_{\lambda_j}(K)$ tend to $0$ for all compact $K$. Given two pseudodifferential operators on $M$ with the same principal symbol of order zero, their difference is an operator of negative order and thus compact. It follows that $Q$ is independent of the choice of quantization.

Using standard estimates on pseudodifferential operators one shows (\cite{Z2009a}, $\S6$) that any weak limit is continuous on $C(S^*M)$. It is a positive functional since each $\rho_{\lambda_k}$ is and hence any limit is a probability measure.

By the invariance of the $\rho_{\lambda_k}$ under the automorphisms $\alpha_t$ on $\Psi^0$ and by Egorov's theorem we find that any limit of $\rho_{\lambda_k}(A)$ is a limit of $\rho_{\lambda_k}(Op(\sigma_A\circ g^t))$, and hence the limit is invariant under the geodesic flow $g^t$.

It follows from $\rho_{\lambda_j,\lambda_k}(U^tAU^{-t}) = e^{it(\lambda_k-\lambda_j)}\rho_{\lambda_j,\lambda_k}(A)$ that the off-diagonal matrix elements can only have a limit for subsequences $\left\{\lambda_{j_n}\right\}$ and $\left\{\lambda_{k_n}\right\}$ of eigenvalue-parameters such that the spectral gap $|\lambda_{j_n}-\lambda_{k_n}|$ tends to a limit $\tau\in\rr$. In that case, each limit $\mu$ is an eigenmeasure for the geodesic flow:
\begin{eqnarray*}
\mu(a\circ g^t) = e^{it\tau}\mu(a).
\end{eqnarray*}
A measure is called invariant under time-reversal, if it is invariant under the anti-symplectic involution $(x,\xi)\rightarrow(x,-\xi)$ on $T^*M$. Since the eigenfunctions are (by our assumption) real-valued and hence complex-conjugation invariant, it follows that any quantum limit is invariant under time-reversal.

From the mathematical point of view, one would like to know the behavior of the diagonal matrix elements and the off diagonal matrix elements, when the eigenvalue tends to infinity. One of the principal problems is:
\begin{problem}\label{quantum limits problem}
Determine the set $Q$ of quantum limits.
\end{problem}
As a motivating example, suppose that for a subsequence $k_j$ the functionals $\rho_{k_j}$ tend to the Liouville measure $\mu_L$. Let $E\subseteq M$ denote a measureable set whose boundary has measure zero. Testing against multiplication operators (with symbols given by smoothed versions of the characteristic function of $E$) yields (\cite{Z2009b}, p. 19)
\begin{eqnarray*}
\frac{1}{\vol(M)} \int_E |\phi_{k_j}(x)|^2 \, dx \rightarrow \frac{\vol(E)}{\vol(M)}.
\end{eqnarray*}
We interprete $|\phi_{k_j}(x)|^2 \, dx$ as the probability density of finding a particle of energy $\lambda_{k_j}^2$ in $E$. Then this sequence of probabilities tends to uniform measure and the eigenfunctions become uniformly distributed on $M$. However, the assumption $\rho_{k_j}\rightarrow\mu_L$ is much stronger, since then
\begin{eqnarray*}
\langle Op(1_E)\phi_{k_j},\phi_{k_j}\rangle \rightarrow \frac{\mu_L(\pi^{-1}(E))}{\mu_L(S^*M)},
\end{eqnarray*}
where $\pi: S^*M \rightarrow M$ is the natural projection. The Laplacian of $(M,g)$ is said to be QUE (quantum uniquely ergodic) if the only quantum limit for any orthonormal basis of eigenfunctions is the Liouville measure. The following conjecture was first stated by Rudnick-Sarnak (\cite{RS}):
\begin{conj}
Let $(M,g)$ be a negatively curved manifold. Then $\Delta$ is QUE.
\end{conj}
Off-diagonal matrix elements are also important as transition amplitudes between states. As described above, a sequence of such matrix elements cannot have a weak limit unless the spectral gap tends to a limit $\tau$. We denote the corresponding set of limits by $Q_{\tau}$. Then we can also formulate:
\begin{problem}
Determine the set $Q_{\tau}$ of off-diagonal quantum limits.
\end{problem}
For examples of possible quantum limits we refer to the overview articles \cite{Z2009a} and \cite{Z2009b}, which also describe recent developments of mathematical quantum chaos such as mixing properties of eigenfunctions, boundary quantum ergodicity, converse quantum ergodicity, and other problems.

\newpage
\pagebreak
\thispagestyle{empty}

\subsection*{Outline and statement of results}
Let $X=G/K$\index{$X=G/K$, symmetric space} denote a symmetric space of the noncompact type, where $G$\index{$G$, semisimple Lie group} is a connected semisimple Lie group with finite center and $K$\index{$K$, maximal compact subgroup of $G$} a maximal compact subgroup of $G$. In Section \ref{Preliminaries} we recall basic definitions concerning symmetric spaces and we give detailed descriptions of their geometry. Our setting is as follows: Let $G=KAN$ be an Iwasawa decomposition of $G$ and let $M$ denote the centralizer of $A$ in $K$. The geodesic boundary of $X$ can be identified with the flag manifold $B:=K/M$. Let $o:=K\in G/K$ denote the \emph{origin} of the symmetric space $X$. We fix a cocompact and torsion free discrete subgroup $\Gamma$ of $G$. Let $\Delta$, resp. $\Delta_{\Gamma}$, denote the Laplace operator of $X$, resp. $X_{\Gamma}$.

In \cite{Z86}, S. Zelditch introduced a natural pseudodifferential operator convention for $G/K$, when $G=PSU(1,1)$, $K=PSO(2)$. In Section \ref{Pseudodifferential analysis on symmetric spaces} we generalize this calculus to symmetric spaces of the noncompact type. We sometimes restrict our results to rank one spaces. The interesting aspect of this calculus is its $G$-equivariance: Let $SX$ denote the unit tangent bundle of $X=G/K$. If $a\in C^{\infty}(SX)$ is $\Gamma$-invariant under the natural action of $G$ on $SX$, then it yields a pseudodifferential operator on the quotient $X_{\Gamma}:=\Gamma\backslash G/K$\index{$X_{\Gamma}$, compact hyperbolic manifold}. We can hence use the $G$-equivariant non-Euclidean pseudodifferential calculus to define Wigner distributions on the quotient $X_{\Gamma}=\Gamma\backslash G/K$. 

If $Y$ is a manifold, $u$ a distribution or hyperfunction on $Y$ and $\phi$ a test function, then we denote the pairing $\langle \varphi, u\rangle_Y$ by $\int_Y  \varphi(y)u(dy)$. The starting point of all following observations is Helgason's representation theorem for joint eigenfunctions of the algebra $\mathbb{D}(G/K)$ of translation invariant differential operators: Given a joint eigenfunction $\phi\in\mathcal{E}_{\lambda}(X)$ (see Section \ref{Section Helgason boundary values}), then there is a linear functional $T_{\phi}$ on the space of analytic functions on $B$ such that $\phi$ is given by the Poisson-Helgason-transform $\phi(z)=\langle e^{(i\lambda+\rho)\langle z,b\rangle}, T\rangle_B=\int_B e^{(i\lambda+\rho)\langle z,b\rangle} T(db)$. Here, the function $e_{\lambda,b}:=e^{(i\lambda+\rho)\langle z,b\rangle}$ denotes a generalized Poisson kernel (see Section \ref{Invariant differential operators}).
\iffalse
Here $\lambda\in\la^*_{\cc}$, the complexification of the dual of the Lie algebra of $A$ and $\rho$ is a parameter depending on the root-structure of $\g$, the Lie algebra of $G$. The bracket $\langle z,b\rangle$ is defined on $X\times B$ and the powers $e^{(i\lambda+\rho)\langle z,b\rangle}$ of the Poisson-kernel $e^{2\rho\langle z,b\rangle}$ are called non-Euclidean plane waves.
\fi

In Section \ref{Section Helgason boundary values} we describe the theory of Helgason boundary values. In particular, we describe their regularity as a function of the spectral parameter $\lambda\in\la$, where $\la$ is the Lie algebra of $A$.
\iffalse
When these boundary values are associated to a function belonging to a joint eigenspace $\mathcal{E}_{\lambda}(X)$, we interprete $\lambda\in\la^*$ (or $\la^*_{\cc}$) as a spectral parameter tending to infinity. In fact, in the non-Euclidean context the eigenvalues of the Laplace-Beltrami operator can be parameterized by the elements $\lambda\in\la^*_{\cc}$ (cf. Section \ref{Invariant differential operators}). As usual in mathematics, also the list of properties of the boundary values is long. However, we did not find a result which turned out to be useful for us. We will therefore in Subsection \ref{Regularity} study the regularity of the boundary values in dependence of the spectral parameter $\lambda$. To our knowledge, the results we present are new and have not been written down in the literature before. In order to give these results for arbitrary rank symmetric spaces, we slightly modify the results from our recent preprint \cite{HS09}.
\fi

Wigner distributions tend to measures with certain invariance properties. The question arises whether there exist distributions constructed from eigenfunctions which are related to the Wigner-distributions and which already possess these invariance properties. For hyperbolic surfaces, such distributions were constucted by N. Anantharaman and S. Zeldirch in \cite{AZ}. These distributions were termed Patterson-Sullivan distributions by analogy with their construction of boundary measures associated to ground states on infinite volume hyperbolic manifolds (\cite{Sul}): The Patterson-Sullivan distribution associated to a real eigenfunction $\phi_{ir}$ corresponding to the eigenvalue $1/4+r^2$ and with associated boundary values $T_{ir}$ is the distribution on $B^{(2)}$ (the space consisting of distinct boundary points $b,b'\in B$) defined by
\begin{eqnarray}
ps_{ir}(db,db') := \frac{T_{ir}(db) T_{ir}(db')} {|b-b'|^{1+2ir}}.
\end{eqnarray}
The interesting aspect of quotients $X_{\Gamma}$ lies in the study of $\Gamma$-invariant eigenfunctions on the original symmetric space: If the eigenfunction is $\Gamma$-invariant, then the corresponding Patterson-Sullivan distribution is $\Gamma$-invariant and invariant under time reversal. To obtain a geodesic flow invariant distribution $PS_{ir}$ on $SX$, Anantharaman and Zelditch tensor with $dt$. They also define normalized Patterson-Sullivan distributions by dividing by the integral against $1$. The result is a geodesic flow invariant distribution $\widehat{PS}_{ir}$ constructed as a quadratic expression in the eigenfunctions. Anantharaman and Zelditch then proved that there is an explicit intertwining operator $L_{ir}$ mapping Patterson-Sullivan distributions into Wigner distributions.
\iffalse
, and one finds
\begin{eqnarray}
\langle a,W_{ir}\rangle = \langle a,\widehat{PS}_{ir}\rangle + r^{-1} \langle R(a),\widehat{PS}_{ir}\rangle + \mathcal{O}(r^{-2}),
\end{eqnarray}
where $R$ is a differential operator. Hence the quantum limits problem \ref{quantum limits problem} is the same of determining the weak$^*$-limits of the Patterson-Sullivan distributions.
\fi

We explain how to generalize these definitions to symmetric spaces of the noncompact type: Following \cite{E96} we say that two distinct boundary points $b,b'\in B$ can be joint at infinity if there is a geodesic in $X$ with forward endpoint $b$ and backwards endpoint $b'$. We describe in Section \ref{Preliminaries} the open dense subset $B^{(2)}$ of distinct boundary points that can be joint at infinity. It turns out that this space is invariant under the action of $G$ on $B$ and identifies with the homogeneous space $G/MA$. We introduce functions $d_{\lambda}$ on $B^{(2)}$ and a geodesic Radon transform $\mathcal{R}:C_c^{\infty}(SX)\rightarrow C_c^{\infty}(B^{(2)})$ such that the expression
\begin{eqnarray}\label{expression}
\langle a,PS_{\lambda}\rangle_{SX} := \int_{B^{(2)}} \, d_{\lambda}(b,b') \, \mathcal{R}(a)(b,b') \, T_{\lambda}(db)\, T_{\lambda}(db')
\end{eqnarray}
defines a $\Gamma$-invariant distribution on $SX$, and this is the generalized \emph{Patterson-Sullivan distribution} associated to the eigenfunction $\phi\in\mathcal{E}_{\lambda}$ (Sec. \ref{Invariant differential operators}). The $PS_{\lambda}$ are invariant under the geodesic flow and under time reversal. The weight functions $d_{\lambda}$ will be called \emph{intermediate values} because they satisfy a certain equivariance property, which generalizes a so-called intermediate values formula for hyperbolic surfaces (Sec. \ref{Patterson-Sullivan distributions}).

As was pointed out in the introduction of \cite{AZ} it is of interest to also have analogous definitions for off-diagonal matrix entries. We will in fact also consider these off-diagonal elements and off-diagonal Patterson-Sullivan distributions: In Section \ref{Patterson-Sullivan distributions} we use off-diagonal intermediate values $d_{\lambda,\mu}$ on $B^{(2)}$. Given joint eigenfunctions $\phi$ and $\psi$ we then introduce general off-diagonal Patteson-Sullivan distributions on $SX$.

The point is that all Patterson-Sullivan distributions we consider are $\Gamma$-invariant. We show how this lets the definitions descend to quotients $X_{\Gamma}$. In order to generalize the above mentioned results for hyperbolic surfaces, we will find an explicit intertwining operator that maps off-diagonal Patterson-Sullivan distributions into non-Euclidean Wigner distributions.
\iffalse
Finally we will also have a generalization of the asymptotic equivalence of Patterson-Sullivan distributions and Wigner distributions.
\fi

\newpage
\pagebreak
\thispagestyle{empty} 

\section{Preliminaries}\label{Preliminaries}

A Riemannian manifold $X$ is a called a \emph{homogeneous space} if its group of Riemannian isometries acts transitively on $X$. We consider a point $x$ of a connected Riemannian manifold $X$. Let $U$ denote a symmetric neighborhood of $M$ in the tangent space of $x$ such that the exponential map is well-defined on $U$ and a diffeomorphism onto its image $V$. The symmetry $u\mapsto -u$ of $U$ induces a map $s_x$ on $V$, which we call the \emph{local geodesic symmetry centered at $x$}\index{geodesic symmetry}. We say that $X$ is a \emph{Riemannian locally symmetric} space if for any $x$ in $X$ the corresponding local symmetry at $x$ is a local isometry of $X$. We say that $X$ is \emph{globally symmetric} for any $x$ this isometry may be extended uniquely to $X$. A complete simply connected locally symmetric space is globally symmetric. In this sense, globally symmetric spaces are complete spaces which possess a very large group of isometries. In particular, their group of isometries acts transitively. We recall material from \cite{He01} and \cite{E96} for some background.

A globally symmetric space $X$ is the Cartesian Riemannian product of three globally symmetric spaces $X = \rr^n \times D \times T$ (\emph{de Rham decomposition}), where $D$ has nonpositive curvature, where $T$ has nonnegative curvature, and where $D$ and $T$ may not be written as a product of $\rr$ with another Riemannian manifold. We say that $D$ is of noncompact type and $T$ is of compact type. We will be interested in symmetric spaces of the noncompact type.

The structure of Riemannian symmetric spaces is intrinsically linked with the theory of Lie groups: Let $G$ denote the isometry group of the connected Riemannian manifold $X$. For a compact subset $C$ of $X$ and an open subset $U$ of $X$ put $W(C,U):=\left\{g\in G:g\cdot C\subset U\right\}$. The compact open topology is defined as the smallest topology on $G$ for which all the sets $W(C,U)$ are open. For this topology, $G$ is Hausdorff, separable, locally compact and second countable. If $X$ is globally symmetric, $G$ can be proved to carry a structure of Lie group compatible with this topology. Let $G_0$ denote the identity component of $G$, select a point $p\in X$ and denote by $K$ the subbgroup of $G_0$ which stabilizes $p$. Then $K$ is a maximal compact subgroup of $G_0$ and $G_0/K$ is isometric to $X$. On the other hand, given a connected Lie group $G_0$ and a closed subgroup $K$ of $G_0$, we call $(G_0,K)$ a Riemannian symmetric pair if the group $\Ad_{G_0}(K)$ is compact and if there exists an involutive smooth automorphism $\sigma$ of $G_0$, which is not the identity, such that $(K_{\sigma})_0\subset K \subset K_{\sigma}$, where $K_{\sigma}$ is the set of fixed point of $\sigma$ in $G_0$ and where $(K_{\sigma})_0$ is its identity component. Then there is a Riemannian metric on $G_0/K$ such that $G_0/K$ is a Riemannian symmetric space. We now explain these constructions for symmetric spaces of the noncompact type.

Call a Lie algebra \emph{semisimple} if it is a direct sum of simple (non-abelian) Lie algebras without proper ideals. A (connected) Lie group is said to be semisimple if its Lie algebra is semisimple, that is it has no non-trivial abelian connected normal closed subgroup. We denote the Lie algebra of $G$ by $\mathfrak{g}$ and let $Tr$ denote the \emph{trace}\index{trace} of a vector space endomorphism. We consider the symmetric bilinear form $B(X,Y)=Tr(\ad X \ad Y)$ on $\mathfrak{g}\times\mathfrak{g}$ and call $B(\cdot,\cdot)$\index{$B(\cdot,\cdot)$, Killing form} the \emph{Killing form}\index{Killing form} of $\mathfrak{g}$. A Lie algebra $\mathfrak{g}$ over a field of characteristic $0$ is \emph{semisimple}\index{semisimple} if and only if its Killing form $B$ of $\mathfrak{g}$ is non-degenerate. \iffalse (cf. \cite{He01}, Exercise B.8, Ch. III).\fi

Let $X$ be a symmetric space of the noncompact type. If $p$ is any point of $X$, its stabilizer is a maximal compact subgroup of $G_0$. If $K$ is a maximal compact subgroup of $G_0$, then there is a unique point $p$ in $X$ such that $K$ is the stabilizer of $p$. Any two maximal compact subgroups of $G_0$ are conjugate by an element of $G_0$. If $\mathfrak{k}$ is the Lie algebra of $K$, the Killing form of $\mathfrak{g}$ is strictly negative definite on $\mathfrak{k}$. The group $G_0$ acts transitively on $X$. It is a semisimple Lie group with finite center. \iffalse \cite{E96}, pp. 59 ff, pp. 69 ff. See also \cite{He01}, Ch. VI.\fi Fix a point $p\in X$ and let $K$ denote its stabilizer in $G_0$. Consider the coset space $G_0/K$ and the diffeomorphism $\phi: G_0/K \rightarrow X, \,\, \phi(gK) = g(p)$ for $g\in G_0$. Denote by $\langle\hspace{1mm},\hspace{1mm}\rangle$\index{$\langle\hspace{1mm},\hspace{1mm}\rangle$, metric} the metric on $G_0/K$ obtained by pulling back the metric of $X$ by $\phi$. Then $\phi$ is an isometry and the metric $\langle\hspace{1mm},\hspace{1mm}\rangle$ is left $G_0$-invariant, that is left translations on $G_0/K$ by elements of $G_0$ are isometries of the metric space $(G_0/K,\langle\hspace{1mm},\hspace{1mm}\rangle)$. Hence each globally symmetric space of the noncompact type can be written in the form $G_0/K$ as above. These observations are summarized by
\begin{thm}[E. Cartan]
The Riemannian globally symmetric spaces of the noncompact type are the spaces of the form $G/K$ equipped with a $G$-invariant metric, where $G$ is a connected semisimple Lie group with finite center and $K$ a maximal compact subgroup of $G$.
\end{thm}

\subsection{Symmetric spaces and real semisimple Lie groups}\label{Decomosition of real semisimple Lie groups}
\begin{defn}
A \emph{Riemannian symmetric space of the noncompact type} is a homogeneous space $X=G/K$\index{$X=G/K$, symmetric space of the noncompact type}, where $G$\index{$G$, real semisimple Lie group with finite center} is a real connected semisimple Lie group with finite center and $K$\index{$K$, maximal compact subgroup of $G$} is a maximal compact subgroup of $G$.
\end{defn}

Let $G$ denote a connected Lie group with Lie algebra $\g$ and let $H$ be a closed subgroup of $G$ with Lie algebra $\lh$. By $G/H$ we denote the quotient space consisting of left cosets $gH$, $g\in G$. Let $\pi:G\rightarrow G/H$ denote the natural projection. Choose a complementary subspace $\mathfrak{m}$ of $\g$ such that $\mathfrak{g}=\mathfrak{h}\oplus\mathfrak{m}$. Let $X_1,...,X_r$ and $X_{r+1},...,X_n$ be bases of $\mathfrak{m}$ and $\mathfrak{h}$, respectively. If $g\in G$, the mapping
\begin{eqnarray}\label{G/H coordinates}
(x_1,...,x_r)\mapsto \pi (g\exp (x_1X_1 + ... + x_rX_r))
\end{eqnarray}
is a diffeomorphism of a neighborhood of $0\in\mathfrak{m}$ onto a neighborhood of the point $\pi(g)=gH\in G/H$. The inverse of \eqref{G/H coordinates} is a local coordinate system near $gH$, turning a neighborhood of each $\pi(g)$ and hence $G/H$ into a manifold. \iffalse (\cite{He00}, p. 284).\fi

The Lie algebra $\mathfrak{g}$ is naturally identified with the tangent space $T_e G$ of $G$ at the identity $e\in G$. We list basic results and definitions about semisimple Lie groups. Details can be found in standard sources (\cite{He01}).

For each $X\in\mathfrak{g}\cong T_e G$ there is a unique homomorphism $\gamma_X : (\rr,+)\rightarrow G$ such that $\gamma_X'(0)=X$. The image of $\gamma_X$ is called a \emph{one parameter subgroup of} $G$. The mapping $\exp:\mathfrak{g}\rightarrow G, X\mapsto\exp(X):=\gamma_X(1)$ is called the \emph{exponential map}\index{exponential map} of $G$. We have $e^{tX}=\gamma_X(t)$ for all $t\in\rr$.

Each $g\in G$ defines an inner automorphism $C_g:G\rightarrow G$ by $C_g(h)=ghg^{-1}$ of the group $G$. Taking the derivative we define a Lie algebra automorphism
\begin{eqnarray*}
\Ad(g)=dC_g:\mathfrak{g}\rightarrow \mathfrak{g}.
\end{eqnarray*}
The map $\Ad:G\rightarrow\textnormal{Aut}(\g)$ is called the \emph{adjoint representation} of $G$\index{$\Ad$, adjoint representation}\index{$\ad$, adjoint representation}. We will often denote the corresponding group action of $G$ on $\g$ by $g\cdot X$ ($X\in\g$). For $X\in \mathfrak{g}$ we define a linear transformation
\begin{eqnarray*}
\ad X:\mathfrak{g}\rightarrow\mathfrak{g}, \,\,\,\,  (\ad X)(Y)=[X,Y],
\end{eqnarray*}
where $[\hspace{1mm},\hspace{1mm}]$ denotes the Lie bracket of vector fields on $G$.

If $\sigma$ is an automorphism of $\mathfrak{g}$ then $\ad(\sigma X)=\sigma\circ\ad X\circ\sigma^{-1}$ so by $Tr(AB)=Tr(BA)$ we have $B(\sigma X,\sigma Y)=B(X,Y)$ and $B([X,Y],Z) + B(Y,[X,Z]) = 0.$ If $\mathfrak{a}$ is an ideal in $\mathfrak{g}$, then the Killing form of $\mathfrak{a}$ coincides with the restriction of $B$ to $\mathfrak{a}\times\mathfrak{a}$.

The space $G/H$ is called \emph{reductive}\index{reductive}, if $\mathfrak{m}$ as above can be chosen such that
\begin{eqnarray}\label{reductive}
\mathfrak{g}=\mathfrak{h}\oplus\mathfrak{m},\hspace{4mm}\Ad_G(h)\mathfrak{m}\subset\mathfrak{m} \hspace{4mm}(h\in H).
\end{eqnarray}
If $\Ad(H)$ is compact, then $G/H$ is reductive: In fact, $\g$ will then admit a positive definite quadratic form invariant under $\Ad_G(H)$ and $\mathfrak{m}$ can be chosen to be the orthogonal complement (w.r.t. this quadratic form) of $\mathfrak{h}$ in $\mathfrak{g}$ (\cite{He00}, p. 284).

\subsubsection{Tangent spaces and Cartan decomposition}
For the descriptions of the geometric structure of a symmetric space $X=G/K$ in terms of algebraic data given by the semisimple Lie group $G$ we orient ourselves on \cite{E96}.

We write $o:=K\in G/K$ and call $o$ the \emph{origin} of the symmetric space $X=G/K$. Define an involution $\sigma: G\rightarrow G$ by $\sigma(g) = s \circ g \circ s$ (\cite{E96}, p. 71), where $s$ denotes the geodesic symmetry at $o$. The differential of $\sigma$ at $e$ is $\theta=d\sigma:\mathfrak{g}\rightarrow\mathfrak{g}$, which is also characterized by the equation $\sigma(e^{tX})=e^{t\theta(X)}$ for all $X\in\mathfrak{g}$. Since $\theta^2=\id_{\g}$ we obtain the \emph{Cartan decomposition}\index{Cartan decomposition}
\begin{eqnarray*}
\mathfrak{g} = \mathfrak{k} \oplus \mathfrak{p},
\end{eqnarray*}
where $\mathfrak{k} = \left\{ X\in\mathfrak{g} : \theta X = X\right\}$ and $\mathfrak{p} = \left\{ X\in\mathfrak{g} : \theta X = -X\right\}$ \index{$\mathfrak{p}$, orthogonal complement of $\mathfrak{k}$} are the eigenspaces corresponding to the eigenvalues $+1$ and $-1$. The Lie algebra automorphism $\theta$ preserves Lie brackets, so we have
\begin{eqnarray}\label{K invariant}
\left[\mathfrak{k},\mathfrak{k}\right]\subseteq\mathfrak{k}, \hspace{2mm} \left[\mathfrak{k},\mathfrak{p}\right]\subseteq\mathfrak{p}, \hspace{2mm} \left[\mathfrak{p},\mathfrak{p}\right]\subseteq\mathfrak{k}.
\end{eqnarray}

We consider the map $\pi: G\rightarrow X$ given by $g\rightarrow g\cdot o$. Taking the differential we obtain a linear map $d\pi : \mathfrak{g} \rightarrow \g$, whose kernel is precisely $\mathfrak{k}$. Moreover, $\mathfrak{k}$ is the Lie algebra of the maximal compact subgroup $K=\left\{g\in G: g\cdot o=o\right\}$ of $G$. The restriction $d\pi: \mathfrak{p}\rightarrow \g$ is a monomorphism and we use it to identify $T_o X=\mathfrak{p}$. Although we restricted the above  constructions to the particular point $o$, these results can be obtained at for each $p\in X$.

It also follows that $\Ad(K)$ leaves $\mathfrak{p}$ invariant. Moreover, the elements of $\Ad(K)$ are orthogonal transformations on $\mathfrak{p}$ with respect to the restriction to $\mathfrak{p}$ of the Killing form $B$ of $\mathfrak{g}$. \iffalse (cf. \cite{E96}, p. 71-72).\fi The spaces $\mathfrak{p}$ and $\mathfrak{k}$ are orthogonal with respect to the Killing form $B$ of $\mathfrak{g}$.\iffalse (cf. \cite{E96} and p. 77).\fi

For $X,Y\in\mathfrak{g}$ we set $B_{\theta}(X,Y)=-B(\theta X, Y)$\index{$B_{\theta}$, positive definie bilinear form on $\g$}. Then $B_{\theta}$ is a positive definite bilinear form on $\mathfrak{g}$. \iffalse (cf. \cite{E96}, p. 77/ \cite{He01}, Ch. III, \S 7).\fi We can therefore call $|X|^2=B_{\theta}(X,X)$ the \emph{norm on $\mathfrak{g}$ induced by the Killing form}\index{norm induced by the Killing form}. The restriction of $B_{\theta}$ to $\mathfrak{p}$ equals the Killing form of $\mathfrak{g}$ restricted to $\mathfrak{p}$.\iffalse (cf. \cite{E96}, Ch. 2).\fi

\subsubsection{Rank of Symmetric Spaces}
A totally geodesic submanifold of a globally symmetric space $X$ is necessarily itself a globally symmetric space. If $X$ is of the noncompact type, totally geodesic submanifolds have nonpositive curvature and thus don't have compact type factors (\cite{E96}, Ch. 2). We say that $X$ has \emph{rank} $k$\index{rank of a symmetric space} if it contains a flat totally geodesic submanifold of dimension $k$ and if every other flat totally geodesic submanifold has dimension $\leq k$. As $X$ contains geodesics, its rank is $\geq 1$. A symmetric space has rank one if and only if it has negative sectional curvature, that is its sectional curvature (as a function on the Grassmanian bundle of tangent $2$-planes of $X$) is everywhere negative.

As usual let $X=G/K$ and $\mathfrak{g}=\mathfrak{k}+\mathfrak{p}$ be a Cartan decomposition. The rank of $X$ may also be defined by the dimension of a maximal abelian subspace $\mathfrak{a}$\index{$\mathfrak{a}$, abelian Lie algebra} of $\mathfrak{p}$. It does neither depend on the choice of $\mathfrak{a}$ nor $\mathfrak{p}$ (\cite{E96}, p. 76). These definitions of rank are equivalent (\cite{E96}, 1.12.12 and (2.15.4)).

\subsubsection{Root space decompositions}
Fix a maximal abelian subspace $\mathfrak{a}$ of $\mathfrak{p}$. \iffalse By \cite{E96}, p. 77, \fi Let $\la^*$ be the real dual space of $\la$ and let $\la^*_{\cc}$ be its complexification. The operators $\ad X$ and $\ad Y$ commute in $\textnormal{End}(\mathfrak{g})$ for all $X,Y\in\mathfrak{a}$. Let $\mathfrak{a}^*$ be the real dual of $\mathfrak{a}$ and let $\alpha\in\mathfrak{a}^*$. Then define
\begin{eqnarray*}
\mathfrak{g}_{\alpha} = \left\{ X\in\mathfrak{g}: (\ad H)(X)=\alpha(H)X \textnormal{ for all } H\in\mathfrak{a}\right\}.
\end{eqnarray*}
\index{$\mathfrak{g}_{\alpha}$, root subspace} An element $0\neq\alpha\in\mathfrak{a}^*$ is called a restricted \emph{root}\index{root} if $\mathfrak{g}_{\alpha}\neq\left\{0\right\}$. It also follows that $\ad(\mathfrak{a})$ is a commuting family of linear transformations of $\mathfrak{g}$. We denote the set of roots determined by $\mathfrak{a}$ by $\Sigma$. Then $\Sigma\subset\mathfrak{g}^*$ is a nonempty finite set. We have the $B_{\theta}$-orthogonal direct sum decomposition
\begin{eqnarray*}
\mathfrak{g} = \mathfrak{g}_{0} + \sum_{\alpha\in\Sigma}\mathfrak{g}_{\alpha}
\end{eqnarray*}
(\cite{He01}, p. 263 or \cite{E96}, p. 78). This is called the \emph{root space decomposition}\index{root space decomposition} of $\mathfrak{g}$ determined by $\mathfrak{a}$. For $X\in\mathfrak{g}$ let $Z(X)$\index{$Z(X)$, centralizer of $X$} denote the centralizer of $X$ in $\mathfrak{g}$. An element $X\in\mathfrak{p}$ is called \emph{regular}\index{regular} if $Z(X)\cap\mathfrak{p}$ is a maximal abelian subspace of $\mathfrak{p}$. Otherwise $X$ is called \emph{singular}. \iffalse By \cite{E96}, p. 79,\fi An element $X\neq0$ is regular if and only if $\alpha(X)\neq0$ for every nonzero root $\alpha\in\Sigma$ that occurs in the root space decomposition of $\mathfrak{g}$ determined by $\mathfrak{a}$. Let $\mathfrak{a}'=R(\mathfrak{a})$ denote the set of regular elements. This set is the complement in $\mathfrak{a}$ of the union of the finite collection of hyperplanes
\begin{eqnarray*}
\mathfrak{a}_{\alpha}=\left\{X\in\mathfrak{a}:\alpha(X)=0\right\}, \,\,\,\,\, \alpha\in\Sigma.
\end{eqnarray*}
We write
\begin{eqnarray*}
H \sim H' \,\,\, \Longleftrightarrow \,\,\, \alpha(H)\alpha(H')>0 \,\, \forall \, \alpha\in\Sigma.
\end{eqnarray*}
This $\sim$ defines an equivalence relation in $R(\mathfrak{a})$. The corresponding equivalence classes are called \emph{Weyl chambers}\index{Weyl chambers}. We fix a Weyl chamber $\mathfrak{a}^+$ and call it the \emph{positive Weyl chamber}\index{positive Weyl chamber}. We call a root $\alpha$ \emph{positive}\index{positive root} and write $\alpha>0$ or $\alpha\in\Sigma^+$ if $\alpha$ has positive values on $\mathfrak{a}^+$. A root $\alpha$ is called \emph{simple}\index{simple root} if it is not the sum of two positive roots. \iffalse (see. \cite{He01}, Ch. IX).\fi Then $\mathfrak{a}^+$\index{$\mathfrak{a}^+$, positive Weyl chamber} is given by
\begin{eqnarray*}
\mathfrak{a}^+ = \left\{   H\in\mathfrak{a}: \alpha_1(H),...,\alpha_l(H)>0   \right\},
\end{eqnarray*}
where $\left\{\alpha_1,...,\alpha_l\right\}$ is the set of simple roots. The set of simple roots is a basis of $\mathfrak{a}^*$ (\cite{E96}, p. 81). Let the real dual space $\mathfrak{a}^*$ be ordered lexicographically with respect to this basis (\cite{He01}, p 173).

\subsubsection{The Weyl group}\label{The Weyl group}
Let $\langle\hspace{1mm},\hspace{1mm}\rangle$\index{$\langle\hspace{1mm},\hspace{1mm}\rangle$, Killing form} denote the Killing form. The Riesz representation theorem states that for each $\alpha\in\Sigma$ there is a unique \emph{root vector}\index{root vector} $H_{\alpha}\in\mathfrak{a}$\index{$H_{\alpha}\in\mathfrak{a}$, root vector} such that $\alpha(H)=\langle H,H_{\alpha}\rangle$ for all $H\in\mathfrak{a}$. Given a root $\alpha$, we consider the \emph{reflection}\index{$S_{\alpha}$, reflection in the hyperplane $\mathfrak{a}_{\alpha}$} $S_{\alpha}$ in the hyperplane $\mathfrak{a}_{\alpha}$ of $\mathfrak{a}$ that is orthogonal to $H_{\alpha}$ (the kernel of $\alpha$). This reflection is concretely given by the Householder transformation
\begin{eqnarray*}
S_{\alpha}(H) = H - 2\frac{\langle H_{\alpha}, H\rangle}{\langle H_{\alpha}, H_{\alpha}\rangle}H_{\alpha}.
\end{eqnarray*}
The $S_{\alpha}$ permute the root vectors (\cite{E96}, p. 81).

The \emph{Weyl group}\index{Weyl group} $W=W(\g,\la)$ is defined as the group $W$ of isometries of $\mathfrak{a}$ generated by the $S_{\alpha}$. We write\index{$Z_K(\la)$, centralizer of $\la$ in $K$}
\begin{eqnarray*}
Z_K(\mathfrak{a})=\left\{k\in K: \Ad(k)(H)=H \hspace{2mm}\forall H\in\mathfrak{a}\right\}
\end{eqnarray*}
for the centralizer of $\mathfrak{a}$ in $K$ and
\begin{eqnarray*}
N_K(\mathfrak{a})=\left\{k\in K: \Ad(k)\mathfrak{a}=\mathfrak{a}\right\}
\end{eqnarray*}
for the normalizer \index{$N_K(\la)$, normalizer of $\la$ in $K$} $\la$ in $K$. Then $N_K(\la)$ normalizes $Z_K(\la)$. Both groups are compact and have the same Lie algebra, namely
\begin{eqnarray*}
\mathfrak{m}= \mathfrak{g}_0 \cap \mathfrak{k} = \left\{ X\in\mathfrak{k}: [X,H]=0 \textnormal{ for all } H\in\mathfrak{a}\right\}.
\end{eqnarray*}
The restriction of the exponential map of $G$ to $\la$ is an analytic diffeomorphism onto the abelian subgroup $A:=\exp(\la)$. The inverse diffeomorphism is denoted by $\log$\index{$\log$, inverse of the exponential map}. We can also set $N_K(A)$ and $Z_K(A)$ for the normalizer and the centralizer of $A$ in $K$, respectively. The Weyl group $W$ is isomorphic to the finite group $N_K(A)/Z_K(A)$ (\cite{E96}, p. 82). Write
\begin{eqnarray*}
M:=\left\{k\in K: ka=ak \hspace{2mm}\forall a\in A\right\}
\end{eqnarray*}
for the centralizer\index{$M$, centralizer of $A$ in $K$} of $A$ in $K$ and
\begin{eqnarray*}
M':=\left\{k\in K: kAk^{-1}=A\right\}
\end{eqnarray*}
for the normalizer\index{$M'$, normalizer of $A$ in $K$} of $A$ in $K$. Then $W\cong M'/M\cong N(A)/Z(A)$, where $N(A)$ and $Z(A)$ denote the normalizer and the centralizer of $A$ in $G$, respectively. We always consider $W$ to be the group $W=M'/M$. The Weyl group is acting simply transitively on the collection of Weyl chambers of $\mathfrak{a}$ (\cite{E96}, p. 83). Its action extends to $\la^*$ by duality, to $A$ via the exponential map, and to the complexifications of $\la$ and $\la^*$ by complex linearity. The Weyl group permutes the root vectors and it permutes the root spaces.\index{$W$, Weyl group}

\subsubsection{Decomposition theorems}\label{decomposition theorems}
Let $G$ be a semisimple Lie group and carry over the algebraic data $\mathfrak{g}$, $\theta$, $\mathfrak{k}$, $K$, $\mathfrak{p}$, $\mathfrak{a}$, $A$, $\mathfrak{a}^+$ from the preceding paragraphs. Write $A^+:=\exp\mathfrak{a}^+$ and let $\overline{A^+}$ denote the closure of $A^+$ in $G$. The \emph{real rank}\index{real rank} of $G$ is the dimension $\dim\mathfrak{a}$ (it is independent on the choice of $\la\subset\lp$). We need the following decomposition theorems for $G$ (\cite{He01}, Ch. IX).

\begin{thm}[\emph{Cartan decomposition}]
Each $g\in G$ can be written $g=k_1ak_2$, where $k_1,k_2\in K$. The element $a=a^+(g)\in\overline{A^+}$ is uniquely determined by $g$. Thus $G=K\overline{A^+}K$.
\end{thm}

Recall that we denote the set of positive roots by $\Sigma^+$\index{$\Sigma^+$, set of positive roots}. Let
\begin{eqnarray*}
\mathfrak{n}=\sum_{\lambda\in\Sigma^+}\mathfrak{g}_{\lambda}.
\end{eqnarray*}
\index{$\mathfrak{n}$, nilpotent Lie algebra} Then $\lnn$ is a subalgebra of $\g$. Let $N$ be the corresponding connected subgroup of $G$. Then $\mathfrak{n}$ and $N$ are nilpotent (\cite{He01}, Ch. VI, Thm. 3.4, Ch. IX, Lemma 1.6) and $\mathfrak{a}+\mathfrak{n}$ is a solvable Lie algebra. Each element $a\in A$ normalizes $N$, that is we have $aN=Na$ for all $a\in A$. In particular, $AN=NA$ is a subgroup of $G$.

\begin{thm}[Iwasawa decomposition]
We have $\mathfrak{g}=\mathfrak{k}+\mathfrak{a}+\mathfrak{n}$ (direct vector space sum) and $G=KAN$. The mapping $(k,a,n)\rightarrow kan$ is a diffeomorphism of $K\times A\times N$ onto $G$.
\end{thm}

We fix some notation: If $g\in G$, we will always write
\begin{eqnarray*}
g=k(g)\exp H(g) \,  n(g),
\end{eqnarray*}
where $k(g)\in K$, $H(g)\in\mathfrak{a}$ and $n(g)\in N$. The corresponding projections onto the $K$, $\la$ and $N$ are called \emph{Iwasawa projections}\index{Iwasawa projections}\index{$H(g)$, Iwasawa projection $KAN\rightarrow\la$}. We can also decompose\index{$A(g)$, Iwasawa projection $NAK\rightarrow\la$}
\begin{eqnarray*}
g=n(g)\exp A(g) k(g)
\end{eqnarray*}
corresponding to $G=NAK$, where $A(g)\in\mathfrak{a}$. Clearly $A(g)=-H(g^{-1})$.

\begin{rem}
Each point $p\in X$ gives rise to another Cartan involution and another Cartan decomposition. Let $\theta_p$ be the Cartan involution and $\mathfrak{g}=\mathfrak{k}_p+\mathfrak{p}_p$ be the Cartan decomposition determined by $p\in X$. If $q\in X$ determines $\theta_q$ and $\mathfrak{g}=\mathfrak{k}_q+\mathfrak{p}_q$, then $\mathfrak{k}_q=g\cdot\mathfrak{k}_p$ and $\mathfrak{p}_q=g\cdot\mathfrak{p}_p$ whenever $g\cdot p=q$ (\cite{E96}, \S 2.3, \S 2.8, \cite{He01}, Ch. III, Thm. 7.2). It follows that all Cartan decompositions of $\mathfrak{g}$ are conjugate in $G$. By \cite{He01}, Ch. V, Lem. 6.3 (or \cite{E96}, \S 2.8), any two maximal abelian subspaces $\mathfrak{a}_1$ and $\mathfrak{a}_2$ of $\mathfrak{p}_p$ are conjugate by an element $k\in K$. Since the Weyl group acts simply transitively on the Weyl chambers (\cite{He01}, Ch. VII, Theorem 2.12) we deduce that for another choice $\mathfrak{a}_1$ resp. $A_1$ the corresponding Iwasawa decomposition components $AN$ and $A_1 N_1$ are conjugate by an element of $K$. It follows that all Iwasawa decompositions are conjugate in $G$.
\end{rem}

Note that $\Ad(m)$ ($m\in M$) leaves $\mathfrak{a}$ pointwise fixed, so it maps a root space $\mathfrak{\alpha}$ into itself. Hence $M$ normalizes $N$, so $MN=NM$ is a group. Then $P=MAN$ is a closed subgroup of $G$. For $s\in W=M'/M$ we fix a representative $m_s\in M'$.

\begin{thm}[Bruhat decomposition]
Let $G$ be any noncompact semisimple Lie group. Then $G$ decomposes into double cosets of $P=MAN$, that is
\begin{eqnarray*}
G = \bigcup_{s\in W}Pm_sP \hspace{3mm}(\textnormal{disjoint union}).
\end{eqnarray*}
\end{thm}

We can also write $S=\exp{\lp}$. Then (cf. \iffalse \cite{E96}, p. 117,\fi \cite{He01}, Ch. VI\iffalse , p. 253\fi)
\begin{thm}\label{K and S}
$G=K\cdot S=S\cdot K$. The indicated decomposition of an element of $G$ is unique. The mapping $(X,k)\mapsto (\exp X)k$ is a diffeomorphism of $\mathfrak{p}\times K$ onto $G$. Write $\pi:G\rightarrow G/K$. Then the mapping $\pi\circ\exp$ is a diffeomorphism of $\mathfrak{p}$ onto the globally symmetric space $X=G/K$.
\end{thm}

\begin{defn}\label{rho}
For $\alpha\in\Sigma^+$ we call $m_{\alpha}=\dim\mathfrak{g}_{\alpha}$ the \emph{multiplicity}\index{$m_{\alpha}$, multiplicity of a simple root $\alpha$} of $\alpha$. Once for all we define the parameter $\rho\in\mathfrak{a}^*$ by\index{$\rho=\frac{1}{2}\sum_{\alpha\in\Sigma^+}m_{\alpha}\alpha$, parameter}
\begin{eqnarray*}
\rho=\frac{1}{2}\sum_{\alpha\in\Sigma^+}m_{\alpha}\alpha.
\end{eqnarray*}
\end{defn}

We finish this subsection with a few remarks on the nilpotent subgroup $N$. Let $\cdot$ denote the adjoint action of $G$ on $\g$.
\begin{rem}\label{N on A}
\begin{itemize}
\item[(1)] Let $H\in\mathfrak{a}'$ (regular). The mapping $n\mapsto n\cdot H-H$ defines a diffeomorphism of $N$ onto $\mathfrak{n}$ (\cite{He01}, p. 403).
\item[(2)] Assume $H\in\mathfrak{a}'$ (i.e. $H$ is regular) such that $\alpha(H)>0$ for all $\alpha\in\Sigma^+$. Then (\cite{He01}, p. 278)
\begin{eqnarray*}
N=\left\{g\in G: \lim_{t\rightarrow\infty}\exp(-tH)g\exp tH=e\right\}.
\end{eqnarray*}
\item[(3)] For $X\in\lp$, let $Z_N(X)$ denote the cantralizer of $X$ in $N$\index{centralizer of $A$ in $N$}\index{$Z_N(X)$, centralizer of $X$ in $N$} and let $Z_{\mathfrak{n}}(X)=\left\{X\in\mathfrak{n}:[X,X]=0\right\}$ denote the centralizer of $X$ in $\mathfrak{n}$\index{centralizer of $X$ in $\mathfrak{n}$}. Let $H\in\la$. Then $Z_N(H)=\exp(Z_{\mathfrak{n}}(H))$. Each $X\in\mathfrak{n}$ is of the form $X=\sum_{\alpha\in\Sigma^+}X_{\alpha}$, where $X_{\alpha}\in\mathfrak{g}_{\alpha}$. By definition we thus have $[H,X]=\sum_{\alpha\in\Sigma^+}\alpha(H) X_{\alpha}$.
Now assume $[H,X]=0$. Then $\alpha(H) X_{\alpha}=0$ for all $\alpha$. Then $X=0$, hence $Z_{\mathfrak{n}}(H)=\left\{0\right\}$ and $Z_N(H) = \left\{ e \right\}$. In general, for $X\in\lp$ we have $Z_H(X)=\left\{e\right\}$ if and only if $X$ is regular.
\end{itemize}
\end{rem}

\subsubsection{Measure theoretic preliminaries}\label{Measure theoretic preliminaries}
We establish some conventions about the normalization of invariant measures on the groups and homogeneous spaces we work with. We follow the standard source \cite{He94}, Ch. II.

If $Y$ is any manifold we denote by $C(Y)$ the space of real- or complex-valued continuous functions on $Y$. By $C_c(Y)$ we denote the subspace of $C(Y)$ consisting of functions with compact support.

The Killing form induces Euclidean measures on $A$, its Lie algebra $\mathfrak{a}$ and the dual space $\mathfrak{a}^*$. If $l=\dim(A)$, we multiply these measures by the factor $(2\pi)^{-l/2}$ and thereby obtain invariant measures $da, dH$ and $d\lambda$ on $A,\mathfrak{a}$ and on $\mathfrak{a}^*$. This normalization has the advantadge that the Euclidean Fourier transform on $A$ is inverted without a multiplicative constant. We normalize the Haar measures $dk$ and $dm$ on the compact groups $K$ and $M$ such that the total measure is $1$.

In general, if $U$ is a Lie group and $P$ a closed subgroup, with left invariant measures $du$ and $dp$, the $U$-invariant measure $du_P=d(uP)$ on $U/P$ (when it exists) will be normalized by
\begin{eqnarray}\label{integral quotient}
\int_U f(u)du=\int_{U/P}\left(\int_P f(up)dp\right)du_P.
\end{eqnarray}
This measure exists if $U$ is unimodular and $P$ is a compact subgroup of $U$ (\cite{He00}, Ch. I, Thm. 1.9). In particular, we have a $K$-invariant measure $dk_M=d(kM)$ on $K/M$ of total measure $1$. We also use the notation
\begin{eqnarray}\label{invariant measures}
dx=dg_K=d(gK), \hspace{3mm} d\xi=dg_{MN}=d(gMN)
\end{eqnarray}
for the invariant measures on $X=G/K$ and $\Xi=G/MN$. By uniqueness, $dx$ is a constant multiple of the measure on $X$ induced by the Riemannian structure on $X$ given by the Killing form $B$.

The involutive automorphism $\theta$ of $\mathfrak{g}$ induces a unique (\cite{He01}, Ch. IV, Prop. 3.5) analytic involutive automorphism, also denoted by $\theta$, of $G$ whose differential at $e\in G$ is the original $\theta$. (\cite{He01}, Ch. VI, Thm. 1.1). It thus makes sense to define $\overline{N}=\theta N$. The mapping $(\overline{n},m,a,n)\mapsto \overline{n}man$ is a bijection of $\overline{N}\times M\times A\times N$ onto the open submanifold $\overline{N}MAN$ of $G$, whose complement is a null-set for the Haar measure of $G$ (\cite{He01}, Ch. IX, \S 1). In the Iwasawa decomposition notation, the mapping $\overline{N}\rightarrow K/M$, $\overline{n}\mapsto k(\overline{n})M$, is a diffeomorphism of $\overline{N}$ onto an open subset of $K/M$ whose complement is a null set for the invariant measure $d(kM)$ on $K/M$.

The Haar measures $dn$ and $d\overline{n}$ on the nilpotent groups $N$ and $\overline{N}$ can be normalized (\cite{He00}, Ch. IV, \S 6) such that
\begin{eqnarray*}
\theta(dn)=d\overline{n}, \,\,\,\,\,\,\,\, \int_{\overline{N}}e^{-2\rho(H(\overline{n}))}d\overline{n}=1.
\end{eqnarray*}

By loc. cit., Ch. I, \S 5, we can then normalize the Haar measure on $G$ such that for all $f\in C_c(G)$
\begin{eqnarray}\label{integral formula G}
\int_G f(g)dg &=& \int_{KAN}f(kan)e^{2\rho(\log a)}dkdadn \\
&=& \int_{NAK}f(nak)e^{-2\rho(\log a)}dndadk.
\end{eqnarray}

Recall that each $m\in M$ leaves $\mathfrak{a}$ pointwise fixed, so $m$ maps a root space into and onto itself. Hence $n\mapsto mnm^{-1}$ is an automorphism of $N$ mapping $dn$ into a multiple of $dn$. Since $M$ is compact, $dn$ is preserved. It follows that the product measure $dm dn$ is a bi-invariant measure on $MN=NM$. Let $m^*\in M'$ denote any representative of the the Weyl group element mapping the positive Weyl chamber $\mathfrak{a}^+$ onto $-\mathfrak{a}^+$. Then the mapping $n\mapsto (m^*)^{-1}nm^*$ is a diffeomorphism between $N$ and $\overline{N}=\theta(N)$ (\cite{He94}, p. 102).

We will also need the following integral formulas (\cite{He00}, Ch. I).
\begin{lem}
\begin{itemize}
\item[(1)] Let $f\in C_c(AN)$ and $a\in A$. Then
\begin{eqnarray}\label{integral AN NA}
\int_N f(na) \, dn = e^{2\rho(\log(a))}\int_N f(an) \, dn.
\end{eqnarray}
\item[(2)] Let $f\in C_c(G)$. Then
\begin{eqnarray}\label{integral formula G and ANK}
\int_G f(g) \, dg = \int_{KNA}f(kna) \, dk \, dn \, da = \int_{ANK}f(ank) \, da \, dn \, dk.
\end{eqnarray}
\item[(3)] Let $f\in C_c(X)$. Then
\begin{eqnarray}\label{integral formula AN and G/K}
\int_X f(x) dx = \int_{AN}f(an\cdot o) \, da \, dn.
\end{eqnarray}
\end{itemize}
\end{lem}

\subsubsection{Special functions and the Plancherel density}\label{Special functions and the Plancherel density}
Recall that we denote by $\Sigma^+$ the system of positive roots. The set of all (restricted) roots is the disjoint union of $\Sigma^+$ and $-\Sigma^+$. We write $\Sigma^-:=-\Sigma^+$. A root $\alpha\in\Sigma$ is called \emph{indivisible} if $\alpha/2\notin\Sigma$. For the sets of indivisible, respectively positive indivisible roots, we write $\Sigma_0$ and $\Sigma^+_0$, respectively. We can then define
\begin{eqnarray}
\Sigma^+_0:=\Sigma^+\cap\Sigma_0 \,\,\,\,\, \textnormal{ and } \,\,\,\,\, \Sigma^-_0:=\Sigma^-\cap\Sigma_0.
\end{eqnarray}

Also recall that the Cartan-Killing form $B(\cdot,\cdot)$ is positive definite on $\lp\times\lp$, so $\langle X,Y \rangle:= B(X,Y)$ defines a Euclidean structure in $\lp$ and in $\la\subset\lp$. Given $\gamma\in\la^*$, there is a unique $H_{\gamma}\in\la$\index{$H_{\gamma}$, representation in $\la$ of $\gamma\in\la^*$} such that $\gamma(H)=\langle H_{\gamma},H\rangle$ for all $H\in\la$. We can thus extend $\langle\cdot,\cdot\rangle$ to $\la^*$ by duality, that is we set $\langle\lambda,\mu\rangle=\langle H_{\lambda},H_{\mu}\rangle$ for $\lambda,\mu\in\la^*$. Finally we denote the $\cc$-bilinear extension of $\langle\cdot,\cdot\rangle$ to $\la^*_{\cc}$ by the same symbol. Given $\alpha\in\Sigma$ and $\lambda\in\la^*_{\cc}$ we write\index{$\lambda_{\alpha}$, normalization of $\lambda\in\la^*_{\cc}$}
\begin{eqnarray}
\lambda_{\alpha} := \frac{\langle\lambda,\alpha\rangle}{\langle\alpha,\alpha\rangle}.
\end{eqnarray}
Let $\Gamma$ denote the classical $\Gamma$-function. Here and in the following we adopt the convention that $m_{2\alpha}=0$ if $2\alpha$ is not a root. Harish-Chandra's $c$-function is the meromorphic function on $\la^*_{\cc}$ given by the Gindikin-Karpelevich product formula
\begin{eqnarray}\label{c-function}
c(\lambda) = c_0 \prod_{\alpha\in\Sigma_0^+} c_{\alpha}(\lambda)
\end{eqnarray}
where
\begin{eqnarray}
c_{\alpha}(\lambda) = \frac{2^{-i\lambda_{\alpha}}\Gamma(i\lambda_{\alpha})}{\Gamma(\frac{i\lambda_{\alpha}}{2}+\frac{m_{\alpha}}{4}+\frac{1}{2}) \, \Gamma(\frac{i\lambda_{\alpha}}{2}+\frac{m_{\alpha}}{4}+\frac{m_{2\alpha}}{2})},
\end{eqnarray}
and where the constant $c_0$ is defined by $c(-i\rho)=1$. Note that the function
\begin{eqnarray}
|c(\lambda)|^{-2}=c(\lambda)c(-\lambda)=c(s\lambda)c(-s\lambda) \,\,\, \forall \, s\in W
\end{eqnarray}
is Weyl group invariant (\cite{He00}, p. 451). The singularities of the \emph{Plancherel density}
\begin{eqnarray}
\frac{1}{c(\lambda)c(-\lambda)} = \frac{1}{c_0^2}\prod_{\alpha\in\Sigma^+_0}\frac{1}{c_{\alpha}(\lambda)c_{\alpha}(-\lambda)}
\end{eqnarray}
can be explicitly written down. We recall some formulas given in \cite{HP}. Note that if both $\alpha$ and $2\alpha$ are roots, then $m_{\alpha}$ is even and $m_{2\alpha}$ is odd (\cite{He01}, p. 530). For $\alpha\in\Sigma^+_0$, the singularities of
\begin{eqnarray}
\frac{1}{c_{\alpha}(\lambda)c_{\alpha}(-\lambda)}
\end{eqnarray}
are described by distinguishing the following four cases:
\begin{itemize}
\item[(a)] $m_{\alpha}$ even, $m_{2\alpha}=0$,
\item[(b)] $m_{\alpha}$ odd, $m_{2\alpha}=0$,
\item[(c)] $m_{\alpha}/2$ even, $m_{2\alpha}$ odd,
\item[(d)] $m_{\alpha}/2$ odd, $m_{2\alpha}$ odd.
\end{itemize}
It follows from simple identities for the $\Gamma$-function that
\begin{eqnarray}
\frac{1}{c_{\alpha}(\lambda)c_{\alpha}(-\lambda)} = C_{\alpha} \lambda_{\alpha} p_{\alpha}(\lambda) q_{\alpha}(\lambda),
\end{eqnarray}
where $C_{\alpha}$ is a positive constant depending on $\alpha$ and on the multiplicities, where $p_{\alpha}$ is a polynomial, and where $q_{\alpha}$ is a function. We make the convention that a product taken over the empty set is equal to one. Then the explicit expressions for $p_{\alpha}$ and $q_{\alpha}$ in the four cases listed above are (\cite{HP}, p. 501)
\begin{itemize}
\item[(a)] $p_{\alpha}(\lambda) = \lambda_{\alpha} \prod_{k=1}^{m_{\alpha}/2-1}(\lambda_{\alpha}^2 + k^2)$, \\
$q_{\alpha}(\lambda)=1$,
\item[(b)] $p_{\alpha}(\lambda) = \prod_{k=0}^{(m_{\alpha}-3)/2}[\lambda_{\alpha}^2 + (k+\frac{1}{2})^2]$, \\
$q_{\alpha}(\lambda)=\tanh(\pi\lambda_{\alpha})$,
\item[(c)] $p_{\alpha}(\lambda) = \prod_{k=0}^{m_{\alpha}/4-1}[(\lambda_{\alpha}/2)^2 + (k+\frac{1}{2})^2] \cdot \prod_{l=0}^{m_{\alpha}/4+(m_{2\alpha}-1)/2-1}[(\lambda_{\alpha}/2)^2 + (l+\frac{1}{2})^2]$, \\
$q_{\alpha}(\lambda)=\tanh(\pi\lambda_{\alpha}/2)$,
\item[(d)] $p_{\alpha}(\lambda) = \prod_{k=0}^{(m_{\alpha}-2)/4}[(\lambda_{\alpha}/2)^2 + k^2] \cdot \prod_{l=1}^{(m_{\alpha}+2m_{2\alpha})/4-1}[(\lambda_{\alpha}/2)^2 + l^2]$, \\
$q_{\alpha}(\lambda)=\coth(\pi\lambda_{\alpha}/2)$,
\end{itemize}
Note that in each of the above cases the degree of the polynomial $\lambda_{\alpha} p_{\alpha}(\lambda)$ equals $m_{\alpha}$, and hence the dimension of the root subspace $\g_{\alpha}$. Given $\lambda\in\la^*_+$ we sometimes write $\lambda\rightarrow\infty$ and mean that $\lambda(H)\rightarrow\infty$ for all $H\in\la_+$. Recall that $\tanh\sim 1$ and $\coth\sim 1$ to all orders. Hence if asymptotics $\lambda\rightarrow\infty$ are involved, we can replace the factor $q_{\alpha}(\lambda)$ by $1$, and the Plancherel density is asymptotically a polynomial of degree $\dim(N)$.

For any (restricted) root $\alpha$ we can also write $\alpha_0:=\alpha/\langle\alpha,\alpha\rangle$. We will later need \emph{Harish-Chandra's} $e$\emph{-functions}\index{$e_s(\lambda)$, Harish-chandra's $e$-functions} (\cite{He94}, p. 163)
\begin{eqnarray}\label{e function}
e_s(\lambda) = \prod_{\alpha\in\Sigma_s^+} \Gamma\left(\frac{m_{\alpha}}{4}+\frac{1}{2}+\frac{\langle i\lambda,\alpha_0\rangle}{2}\right) \Gamma\left(\frac{m_{\alpha}}{4}+\frac{m_{2\alpha}}{2}+\frac{\langle i\lambda,\alpha_0\rangle}{2}\right),
\end{eqnarray}
where $s\in W$ and $\Sigma_s^+=\Sigma_0^+\cap s^{-1}\Sigma_0^-$\index{$\Sigma_s^+=\Sigma_0^+\cap s^{-1}\Sigma_0^-$, partial positive roots}.

\subsection{Geodesics, horocycles, and the boundary at infinity}\label{geodesics}
Let $X$ be a symmetric space of the noncompact type, hence $X=G/K$, where $G$ is a noncompact connected semisimple Lie group with finite center and where $K$ is a maximal compact subgroup of $G$. We carry over the notations from the previous section.\iffalse, that is the algebraic data $\left\{G,\mathfrak{g},\theta,\mathfrak{k},\mathfrak{p},\mathfrak{a},\mathfrak{n},\mathfrak{m},K,A,N,M,M'\right\}$.\fi The \emph{origin}\index{origin} of $X$ is the identity coset $o:=K\in G/K$. A basic remark which follows from Theorem \ref{K and S} is that the geodesics through the origin are (\iffalse \cite{He01}, p. 226,\fi\cite{E96}, p. 74) the curves
\begin{eqnarray}
\gamma_X : t\mapsto e^{tX}\cdot o, \,\,\,\,\,\,\, (X\in\mathfrak{p})
\end{eqnarray}
As $X$ is a simply connected manifold of nonpositive sectional curvature, for each points $p\neq q$ in $X$ there exists a unique unit speed geodesic $\gamma_{p,q}:\rr\rightarrow X$ with $\gamma_{p,q}(0)=p$ and $\gamma_{p,q}(a)=q$, where $d(p,q)=a$, and where $d$ denotes the distance function on $X$ (loc. cit, p. 20).

\begin{defn}
Two unit speed geodesics $\gamma$ and $\sigma$ of $X$ are \emph{asymptotes} or \emph{asymptotically equivalent}\index{asymptotically equivalent geodesics} if there exists $C\geq0$ such that the $d(\gamma(t),\sigma(t))\leq C$ for all $t\geq0$. Two unit vectors $v,w\in SX$ are said to be \emph{asymptotes} or \emph{asymptotically equivalent} if the corresponding geodesics $\gamma_v$ resp. $\gamma_w$ with initial velocity $v$ and $w$ have this property.
\end{defn}

The asymptote relation is an equivalence relation on the unit speed geodesics of $X$ and on the unit vectors of $SX$.

\begin{defn}
A \emph{point at infinity}\index{point at infinity} for $X$ is an equivalence class of asymptotic geodesics of $X$ (\cite{E96}, p. 27). The set of all points at infinity for $X$ is denoted by $X(\infty)$\index{$X(\infty)$, set of points at infinity}. The equivalence class represented by a geodesic $\gamma$ is denoted by $\gamma(\infty)$\index{$\gamma(\infty)$, equivalence class of a geodesic} and the equivalence class represented by the oppositely oriented geodesic $\gamma^{-1}: t\mapsto\gamma(-t)$ is denoted by $\gamma(-\infty)$\index{$\gamma(-\infty)$, equivalence class of a geodesic}.
\end{defn}

If $\gamma$ is any geodesic of the complete, simply connected space $X$ with nonpositive curvature, then for each $p\in X$ there exists a unique geodesic $\sigma$ of $X$ such that $\sigma(0)=p$ and $\sigma$ is asymptotic to $\gamma$ (\cite{E96}, p. 28).

\begin{defn}
We say that points $x\neq y$ in $X(\infty)$ can be \emph{joined by a geodesic}\index{joined by a geodesic} of $X$ if there exists a geodesic $\gamma$ of $X$ with $\gamma(\infty)=x$ and $\gamma(-\infty)=y$. The geodesic $\gamma$ is said to \emph{join} $x$ and $y$.
\end{defn}

Throughout this work we will mainly be interested in points at infinity that can be joined by a geodesic. We first recall a basis result (\cite{EO}, Proposition 4.4\iffalse See also \cite{E96}, p. 150\fi):

\begin{thm}\label{joining}
Let $X$ have rank one. The sectional curvature of $X$ is strictly negative. Any two distinct points $x,y\in X(\infty)$ can be joined by a geodesic of $X$.
\end{thm}

To motivate this setting, we will now describe the geometry of a rank one space in detail. The group theoretical aspects can then be generalized to higher rank spaces.

\subsubsection{The boundary at infinity}\label{boundary}
Let $X=G/K$ have rank one. We call $B=X(\infty)$ the \emph{boundary at infinity}\index{$B$, boundary at infinity}. For $X\in \mathfrak{p}$ let $\gamma_X=e^{tX}\cdot o$\index{$\gamma_X$, geodesic} denote the geodesic through the origin $o\in X$ with inital direction $X$. We introduce an action of $G$ on $B$. For $b=\lim_{t\rightarrow \infty}\gamma_{X}(t)$ and $g\in G$, define
\begin{eqnarray*}
g\cdot b := g\cdot\lim_{t\rightarrow \infty}\gamma_{X}(t)=\lim_{t\rightarrow \infty}\gamma_{g\cdot X}(t)\in B.
\end{eqnarray*}
(Here, $g\cdot X$ denotes the adjoint action.) Since $G/K$ has rank one, we define once and for all $H$ to be the unique unit vector (w.r.t. the norm induced by the Killing form) in the positive Weyl chamber $\mathfrak{a}^+$. We write $S(\mathfrak{p})$ for the unit sphere of $\mathfrak{p}$. Let $b_{\infty}\in B$\index{$b_{\infty}$, boundary point} denote the boundary point $\lim_{t\rightarrow\infty}\gamma_H(t)$. Let $b_{-\infty}\in B$\index{$b_{-\infty}$, boundary point} denote the boundary point $\lim_{t\rightarrow -\infty}\gamma_H(t)=\lim_{t\rightarrow \infty}\gamma_{-H}(t)$.

The only orthogonal transformations of the one-dimensional space $\mathfrak{a}$ are $\pm\id$. It follows that (in the rank one case) the Weyl group has exactly two elements. Let $w\in M'$ denote any representative of the nontrivial Weyl group element. The adjoint action of $w$ on $\mathfrak{a}$ is $-id$, so $w\cdot H=-H$. It follows that $w\cdot b_{\infty}=b_{-\infty}$ and vice versa.

For $b\in B$ there exists $X\in S(\mathfrak{p})$ such that $b=\gamma_X(\infty)$ for $\gamma_{X}(t)=e^{tX}\cdot o$. Since $K$ acts transitively on $S(\mathfrak{p})$, there is $k\in K$ such that $k\cdot H=X$. Hence $k\cdot b=b_{\infty}$. In particular $K$ acts transitively on $B$. The stabilizer of $b_{\infty}$ is by definition the stabilzer $M$ of $H$. The action of $K$ on $B$ is continuous (\cite{E96}, Ch. 3) and since it is transitive, $B$ is compact. Hence under the mapping\index{$A_H$, homeomorphism}
\begin{eqnarray}\label{identification}
A_H: K/M\rightarrow B, \hspace{1mm} kM\mapsto\lim_{t\rightarrow\infty}\gamma_{k\cdot H}(t),
\end{eqnarray}
$B$ is in a natural way homeomorphic to the compact space $K/M$. We make $B$ a smooth manifold by giving it the differentiable structure that makes $A_H$ a diffeomorphism (\cite{E96}, Ch. 3.8). The natural Lie topology of $K/M$ agrees with the compact open topology of the homeomorphism group of $B$, so $B=K/M$ as homogeneous spaces.

\subsubsection{The real flag manifold}\label{The real flag manifold}
We drop the rank one assumption and let $X=G/K$ be a general symmetric space of the noncompact type. Each $g\in G$ can be written $g=k(g)a(g)n(g)$ corresponding to $G=KAN$. We introduce the map
\begin{eqnarray}\label{G on K}
G\times K\rightarrow K, \hspace{3mm} (g,k)\mapsto g\cdot k := T_g(k) := k(gk)
\end{eqnarray}
Then $T_g$\index{$T_g$, group action} is a group action of $G$ on $K$. In particular, $T_g$ is inverted by $T_{g^{-1}}$ and defines a diffeomorphism of $K$ onto itself. This can easily be verified using the Iwasawa decomposition. For $g\in G$, $k\in K$ and $m\in M$ we clearly have $k(gkm)=k(gk)m$, since $m$ normalizes $N$ and centralizes $A$. Hence $k\mapsto k(gk)$ is right-$M$-equivariant, so \eqref{G on K} descends to an action of $G$ on the quotient $K/M$. We write $\overline{T}_g: K/M \rightarrow K/M, \,\, kM\mapsto k(gk)M$ for this action\index{$\overline{T}_g$, group action}.

Let $man\in P=MAN$. Then $man\cdot M = k(man)M = M$. Thus $MAN$ is the centralizer in $G$ of $M\in K/M$. The group $G$ acts naturally (by left-translations) on $G/P$. The mapping $\phi: K/M\rightarrow G/P, \,\, kM\mapsto kP$, is a bijection of $K/M$ onto $G/P$ which is regular at the origin, hence everywhere, so it is a diffeomorphism (\cite{He01}, p. 407). The identification $\phi: K/M\rightarrow G/P$ intertwines the actions of $G$ on $K/M$ and the natural group action of $G$ on $G/P$:
\begin{eqnarray*}
\phi(g\cdot kM) = \phi(k(gk)M) = k(gk)MAN = gkMAN = g\cdot\phi(kM).
\end{eqnarray*}
The spaces $K/M$ and $G/P$ are thus equivalent from this group theoretical point of view. We will write $B:=K/M=G/P$. We also recall the following useful lemma (\cite{He01}, p. 407):
\begin{lem}
The mapping $\overline{n}\mapsto k(\overline{n})M$ is a diffeomorphism of $\overline{N}$ onto an open submanifold of $K/M$ whose complement consists of finitely many disjoint manifolds of lower dimension.
\end{lem}

\begin{rem}
A \emph{Hadamard manifold} is a simply connected complete Riemannian manifold of nonpositive curvature and arbitrary dimension. We say that a Hadamard manifold $X$ satisfies the \emph{visibility axiom}\index{visibility axiom}, if any two points of the geodesic boundary (\cite{EO}) can be joined by a geodesic $X$. A Hadamard manifold may or may not satisfy the visibility axiom. The extreme cases are as follows:
\begin{itemize}
\item[(a)] The sectional curvature is zero. Then asymptoticity of geodesics coincides with ordinary parallelism, hence the visibility axiom is not satisfied.
\item[(b)] The sectional curvature is negative and bounded away from zero. In this case the behaviour of geodesic rays is qualitatively the same as in hyperbolic geometry, the visibility axiom is satisfied, and the geodesic joining two given boundary points is unique (\cite{EO}, Cor. 5.2).
\end{itemize}
A special class of Hadamard manifolds consists of Riemannian symmetric spaces of the noncompact type. If the symmetric space has rank one, then its sectional curvature is bounded between two negative constants (and thus the space falls into category (b) from above), so the visibility axiom is satisfied. On the other hand, higher rank spaces are characterized by the existence of totally geodesic flat subspaces, in which the visibility axiom fails, and hence it fails the ambient space as well (\cite{Hof}).

The description of the geodesic boundary of a higher rank space $X=G/K$ differs from the rank one case. For details we refer to \cite{E96}. If $X\cup X(\infty)$ is given the so-called \emph{cone topology} (loc. cit., p. 28), then isometries and geodesic symmetries of $X$ extend to the boundary $X(\infty)$ (loc. cit. p. 30).
\end{rem}

\begin{rem}
Given a boundary point $x\in X(\infty)$, let $G_x\subset G$ denote its stabilizer. Then $G_x$ acts transitively on $X=G/K$ (loc. cit., p.101). Suppose that another point $y$ at infinity can be joined with $x$ by a geodesic. Then the set of points to which $x$ can be joined is the orbit $G_x(y)$ (loc. cit., p. 151). If $X$ has rank one, then $G_x$ acts transitively on $X(\infty)\setminus\left\{x\right\}$. This fails whenever the rank of $X$ is $\geq 2$.
\end{rem}

Irrespective from the geometric point of view, many group theoretical aspects generalize to the higher rank case. We take the preceding remark as a motivation.

\begin{defn}
A subgroup of $P^*$ of $G$ is \emph{parabolic}\index{parabolic} if there exists a point $b\in B$ such that $P^*=G_b=\left\{ g\in G: gb=b\right\}$ is the stabilizer of $b$ in $G$.
\end{defn}

\begin{rem}
\begin{itemize}
\item[(1)] Our definition of a parabolic subgroup follows \cite{E96} and does only consider the \emph{minimal parabolic subgroups} of $G$.
\item[(2)] Unlike the subgroups of $G$ that fix a point in $X$, the parabolic subgroups are noncompact.
\item[(3)] The parabolic subgroup fixing $b=M\ni K/M$ is $P=MAN$ ($M\in K/M$ corresponds to $P\in G/P$).
\item[(4)] Let $b=hP\in G/P$ ($h\in G$). Then $g\cdot b = b \Leftrightarrow g\in hPh^{-1}$, so all parabolic subgroups of $G$ are conjugate to each other.
\item[(5)] $AN$ acts transitively on $X$, so the same holds for $P=MAN$. It follows that all parabolic subgroups act transitively on $X$.
\end{itemize}
\end{rem}

\subsubsection{The rank one case}
Let $X=G/K$ have rank one. The Weyl group $W=M'/M$ has exactly two elements. Let $w\in M'$ denote any representative of the nontrivial Weyl group element. As before, let $H$ denote the unit vector in $\la^+$. We also write\index{$a_t$, one-dimensional parameterization of $A$}
\begin{eqnarray}
a_t:=\exp(tH)\in A.
\end{eqnarray}
We consider the geodesic $t\mapsto a_t\cdot o$. Its forward limit point is $b_{\infty}$ and it identifies with $M\in K/M$ (that is $P\in G/P$). Its backward limit point $b_{-\infty}$ identifies with $wM\in K/M$ (that is $wP\in G/P$).

Since $wM\neq M$ in $K/M$, the geodesic $t\mapsto a_t\cdot o$ is the unique (up to parameter translation and time reversal) geodesic of $X$ that joins the boundary points $M\in K/M$ and $wM\in K/M$ at infinity (\cite{Q}).

We consider the homogeneous space $G/M$. The group $M$ is the stabilizer in $K$ of the unit vector tangent at $o$ to the geodesic $t\mapsto a_t\cdot o$. As $K$ acts transitively on the set of unit vectors in $T_o X\cong \lp$, the unit tangent bundle of $X$ identifies $G$-equivariantly with $G/M$ and the geodesic flow reads as the
action of $A$ by right translations on $G/M$.

\begin{lem}
Let $b\in B$. Then $G_b$ acts transitively on $B\setminus\left\{b\right\}$. In particular, $P$ acts transitively on $B\setminus\left\{b_{\infty}\right\}$.
\end{lem}
\begin{proof}
Since all parabolic subgroups are conjugate, it suffices to prove the assertion for $G_{b_{\infty}}=P$. Recall the Bruhat decomposition
\begin{eqnarray*}
G = P \cup PwP \,\,\,\,\,\,\,\, \textnormal{(disjoint union)},
\end{eqnarray*}
Let $b\in B\setminus\left\{b_{\infty}\right\}$ and select $g\in G$ such that $b=g\cdot b_{\infty}$. Note that $p\cdot b_{\infty}=b_{\infty}$ for each $p\in P$. Thus $g=p_1wp_2$ ($p_1,p_2\in P$). Hence $b=p_1wP=p_1\cdot b_{-\infty}$, which shows that $b\in P\cdot b_{-\infty}$, as desired.
\end{proof}

\begin{defn}
Let $\Delta=\left\{(b,b)\in B\times B\right\}$ denote the diagonal of $B\times B$. Let $B^{(2)}:=(B\times B)\setminus\Delta$ denote the set of distinct boundary points.
\end{defn}
We study the space of geodesics and the geodesic connections in the rank one case and describe the map that assigns to a geodesic its forward and backward limit points. We consider the diagonal action of $G$ on $B^{(2)}$ given by
\begin{eqnarray}\label{G on B^{(2)}}
G \times B^{(2)} \rightarrow B^{(2)}, \hspace{3mm} g\cdot (b_1,b_2) = (g\cdot b_1, g\cdot b_2).
\end{eqnarray}
Note that $g\cdot b_1 = g\cdot b_2$ implies $b_1 = b_2$, so \eqref{G on B^{(2)}} is well-defined.

\begin{lem}\label{space of geodesics}
$G$ acts transitively on $B^{(2)}$. The stabilizer of $(b_{\infty},b_{-\infty})$ is $MA$. In particular, $B^{(2)}=G/MA$ as a homogeneous space.
\end{lem}
\begin{proof}
Let $b_1 \neq b_2$ be points in $B$. Since $K$ acts transitive on $B$, we find $k\in K$ such that $k\cdot b_1 = b_{\infty}$. Since $P$ acts transitively on $B\setminus\left\{b_{\infty}\right\}$, we find $p\in P$ such that $p\cdot k\cdot b_2=b_{-\infty}$. Let $g=pk$. Then
$g\cdot (b_1,b_2) = (b_{\infty},b_{-\infty})$, so the group action is transitive.

It remains to show that $g\cdot(b_{\infty},b_{-\infty})=(b_{\infty},b_{-\infty}) \Leftrightarrow g\in MA$. Clearly an element $ma\in MA$ fixes both $M\in K/M$ and $wM\in K/M$, since $M'$ normalizes both $A$ and $M$.

Conversely assume that $g\cdot(b_{\infty},b_{-\infty})=(b_{\infty},b_{-\infty})$. Then $g\cdot b_{\infty}=b_{\infty}$, hence $g = man \in MAN$. It suffices to prove that $n=e$. By the assumption we have $n\in G_{b_{\infty}}\cap G_{b_{-\infty}} = MAN \cap wMANw^{-1}\subset\theta(N)$. Hence $n\in N\cap\theta N=\left\{e\right\}$. (Recall that $\mathfrak{g}$ is the direct vector space sum of the root-subspaces $\mathfrak{g}_{\alpha}$.)
\end{proof}

\begin{rem}
\begin{itemize}
\item[(1)] The unit tangent bundle $SX\cong G/M$ identifies with the set of pointed oriented complete geodesics of $X$.
\item[(2)] $B^{(2)} \cong G/MA$ is the set of oriented geodesics up to parameter translation. We can also write $SX\cong B^{(2)} \times \rr$.
\item[(3)] One could also prove Lemma \ref{space of geodesics} by using that the flats $nA\cdot o$ and $A\cdot o$ (\cite{He94}) coincide if and only if $n=e$.
\item[(4)] Lemma \ref{space of geodesics} is false for $G/K$ of rank $\geq 2$. This follows from the Bruhat decomposition, too. We will later see which subspace of $B\times B$ identifies with the homogeneous space $G/MA$.
\end{itemize}
\end{rem}

We can now give group-theoretical proof of Theorem \ref{joining}. See also \cite{Q}.
\begin{thm}\label{rank one connections}
Each geodesic $\sigma$ of $G/K$ has two distinct limit points in $B$. For $(b_1,b_2)\in B^{(2)}$ there exists up to parameter translation a unique geodesic $\sigma$ with limit points $b_1$ and $b_2$. For $(x,b)\in X\times B$ there is a unique geodesic through $x$ with limit point $b$.
\end{thm}
\begin{proof}
The first point is true for the geodesic $t\mapsto a_t\cdot o$ and therefore for a general geodesic as $G$ acts transitively on the set of geodesics, since it acts transitively on $X$ and $K$ acts transitively on the unit sphere of $T_o X$. The third point is true for $x=o$ as $K$ acts transitively on $B$ and therefore for any $x$, as for every $b\in B$ its stabilizer $G_b$ in $G$ acts transitively on $X$. The second point is true for $b_{\infty}$ and $b_{-\infty}$, hence by transitivity of $G$ on $B^{(2)}$ for all pairs of limit points.
\end{proof}

\begin{rem}
$B\setminus\left\{b_{\infty}\right\} \cong N$ as homogeneous spaces. In fact, the action of $N$ on $B\setminus \left\{b_{\infty}\right\}$ is already transitive, since the action of $P=MAN$ is and $MA$ fixes $b_{-\infty}$. It follows from \ref{N on A} that the stabilizer in $N$ of $b_{-\infty}$ is $\left\{e\right\}$.
\end{rem}

\subsubsection{The general case}\label{The general case}
We drop the rank one assumption and let $X=G/K$ be a general symmetric space of the noncompact type. Consider the diagonal action of $G$ on $G/K \times G/P$ given by
\begin{eqnarray}\label{G on X and B}
\gamma\cdot(gK,hP):=(\gamma gK,\gamma hP),\hspace{2mm}\gamma,g,h\in G.
\end{eqnarray}
Note that in the customary sense, $P=MAN$ is still a \emph{minimal parabolic subgroup} of $G$ (we do not describe this concept here). The action \eqref{G on X and B} yields the useful identification $G/M\cong X\times B$ (as homogeneous spaces). To describe this identification, we use simple Iwasawa decomposition arguments. First, let $b_0$ denote the identity coset of $K/M\cong G/P$. For $\gamma\in G$ we observe $\gamma\cdot(o,b_0)=(o,b_0) \Leftrightarrow \gamma\in K\cap P=M$. It follows that $M\subset G$ is the stabilizer of $(o,b_0)\in X\times B$. For a proof of $X\times B\cong G/M$ it remains to show that the diagonal action of $G$ on $X\times B$ is transitive. We say that cosets $\gamma P\in G/P$ and $hK\in G/K$ are \emph{incident}\index{incident}, if as subsets of $G$ they are not disjoint.
\begin{lem}\label{incident}
Let $g\in G$. Then $gK\in G/K$ and $P\in G/P$ are incident. Let $h\in G$. Then $gK$ and $hP\in G/P$ are incident.
\end{lem}
\begin{proof}
Write $g=nak$. Then $gK=naK\subset G$ contains $p=na\in MAN=P$. For general $hP\in G/P$ select $p\in h^{-1}gK \cap P$. Then $p=h^{-1} gk$ for some $k\in K$, so $hP \ni hp = gk \in gK$.
\end{proof}
\begin{cor}
$G$ acts transitively on $G/K\times G/P$.
\end{cor}
\begin{proof}
First, given $(gK,hP)\in G/K\times G/P$, we apply Lemma \ref{incident} and write $gk=hp$, where $k\in K$ and $p\in P$. Then $gk\cdot(o,b_0)=(gk\cdot o,gk\cdot b_0)=(g\cdot o, h\cdot b_0)$.
\end{proof}
\begin{cor}\label{KAN on X and B}
Each element $(gK,kM)\in G/K\times K/M$ can be written in the form $(kanK,kM)$. If $(z,b)\in X\times B$, then there is $g\in G$ such that $g\cdot(o,b_0)=(z,b)$. The element $g\in G$ is uniquely determined modulo $M$.
\end{cor}

If $H\in\la$, the geodesic $t\mapsto\exp(tH)\cdot o$ in $X$ is said to be \emph{regular} if the vector $H$ is regular. A general geodesic $\gamma$ in $X$ is said to be \emph{regular} if its stabilizer $\left\{g\in G: g\cdot\gamma=\gamma\right\}$ in $G$ has minimum dimension (\cite{He94}, p.82). A \emph{flat} in $X$ is a totally geodesic submanifold of $X$ whose curvature tensor vanishes identically. The \emph{maximal flats} in $X$ are of the form $gA\cdot o$ ($g\in G$) (\cite{He01}, Ch. V, \S6).

Recall the Bruhat decomposition
\begin{eqnarray*}
G = \bigcup_{s\in W}Pm_sP \hspace{3mm}(\textnormal{disjoint union}),
\end{eqnarray*}
where for $s\in W$ (Weyl group) we picked a representative $m_s\in M'$. Exactly one of the above sets $Pm_sP$ is open and dense in $G$, namely $PwP$, where $w$ is the longest Weyl group element. The other summands have lower dimension. Recall $\overline{N}=wNw^{-1}$ (conjugation by $w$ is not necessarily $\theta_{|N}$, the restriction of $\theta$ to $N$). It follows that the manifold $\overline{N}MAN$ is open and dense in $G$. Thus the space of flats can be naturally identified with $G/MA$, or a dense open subset of $G/P \times G/\overline{P}$, where $\overline{P}:=MA\overline{N}$, via the $G$-equivariant map
\begin{eqnarray*}
G/MA \ni gMA \mapsto (gP,gwPw^{-1}) \in G/P \times G/\overline{P}.
\end{eqnarray*}
We also consider the $G$-equivariant map
\begin{eqnarray*}
G/MA \ni gMA \mapsto (gP,gwP) \in G/P \times G/P = B\times B.
\end{eqnarray*}
It follows from the Bruhat decomposition that its image is an open and dense subset of $G/P\times G/P = B\times B$, namely $\left\{(gP,hP)\in G/P\times G/P: h^{-1}g\in PwP\right\}$. This open and dense subset of $B\times B$ is the $G$-orbit of $(P,wP)$ in $B\times B$. We will from now on write $B^{(2)}:\cong G/MA$ for this $G$-orbit. If $X$ has rank one, then $B^{(2)}=(B\times B)\setminus\Delta$, where $\Delta$ denotes the diagonal of $B\times B$.

\subsubsection{The space of horocycles}
\begin{defn}
A \emph{horocycle}\index{horocycle} $\xi$ in $X$ is any orbit $\xi=N'\cdot x$, where $x\in X$ and $N'=g^{-1}Ng$ is a subgroup of $G$ conjugate to $N$. In particular, we define $\xi_0$ to be the horocycle $N \cdot o$.
\end{defn}
The choice of Iwasawa-decomposition is immaterial since all such decompositions are conjugate (\cite{E96}, p.105). We note that  each horocycle is a closed submanifold of $X$. The group $G$ acts transitively on the set of horocycles. The subgroup of $G$ which maps the horocycle $\xi_0$ into itself equals $MN$ (\cite{He94}, Ch. II, \S1).

The set of horocycles in $X$ with the differentiable structure of $G/MN$ is called the \emph{dual space}\index{dual space} of $X$ and will be denoted by $\Xi$. We write $\Xi=G/MN$. Then each $\xi\in\Xi$ can be written in the form $\xi=gMN$, where $g\in G$. Decompose $g=kan$ corresponding to the Iwasawa decomposition. Then $\xi=kanMN=kaMN$, since $M$ normalizes $N$. Let $h$ be another representative of $\xi$, that is $hMN=gMN$, so $h=gmn'$, since $M$ normalizes $N$. Then $\xi=hMN=kanmn'MN=kaMN=kmaMN$. It follows that each horocycle $\xi\in \Xi$ can be written in the form $kaMN$, where $kM\in K/M$ and $a\in A$ are unique.

\begin{defn}
If $\xi=kaMN$ is any horocycle, then $b=kM\in B=K/M$ is said to be \emph{normal}\index{normal} to $\xi$.
\end{defn}

\begin{lem}
Each horocycle $\xi=gNg^{-1}\cdot x$ ($g\in G, x\in X$) can be written in the form $\xi=ka\cdot\xi_0$, where $kM\in K/M$ and $a\in A$ are unique.
\end{lem}
\begin{proof}
Write $g=kan$ and $g^{-1}\cdot x = \tilde{n}\tilde{a} \cdot o$ corresponding to the Iwasawa decomposition. Since $A$ normalizes $N$ we obtain
$\xi = gNg^{-1}\cdot x = kaN\tilde{n}\tilde{a}K = ka_1NK = ka_1\cdot \xi_0$. The uniqueness follows from the fact that $MN$ is the stabilizer of the horocycle $\xi_0$.
\end{proof}

\begin{defn}
Let $\xi = kaMN \in \Xi$ be any horocycle. We call $\log(a)$ the \emph{composite distance}\index{composite distance} from $o$ to $\xi$. In general, for $x=g_1 K\in X$ and $\xi=g_2 MN \in \Xi$ we call $\left\langle x,\xi \right\rangle = H(g_1^{-1}g_2)$ the \emph{composite distance}\index{$\left\langle x,\xi \right\rangle$, composite distance from a point $x$ to a horocycle $\xi$} from $x$ to $\xi$.
\end{defn}

Recall that $H:G\rightarrow\mathfrak{a}$ is left-$K$-invariant and right-$MN$-invariant, so $\left\langle x,\xi \right\rangle$ is well-defined and invariant under the natural diagonal action of $G$ on the product space $X\times \Xi\cong G/K\times G/MN$. We also state the following uniqueness result (\cite{He94}, p. 81).

\begin{lem}\label{lemma xi}
Given $x\in X$, $b\in B$, there exists a unique horocycle passing through $x$ with normal $b$. For $x=gK \in G/K$ and $b=kM \in K/M$,
\begin{eqnarray}\label{xi}
\xi = \xi(x,b)=k\exp(-H(g^{-1}k))\xi_0
\end{eqnarray}
is the unique horocycle in question.
\end{lem}
\iffalse\begin{proof} The proof is quite ugly:
\begin{eqnarray*}
x\in k\exp(-H(g^{-1}k))\xi_0 &\Longleftrightarrow& gK \in k\exp(-H(g^{-1}k))\xi_0 \\
&\Longleftrightarrow& gK \in k\exp(-H(g^{-1}k)NK \\
&\Longleftrightarrow& \exists n_1\in N:                     gK \in k\exp(-H(g^{-1}k))n_1K \\
&\Longleftrightarrow& \exists n_1\in N \exists k_1\in K:    gk_1 \in k\exp(-H(g^{-1}k))n_1 \\
&\Longleftrightarrow& \exists n_1\in N \exists k_1\in K:    k^{-1}gk_1 \in \exp(-H(g^{-1}k))n_1 \\
&\Longleftrightarrow& \exists n_1\in N \exists k_2\in K:    k^{-1}g \in \exp(-H(g^{-1}k))n_1k_2 \\
&\Longleftrightarrow& \exists n_2\in N \exists k_2\in K:    k^{-1}g \in n_2\exp(-H(g^{-1}k))k_2 \\
&\Longleftrightarrow& \exists n_2\in N \exists k_2\in K:    k^{-1}g \in n_2\exp(A(k^{-1}g))k_2 \\
&\Longleftrightarrow& k^{-1}g \in N\exp(A(k^{-1}g))K.
\end{eqnarray*}
The last statement is true because of the Iwasawa decomposition.
\end{proof}\fi

\iffalse
\begin{rem}\label{invariant measure on G/MN}
Each horocycle $\xi\in G/MN$, as a submanifold of $X$, inherits a Riemannian structure from $X$ and hence a measure. The mapping $n\mapsto n\cdot o$ is a diffeomorphism of $N$ onto the standard horocycle $\xi_0=N\cdot o$. We carry the measure $dn$ over to $\xi_0$. This gives a measure $ds$ on $\xi_0$, which is invariant under the action of $MN$. Thus we get by translation by elements of $G$ a well-defined measure $ds$ on each horocycle $\xi$. By uniqueness this is a multiple of the Riemannian measure on each horocycle $\xi$ by a constant factor not depending on $\xi$.
\end{rem}
\fi

\subsubsection{Horocycles brackets and the Iwasawa-projection}\label{Horocycles brackets and the Iwasawa-projection}
For $x\in X$ and $b\in B$ let $\xi(x,b)$ denote the unique horocycle passing through the point $x\in X$ with normal $b\in B=K/M$. We denote by
$\left\langle x,\xi \right\rangle\in\la$ the \emph{composite distance} from the origin $o$ to the horocycle $\xi(x,b)$. This vector-valued inner product has a simple expression in terms of the Iwasawa decomposition $G=KAN=NAK$. Therefore recall the projections $H:KAN\rightarrow\la$ and $A:G=NAK\rightarrow\la$. In view of \eqref{xi} we define $A:X\times B \rightarrow \mathfrak{a}$ via
\begin{eqnarray*}
(x,b)\mapsto A(x,b)= \left\langle x,b\right\rangle = \left\langle gK,kM\right\rangle := A(k^{-1}g) = -H(g^{-1}k).
\end{eqnarray*}
We will mostly use the notation $\langle\,,\,\rangle$ for this inner product and call it the \emph{horocycle bracket}\index{horocycle bracket}. Sometimes, when this horocycle bracket is needed in one equation with the Killing form, we use the notation $A(x,b)$, which is also used in \cite{He94}. We clearly have
\begin{lem}
$\left\langle \cdot ,\cdot \right\rangle$ is invariant under the diagonal action of $K$ on $X\times B$.
\end{lem}

Recall that $g\in G$ acts on $K$ by $g\cdot k=k(gk)$, where $k:G=KAN\rightarrow K$ denotes the Iwasawa projection. By the right-$M$-equivariance of this projection the action descends to an action of $G$ on $K/M$.

\begin{lem}\label{invariance0}
Let $g_1,g_2\in G$, $k\in K$. Then $H(g_1 g_2 k)=H(g_1 k(g_2 k)) + H(g_2 k)$.
\end{lem}
\begin{proof}
Decompose $g_2k=\tilde{k}\tilde{a}\tilde{n}$ and $g_1\tilde{k}=k'a'n'$. Then $$H(g_1g_2k) = H(k'a'n'\tilde{a}\tilde{n}) = H(a'n'\tilde{a}).$$ Since $A$ normalizes $N$ this equals $\log(a') + \log(\tilde{a})$.
\end{proof}

\begin{lem}
Let $x=hK\in G/K$, $b=kM\in K/M$, $g\in G$. Then
\begin{eqnarray}\label{equivariance}
\left\langle g\cdot x,g\cdot b \right\rangle = \left\langle x,b \right\rangle + \left\langle g\cdot o,g\cdot b \right\rangle.
\end{eqnarray}
\end{lem}
\begin{proof}
By definition, $\left\langle g\cdot x,g\cdot b \right\rangle = -H(h^{-1}g^{-1}k(gk))$. Then by Lemma \ref{invariance0} with $g_1=h^{-1}g^{-1}$ and $g_2=g$ this equals
\begin{eqnarray*}
-H(h^{-1}g^{-1}gk) + H(gk) = -H(h^{-1}k) + H(gk).
\end{eqnarray*}
Also $\langle g\cdot o, g\cdot b\rangle = -H(k) + H(gk) = H(gk)$ as above for $h=e$. Hence
\begin{eqnarray*}
\left\langle g\cdot x,g\cdot b \right\rangle - \left\langle g\cdot o,g\cdot b \right\rangle &=& [-H(h^{-1}k) + H(gk)] - [-H(k) + H(gk)],
\end{eqnarray*}
and the right hand side equals $-H(h^{-1}k) = \langle hK, kM\rangle = \langle x,b\rangle$.
\end{proof}

\begin{cor}\label{cor need}
$\left\langle g^{-1}\cdot o, b \right\rangle = - \left\langle g\cdot o, g\cdot b \right\rangle$.
\end{cor}
\begin{proof} $0=\left\langle o, g\cdot b \right\rangle = \left\langle g\cdot g^{-1}\cdot o, g\cdot b \right\rangle = \left\langle g^{-1}\cdot o, b \right\rangle + \left\langle g\cdot o, g\cdot b \right\rangle$, since the distance to the origin of a horocycle passing through the origin is $0$.
\end{proof}

We go on using the Iwasawa decomposition and easily derive
\begin{lem}\label{formulae for bracket}
\begin{itemize}
\item[(i)] $\langle g^{-1}z,M\rangle = \langle z,g\cdot M\rangle - \langle g\cdot o, g\cdot M\rangle$,
\item[(ii)] $\langle g^{-1}z,g^{-1}b\rangle = \langle z,b\rangle - \langle g\cdot o,b\rangle$.
\end{itemize}
\end{lem}

\begin{lem}\label{Iwasawa and bracket}
Let $g\in G$. Then $\langle g\cdot o,g\cdot M\rangle = H(g)$.
\end{lem}
\begin{proof}
Write $g=kan$ corresponding to the Iwasawa decomposition. Then $g^{-1}k=n^{-1}a^{-1}=a^{-1}\tilde{n}$, so $\langle kan\cdot o,kan\cdot M\rangle = -H(g^{-1}k) = \log(a) = H(g)$.
\end{proof}
Note that one could also prove \eqref{equivariance} using Lemma \ref{Iwasawa and bracket}. We will need some more component computations for later reference. Under $X\times B\cong G/M$, each $(z,b)\in X\times B$ can be written $(g\cdot o,g\cdot M)$. Then $\langle z,b\rangle=H(g)$ follows from Lemma \eqref{Iwasawa and bracket}. We go on using the Iwasawa decomposition and easily derive

\begin{cor}\label{corollary bracket}
Given $z,w\in X$, $b,b'\in B$, let $(z,b)\in X\times B$ correspond to $gM\in G/M$ and let $(w,b')=(hK,h\cdot M)\in X\times B$ correspond to $hM\in G/M$, respectiveley. Then
\begin{itemize}
\item[(1)] $\langle z,b\rangle = H(g)$,
\item[(2)] $\langle z,b'\rangle = -H(g^{-1}k(h)) = -H(g^{-1}h) + H(h)$,
\item[(3)] $\langle w,b\rangle = -H(h^{-1}k(g)) = -H(h^{-1}g) + H(g)$,
\item[(4)] $\langle w,b'\rangle = H(h)$.
\end{itemize}
\end{cor}

\subsection{Invariant differential operators}\label{Invariant differential operators}
We recall the theory of invariant differential operators to put results concerning the Laplacian of a symmetric space into a general context. We will need to recall relations between invariant differential operators and invariant polynomials for the Weyl group. We recall the definition of the Laplace-Beltrami operator and give the explicit and important formula \eqref{character of the Laplacian} for the so-called complete symbol of this invariant differential operator. The material is taken mostly from \cite{He00}.

If $V$ is an open subset of $\rr^n$ we let $\mathcal{E}(V)=C^{\infty}(V)$\index{$\mathcal{E}(V)$, smooth functions on an open set $V$} denote the set of smooth functions on $V$ and $\mathcal{D}(V)$\index{$\mathcal{D}(V)$, compactly supported smooth functions on $V$ an open set $V$} denote the space of functions in $\mathcal{E}(V)$ with compact support contained in $V$. Let $\partial_j$ denote partial differentiation with respect to $x_j$, where $x=(x_1,\ldots,x_n)\in\rr^n$. If $\alpha=(\alpha_1,...,\alpha_n)\in\nn_0^n$, put
\begin{eqnarray}
D^{\alpha} &=& \partial_1^{\alpha_1}\cdots\partial_n^{\alpha_n}, \hspace{20mm} x^{\alpha}=x_1^{\alpha_1}\cdots x_n^{\alpha_n},\\
|\alpha| &=& \alpha_1 + \cdots + \alpha_n, \hspace{15mm} \alpha! = \alpha_1!\cdots\alpha_n!.
\end{eqnarray}
If $S$ is any subset of the open set $V$ and $m\in\nn_0$ we put
\begin{eqnarray}\label{seminorms}
\|f\|^S_m = \sum_{|\alpha|\leq m}\sup_{x\in S}|D^{\alpha}f(x)|.
\end{eqnarray}
A \emph{differential operator}\index{differential operator on $\rr^n$} on $V$ is a linear mapping $D:\mathcal{D}(V)\rightarrow\mathcal{D}(V)$ such that for each relatively compact open set $U\subset V$ such that $\overline{U}\subset V$ (closure in $\rr^n$), there exists a finite famliy of functions $a_{\alpha}\in\mathcal{E}(U)$, $\alpha\in\nn_0^n$, such that
\begin{eqnarray}
D\phi = \sum_{\alpha}a_{\alpha}D^{\alpha}\phi, \,\,\,\,\, \phi\in\mathcal{D}(U).
\end{eqnarray}
Differential operators decrease supports\index{support}:
\begin{eqnarray}\label{decrease supports}
\supp(D\phi)\subset\supp(\phi).
\end{eqnarray}
Conversely, Peetre's theorem states that any linear mapping $D:\mathcal{D}(V)\rightarrow\mathcal{D}(V)$ decreasing supports is a differential operator (\cite{He00}, p. 236).\\

Let $M$ be a manifold. A \emph{differential operator}\index{differential operator on a manifold} $D$ on $M$ is a linear mapping of $C_c^{\infty}(M)$ into itself which decreases supports:
\begin{eqnarray*}
\supp(Df) \subset \supp(f), \hspace{4mm} f\in C_c^{\infty}(M).
\end{eqnarray*}
The definition of a differential operator extends naturally to $C^{\infty}(M)$ if one puts $(Df)(x)=(D\phi)(x)$, where $\phi\in C_c^{\infty}$ equals $f\in C^{\infty}$ in a neighborhood of $x\in M$.

To describe the function and distribution spaces we work with, we follow \cite{He00}, Ch. II. Let $M$ satisfy the second axiom of countability\index{second countable manifold}, that is the topology of $M$ admits a countable base for the open sets. If $(U,\phi)$ is a local coordinate system on $M$, the mapping
\begin{eqnarray*}
D^{\phi} : F \mapsto (D(F\circ\phi))\circ\phi^{-1}, \hspace{4mm} F\in C_c^{\infty}(\phi(U)),
\end{eqnarray*}
is support-decreasing. It follows that for each open relatively compact set $W$ such that $\overline{W}\subset U$ there are finitely many $a_\alpha\in C^{\infty}(W)$ such that
\begin{eqnarray*}
Df = \sum_{\alpha}a_{\alpha}(D^{\alpha}(f\circ\phi^{-1}))\circ\phi, \hspace{4mm} f\in C_c^{\infty}(W).
\end{eqnarray*}
Just as for open sets in $\rr^n$ the definition of differential operators extends to $C^{\infty}(M)$. We write 
\begin{eqnarray*}
\mathcal{D}(M)=C_c^{\infty}(M) \,\,\,\,\, \textnormal{ and } \,\,\,\,\, \mathcal{E}(M)=C^{\infty}(M).
\end{eqnarray*}\index{$\mathcal{D}(M)=C_c^{\infty}(M)$, space of compactly supported smooth function on $M$}\index{$\mathcal{E}(M)=C^{\infty}(M)$, space of smooth functions on $M$}If $K$ is a compact subset of $M$, we denote by $\mathcal{D}_K(M)$\index{$\mathcal{D}_K(M)$, space of functions supported with compact support in $K$} the subset of functions in $\mathcal{D}(M)$ with support in $K$.

For an open set $V$ of $\rr^n$ the spaces $\mathcal{E}(V)$ are topologized by the seminorms $\left\|f\right\|^C_m$, as $C$ runs through the compact subsets of $V$ and $k$ runs through $\nn_0$. If $(U, \phi)$ runs through all local coordinate systems on $M$, this gives a topology of $\mathcal{E}(U)$ with the property that a sequence $f_n$ in $\mathcal{E}(U)$ converges to $0$ if and only if for each differential operator $D$ on $U$, the sequence $Df_n\rightarrow0$ uniformly on each compact subset of $U$. It follows that the topology of $\mathcal{E}(U)$ is independent of the coordinate system.

The space $\mathcal{E}(M)$ is provided with the weakest topology for which the restrictions $f\mapsto f_{|U}$, when $(U,\phi)$ runs through the local coordinate systems of $M$, are continuous. By the countability assumption, we may restrict the $(U,\phi)$ to a countable family of charts $(U_j,\phi_j)$. It follows that $\mathcal{E}(M)$ is a Fr\'{e}chet space and again the topology is described by uniform convergence (of all derivatives) on compact subsets. Since $M$ is the union of an increasing sequence of compact subsets, this implies that $\mathcal{D}(M)$ is dense in $\mathcal{E}(M)$.

When $K$ is a compact subset of $M$, the space $\mathcal{D}_K$ is given the topology induced by $\mathcal{E}(M)$. As a closed subspace of $\mathcal{E}(M)$ it is a Frechet space.

A linear functional $T$ on $\mathcal{D}(M)$ is called a \emph{distribution}\index{distribution} if for any compact subset $K\subset M$ the restriction of $T$ on $\mathcal{D}_K(M)$ is continuous. The set of distributions is denoted by $\mathcal{D}'(M)$. We often write $\int_M f(m)dT(m)$ instead of $T(f)$.

The space $\mathcal{D}(M)$ is given the \emph{inductive limit topology}\index{inductive limit topology} of the spaces $\mathcal{D}_K(M)$ by taking as a fundamental system of neighborhoods of $0$ the convex sets $W$ such that for each compact subset $K\subset M$ the space set $W\cap\mathcal{D}_K(M)$ is a neighborhood of $0$ in $\mathcal{D}_K(M)$. It follows that $\mathcal{D}'(M)$ is the dual space of $\mathcal{D}(M)$.

A distribution $T$ is said to \emph{vanish} on an open set $V\subset M$ if $T(f)=0$ for all $f\in\mathcal{D}(V)$. The \emph{support}\index{support of a distribution} of $T$ is the complement of the largest open subset of $M$ on which $T$ vanishes. Let $\mathcal{E}'(M)$ denote the set of distributions with compact support. The restriction of a functional from $\mathcal{E}(M)$ to $\mathcal{D}(M)$ identifies the dual of $\mathcal{E}(M)$ with $\mathcal{E}'(M)$ (cf. \cite{He00}, p. 240).

If $N$ is another manifold and $\phi$ is a diffeomorphism of $M$ and $N$ and if $f\in\mathcal{D}(N)$, $g\in\mathcal{E}(N)$, $T\in\mathcal{D}'(M)$, $D\in E(M)$, we write
\begin{eqnarray*}
g^{\phi^{-1}}=g\circ\phi,\hspace{4mm}T^{\phi}=T(f^{\phi^{-1}}),\hspace{4mm}D^{\phi}(g)=(D(g^{\phi^{-1}}))^{\phi}.
\end{eqnarray*}
If $\phi$ is a diffeomorphism of $M$ onto itself, then $D$ is said to be \emph{invariant under}\index{invariant under a diffeomorphism}\index{invariant differential operator} $\phi$, if $D^{\phi}=D$, that is
\begin{eqnarray*}
Dg = (D(g\circ\phi))\circ\phi^{-1}\textnormal{ for all } g\in\mathcal{E}(M).
\end{eqnarray*}

Given a measure $\mu$ on $M$, the space $\mathcal{E}(M)$ is imbedded in $\mathcal{D}'(M)$ associating with $f\in\mathcal{D}(M)$ the distribution
\begin{eqnarray}\label{imbedding of functions}
f \mapsto I_f:= \left( g\mapsto \int_M f \, g \, d\mu\right)
\end{eqnarray}
on $M$. We call this the \emph{canonical imbedding}\index{imbedding of functions into distributions} of functions into distributions.

\subsubsection{The Laplace-Beltrami operator}\label{The Laplacian}
Let $M$ be a pseudo-Riemannian manifold with pseudo-Riemannian structure $g$ and let $\phi:q\mapsto(x_1(q),...,x_n(q))$ be a coordinate system valid on an open subset $U\subset M$. As customary we define the functions $g_{ij}$, $g^{ij}$ and $\overline{g}$ on $U$ by
\begin{eqnarray}
g_{ij} = g(\frac{\partial}{\partial x_i},\frac{\partial}{\partial x_j}), \,\,\,\, \sum_j g_{ij} g^{jk} = \delta_{ik}, \,\,\,\, \overline{g}=|\det(g_{ij})|.
\end{eqnarray}
In this section we write $\left\langle\hspace{1mm}|\hspace{1mm}\right\rangle$ in place of $g$ and extend it $\cc$-bilinearly to complex vector fields. Each $f\in C^{\infty}(M)$ gives rise to the vector field $\grad f$ (gradient of $f$)\index{gradient} defined by
\begin{eqnarray}
\left\langle\grad f|X\right\rangle = Xf
\end{eqnarray}
for each vector field $X$.

On the other hand, if $X$ is a vector field on $M$, the \emph{divergence}\index{divergence} of $X$ is the function on $M$ which on $U$ is given by
\begin{eqnarray}
\diverg(X) = \frac{1}{\sqrt{\overline{g}}}\sum_i \partial_i(\sqrt{\overline{g}}X_i),
\end{eqnarray}
if $X=\sum_i X_i (\partial/{\partial x_i})$ on $U$. Then $\diverg(X)$ is well-defined (\cite{He00}, p.243) and independent of the coordinate system.

The \emph{Laplace-Beltrami operator}\index{Laplace-Beltrami operator} on $M$ is defined by
\begin{eqnarray}
Lf = \diverg \grad f, \,\,\,\,\,\,\,\,\, f\in\mathcal{E}(M).
\end{eqnarray}
In terms of local coordinates one has (loc. cit., p.245)
\begin{eqnarray}
Lf = \frac{1}{\sqrt{\overline{g}}}\sum_k \partial_k\left(\sum_i g^{ik}\sqrt{\overline{g}}\partial_i f\right),
\end{eqnarray}
so $L$ is a differential operator on $M$. The Laplace-Beltrami operator $L$ of a pseudo-Riemannian manifold $M$ is symmetric:
\begin{eqnarray}
\int_M u(x)(Lv)(x)dx = \int_M (Lu)(x)v(x)dx, \,\,\,\,\,\, u\in \mathcal{D}(M), \,\,\, v\in\mathcal{E}(M),
\end{eqnarray}
where $dx$ is the Riemannian measure on $M$. If $\Phi$ is a diffeomorphism of $M$, then $\Phi$ leaves the Laplace-Beltrami operator invariant if and only if it is an isometry.

Let $M$ be an $m$-dimensional Riemannian manifold and $p$ a point in $M$. Given normal coordinates $(x_1,\ldots,x_m)$ around $p$ such that $(\partial/{\partial x_i)_p}$ ($1\leq j\leq m$) is an orthonormal basis of the tangent space at $p$, then the Laplace-Beltrami operator $L$ of $M$ is given at $p$ by (\cite{He01}, p. 330)
\begin{eqnarray}
(Lf)(p) = \sum_i \frac{\partial^2 f}{\partial x_i^2} (p), \,\,\,\,\,\, f\in\mathcal{E}(M).
\end{eqnarray}

Suppose that $M$ is a compact Riemannian manifold of dimension $m\geq 2$. Let $d$ denote the \emph{distance function}\index{distance function} on $M$ and write
\begin{eqnarray}
( f_1|f_2 ) = \int_M f_1(x) \, \overline{f_2(x)} \, dx, \,\,\,\,\,\,\,\,\,\,\, f_1, f_2 \in L^2(M),
\end{eqnarray}
for the customary $L^2$-product of $M$. Given $\lambda\in\cc$, define the \emph{eigenspace}\index{eigenspace} $\mathcal{E}_{\lambda}$ by
\begin{eqnarray}
\mathcal{E}_{\lambda} = \left\{ u\in\mathcal{E}(M): Lu=\lambda u \right\}
\end{eqnarray}
and $\Lambda$ the \emph{spectrum}\index{spectrum}
\begin{eqnarray}
\Lambda = \left\{ \lambda\in\cc: \mathcal{E}_{\lambda}\neq0 \right\}.
\end{eqnarray}
Then (\cite{War}, Chapter 6)
\begin{itemize}
\item[(a)] $\Lambda$ is a discrete subset of $\cc$ and $\lambda\leq0$ for each $\lambda\in\Lambda$.
\item[(b)] Each eigenspace $\mathcal{E}_{\lambda}$ is finite-dimensional: $\dim\mathcal{E}_{\lambda}<\infty$ for each $\lambda$.
\item[(c)] In accordance with (a) and (b), let $\phi_0,\phi_1,\phi_2,\ldots$ be an orthonormal system in $L^2(M)$ such that each $\mathcal{E}_{\lambda}$ is spanned by some of the $\phi_i$. Then, if $f\in L^2(M)$,
\begin{eqnarray}
f = \sum_0^{\infty} \langle f,\phi_n\rangle\phi_n,
\end{eqnarray}
where the sum converges in $L^2(M)$.
\item[(d)] If $f\in\mathcal{E}(M)$, the expansion in (c) converges absolutely and uniformly.
\end{itemize}

\subsubsection{Harish-Chandra's isomorphism and radial parts}\label{Harish-Chandra's isomorphism}
Suppose $H$ is a closed subgroup of $G$ with Lie algebra $\lh$. Let $\dd(G/H)$\index{$\dd(G/H)$, algebra of translation-invariant differential operators on $G/H$} be the algebra of differential operators on $G/H$ which are invariant under the translations $\tau(g): xH \mapsto gxH$ ($g\in G$) of $G/H$ onto itself\index{$\tau(g)$, translation on $G/H$}. We write $\dd(G)$\index{$\dd(G)$, algebra of translation-invariant differential operators on $G$} instead of $\dd(G/\left\{e\right\})$. For $g\in G$, let $\rho_g$\index{$\rho_g$, right-translation by $g$ in $G$} denote the \emph{right-translation} by $g$ in $G$. Then define
\begin{eqnarray}
\dd_H(G) = \left\{ D\in \dd(G): D^{\rho_h}=D \textnormal{ for all } h\in H  \right\}.
\end{eqnarray}
Write $\pi: G\rightarrow G/H$. If $f$ is a function on $G/H$, we put $\widetilde{f}=f\circ\pi$, so that $\widetilde{f}$ is a function on $G$. Given $u\in\dd_K(G)$ and $f\in\mathcal{E}(G/K)$, let $D_u\in\dd(G/K)$ denote the operator defined by $(D_u f)^{\sim} = u\widetilde{f}$. Then we have (\cite{He00}, p. 285):

\begin{thm}\label{mu}
The mapping $\mu: u\mapsto D_u$ is a homomorphism of $\dd_K(G)$ onto $\dd(G/K)$. The kernel of $\mu$ is $\dd_K(G)\cap\dd(G)\lk$.
\end{thm}

Recall the Iwasawa decomposition $G=KAN$. Let $\dd(A)$\index{$\dd(A)$, algebra of translation-invariant differential operators on $A$} denote the algebra of translation-invariant differential operators (with constant coefficients) on $A$ and let $\dd_W(A)\subset\dd(A)$\index{$\dd_W(A)$, Weyl group invariant differential operators on $A$} denote the subalgebra consisting of $W$-invariant differential operators on $A$. If $D\in\dd(G)$, there is (\cite{He00}, p. 302) a unique element $D_{\mathfrak{a}}\in \dd(A)$\index{$D_{\mathfrak{a}}$, projection of $D\in\dd(G)$ onto $\dd(A)$} such that
\begin{eqnarray}\label{Da radial part}
D-D_{\mathfrak{a}} \in \mathfrak{n}\hspace{0.1mm}\dd(G) + \dd(G)\mathfrak{k}.
\end{eqnarray}
If $\nu$ is a linear function on $\la$ we denote by $e^{\nu}: A\rightarrow\cc$ the function $a\mapsto e^{\nu(\log(a))}$. Let $\circ$ denote the composition of differential operators. The mapping 
\begin{eqnarray*}
\gamma: D\mapsto e^{-\rho} D_{\mathfrak{a}}\circ e^{\rho}
\end{eqnarray*}
is a surjective homomorphism of $\dd_K(G)$ onto $\dd_W(A)$ with kernel $\dd_K(G)\cap\dd(G)\lk$ (\cite{He00}, 304). The next theorem (\cite{He00}, 306) involves \emph{Harish-Chandra's isomorphism}\index{$\Gamma$, Harish-Chandra's isomorphism} $\Gamma$:
\begin{thm}\label{Gamma isomorphism}
Let $\mu$ denote the isomorphism from Theorem \ref{mu}. Consider the diagram
\[
\begin{xy}
  \xymatrix{
  &    \dd_K(G) \ar[ld]_{\mu} \ar[rd]^{\gamma}   \\ 
  \dd(G/K) \ar[rr]^{\Gamma}    &       &  \dd_W(A).
  }
\end{xy}
\]
Then $\gamma$ factors through $\mu$ to yield an isomorphism $\Gamma$ of algebras
\begin{eqnarray}\label{canonical isomorphism}
\Gamma: \dd(G/K)\rightarrow \dd_W(A),
\end{eqnarray}
given by $\Gamma(\mu(D))=\gamma(D)$ for $D\in D_K(G)$.
\end{thm}

When the nilpotent subgroup $N$ of $G$ acts on the symmetric space $G/K$, the orbits are transversal (in the sense of \cite{He00}, Ch. II, \S3 (29)) to the submanifold $A\cdot o$ (\cite{He00}, p. 266). Thus if $D$ is a differential operator on $X$, it follows from \cite{He01}, Ch. II, Theorem 3.6 that there is a uniquely determined differential operator $\Delta_N(D)$\index{$\Delta_N$, radial part} on $A\cdot o$ such that for each $N$-invariant function on $G/K$
\begin{eqnarray}\label{Radial part 1}
(Df)(a\cdot o) = (\Delta_N(D)f_{|A\cdot o})(a\cdot o),
\end{eqnarray}
where $f_{|A\cdot o}$ denotes the restriction of $f$ to $A\cdot o$. The operator $\Delta_N(D)$ is called the \emph{radial part} of $D$\index{radial part of a differential operator}. The isomorphism \eqref{canonical isomorphism} is then given by (\cite{He00}, p. 306)
\begin{eqnarray}\label{Radial Part 2}
\Gamma(D) = e^{-\rho}\Delta_N(D)\circ e^{\rho}.
\end{eqnarray}
In particular (loc. cit.), for the Laplacian $L_X$ on $X=G/K$ we have
\begin{eqnarray}\label{symbol Laplace 1}
\Gamma(L_X) = L_A - \langle\rho,\rho\rangle,
\end{eqnarray}
where $L_A$\index{$L_A$, Laplace operator on $A\cdot o$} denotes the Laplace operator of the submanifold $A\cdot o$ of $G/K$.

\subsubsection{Joint eigenfunctions and joint eigenspaces}\label{Joint eigenfunctions and joint eigenspaces}
If $V$ is a finite-dimensional vector space over $\rr$, the \emph{symmetric algebra}\index{symmetric algebra} $S(V)$ over $V$ is defined as the algebra of complex-valued polynomial functions on the dual space $V^*$ (\cite{He00}, p. 280). If $X_1,...,X_n$ is a basis of $V$, then $S(V)$ can be identified with the commutative algebra of polynomials
\begin{eqnarray}
\sum_{k}a_{k_1...k_n}X_1^{k_1}...X_n^{k_n}.
\end{eqnarray}
Let $U$ be any Lie group with Lie algebra $\mathfrak{u}$. Consider the exponential mapping $\exp:\mathfrak{u}\rightarrow U$, which maps a line $\rr X$ through $0$ in $\mathfrak{u}$ onto the one-parameter subgroup\index{one-parameter subgroup} $t\mapsto \exp(tX)$ of $U$. As usual, if $X\in\mathfrak{u}$, let $\widetilde{X}\in\dd(U)$ (\cite{He00}, p.280) denote the vector field of $U$ given by
\begin{eqnarray}
(\widetilde{X}f)(g) = X(f\circ \lambda_g) = \left(\frac{d}{dt}f(g\exp tX)\right)_{t=0}, \hspace{6mm} f\in\mathcal{E}(G).
\end{eqnarray}
%The bracket operation of $\mathfrak{g}$ is given by $[X,Y]^{\sim}=[\widetilde{X},\widetilde{Y}]$, when $X,Y\in \mathfrak{g}$.
The relation between $S(\mathfrak{u})$ and $\dd(U)$ is as follows (\cite{He01}, p. 281):
\begin{thm}\label{symmetrization}
There exists a unique linear bijection $\lambda: S(\mathfrak{u})\rightarrow\dd(U)$\index{$\lambda$, symmetrization map} such that $\lambda(X^m)=\widetilde{X}^m$ for all $X\in \mathfrak{u}$ and all $m\in\nn_0$.
%If $X_1,...X_n$ is any basis of $\mathfrak{u}$ and $P\in S(\mathfrak{u})$, then
%\begin{eqnarray}
%(\lambda(P)f)(g) = \left\{ P(\partial_1,...,\partial_n)f(g\exp(t_1X_1+...+t_nX_n))    \right\}_{t=0},
%\end{eqnarray}
%where $f\in \mathcal{E}(U)$, $\partial_i=\frac{\partial}{\partial_i}$ and $t=(t_1,...,t_n)$.
\end{thm}

\iffalse
Given a closed subgroup $H$ of $G$, let $\pi$ denote the canonical projection of $G$ onto $G/H$. If $f$ is a function on the homogeneous space $G/H$, we write $\widetilde{f}$ for the function $f\circ\pi$ on $G$.

Let $S(\mathfrak{\lp})^K$ denote the set of $\Ad_G(K)$-invariants in $S(\mathfrak{\lp})$. Recall the Cartan decomposition $\g=\lk\oplus\lp$. Given $Q\in S(\lp)$, let $D_{\lambda(Q)}$ be defined as in Theorem \ref{mu}.
\begin{thm}[\cite{He01}, p. 287]
The mapping $Q\mapsto D_{\lambda(Q)}$ is a linear bijection $S(\mathfrak{\lp})^K \rightarrow D(G/K)$.
\end{thm}

Since $\la$ is abelian, the symmetric algebra $S(\la)$ and the universal envelopoing algebra $\mathcal{U}(\la)$ coincide. It follows that the symmetrization map $\lambda$ identifies $S(\la)^W$ with $\dd_W(A)$. The point is that every $W$-invariant polynomial on $\la$ can be uniquely extended to a $K$-invariant polynomial on $\lp$ (\cite{He00}, p. 299). It follows that $\dd_W(A)\subseteq\dd(G/K)$.

The action of the Weyl group $W$ on $\la^*$ is given by $(s\cdot\alpha)(H) = \alpha(s^{-1}\cdot H)$, where $s\in W$, $\alpha\in\la^*$, $H\in\la$, the group action on $\la$ being the adjoint action. By $S(\la)^W$ we denote the Weyl group invariant polynomials on $\la$.
\fi

Theorem \ref{symmetrization} states that the algebra $\dd(U)$ of translation invariant differential operators on $U$ is generated by the $\widetilde{X}$ ($X\in\mathfrak{u}$). The mapping $\lambda$ is called \emph{symmetrization}\index{symmetrization} and identifies the commutative algebras $S(\la)$ and $\dd(A)$. Further, it identifies the set $S(\la)^W$ of $W$-invariants in $S(\la)$ with the set $\dd_W(A)$ of $W$-invariant differential operators on $A\cdot o$ with constant coefficients.

Given a homomorphism $\chi: D(G/K)\rightarrow \cc$ we introduce the \emph{joint eigenspace}\index{joint eigenspace}
\begin{eqnarray*}
\mathcal{E}_{\chi}(X) = \left\{f\in\mathcal{E}(G/K): Df=\chi(D)f \textnormal{ for all } D\in\dd(G/K)  \right\}.
\end{eqnarray*}

We know from \eqref{canonical isomorphism} that $\dd(G/K)\cong S(\la)^W$. Since $\dd(A)$ is a commutative polynomial ring, each $\nu\in\la^*_{\cc}$ extends uniquely to a homomorphism of $\dd(A)$ into $\cc$, denoted by $D\mapsto D(\nu)$. We then have (\cite{He01}, Chapter III, Lemma 3.11):
\begin{lem}\label{important lemma}
The homomorphisms of $S(\la)^W$ into $\cc$ are precisely
\begin{eqnarray*}
\chi_\mu: P\mapsto P(\mu),
\end{eqnarray*}
\index{$\chi_\mu$, homomorphism}where $\mu$ is an element of $\mathfrak{a}_{\cc}^*$,
\end{lem}

It follows that the characters of $\dd(G/K)$ (and hence the joint eigenspaces) are parameterized by the orbits of $W$ in $\la^*_{\cc}$: Given $\lambda\in\mathfrak{a}^*_{\cc}$ we define
\begin{eqnarray}\label{convenient}
\mathcal{E}_{\lambda}(X) = \left\{f\in\mathcal{E}(X): Df=\Gamma(D)(i\lambda)f \textnormal{ for all } D\in\dd(X)  \right\}.
\end{eqnarray}
Lemma \ref{important lemma} implies that each $\mathcal{E}_{\chi}(X)$ is given by a suitably chosen $\mathcal{E}_{\lambda}(X)$.

\begin{defn}
Let $\lambda\in\mathfrak{a}^*_{\cc}$ and $b\in B$. We define
\begin{eqnarray}\label{non-Euclidean plane waves}
e_{\lambda,b}: X\rightarrow\cc, \hspace{2mm} z\mapsto e^{(i\lambda+\rho)\langle z,b\rangle}.
\end{eqnarray}
The exponential functions $e_{\lambda,b}$ are called \emph{non-Euclidean plane waves}\index{$e_{\lambda,b}$, non-Euclidean plane waves}.
\end{defn}

Recall our notation $b_{\infty}=eM\in K/M$. Let $\lambda\in\la^*_{\cc}$. The function
\begin{eqnarray}
e_{\lambda,b_{\infty}}: G/K\rightarrow\cc, \,\,\,\,\,\,\, gK \mapsto e^{-(i\lambda+\rho)H(g^{-1})}
\end{eqnarray}
is $N$-invariant and its restriction to $A\cdot o$ is given by ${e_{\lambda,b_{\infty}}}_{|A\cdot o}(a\cdot o) = e^{(i\lambda+\rho)(\log a)}$. By \eqref{Radial part 1} and \eqref{Radial Part 2}, if $D\in\dd(G/K)$,
\begin{eqnarray*}
(D e_{\lambda,b_{\infty}})_{|A\cdot o} &=& \Delta_N(D)(e_{\lambda,b_{\infty}})_{|A\cdot o} \\
&=& (e^{\rho}\Gamma(D)\circ e^{-\rho})(e_{\lambda,b_o})_{|A\cdot o} \\
&=& \Gamma(D)(i\lambda)(e_{\lambda,b_o})_{|A\cdot o}.
\end{eqnarray*}
Hence $D e_{\lambda,b_{\infty}} = \Gamma(D)(i\lambda) e_{\lambda,b_{\infty}}$, since both sides are $N$-invariant. In general, when $b=kM\in K/M$ is arbitrary, then $e_{\lambda,b}(x) = e_{\lambda,b_{\infty}}(k^{-1}\cdot x)$, so the $K$-invariance of $D$ implies
\begin{eqnarray}\label{ilambda}
D e_{\lambda,b} = \Gamma(D)(i\lambda) e_{\lambda,b}
\end{eqnarray}
for all $\lambda\in\la^*_{\cc}$, $b\in B$ and $D\in\dd(G/K)$. It follows that each $e_{\lambda,b}$ is a joint eigenfunction and belongs to $\mathcal{E}_{\lambda}(X)$. Moreover, \eqref{ilambda} explains why one takes $i\lambda$ instead of $\lambda$ in the definition \eqref{convenient} of the $\mathcal{E}_{\lambda}(X)$. Finally, \eqref{symbol Laplace 1} implies
\begin{eqnarray}\label{character of the Laplacian}
L_X \, e_{\lambda,b} = \Gamma(L_X)(i\lambda) \, e_{\lambda,b} = - (\langle\lambda,\lambda\rangle + \langle\rho,\rho\rangle) \, e_{\lambda,b}.
\end{eqnarray}
This explicit formula for the eigenvalues of the Laplacian is of particular importance and will be applied a couple of times in the following sections.

\begin{rem}\label{polynomials}
A Riemannian manifold $X$ with distance function $d$ is called \emph{two-point homogeneous}\index{two-point homogeneous space} if whenever $d(p,q)=d(p',q')$, then there is an isometry $g$ of $X$ such that $g(p)=p'$ and $g(q)=q'$. A Riemannian symmetric space of the noncompact type is two-point homogeneous if and only if its real rank is one. If $X\cong G/K$ is a two-point homogeneous space, then $\dd(G/K)$ is generated by the Laplacian, that is the algebra of invariant differential operators consists of the polynomials in the Laplace-Beltrami operator (\cite{He01}, p. 288).
\end{rem}

\subsection{The classical examples}
It is always useful to have concrete examples in mind. The classification of globally symmetric spaces of noncompact type is the
same as the classification of semisimple Lie groups. As often in Lie group theory, the classification contains a finite number of infinite lists (as the one of special linear groups $\textnormal{SL}_n(\rr)$ for $n\geq 2$), the so-called classical groups, and a finite set of ``exceptional'' examples.

\subsubsection{Hyperbolic spaces and their realizations}
For rank one symmetric spaces, there are three lists of classical spaces: Real, complex, and quaternionic hyperbolic spaces. There is only one exceptional one, the Cayley hyperbolic plane. For the latter we refer to the standard literature on exceptional Lie groups and Lie algebras, for example \cite{D78}. In this Section we describe the realizations of the classical hyperbolic spaces. We follow \cite{DH97}.

Let $\ff\in\left\{\rr,\cc,\hh\right\}$ denote the field of real numbers, complex numbers, or the quaternions. On $\ff^{n+1}$, regarded as a right-vector space over $\ff$, we consider the Hermitian form
\begin{eqnarray*}
[x,y] = \overline{y_0}x_0 - \overline{y_1}x_1 - \cdots - \overline{y_n}x_n.
\end{eqnarray*}

Let $G=U(1,b;\ff)$ be the group of $(n+1)\times(n+1)$ matrices with coefficients in $\ff$ which preserve this Hermitian form. The group $G$ acts on the projective space $P_n(\ff)$ and the stabilizer of the line generated by the vactor $(1,0,\ldots,0)$ is the group $K=U(1;\ff)\times U(n;\ff)$, which is compact. We call $X=G/K$ a hyperbolic space. $X$ is a Riemannian symmetric space of the noncompact type of rank one. By $\pi$ we denote the natural projection map
\begin{eqnarray*}
\pi:\ff^{n+1}\setminus\left\{o\right\}\rightarrow P_n(\ff).
\end{eqnarray*}
The hyperbolic space $X$ is then the image under $\pi$ of the open set
\begin{eqnarray*}
\left\{x\in\ff^{n+1}:[x,x]>0\right\}.
\end{eqnarray*}
On $\ff^n$ we have the inner product $(x,y)=\sum_j \overline{y_j}x_j$ with norm $\left\|x\right\|=(x,x)^{1/2}$. Let $B(\ff^n)$ denote the unit ball in $\ff^n$. Then the space $X$ can also be realized as the unit ball in $\ff^n$. In fact, the map
\begin{eqnarray*}
\left\{x\in\ff^{n+1}:[x,x]>0\right\}\rightarrow\ff^n
\end{eqnarray*}
given by $x\mapsto y$, where $y_p=x_px_0^{-1}$, defines, after going to the quotient space, a real analytic bijection of $X$ onto $B(\ff^n)$ and $G$ acts transitively by fractional linear transformations (\cite{DH97}).

Let $d$ denote the dimension of $\ff$ over $\rr$, so $d=1,2$ or $4$ respectively. On $\left\{x\in\ff^n:[x,x]>0\right\}$ we consider the Riemannian metric
\begin{eqnarray*}
ds^2 = - \frac{[dx,dx]}{[x,x]}.
\end{eqnarray*}
This metric is invariant under $x\mapsto x\lambda$ ($\lambda\in\ff\setminus\left\{0\right\})$ and thus defines a Riemannian metric on $X$, which is invariant under $G$, of signature $(dn,0)$.

We can now describe the group theoretical deompositions of $G$. Let $J$ be the $(n+1)\times(n+1)$ diagonal-matrix
\begin{eqnarray*}
J &:=& \begin{pmatrix} -1 &  & \\  & 1 & & \\ & & \ddots \\ & & & 1 & \end{pmatrix}.
\end{eqnarray*}
it will turn out that this matrix is a representative in $M'$ for longest Weyl group element $w\in W$. For any $(n+1)\times(n+1)$-matrix $X$ with coefficients in $\ff$ we set $X^*:= J \overline{X}^{\textnormal{tr}}J$. The Lie algebra $\g$ of $G$ consists of matrices $X$ which satisfy $X+X^*=0$. These are the matrices of the form
\begin{eqnarray*}
X &=& \begin{pmatrix} Z_1 & Z_2 \\ \overline{Z_2}^{\textnormal{tr}} & Z_3 \end{pmatrix},
\end{eqnarray*}
where $Z_1$ and $Z_3$ are anti-Hermitian and $Z_2$ is arbitrary. The involutive automorphism $\theta$ of $\g$ is given by
\begin{eqnarray*}
\theta(X) := JXJ.
\end{eqnarray*}
This $\theta$ is the Cartan involution with the usual decomposition $\g=\lk + \lp$ into eigenspaces to the eigenvalues $+1$ and $-1$. The space $\lk$ is the Lie algebra of the subgroup $K=U(1;\ff)\times U(n;\ff)$.

Let $L$ be the element
\begin{eqnarray*}
L &:=& \begin{pmatrix} 0 & 0 & 1 \\  0 & 0_{n-1} & 0 \\ 1 & 0 & 0 \end{pmatrix} \in \g.
\end{eqnarray*}
Then $L\in\lp$ and $\la:=\rr L$ is a maximal abelian subspace of $\lp$. The centralizer of $L$ in $\lk$ is
\begin{eqnarray*}
\mathfrak{m} &:=& \left\{ \begin{pmatrix} u & 0 & 0 \\  0 & v & 0 \\ 0 & 0 & u \end{pmatrix} : u\in\ff, u+\overline{u}=0, v\in\mathfrak{u}(n-1;\ff) \right\},
\end{eqnarray*}
where $\mathfrak{u}(n-1;\ff)$ denotes the Lie algebra of $U(n-1,\ff)$. Let $\alpha:=1$ The nonzero eigenvalues of $L$ are $\pm\alpha$ if $\ff=\rr$ and $\pm\alpha, \pm\alpha$ if $\ff=\cc$ or $\ff=\hh$. The root-space $\g_{\alpha}$ consists of the matrices
\begin{eqnarray*}
X &=& \begin{pmatrix} 0 & z^* & 0 \\  z & 0_{n-1} & -z \\ 0 & z^* & 0 \end{pmatrix},
\end{eqnarray*}
where $z$ is an $(n-1)\times 1$-matrix with coefficients in $\ff$, and where $z^*=-\overline{z}^{\textnormal{tr}}$. We have $m_{\alpha}:=\dim(\g_{\alpha})=d(n-1)$. The space $\g_{2\alpha}$ consists of matrices of the form
\begin{eqnarray*}
X &=& \begin{pmatrix} w & 0 & -w \\  0 & 0_{n-1} & 0 \\ w & 0 & -w \end{pmatrix},
\end{eqnarray*}
where $w\in\ff$ with $w+\overline{w}=0$. Then $m_{2\alpha}:=\dim(\g_{2\alpha})=d-1$. We have
\begin{eqnarray*}
\g = \g_{-2\alpha} + \g_{-\alpha} + \la + \mathfrak{m} + \g_{\alpha} + \g_{2\alpha}.
\end{eqnarray*}
The subgroup $A=\exp(\la)$ of $G$ is given by the matrices
\begin{eqnarray*}
a_t &:=& \begin{pmatrix} \cosh t & 0 & \sinh t \\  0 & \id_{n-1} & 0 \\ \sinh t & 0 & \cosh t \end{pmatrix},
\end{eqnarray*}
where $t\in\rr$. The centralizer of $A$ in $K$ is the subgroup $M$ of matrices
\begin{eqnarray*}
\begin{pmatrix} u & 0 & 0 \\  0 & v & 0 \\ 0 & 0 & u \end{pmatrix},
\end{eqnarray*}
where $u\in\ff$, $|u|=1$, $v\in U(n-1;\ff)$. The Lie algebra of $M$ is $\mathfrak{m}$. The subspace $\lnn = \g_{\alpha} + \g_{2\alpha}$ is a nilpotent subalgebra of $\g$ and the Lie algebra of the analytic subgroup $N$ of $G$ given by the matrices
\begin{eqnarray*}
n(w,z) := \begin{pmatrix} 1+w-\frac{1}{2}[z,z] & z^* & -w+\frac{1}{2}[z,z] \\  z & \id_{n-1} & -z \\ w-\frac{1}{2}[z,z] & z^* & 1-w+\frac{1}{2}[z,z] \end{pmatrix},
\end{eqnarray*}
where $w\in\ff$, $w+\overline{w}=0$, where $z$ is an $(n-1)\times 1$-matrix with coefficients in $\ff$ and with $z^*=-z^{\textnormal{tr}}$. If
\begin{eqnarray*}
z = \begin{pmatrix} z_2 \\ \vdots \\ z_n  \end{pmatrix}, \,\,\,\, z' = \begin{pmatrix} z'_2 \\ \vdots \\ z'_n  \end{pmatrix},
\end{eqnarray*}
then $[z,z']=-\overline{z'_2}z_2 - \cdots - \overline{z'_n}z_n$. The composition law in $N$ is
\begin{eqnarray*}
n(w,z) \cdot n(w',z') = n(w+w'+\Im[z,z'],z+z').
\end{eqnarray*}
In particular, since $[z,z]$ is real, the inverse of $n(w,z)$ is $n(-w,-z)$. The subgroup $A$ normalizes $N$:
\begin{eqnarray*}
a_t n(w,z) a_{-t} = n(e^{2t}w,e^{t}z).
\end{eqnarray*}
The parameter $\rho$ is given by $\rho=\frac{1}{2}(m_{\alpha}+m_{2\alpha})$. The Iwasawa decomposition reads $G=KAN=NAK$. Each $g\in G$ can be written $g=k\exp H(g) n$, where $k\in K$, $n\in N$, and $H(g)\in\la$.  Let $|\cdot|$ denote the norm in $\ff$.
\begin{lem}
Let $g=(g_{i,j})$ with $i,j=0,1,...,n$ be an element in $G=U(1,n;\ff)$. Then $H(g)=tL$, and $t=\ln|g_{0,0}+g_{0,n}|$.
\end{lem}
\begin{proof}
Set $f(g):=\ln|g_{0,0}+g_{0,n}|$. Let $b_t:=\exp(tL) \in A$. Then
\begin{eqnarray*}
\ln|(b_t)_{00}+ (b_t)_{0n}| &=& \log|\cosh t + \sinh t| \\
&=& \ln(e^t) \\
&=& t.
\end{eqnarray*}
Hence $H(g)=f(g)$. Moreover, $f(g)$ is left-$K$-invariant, since $k\in K=U(1;\ff)\times U(n;\ff))$ is unitary, and right $N$-invariant (this follows from the explicit expression of $n(w,z))$. Hence $f(g)=H(g)$ for all $g\in G$.
\end{proof}

An explicit computation shows the following: If $g=n(w,z)$, then
\begin{eqnarray*}
|(gw)_{00}+ (gw)_{0n}|^2 &=& |1-2w+[z,z]|^2 \\
&=& (1+[z,z]^2)^2+4w^2,
\end{eqnarray*}
since $[z,z]$ is real and $w$ is purely imaginary. This formula is even in $z$ and $w$, so considering $n(-w,-z)=n(w,z)^{-1}$ instead of $n(z,w)$ gives the same result.

\begin{cor}
Let $X=G/K$ be a hyperbolic space. Then $H(nw)=H(n^{-1}w)$, where $n\in N$, $G=KAN$, and where $w\in W$ denotes the longest Weyl group element.
\end{cor}

\begin{rem}
The formula $H(nw)=H(n^{-1}w)$ can alternatively be shown as follows: The set of positive roots $\sum^+$ consists of $\alpha$ and possibly $2\alpha$. Recall that $m_{\alpha}$ and $m_{2\alpha}$ denote the multiplicities of these roots. We write $B(\cdot,\cdot)$ for the Killing form and put $|Z|^2=-B(Z,\theta Z)$ for $Z\in\g$. If $\overline{n}\in\overline{N}$ we write $\overline{n}=\exp(X+Y)$, where $X\in\g_{-\alpha}$ and $Y\in\g_{-2\alpha}$. Set
\begin{eqnarray*}
c^{-1}:=4(m_{\alpha}+4m_{2\alpha}).
\end{eqnarray*}
Then by \cite{He94}, p. 180, we have
\begin{eqnarray*}
e^{\rho H(\overline{n})}
= [(1+c|X|^2)^2 + 4c|Y|^2]^{\frac{1}{4}\left(m_{\alpha}+2m_{2\alpha}\right)}.
\end{eqnarray*}
We always have $wNw^{-1}=\overline{N}$, although conjugation with $w$ does not have to coincide with the involution $\theta$ in all cases (it is true for the classical hyperbolic spaces). It follows that the inverse of $\overline{n}=\exp(X+Y)$ is given by $\overline{n}=\exp(-X-Y)$. In particular, the formula for $e^{\rho H(\overline{n})}$ is even in $X$ and $Y$, so $H(\overline{n})=H(\overline{n}^{-1})$ for all $\overline{n}\in\overline{N}$. This implies $H(nw)=H(n^{-1}w)$ for all $n\in N$, since $H(\cdot)$ is left-$K$-invariant.
\end{rem}

\subsubsection{The special linear groups}\label{The special linear groups}
The groups $G=\textnormal{SL}_n(\rr)$ are generic examples for higher rank spaces. In particular, if $K=\textnormal{SO}_n(\rr)$, then $G/K$ is a Riemannian symmetric space of the noncompact type of rank $n-1$. We will briefly recall the Iwasawa-decomposition components of this group and give a counterexample for the formula $H(nw)=H(n^{-1}w)$ we already analyzed in the case of rank one spaces. The interest of the function $n\mapsto H(nw)$ arises in the fact that it is the phase function of several integrals, such as the Harish-Chandra's $c$-function, and another family of operators we will consider in Section \ref{Patterson-Sullivan distributions}.

Let $G=\textnormal{SL}_n(\rr)$. The subgroup $A$ arising in the Iwasawa decomposition consists of the $n\times n$-diagonal matrices
\begin{eqnarray*}
a := \begin{pmatrix} a_1 &  &  \\   & \ddots &  \\  &  & a_n \end{pmatrix},
\end{eqnarray*}
where $a_1\cdots a_n=1$ and $a_j>0$ for all $1\leq j\leq n$. The nilpotent subgroups $N$ and $\overline{N}$ are given by upper, respectively lower, triangular matrices with $1's$ in the main diagonal. The subgroup $M'$ of $K$ is generated by the subgroup $M$ and by the diagonal-matrices
\begin{eqnarray*}
s_i := \begin{pmatrix} 1 & & & & & & & \\ & \ddots & & & & & & \\ & & 1 & & & & \\ & & & s & & & & \\ & & & & 1 & & \\ & & & & & \ddots & & \\ & & & & & & 1 \end{pmatrix},
\end{eqnarray*}
where the matrix
\begin{eqnarray*}
s := \begin{pmatrix} 0 & 1 \\ -1 & 0 \end{pmatrix}
\end{eqnarray*}
is placed in the $i-th$ and $(i+1)-th$ rows. The Weyl group $W$ (imbedded into the subgroup $M'$) is generated by the matrices $s_i$. The action of $W$ on $A$ is defined by the formula $w'\cdot a := w'aw'^{-1}$ ($w'\in W$, $a\in A$). The group $W$ coincides with the symmetric group $S_n$ and therefore has $n!$ elements. The matrix $w$ with all zero entries, except for the entries $(w)_{k,n-k+1}=\pm 1$, is the longest element in $W$. It permutes the entries $a_k$ and $a_{n-k+1}$ ($k=1,2,\ldots,n$) of the matrices $a=\textnormal{diag}(a_1, a_2, \ldots, a_n \in A$. Moreover, we have $\overline{N}=wNw^{-1}$.

Let $G=\textnormal{SL}_3(\rr)$. We will now find an $n\in N$ such that $H(nw)\neq H(n^{-1}w)$. An element $a\in A$ has the form
\begin{eqnarray*}
a = \begin{pmatrix} e^{s} & 0 & 0 \\  0 & e^{t} & 0 \\ 0 & 0 & e^{-s-t} \end{pmatrix},
\end{eqnarray*}
where $s,t\in\rr$. The longest Weyl group element $w\in W$ is
\begin{eqnarray*}
w = \begin{pmatrix} 0 & 0 & 1 \\  0 & -1 & 0 \\ 1 & 0 & 0 \end{pmatrix}.
\end{eqnarray*}
We fix an element $n\in N$. Then there are $d,e,f\in\rr$ such that
\begin{eqnarray*}
n = \begin{pmatrix} 1 & d & e \\  0 & 1 & f \\ 0 & 0 & 1 \end{pmatrix}.
\end{eqnarray*}
Multiplying out we find
\begin{eqnarray*}
nw = \begin{pmatrix} e & -d & 1 \\  f & -1 & 0 \\ 1 & 0 & 0 \end{pmatrix} \,\,\, \textnormal{ and } \,\,\, n^{-1}w = \begin{pmatrix} df-e & d & 1 \\  -f & -1 & 0 \\ 1 & 0 & 0 \end{pmatrix}.
\end{eqnarray*}
Now suppose that $nw=\tilde{k}\tilde{a}\tilde{n}$ is written corresponding to the Iwasawa decomposition, where $\tilde{a}=a(s,t)$ as above. Then $\tilde{k}=nw\tilde{n}^{-1}\tilde{a}^{-1}\in\textnormal{SO}_3(\rr)$ yields
\begin{eqnarray}\label{s eq}
e^2 + f^2 + 1 = e^{2s}.
\end{eqnarray}
Similarly, if $n^{-1}w$ is Iwasawa decomposed with $A$-part $a(s',t')$, then
\begin{eqnarray}\label{s' eq}
(df-e)^{2} + f^2 + 1 = e^{2s'}.
\end{eqnarray}
If we now assume that $H(nw)=H(n^{-1}w)$ then in particular $s=s'$. The equations \eqref{s eq} and \eqref{s' eq} have solutions for suitable chosen $d,e,f$ and $s$, but surely not for all choices. For example, the equations contradict if $d=e=f=1$, which shows that $H(nw)=H(n^{-1}w)$ is not a general property in SL$_3(\rr)$. The method used here can be extended to all special linear groups SL$_n(\rr)$ for $n\geq 3$. We always have $H(nw)=H(n^{-1}w)$ in the group SL$_2(\rr)$ (see Section \ref{An integral formula}).

\newpage
\pagebreak
\thispagestyle{empty} 

\section{Component computations}\label{Component Computations}

For later reference, we outhouse long and technical computations.

\subsection{Some integral formulas}
If $U$ is a Lie group with closed subgroup $V$ and with a left-invariant positive measure on $V$ we put
\begin{eqnarray}\label{tilde}
\widetilde{F}(uV) := \int_V F(uv) \, dv, \hspace{5mm} F\in C_c(U).
\end{eqnarray}
Note that this factorization $\widetilde{F}$ is not the same as the lift $F\circ\pi$ from the preceding sections. The mapping $F\mapsto\widetilde{F}$ is a linear and surjective mapping of $C_c(U)$ onto $C_c(U/V)$ (\cite{He01}, p. 91). In what follows, we will often use the following integral formula due to Harish-Chandra (\cite{He01}, p. 197).
\begin{lem}\label{integral formula 1}
Let $g\in G$. Then
\begin{eqnarray}\label{often used}
\int_K f(k(g^{-1}k)) dk = \int_K f(k)e^{-2\rho H(gk)} dk, \hspace{3mm}f\in C(K).
\end{eqnarray}
\end{lem}
Hence $(T_g)^*(dk)=e^{-2\rho(H(gk))}dk$, where $(T_g)^*(dk)$ denotes the pull-back measure corresponding to the $G$-action on $K$. We write $\frac{d k(gk)}{dk}=e^{-2\rho(H(gk))}$ to express the Jacobian $|\det d T_g(k)|$. We will need a similar formula for the quotient $K/M$. Therefore first observe that
\begin{eqnarray*}
\int_M (F\circ T_g^{-1})(km)dm &=& \widetilde{F}(k(g^{-1}k)M) \\
&=& \widetilde{F}(\overline{T}_g^{-1}(kM)) \\
&=& \widetilde{F}\circ\overline{T}_g^{-1}(kM).
\end{eqnarray*}
Hence
\begin{eqnarray}\label{simple computation}
(F\circ T_g^{-1})^{\sim}(kM) = \widetilde{F}\circ \overline{T}_g^{-1}(kM).
\end{eqnarray}
Recall that the Iwasawa projection $g\mapsto H(g)$ is $M$-bi-invariant. It follows that the Jacobian $e^{-2\rho H(gk)}$ of the action of $g$ on $K$ is a function on $K/M$.

\begin{cor}\label{Jacobian}
The Jacobian of $\overline{T}_g: K/M\rightarrow K/M, kM\mapsto k(gk)M,$ is $|\det d\overline{T}_g(kM)|=e^{-2\rho H(gk)}$.
\end{cor}
\begin{proof}
We need to show that for each $f\in C(K/M)$
\begin{eqnarray}
\int_{K/M} (f\circ \overline{T}_g^{-1})(kM) dkM = \int_{K/M} f(kM) e^{-2\rho H(gk)} d(kM) .
\end{eqnarray}
Select $F\in C(K)$ such that $f=\widetilde{F}$. Then by \ref{integral quotient} and the $M$-equivariance of $T_g$
\begin{eqnarray*}
\int_{K/M} |\det dT_g(k)| f(kM) d(kM) &=& \int_{K/M} |\det dT_g(k)| \widetilde{F}(kM)d(kM) \\
&=& \int_{K/M} |\det dT_g(k)| \left(\int_M F(km)dm \right) d(kM).
\end{eqnarray*}
(Recall $\int_M dm = 1$.) Then the last expression equals
\begin{eqnarray*}
\int_K F(k)|\det dT_g(k)| dk &=& \int_K (F\circ {T_g}^{-1})(k)dk \\
&=& \int_{K/M} \left(\int_M F \circ {T_g}^{-1}(km)dm\right) d(kM),
\end{eqnarray*}
and by \eqref{simple computation} the last term equals $\int_{K/M} (\widetilde{F}\circ \overline{T}_g^{-1})(kM) d(kM)$, as desired.
\end{proof}

\begin{rem}
The measure $dp=dmdadn$ (in the notation of \ref{Measure theoretic preliminaries}) is a left-invariant measure on $P=MAN$. Let $db$ denote the normalized $K$-invariant measure on $K/M=G/P$. Using \eqref{integral formula G} we get (\cite{He94}, p. 512) for $f\in C_c(G)$
\begin{eqnarray}\label{integral formula 2}
\int_G f(g) e^{-2\rho(H(g))} dg = \int_{G/P} db(gP) \int_P f(gp) dp.
\end{eqnarray}
\end{rem}

Corollary \ref{Jacobian} states that $\frac{d k(gk)M}{d(kM)} = e^{-2\rho(H(gk))}$. Given $b=kM$ we use \ref{cor need} to find
\begin{eqnarray}\label{dgb / db}
\frac{d(g\cdot b)}{db} &=& e^{-2\rho(H(gk))} = e^{+2\rho(\langle g^{-1}K,kM\rangle)} \\
&=& e^{-2\rho(\langle gK,g\cdot kM \rangle)} = e^{-2\rho(\langle g\cdot o, g\cdot b \rangle)}.
\end{eqnarray}
It follows for $f\in C(B)$ that
\begin{eqnarray}
\int_B f(g\cdot b) e^{2\rho(\langle g\cdot o, g\cdot b \rangle)} d(gb) = \int_B f(b) db = \int_{K} f(kM) \, dk.
\end{eqnarray}

\begin{rem}
Let $C_c(G)^M$ denote the right-$M$-invariant functions in $C_c(G)$. Then $C_c(G/M)\cong C_c(G)^M$ via \eqref{tilde}, so $M$-invariant functions on $G$ are functions on $G/M$ and vice versa. Under $G/K \times K/M \cong G/M$ a function $g\mapsto f(gM)$ on $G/M$ becomes a function $(gK,g\cdot B)\mapsto f(gM)$ on $X\times B$.
\end{rem}

\begin{lem}
Let $g\in G$ and $z\in X$. Then
\begin{eqnarray}
\int_K e^{-2\rho H(gk)} \, dk = 1 \,\,\,\,\, \textnormal{ and } \,\,\,\,\, \int_B e^{2\rho\langle z,b\rangle} \, db = 1.
\end{eqnarray}
\end{lem}
\begin{proof}
Apply Harish-Chandra's formula \eqref{often used} to $f(k)=1$:
\begin{eqnarray*}
1 &=& \int_K f(k)\, dk = \int_K f(k(g^{-1}k)\, dk = \int_K f(k)\cdot e^{-2\rho H(gk)}\, dk.
\end{eqnarray*}
Given $z=g\cdot o$, where $g\in G$, we then find
\begin{eqnarray*}
\int_B e^{2\rho\langle z,b\rangle} \, db = \int_K e^{-2\rho H(g^{-1}k)} \, dk = 1,
\end{eqnarray*}
as desired.
\end{proof}
Recall the formulas $\frac{d(g\cdot b)}{db} = e^{-2\rho(\langle g\cdot o, g\cdot b \rangle)}$ and $\langle g\cdot z,g\cdot b\rangle=\langle z,b\rangle + \langle g\cdot o,g\cdot b\rangle$. Let $f\in C_c(X\times B)$. The $G$-invariance of $dz$ then yields
\begin{eqnarray*}
\int_{X\times B} f(z,b) e^{2\rho\langle z,b\rangle} \, dz\, db &=& \int_{X\times B} f(g\cdot z,g\cdot b) e^{2\rho\langle g\cdot z,g\cdot b\rangle e^{-2\rho\langle g\cdot o,g\cdot b\rangle}} \, dz\, db \\
&=& \int_{X\times B} f(g\cdot z,g\cdot b) e^{2\rho\langle z,b\rangle} \, dz\, db.
\end{eqnarray*}
This proves:
\begin{prop}
$e^{2\rho\langle z,b\rangle}\,dz\,db$ is a $G$-invariant measure on $X\times B$.
\end{prop}
Hence by uniqueness, under the inverse of the $G$-equivariant diffeomorphism $G/M\rightarrow X\times B$, $gM\mapsto (g\cdot o,g\cdot M)$, the measure $e^{2\rho\langle z,b\rangle}\,dz\,db$ is mapped into a scalar multiple of $d(gM)$, the $G$-invariant measure on $G/M$. To compute this scalar $c$, select $f(z)\in C_c^{\infty}(X)$ such that $\int_X f(z)dz=1$. Lift $f(g):=f(g\cdot o)$ to a $K$-invariant function on $G$. Then $\int_G f(g)dg=1$. Also lift $f(z,b):=f(z)$ to a function on $X\times B$, which is independent of $b$. Then $f(g)=f(g\cdot o,g\cdot M)$, so
\begin{eqnarray}\label{last integral 2}
c &=& c \int_G f(g\cdot o,g\cdot M) \, dg \nonumber \\
&=& \int_{X\times B} f(z,b)\, e^{2\rho\langle z,b\rangle}\, dz\, db \nonumber \\
&=& \int_X f(z) \int_B e^{2\rho\langle z,b\rangle}\, db\, dz.
\end{eqnarray}
But $\int_B e^{2\rho\langle z,b\rangle}\, db = 1$ and hence \eqref{last integral 2} equals $\int_X f(z) \, dz = 1$. Thus $c=1$.

\begin{cor}
Let $f\in C_c(X\times B)$. Then
\begin{eqnarray}\label{m-invariant integral}
\int_{G/M} f(g\cdot o,g\cdot M) \, dg = \int_{X\times B} f(z,b) e^{2\rho\langle x,b\rangle} \, dz \, db.
\end{eqnarray}
\end{cor}

Given $(z,b)\in X\times B$ we can find $g\in G$ such that $(z,b)=g\cdot (o,M)$. Then $\langle g\cdot o,g\cdot M\rangle=H(g)$. If we replace $f(g\cdot o,g\cdot M)$ by $f(g\cdot o,g\cdot M)e^{-2\rho H(g)}$ it follows from \eqref{m-invariant integral} that
\begin{eqnarray}\label{haar measure corollary}
\int_{X\times B} f(z,b) \, dz \, db = \int_{G/M} f(g\cdot o,g\cdot M) e^{-2\rho H(g)} \, d(gM).
\end{eqnarray}

One can directly prove \eqref{haar measure corollary} by using the $G$-invariance of $dz$ and the integral formulas \eqref{integral formula G} and \eqref{integral formula AN and G/K}:
\begin{eqnarray*}
\int_{X\times B} f(z,b)\, db,\ dz &=& \int_{K}\int_X f(k\cdot z,kM) \, dz \, dk \\
&=& \int_{KAN} f(kan\cdot o, kan\cdot M) \, dk \, da \, dn \\
&=& \int_{G/M} f(g\cdot o, g\cdot M) e^{-2\rho H(g)} \, d(gM).
\end{eqnarray*}

\subsection{Derivatives corresponding to the Iwasawa decomposition}\label{Iwasawa derivatives}
We begin this subsection by recalling some material from \cite{DKV} concerning derivatives of the Iwasawa projection. We will later apply these derivatives to functions defined by the Iwasawa decomposition.

Let $g,h\in G$. We write $h^{g}=ghg^{-1}$\index{$h^{g}$, conjugation in $G$}. Let $U(\mathfrak{g})$ be the universal enveloping algebra of the complexification of $\mathfrak{g}$. The adjoint representation of $G$ on $\mathfrak{g}$ extends to a representation of $G$ on $U(\mathfrak{g})$ by automorphisms. We write $u^g=Ad(g)u$, if $u\in U(\g)$\index{$u^g$, action of $g\in G$ on $u\in U(\g)$}. Then we have
\begin{eqnarray*}
u^{gh} = (u^h)^g, \hspace{3mm} (uv)^g = u^g v^g, \hspace{3mm} (g,h\in G, u,v \in U(\mathfrak{g})).
\end{eqnarray*}
We shall view elements of $U(\mathfrak{g})$ as left invariant differential operators acting on functions on $G$. To explain this interpretation, we now specify how an element $u=X_1\cdots X_r$ ($X_j\in\mathfrak{g}$), acts as a differential operator. Let $f:G\rightarrow\cc$ be a function on $G$ and define
\begin{eqnarray*}
\partial(u) f(g) := f(g; u) := \frac{\partial^r}{\partial t_1\cdots\partial t_r}_{|t_1=...=t_r=0}f(g \exp t_1 X_1 \cdots \exp t_r X_r).
\end{eqnarray*}
If $u\in U(\g)$ is a complex number $c\in\cc$, then $f(g;c)=cf(g)$.

The Iwasawa decomposition $\mathfrak{g}=\mathfrak{k}\oplus\mathfrak{a}\oplus\mathfrak{n}$ gives rise to the decomposition
\begin{eqnarray*}
U(\mathfrak{g}) = (\mathfrak{k}U(\mathfrak{g}) + U(\mathfrak{g})\mathfrak{n}) \oplus U(\mathfrak{a}).
\end{eqnarray*}
Therefore it makes sense to speak of the projection\index{$E_{\mathfrak{a}}$, projection}
\begin{eqnarray*}
E_{\mathfrak{a}}:U(\mathfrak{g})\rightarrow U(\mathfrak{a}).
\end{eqnarray*}
It is clear that this projection preserves the degree filtrations on both sides. Let\index{$\varepsilon$, homomorphism}
\begin{eqnarray*}
\varepsilon: U(\mathfrak{g}) \rightarrow \cc
\end{eqnarray*}
be a homomorphism that sends all elements of $\mathfrak{g}$ to $0$. We call $\varepsilon(u)$ the \emph{constant term}\index{constant term of $u\in U(\mathfrak{g})$} of $u\in U(\mathfrak{g})$. If $u$ has zero constant term, then the same is true for $E_{\mathfrak{a}}(u)$. Since $\mathfrak{a}$ is abelian, $U(\mathfrak{a})$ is canonically isomorphic to the symmetric algebra (see Subsection \ref{Joint eigenfunctions and joint eigenspaces}) over $\mathfrak{a}$. Thus, on $U(\mathfrak{a})$ the degree filtration arises from a grading. So in $U(\mathfrak{a})$ we may speak of the homogeneous components of an element.

We can now give the main calculation on the derivatives of the Iwasawa projection
\begin{eqnarray*}
H:G\rightarrow\mathfrak{a}, \,\,\,\,\,\,\,\,\,\,\, kan\mapsto\log(a).
\end{eqnarray*}
We will later use these formulas several times in applications of the method of stationary phase.

\begin{lem}
Let $g\in G$, $b\in U(\mathfrak{g})$. Then we have the formula
\begin{eqnarray*}
H(g;b) = \varepsilon(b)H(g) + \left( E_{\mathfrak{a}}(b^{t(g)})  \right)_1,
\end{eqnarray*}
where the subscript $1$ means the homogeneous component of degree $1$, and $t(g)=a(g)n(g)$ is the ``triangular part`` \index{triangular part of the Iwasawa decomposition} of the $KAN$ Iwasawa decomposition of $g\in G$.
\end{lem}
\begin{proof}
We copy the proof given in \cite{DKV}, p. 337 to fix some notation. Since $H$ is left-invariant under $K$ and right-invariant under $N$ we have
\begin{eqnarray*}
H(1;u)=0 \hspace{2mm} \forall u \in \mathfrak{k}U(\mathfrak{g}) + U(\mathfrak{g})\mathfrak{n}.
\end{eqnarray*}
Let $g,h\in G$ and Iwasawa decompose $g=kan$. Then
\begin{eqnarray}\label{clear}
H(gh) = H(kanh) = H(anh) = H(h^{t(g)}t(g)) = H(h^{t(g)}a(g)),
\end{eqnarray}
where $a(g)=a$. The right hand side of \eqref{clear} equals $H(h^{t(g)}) + H(a(g)) = H(h^{t(g)}) + H(g)$ and hence
\begin{eqnarray}
H(g;b) = H(1;b^{t(g)}) = \varepsilon(b)H(g) + H(1;E_{\mathfrak{a}}(b^{t(g)})).
\end{eqnarray}
But as $H(\exp X_1 \cdots \exp X_r)=X_1 + \cdots + X_r$ for $X_j\in\mathfrak{a}$ it follows that for any $c\in U(\mathfrak{a})$ we have $H(1;c)=c_1$.
\end{proof}

\begin{lem}
Let $\langle\cdot,\cdot\rangle$ denote the Killing form of $\mathfrak{g}$ and let $H\in \mathfrak{a}$. Let $\phi$ denote the function
\begin{eqnarray*}
\phi: G\rightarrow\rr, \,\,\,\,\,\,\,\,\,\,\, g \mapsto \langle H(g), H \rangle.
\end{eqnarray*}
Let $g\in G$ and $X\in\mathfrak{g}$. Then
\begin{eqnarray}\label{derivative}
\phi(g;X) = \langle X^{t(g)},H\rangle = \langle X, H^{n(g)^{-1}}\rangle.
\end{eqnarray}
\end{lem}
\begin{proof}
$X\in\mathfrak{g}$ has constant term $0$, so $H(g;X)=E_{\mathfrak{a}}(X^{t(g)})$. The linear functional $\lambda(Y)=\langle Y,H\rangle$ ($Y\in\mathfrak{a}$) has derivative $\lambda(Y)$ and from the chain rule we obtain for $\phi = \lambda \circ H$ that
\begin{eqnarray*}
\phi(g;X) = \langle E_{\mathfrak{a}}(X^{t(g)}), H \rangle = \langle X^{t(g)}, H \rangle,
\end{eqnarray*}
since $\mathfrak{a}$ is orthogonal to $\mathfrak{k}\oplus\mathfrak{n}$ with respect to the Killing form, while $H^{t(g)^{-1}} = H^{n(g)^{-1}a(g)^{-1}} = H^{n(g)^{-1}}$, since $a\in A$ fixes $H$, since $\mathfrak{a}$ is abelian.
\end{proof}

Given any Lie group $G$, we denote by $L_g$\index{$L_g$, left translation} the \emph{left translation} by a group element $g\in G$. The tangent vector to the curve $t\mapsto g\exp{tX}$ at $g$ is $dL_g(X)$. Suppose $\mathfrak{g}$ is a direct sum $\mathfrak{g}=\mathfrak{u}\oplus\mathfrak{v}$, where $\mathfrak{u}$ and $\mathfrak{v}$ are subalgebras of $\mathfrak{g}$ (not necessarily ideals). Let $U$ and $V$ be the analytic subgroups of $G$ with Lie algebras $\mathfrak{u}$ and $\mathfrak{v}$. Let $\alpha: U\times V\rightarrow G$ denote the mapping $(u,v)\mapsto uv$. We identify $U$ and $V$ with the subgroups $(U,e)$ and $(e,V)$ of the product group $U\times V$ and we also identify the tangent space $T_{(u,v)}(U\times V)$ with the direct sum $T_u U + T_v V$ ($u\in U, v\in V$). Let $g\cdot X$ ($g\in G, X\in\g$) denote the adjoint action. Let $Y\in \mathfrak{u}, Z\in \mathfrak{v}$. We then have
\begin{eqnarray*}
\alpha(u\exp{tY,v}) = uv\exp(t \hspace{1mm} v^{-1}\cdot Y), \hspace{2mm} t\in\rr
\end{eqnarray*}
and
\begin{eqnarray*}
\alpha(u,v\exp{tZ}) = uv\exp{tZ}, \hspace{2mm} t\in\rr.
\end{eqnarray*}
It follows that the differential of $\alpha$ at $(u,v)\in U\times V$ is given by
\begin{eqnarray}\label{follows from here 4}
d\alpha_{(u,v)}(dL_uY,dL_vZ)=dL_{uv}(v^{-1}\cdot Y+Z).
\end{eqnarray}
Identifying $T_u U=\mathfrak{u}$ and $T_v V=\mathfrak{v}$ we will from now on denote the differential $d\alpha=\alpha'$ of the product map $\alpha$ by
\begin{eqnarray}
\alpha'(u,v)(X,Y) = v^{-1}\cdot X + Y,
\end{eqnarray}
where $u\in U, \, v\in V, \, X\in\mathfrak{u}, \, Y\in\mathfrak{v}$.
\begin{cor}
The mapping $\alpha$ from above is everywhere regular.
\end{cor}
\begin{proof}
$h^{-1}\cdot Y + Z = 0 \Leftrightarrow Y = -h\cdot Z \in \mathfrak{u}\cap\mathfrak{v}=\left\{0\right\} \Leftrightarrow Y = Z = 0$.
\end{proof}

Assume that $G$ is a semisimple Lie group with Iwasawa decomposition $G=NAK$. Then $NA$ is a group, since $A$ normalizes $N$. We consider the following mappings:
\begin{itemize}
\item[(i)] $\sigma_1: N\times A\rightarrow NA$, $(n,a)\mapsto na$,
\item[(ii)] $\sigma_2: NA\times K\rightarrow NAK=G$, $(na,k)\mapsto nak$,
\item[(iii)] $\sigma_3: N\times A\times K \rightarrow NAK=G$, $(n,a,k)\mapsto nak$,
\item[(iv)] $\sigma_4: A\times N\rightarrow AN$, $(a,n)\mapsto an$,
\item[(v)] $\sigma_5: A\times N\times K\rightarrow AN \times K$, $(an,k)\mapsto ank$,
\item[(vi)] $\sigma_6: A\times N\times K\rightarrow ANK=G$, $(a,n,k)\mapsto ank$.
\end{itemize}

Then $\sigma_3 = \sigma_2 \circ (\sigma_1\times\id_K)$. It follows from the chain rule that
\begin{eqnarray*}
\sigma_3'(n,a,k):\mathfrak{n}\times\mathfrak{a}\times\mathfrak{k}\rightarrow\mathfrak{g}
\end{eqnarray*}
is given by
\begin{eqnarray*}
\sigma_3'(n,a,k)(X,Y,Z)=\Ad(k^{-1}) (\Ad(a^{-1})X+Y) +Z,
\end{eqnarray*}
where $X\in\lnn, Y\in\la, Z\in\lk$. Then
\begin{eqnarray}\label{replace}
\sigma_3'(n,a,k)(X,Y,Z)=k^{-1}a^{-1}\cdot X+ k^{-1}\cdot Y + Z.
\end{eqnarray}

Similarly, we obtain
\begin{eqnarray}\label{this yields}
\sigma_6'(a,n,k)(X,Y,Z)=k^{-1}n^{-1}\cdot X+ k^{-1}\cdot Y + Z,
\end{eqnarray}
for $(a,n,k)\in A\times N\times K$ and $(X,Y,Z)\in\la\times\lnn\times\lk$.

Fix $H\in\mathfrak{a}$, $H\neq0$ and let $\langle\cdot,\cdot\rangle$ denote the Killing form. We introduce the $C^{\infty}$-functions
\begin{itemize}
\item[(i)] $\phi_1: N\times A\times K\rightarrow\rr, \hspace{3mm} \phi_1(n,a,k)=\langle H(nak),H\rangle$,
\item[(ii)] $\phi_2: A\times N\times K \rightarrow\rr, \hspace{3mm} \phi_3(a,n,k)=\langle H(ank),H\rangle$.
\end{itemize}

We factorize $\phi_1$ in the following way: As above, let $\sigma_3:N\times A\times K\rightarrow G$ denote the map $(n,a,k)\mapsto nak$ and let $\lambda_0$ denote the linear functional $X\mapsto\langle X,H\rangle$ on $\mathfrak{a}$. Then $\phi_1=\lambda_0 \circ H \circ \sigma_3$. For the differential of $\phi_1$ we obtain from the chain rule
\begin{eqnarray*}
\phi_1'(n,a,k) = \lambda_0'(H(\sigma_3(n,a,k))) \circ H'(\sigma_3(n,a,k)) \circ \sigma_3'(n,a,k).
\end{eqnarray*}
Now replace $X$ in \eqref{derivative} by $k^{-1}a^{-1}\cdot X + k^{-1}\cdot Y + Z$ from \eqref{replace}. Then $\phi_1'(n,a,k)$ is a map
\begin{eqnarray*}
\phi_1'(n,a,k): T_{(n,a,k)}(N\times A\times K) = \mathfrak{n}\times\mathfrak{a}\times\mathfrak{k} \rightarrow T_{nak}G = \mathfrak{g} \rightarrow \mathfrak{a} \rightarrow \rr
\end{eqnarray*}
given by
\begin{eqnarray*}
(X,Y,Z) \mapsto \langle k^{-1}a^{-1}\cdot X + k^{-1}\cdot Y + Z, H^{n(nak)^{-1}}\rangle.
\end{eqnarray*}
We can now write $nak=\tilde{k}\tilde{a}\tilde{n}$ corresponding to the Iwasawa decomposition. Then
\begin{eqnarray}\label{phase function example}
\phi_1'(n,a,k)(X,Y,Z) &=& \langle k^{-1}a^{-1}\cdot X + k^{-1}\cdot Y + Z, H^{n(nak)^{-1}}\rangle \\
&=& \langle \tilde{n}\cdot k^{-1}\cdot a^{-1}\cdot X, H \rangle + \langle \tilde{n}\cdot k^{-1}\cdot Y,H \rangle + \langle \tilde{n}\cdot Z,H\rangle \nonumber.
\end{eqnarray}
For the derivatives of $\phi_2$, write $ank=\tilde{k}\tilde{a}\tilde{n}$. Then \eqref{this yields} yields
\begin{eqnarray}\label{phase function example 2}
\phi_2'(a,n,k)(X,Y,Z) \hspace{-1mm} = \hspace{-1mm} \langle \tilde{n}\cdot k^{-1}\cdot n^{-1}\cdot X, H \rangle \hspace{-1mm} + \hspace{-1mm} \langle \tilde{n}\cdot k^{-1}\cdot Y,H \rangle \hspace{-1mm} + \hspace{-1mm} \langle \tilde{n}\cdot Z,H\rangle.
\end{eqnarray}

\subsection{Critical sets and Hessian forms}\label{Critical sets and Hessian forms}
Let $\langle\cdot,\cdot\rangle$ denote the Killing form and let $H\in\la_+$ with $\|H\|=1$. We investigate the critical set of the phase function $\psi: \mathfrak{a}^*_+ \times N\times A \times K \rightarrow \rr$,
\begin{eqnarray}\label{psi11}
(\mu,n,a,k) \mapsto \mu(\log(a)) - \langle H(nak),H\rangle,
\end{eqnarray}
arising in Section \ref{Pseudodifferential analysis on symmetric spaces} for an oscillatory integral named $Ua$. We analyze the critical set of $\psi$ and write it down explicitly in the case when $X=G/K$ has rank one. Viewed as a function on $\mathfrak{a}^*_+ \times N\times A \times K/M$, the critical set will then consist of one single point. We then prove the non-degeneracy of the Hessian form of $\psi$ at this critical point.

Note that in order to determine the critical set of $\psi$, we have to solve
\begin{eqnarray}\label{have to solve}
d\psi(\mu,n,a,k)=0.
\end{eqnarray}
Written out, \eqref{have to solve} is equivalent to the equations
\begin{itemize}\label{coordinates}
\item[(a)] $\frac{\partial\psi}{\partial \mu}(\mu,n,a,k) = 0$,
\item[(b)] $\frac{\partial\psi}{\partial s}_{|s=0}(\mu,n\exp sX,a,k) = 0 \, \textnormal{ for all } X\in\mathfrak{n}$,
\item[(c)] $\frac{\partial\psi}{\partial t}_{|t=0}(\mu,n,a\exp tY,k) = 0 \, \textnormal{ for all } Y\in\mathfrak{a}$,
\item[(d)] $\frac{\partial\psi}{\partial \theta}_{|\theta=0}(\mu,n,a,k\exp\theta Z) = 0 \, \textnormal{ for all } Z\in\mathfrak{k}$.
\end{itemize}

\begin{lem}\label{p and n}
Let $\lnn^{\bot}$\index{$\lnn^{\bot}$, orthogonal complement of $\lnn$} denote the orthogonal complement (w.r.t. the Killing form) of $\lnn$ in $\g$. Then $\lnn^{\bot}\cap\lp=\la$.
\end{lem}
\begin{proof}
Let $Z\in\lnn^{\bot}\cap\lp$. Write $Z=Z_{\la}+Z_{\mathfrak{q}}$ corresponding to the orthogonal decomposition $\g=\lk+\la+\mathfrak{q}$. For $Y\in\lnn$ we then have
\begin{eqnarray}
0=\langle Z,Y\rangle=\langle Z_{\la}+Z_{\mathfrak{q}},Y\rangle = \langle Z_{\mathfrak{q}}, Y\rangle,
\end{eqnarray}
since $\la\bot\lnn$. It follows that $Z_{\mathfrak{q}}\bot\g$, so $Z_{\mathfrak{q}}=0$, so $Z\in\la$. Conversely, if $Z\in\la$, then $Z\in\lp$ and $Z\bot\lnn$.
\end{proof}

\begin{lem}\label{as in lemma}
Let $X=G/K$ have rank one. If $\mu\neq 1$ or $kM\neq M$, then the phase function $\psi$ given in \eqref{psi11} has no critical points in $\left\{\mu\right\}\times A\times N\times\left\{k\right\}$.
\end{lem}
\begin{proof}
Suppose that $(\mu,n,a,k)$ is a critical point for $A\times N$. Write $nak=\tilde{k}\tilde{a}\tilde{n}$ corresponding to the Iwasawa decomposition. We rewrite the $A$-derivative given in \eqref{phase function example} as follows:
\begin{eqnarray*}
\phi_1'(n,a,k)(0,H,0) &=& \langle \tilde{n}k^{-1} \cdot H, H \rangle \\
&=& \langle \tilde{k}\tilde{a}\tilde{n}k^{-1} \cdot H,  \tilde{k} \cdot H \rangle \\
&=& \langle nakk^{-1} \cdot H, \tilde{k} \cdot H \rangle \\
&=& \langle  n \cdot H, \tilde{k} \cdot H \rangle.
\end{eqnarray*}
Similarly we find for $X\in\lnn$
\begin{eqnarray*}
\phi_1'(n,a,k)(X,0,0)=\langle  n \cdot X, \tilde{k} \cdot H \rangle.
\end{eqnarray*}
The assumption that $(\mu,n,a,k)$ is critical is then equivalent to the conditions
\begin{itemize}
\item[(a')] $\langle n\cdot H,\tilde{k}\cdot H\rangle = \mu$,
\item[(b')] $\langle n\cdot X,\tilde{k}\cdot H \rangle = 0 \,\,\, \forall\, X\in\lnn$.
\end{itemize}
It follows from (b') that $\tilde{k}\cdot H\bot\lnn$. But since also $\tilde{k}\cdot H\in\lp$, Lemma \ref{p and n} yields $\tilde{k}\cdot H\in\la$. Hence $\tilde{k}\in M'$. (In higher rank, the same argument applies for $H$ regular.) Now equation (a') yields
\begin{eqnarray*}
0 < \mu = \langle n\cdot H,\tilde{k}\cdot H\rangle = \pm \langle n\cdot H,H\rangle = \pm 1,
\end{eqnarray*}
since $n\cdot H - H \in \lnn$. It follows that $\mu=1$ and that $\tilde{k}=m\in M$. Finally, $nak=m\tilde{a}\tilde{n}$ yields (by uniqueness of the Iwasawa decomposition) that $k\in M$, and the lemma is proven.
\end{proof}

\subsubsection{Critical points}
For $H\in\la$, let $Z_N(H)$ denote the centralizer of $H$ in $N$\index{$Z_N(H)$, the centralizer of $H$ in $N$}. Recall that $\g=\lk+\lp$, where $\mathfrak{p}$ denotes the orthogonal complement (with respect to the Killing form) of $\mathfrak{k}$ in $\mathfrak{g}$.
\begin{lem}\label{H N K}
Let $H\in\la$, $n\in N$. Then
\begin{eqnarray*}
n\cdot H \in \lp \Longleftrightarrow n\in Z_N(H).
\end{eqnarray*}
\end{lem}
\begin{proof}
$n\cdot H\in\mathfrak{p}$ is satisfied if and only if $\langle Z, n\cdot H\rangle = 0 \,\,\forall\,Z\in\lk$. But $\mathfrak{a}\subseteq\mathfrak{p}$ yields $\langle Z,H\rangle = 0 \hspace{2mm} \forall Z\in\mathfrak{k}$ and since $n\in N$ we obtain $n\cdot H\in H+\mathfrak{n}$. Thus
\begin{eqnarray*}
0=\langle Z,n\cdot H\rangle \hspace{2mm} \forall Z\in\mathfrak{k} &\Longleftrightarrow& 0=\langle Z,n\cdot H - H\rangle \,\,\,\, \forall Z\in\mathfrak{k} \\
&\Longleftrightarrow& n\cdot H - H \in \mathfrak{p}\cap\mathfrak{n}.
\end{eqnarray*}
We may now use $\mathfrak{p}\cap\mathfrak{n} = \left\{0\right\}$, which follows from the fact that the elements of $\mathfrak{p}$ are semisimple, while the elements of $\mathfrak{n}$ are nilpotent. Thus $n\cdot H=H$, as desired.
\end{proof}

Assume that $(\mu,n,a,k)$ is a critical point in all variables. It follows from (a) that $\log(a) = 0$, that is
\begin{eqnarray}
a=e.
\end{eqnarray}
We use the notation of \eqref{phase function example} and Iwasawa decompose 
\begin{eqnarray}
nak=\tilde{k}\tilde{a}\tilde{n}.
\end{eqnarray}
Condition (d) yields
\begin{eqnarray}\label{condition}
0 = -\frac{\partial\psi}{\partial \theta}_{|\theta=0}(\mu,n,a,k\exp\theta Z) = \langle \tilde{n}\cdot Z,H\rangle = \langle Z,H^{\tilde{n}^{-1}}\rangle \,\,\,\,\,\,\,\,\, \forall Z\in\mathfrak{k}.
\end{eqnarray}
It follows from Lemma \ref{H N K} that $\tilde{n}\in N_H = Z_N(H)$, the centralizer of $H$ in $N$.

\begin{rem}
\begin{itemize}
\item[(1)] It is sufficient for \eqref{condition} to be satisfied only for all $Z\in\mathfrak{m}^{\bot}$, the orthogonal complement of $\mathfrak{m}$ in $\lk$. This can be seen as follows: If
$X = \tilde{n}\cdot H-H = X_{\mathfrak{p}} + X_{\mathfrak{m}} \in (\mathfrak{p}+\mathfrak{m})\cap\lnn$, then $2X_{\mathfrak{m}}=X+\theta(X)\in\mathfrak{m}\cap\mathfrak{m}^{\bot}=\left\{0\right\}$, so $X\in\mathfrak{p}\cap\lnn=\left\{0\right\}$, so $\tilde{n}\in N_H = Z_N(H)$. 
\item[(2)] However, given $\tilde{n}\in N$ and $Z\in\mathfrak{m}$, set $X=\tilde{n}^{-1}\cdot H - H\in\lnn$ and $Y=X+\theta(X)\in\mathfrak{m}^{\bot}$. Since $\mathfrak{m}\bot\la$ we have $2 \langle\tilde{n}\cdot Z,H\rangle = \langle Z,\tilde{n}^{-1}\cdot H - H\rangle + \langle Z,\tilde{n}^{-1}\cdot H - H\rangle = \langle Z,X\rangle + \langle Z,\theta(X)\rangle = \langle Z,Y\rangle = 0$, so \eqref{condition} holds for all $Z\in\mathfrak{m}$ and $\tilde{n}\in N$.
\end{itemize}
\end{rem}

Next, recall $a=e$ and also note that $\tilde{n}\cdot H=H$ is equivalent to $\tilde{n}^{-1}\cdot H=H$. We may then plug \eqref{phase function example} into equation (b) above and obtain the condition
\begin{eqnarray*}
0 = -\frac{\partial\psi}{\partial s}_{|s=0}(\mu,n\exp sX,a,k) = \langle \tilde{n}\cdot k^{-1}\cdot a^{-1}\cdot X,H \rangle = \langle X,k\cdot H \rangle \,\,\,\,\,\, \forall X\in\mathfrak{n}.
\end{eqnarray*}
It is immediate from Lemma \ref{p and n} that
\begin{lem}\label{follows from here 2}
Let $0\neq H\in\la$, $k\in K$. The following assertions are equivalent:
\begin{itemize}
\item[(i)] $\langle X,k\cdot H\rangle = 0 \,\,\,\,\,\, \forall X\in\mathfrak{n}$,
\item[(ii)] $k\cdot H \in \mathfrak{a}$.
\end{itemize}
\end{lem}

\subsubsection{Regular elements}
Let from now on $H\in\la_+$ be regular. Let $\lambda_H\in\la^*_{+}$ denote the linear functional on $\la$ given by $\lambda_0(X)=\langle X,H\rangle$ for $X\in\mathfrak{a}$ (Killing form). Then $\lambda_0\in\la^*_+$, the dual positive Weyl chamber. Also $H=H_{\lambda_0}$ in the notation of the Riesz representation (Section \ref{The Weyl group}). As before, we study the critical set of
\begin{eqnarray*}
\psi: \mathfrak{a}^*_+ \times N\times A \times K \rightarrow \rr, \hspace{3mm} (\mu,n,a,k) \mapsto \mu(\log(a)) - \langle H(nak),H_{\lambda_0}\rangle,
\end{eqnarray*}
Let $(\mu,n,a,k)$ be a critical point of $\psi$. We already know $a=e$. Lemma \ref{follows from here 2} states $k\cdot H\in\mathfrak{a}$. For regular elements we have the following refinement:
\begin{lem}\label{follows from here 3}
Let $0\neq H\in\la$, $k\in K$. The following assertions are equivalent:
\begin{itemize}
\item[(i)] $\langle X,k\cdot H\rangle = 0 \,\,\,\,\,\, \forall X\in\mathfrak{n}$,
\item[(ii)] $k=m'\in M'$, where $M'$ is the normalizer of $A$ in $K$.
\end{itemize}
\end{lem}

Since $H$ is regular, it follows that $k=m'\in M'$.

Next, we Iwasawa decompose $nak=\tilde{k}\tilde{a}\tilde{n}$. Then by the above observations we have $\tilde{n}\in  N_H$. But since $H$ is regular, Remark \ref{N on A} implies $\tilde{n}=e$. Using $a=e$ and $k=m'$ we then observe
\begin{eqnarray}
nm'=nak=\tilde{k}\tilde{a}\tilde{n}=\tilde{k}\tilde{a},
\end{eqnarray}
which implies (by uniqueness of the Iwasawa decomposition)
\begin{eqnarray}
n = \tilde{k}\tilde{a}(m')^{-1} = \tilde{k}{m'}^{-1}\tilde{\tilde{a}} \in N\cap KA = \left\{e\right\},
\end{eqnarray}
so $n=e$, $\tilde{k}=m'$ and $\tilde{a}=e$.

Condition (c) above and \eqref{phase function example} yield
\begin{eqnarray}\label{refer condition c}
0 = \frac{\partial\psi}{\partial t}_{|t=0}(\mu,n,a\exp tY,k) = \mu(Y) - \langle \tilde{n}\cdot k^{-1}\cdot Y,H \rangle.
\end{eqnarray}
Evaluating this at the critical point $(\mu,n,a,k)$, where $\tilde{n}=e$ and $k=m'$, we get
\begin{eqnarray}
\mu(Y) = \langle Y,m'\cdot H \rangle \,\,\, \forall Y\in\mathfrak{a}.
\end{eqnarray}
Recall that $H\in\la_+$ induces the linear function $\lambda_0(Y)=\langle Y,H\rangle$ ($Y\in\la$) on $\la$. It follows that $\mu\in\la^*_+$ is in the $W$-orbit of $\lambda_0\in\la^*_+$. Hence $k=m'=m\in M$ and $\mu=\lambda_0$. We summarize this as follows:

\begin{prop}\label{critical points calculus}
Let $H\in\la_+$ be regular. Write $\lambda_0(Y)=\langle Y,H\rangle$ ($Y\in\la$). The critical points $(\mu,n,a,k)$ of
\begin{eqnarray*}
\psi: \mathfrak{a}^*_+ \times N\times A \times K \rightarrow \rr, \hspace{3mm} (\mu,n,a,k) \mapsto \mu(\log(a)) - \langle H(nak),H_{\lambda_0}\rangle,
\end{eqnarray*}
are precisely
\begin{eqnarray}
(\mu,n,a,k) = (\lambda_0,e,e,m), \,\,\,\,\,\, m\in M.
\end{eqnarray}
On the quotient $\mathfrak{a}^*_+\times N\times A\times K/M$, the phase function $\psi$ has exactly one critical point, namely $(\mu,n,a,k) = (\lambda_0,e,e,M)$.
\end{prop}
\begin{proof}
We have seen that each critical point has this form. In the $K$-variable, $\psi$ is $M$-invariant, since $H:KAN\rightarrow\la$ is invariant. The proposition follows.
\end{proof}

\subsubsection{The Hessian form}\label{The Hessian form}
Let $X$ have rank one. Then $\mathfrak{a}=\rr H$, where $H\in\la_{+}$ is the unique vector such that $\|H\|=1$ (the norm on $\la$ induced by the Killing form). Let $\lambda_0\in\la^*_{+}$ denote the linear functional on $\la$ given by $\lambda_0(X)=\langle X,H\rangle$ for $X\in\mathfrak{a}$ (Killing form). Then $\mathfrak{a}^* = \rr \lambda_0$. Then $\lambda_0\in\la^*_+$, the dual positive Weyl chamber. Also $H=H_{\lambda_0}$.

We compute the Hessian form of $\psi$ for the rank one case. First, we note that the second order derivatives of $\psi$ are clear if they contain at least one derivative in direction $\mu$. We will now also compute the Hessian matrix of the $C^{\infty}$-function
\begin{eqnarray}
\phi_1: N\times A\times K\rightarrow\rr, \hspace{3mm} \phi_1(n,a,k)=\langle H(nak),H\rangle
\end{eqnarray}
at the points $(e,e,m)\in N\times A\times K$ and conclude that as a function on the quotient $\mathfrak{a}^*_+\times N\times\rr\times K/M$ the phase function $\psi$ has a non-degenerate Hessian form at the critical point $(\lambda_0,e,e,M)$. Note that under $\la\cong\rr$ we identify $\lambda_0$ with $1\in\rr$.

Recall that $H:KAN\rightarrow\la$ is a smooth mapping and hence $\phi_1$ is smooth as well. Derivatives of $\phi_1$ are given by
\begin{eqnarray}
\phi_1'(n,a,k)(X,Y,Z) \hspace{-1mm} = \hspace{-1mm} \langle \tilde{n}\cdot k^{-1}\cdot a^{-1}\cdot X, H \rangle \hspace{-1mm} + \hspace{-1mm} \langle \tilde{n}\cdot k^{-1}\cdot Y,H \rangle \hspace{-1mm} + \hspace{-1mm} \langle \tilde{n}\cdot Z,H\rangle,
\end{eqnarray}
where $(X,Y,Z)\in\mathfrak{n}\times\mathfrak{a}\times\mathfrak{k}$ and where $nak=\tilde{k}\tilde{a}\tilde{n}$ is written corresponding to the Iwasawa decomposition.

The Hessian form is bilinear, hence we must prove its non-degenerateness only with respect to a certain basis, which we will later construct. We will now successively fill up the following $3\times 3$-matrix of question marks, where each row and each column corresponds to the Lie algebra direction in which we differentiate:
\begin{eqnarray}
\begin{bmatrix} \hspace{1mm} & \mu & n & a_t & k \\ \mu & 0 & 0 & 1 & 0 \\ n & 0 & ? & ? & ? \\ a_t & 1 & ? & ? & ? \\ k & 0 & ? & ? & ? \end{bmatrix}
\end{eqnarray}

Because of symmetry we only have to consider the following $6$ cases:\\
(1) $X,X'\in\mathfrak{n}$. Then
\begin{eqnarray*}
XX'(\langle H(nak),H\rangle)|_{n=e,a=e,k=m} \\
&\hspace{-90mm}=& \hspace{-45mm} \frac{d}{dt}|_{t=0}\langle n(n\exp(tX)ak)k^{-1}a^{-1}\cdot X',H\rangle|_{n=e,a=e,k=m} \\
&\hspace{-90mm}=& \hspace{-45mm} \frac{d}{dt}|_{t=0} \hspace{1mm} 0 = 0,
\end{eqnarray*}
since at the critical points we have $n=e$, $k=m\in M$ and $a=e$, so the left vector in the Killing form is an element of $\mathfrak{n}$ for each $t$ in a neighborhood of $0\in\rr$ and $\mathfrak{n}\,\bot\,\mathfrak{a}$ with respect to the Killing form.\\
(2) $X\in\mathfrak{n}$, $Y\in\mathfrak{a}$. Then
\begin{eqnarray*}
\frac{d}{dt}|_{t=0}\langle n(na\exp(tY)k)k^{-1}a^{-1}\cdot X,H\rangle|_{n=e,a=e,k=m} = 0
\end{eqnarray*}
as above.\\
(3) $Y,Y'\in\mathfrak{a}$. Then
\begin{eqnarray*}
\frac{d}{dt}|_{t=0}\langle n(na\exp(tY')k)\cdot Y,H  \rangle|_{n=e,a=e,k=m} = \frac{d}{dt}|_{t=0}\langle Y,H\rangle = 0.
\end{eqnarray*}
(4) $Y\in\mathfrak{a}$, $Z\in\mathfrak{k}$. Then
\begin{eqnarray*}
\frac{d}{dt}|_{t=0}\langle n(na\exp(tY)k)\cdot Z,H\rangle|_{n=e,a=e,k=m} = 0,
\end{eqnarray*}
since the left vector in the Killing form is an element of $\mathfrak{k}$ and $\mathfrak{k}\bot\mathfrak{a}$ with respect to the Killing form (recall that $M$ and $A$ commute elementwise).\\
(5) $Z,Z'\in\mathfrak{k}$. Then
\begin{eqnarray*}
\frac{d}{dt}|_{t=0}\langle n(nak\exp(tZ'))\cdot Z,H\rangle|_{n=e,a=e,k=m} = 0,
\end{eqnarray*}
since $\mathfrak{k}\bot\mathfrak{a}$.\\
(6) $X\in\mathfrak{n}$, $Z\in\mathfrak{k}$. Then
\begin{eqnarray}\label{plug 1}
\frac{d}{dt}|_{t=0}\langle n(nak\exp(tZ))\exp(-tZ)k^{-1}a^{-1}\cdot X,H\rangle|_{n=e,a=e,k=m} \nonumber \\
&\hspace{-19cm}=& \hspace{-9.4cm} \frac{d}{dt}|_{t=0}\langle \exp(-tZ)m^{-1}\cdot X,H\rangle \nonumber \\
&\hspace{-19cm}=& \hspace{-9.4cm} \frac{d}{dt}|_{t=0}\langle\widetilde{X},\exp(tZ)\cdot H\rangle \hspace{2mm}(\widetilde{X}=m^{-1}\cdot X\in\mathfrak{n}) \nonumber \\
&\hspace{-19cm}=& \hspace{-9.4cm} \langle \widetilde{X},[Z,H]\rangle,
\end{eqnarray}
since $M$ normalizes $N$. This vanishes for $Z\in\mathfrak{m}$, since then $[Z,H]=0$.

We now analyse the last expression with respect to the transversal direction $\mathfrak{m}^{\bot}$. If $\alpha>0$ is a positive root, we find vectors $X_{\alpha}\in\mathfrak{g}_{\alpha}$ such that
\begin{eqnarray}\label{plug 2}
Z=\sum_{\alpha>0}\left( X_{\alpha} + \theta X_{\alpha}\right).
\end{eqnarray}
Plugging \eqref{plug 2} into the commutator-bracket of $\g$ we obtain
\begin{eqnarray}\label{plug 3}
[Z,H] &=& \left[\sum_{\alpha>0}\left( X_{\alpha} + \theta X_{\alpha}\right),H\right] \nonumber \\
&=& \sum_{\alpha>0} -\alpha(H)X_{\alpha} + \alpha(H)\theta X_{\alpha} \nonumber \\
&=& \sum_{\alpha>0} \alpha(H)\left( \theta X_{\alpha} - X_{\alpha} \right) \in \mathfrak{p}.
\end{eqnarray}
Next, we write $\widetilde{X}\in\mathfrak{n}$ as a sum
\begin{eqnarray}\label{plug 4}
\widetilde{X} = \sum_{\alpha>0}\widetilde{X}_{\alpha} \hspace{4mm} \left(\widetilde{X}_{\alpha}\in\mathfrak{g}_{\alpha}\right).
\end{eqnarray}
Plugging \eqref{plug 3} and \eqref{plug 4} into the Killing form \eqref{plug 1} we obtain
\begin{eqnarray}
\langle \widetilde{X},[Z,H]\rangle &=& \left\langle \sum_{\alpha>0} \alpha(H)\left( \theta X_{\alpha} - X_{\alpha} \right) , \sum_{\beta>0}\widetilde{X}_{\beta} \right\rangle \\
&=& \sum_{\alpha>0}\alpha(H)\left\langle \theta X_{\alpha} , \widetilde{X}_{\alpha} \right\rangle,
\end{eqnarray}
since $\langle \mathfrak{g}_{\alpha},\mathfrak{g}_{\beta}\rangle \neq 0 \Leftrightarrow \alpha+\beta=0$ and since $\theta X_{\alpha}\in\mathfrak{g}_{-\alpha}$.

Now let $\left\{X_1,...,X_s\right\}$ be a basis for $\mathfrak{n}$ consisting of root vectors such that $\left\langle X_j , \theta X_i \right\rangle =\delta_{ij}$. Then
\begin{eqnarray}
\left\{X_1+\theta X_1,...,X_s+\theta X_s\right\}
\end{eqnarray}
is a basis of $\mathfrak{m}^{\bot}$. Hence $\widetilde{X}$ and $Z\in\mathfrak{m}^{\bot}$ are linear combinations
\begin{eqnarray}
\widetilde{X} = \sum_j a_j X_j, \hspace{9mm} Z = \sum_j b_j(X_j+\theta X_j).
\end{eqnarray}
It follows that
\begin{eqnarray}
\langle \widetilde{X},[Z,H]\rangle = \sum_j \alpha_j(H) \, b_j \, a_j,
\end{eqnarray}
where $\alpha_j=\alpha$ if $X_j\in\mathfrak{g}_{\alpha}$. Hence in this basis, for $\lnn\times\mathfrak{m}^{\bot}$ the second derivatives $\langle \widetilde{X},[Z,H]\rangle$ at the critical points are given by the invertible diagonal matrix
\begin{eqnarray}
Q_0 := \begin{pmatrix} \alpha_1(H) & \hspace{1mm} & \hspace{1mm} \\ \hspace{1mm} & \ddots & \hspace{1mm} \\ \hspace{1mm} & \hspace{1mm} & \alpha_s(H) \end{pmatrix}.
\end{eqnarray}
Finally, with respect to this basis, the second derivatives of $\psi$ are
\begin{eqnarray}\label{Hessian form}
Q := \begin{bmatrix} \hspace{1mm} & s & \mathfrak{n} & t & \mathfrak{m} & \mathfrak{m}^{\bot} \\ s & 0 & 0 & 1 & 0 & 0\\ \mathfrak{n} & 0 & 0 & 0 & 0 & -Q_0 \\ t & 1 & 0 & 0 & 0 & 0 \\ \mathfrak{m} & 0 & 0 & 0 & 0 & 0 \\ \mathfrak{m}^{\bot} & 0 & -Q_0 & 0 & 0 & 0 \end{bmatrix}.
\end{eqnarray}
We drop the $\mathfrak{m}$-rows and columns, which describe the stable direction. 

\begin{thm}
The phase function $\psi$ has a non-degenerate Hessian form at its critical point $(\lambda_0,e,e,M)$.
\end{thm}

\subsubsection{Another phase function}
Let $X=G/K$ have rank one and as usual, denote by $H\in\la^+$ the unique unit vector. We also need to determine the critical points of $\psi_t: \rr^+\times A\times N\times K\rightarrow\rr$,
\begin{eqnarray*}
(\mu,n,a,k) \mapsto t(\mu-1) - \langle\log(a),H\rangle - \mu H(a^{-1}n^{-1}k).
\end{eqnarray*}
First, $d\psi_t(\mu,n,a,k)=0$ is equivalent to
\begin{itemize}
\item[(a)] $\frac{\partial\psi_t}{\partial \mu}(\mu,n,a,k) = 0$,
\item[(b)] $\frac{\partial\psi_t}{\partial t}_{|t=0}(\mu,n,a\exp tY,k) = 0 \, \textnormal{ for all } Y\in\mathfrak{a}$,
\item[(c)] $\frac{\partial\psi_t}{\partial s}_{|s=0}(\mu,n\exp sX,a,k) = 0 \, \textnormal{ for all } X\in\mathfrak{n}$,
\item[(d)] $\frac{\partial\psi_t}{\partial \theta}_{|\theta=0}(\mu,n,a,k\exp\theta Z) = 0 \, \textnormal{ for all } Z\in\mathfrak{k}$.
\end{itemize}
We may first consider the mapping $\phi_2: A\times N\times K \rightarrow\rr$,
\begin{eqnarray*}
\phi_2(a,n,k)=\langle H(ank),H\rangle.
\end{eqnarray*}

Given $(X,Y,Z)\in\la\times\lnn\times\lk$, the differential of $\phi_2$ at $(a,n,k)$ is (cf. Sec. \ref{Iwasawa derivatives})
\begin{eqnarray}\label{and}
\phi_2'(a,n,k)(X,Y,Z) \hspace{-1mm} = \hspace{-1mm} \langle \tilde{n}\cdot k^{-1}\cdot n^{-1}\cdot X, H \rangle \hspace{-1mm} + \hspace{-1mm} \langle \tilde{n}\cdot k^{-1}\cdot Y,H \rangle \hspace{-1mm} + \hspace{-1mm} \langle \tilde{n}\cdot Z,H\rangle.
\end{eqnarray}

Assume that $(\mu,n,a,k)$ is a critical point of $\phi_2$ and Iwasawa decompose $ank=\tilde{k}\tilde{a}\tilde{n}$. Then by \eqref{and}
\begin{eqnarray*}
\langle \tilde{n}\cdot Z,H\rangle = 0 \,\,\,\,\, \textnormal{ for all } Z\in\lk.
\end{eqnarray*}
It follows follows from Lemma \ref{H N K} that $\tilde{n}\in N_H$, where $N_H$ denotes the centralizer of $H$ in $N$. Since $G/K$ has rank one this yields $\tilde{n}=e$. Again by \eqref{and} we have that
\begin{eqnarray*}
\langle \tilde{n}\cdot k^{-1}\cdot Y,H \rangle = 0  \,\,\,\,\, \textnormal{ for all } Y\in\lnn.
\end{eqnarray*}
It follows from Lemma \ref{p and n} that $k=m'\in M'$, where $M'$ is the normalizer of $A$ in $K$. Then
\begin{eqnarray*}
anm' = ank = \tilde{k}\tilde{a}\tilde{n} = \tilde{k}\tilde{a},
\end{eqnarray*}
and by uniqueness of the Iwasawa decomposition this implies
\begin{eqnarray*}
\tilde{k} = anm'\tilde{a}^{-1} \,\, \Longrightarrow \,\, \tilde{k}=m'.
\end{eqnarray*}
Again by uniqueness of the Iwasawa decomposition we find
\begin{eqnarray*}
anm'=m'\tilde{a} \,\, \Longrightarrow \,\,  n=e.
\end{eqnarray*}

Now assume that $(\mu,n,a,k)$ is a critical point of $\psi_t$. Since the first two summands in the definition of $\psi_t$ are independent of $k$ and $n$, it follows that for the critical point of $\psi_t$ we have $n=e$ and $k=m'$ as well. Then by the assumption, we have for the derivatives of $\psi_t$ with respect to $a$ and $\mu$ (given by equations (a) and (b) above)
\begin{itemize}
\item[(i)] $\frac{\partial\psi_t}{\partial\mu} = t - \langle H(a^{-1}n^{-1}k),H\rangle = 0$,
\item[(ii)] $\psi_a = -1 + \mu\cdot\langle m'^{-1}\cdot X,H  \rangle = 0$.
\end{itemize}

Recall that we identify the unit vector $H\in\la_+$ with the real number $1\in\rr^+$. Condition (ii) yields
\begin{eqnarray*}
0 < 1/\mu = \langle H,m'\cdot H\rangle = \pm 1,
\end{eqnarray*}
since $M'$ is acting by orthogonal transformations on $\la$. It follows from $\mu>0$ that $\mu=1$ and $m'=m\in M$, where $M$ is the centralizer of $A$ in $K$. Evaluating (i) at the critical point we obtain
\begin{eqnarray*}
t - \langle H(a^{-1}m),H\rangle = 0 \,\, \Longrightarrow \,\, \log(a) = -t.
\end{eqnarray*}

Summarizing we have proven:
\begin{prop}
The critical points $(\mu,n,a,k)$ of $\psi_t$ are precisely
\begin{eqnarray*}
(\mu,n,a,k) = (1,e,a_{-t},m), \,\,\,\,\,\, m\in M.
\end{eqnarray*}
On the quotient $\rr^+\times N\times A\times K/M$, the phase function $\psi_t$ has exactly one critical point, namely $(\mu,n,a,k) = (1,e,a_{-t},M)$.
\end{prop}

The Hessian matrix of $\psi_t$ at the critical points is then given by
\begin{eqnarray}\label{Hessian t}
\begin{bmatrix} \hspace{1mm} & \mu & \mathfrak{n} & \la & \mathfrak{m} & \mathfrak{m}^{\bot} \\ \mu & 0 & 0 & 1 & 0 & 0\\ \mathfrak{n} & 0 & 0 & 0 & 0 & -Q_0 \\ \la & 1 & 0 & 0 & 0 & 0 \\ \mathfrak{m} & 0 & 0 & 0 & 0 & 0 \\ \mathfrak{m}^{\bot} & 0 & -Q_0 & 0 & 0 & f(t) \end{bmatrix},
\end{eqnarray}
where $Q_0$ is as in \eqref{Hessian form} and all other computations are exactly as in Subsection \emph{The Hessian form} of this ongoing section, and where $f(t)$ is the matrix
\begin{eqnarray}\label{f(t)}
f(t) := \frac{\partial^2}{\partial\theta_1\partial\theta_2} \langle a_{-t}\cdot o,k_{\theta_1 Z,\theta_2 Z'}M\rangle_{|\theta_1=\theta_2=0},
\end{eqnarray}
where $Z,Z'\in\mathfrak{m}^{\bot}$ and $k_{\theta Z}=\exp{\theta Z}$ for small $t$ and $Z\in\lk$ (we can restrict this to the critical point $m=e$, since all functions involved are $M$-invariant). We can also rewrite $f(t)$ in terms of \eqref{and} and the Killing form to obtain
\begin{eqnarray}\label{f(t) also}
f(t) = \frac{d}{d\theta}_{|\theta=0} \langle n(a_t k_{\theta Z'})\cdot Z,H\rangle.
\end{eqnarray}

\begin{ex}
If $G=PSU(1,1)$ and $K=PSO(2)$, then $G/K$ is identifiesd with the open unit disk $\mathbb{D}$. Then (\cite{Z86}, p. 103)
\begin{eqnarray}
f(t) = -2\tanh(t) \, \tanh(t+1)^{-2} \neq 0 \,\,\,\,\, \forall\, t\neq0.
\end{eqnarray}
\end{ex}

Note that given $a\in A$ and $k\in K$, the horocycle bracket $\langle a\cdot o,kM\rangle$ equals the Iwasawa projection $-H(a^{-1}k)$. It seems not to be easy to give a short derivation for an explicit formula for the matrix $f(t)$, but we observe the following: The matrix coefficients of the \emph{principal series of representations}\index{principal series representations} of $G$ (see Section \ref{Section Helgason boundary values}) may be expressed as integrals of the form
\begin{eqnarray}\label{becomes}
\int_K e^{\lambda(H(ak))} \, \phi(k) \, dk.
\end{eqnarray}
Here $a\in A$ and $\phi$ is an analytic function on $K$ expressed in terms of matrix coefficients of representations of $K$, and the eigenvalue parameter is $\lambda\in\la^*_{\cc}$. We can keep $Re(\lambda)=\nu$ fixed and absorb the factor $e^{\nu(H(ak))}$ into the amplitude $\phi$. We write $\xi=Im(\lambda)\in\la^*$ and denote by $H=H_{\xi}\in\la$ the vector satisfying $\xi(Y)=\langle Y, H_{\xi}\rangle$ with respect to the Killing form. Then \eqref{becomes} becomes
\begin{eqnarray}\label{to the}
\int_K e^{i\langle H(ak),H\rangle} \ \phi(k) \, dk.
\end{eqnarray}
If we replace $H$ by $\tau H$ and let $\tau\rightarrow\infty$, then the principle of stationary phase states that the main contributions to the asymptotic expansion of \eqref{to the} come from the critical points of the phase function $F_{a,H}$ on $K$ defined by
\begin{eqnarray}
F_{a,H}(k) = \langle H(ak), H\rangle, \,\,\,\,\,\,\,\,\,\,\,\,\,\, (k\in K).
\end{eqnarray}
These functions have been studied in \cite{DKV} (for proofs see Sections 5 and 6 loc. cit.). Let $K_a$, respectively $K_H$, denote the centralizer of $a$ in $K$, respectively of $H$ in $K$. The study of the critical points of $F_{a,H}$ reveals that the critical set of $F_{a,H}$ is equal (for $X$ being of arbitrary rank) to the disjoint union of smooth manifolds $K_a w K_H$, where $w$ runs through the Weyl group. Note that the notation $w K_H$ makes sense, as always $M\subseteq K_H$ for all $H\in\la$. The Hessians of the $F_{a,H}$ are tranversally non-degenerate to the critical manifolds. In particular, if $X$ has rank one, then the subgroup $M=Z_K(A)$ is a critical manifold for $F_{a,H}$ and its Hessian is non-degenerate in transversal direction. Since our matrix $f(t)$ equals the Hessian form of $F_{a_t,H}$ we can summarize:

\begin{thm}
The Hessian form $\psi_t$ is non-degenerate at the critical point.
\end{thm}

\newpage
\pagebreak
\thispagestyle{empty} 

\section{Equivariant pseudodifferential operators on symmetric spaces}\label{Pseudodifferential analysis on symmetric spaces}

The Euclidean Fourier transform\index{Euclidean Fourier transform} of a sufficiently regular function on $\rr^n$ is
\begin{eqnarray}\label{Euclidean Fourier transform}
\widehat{f}(\xi) = (2\pi)^{-n}\int f(x) e^{-ix\cdot\xi}\, dx.
\end{eqnarray}
Writing $D_j = -i (\partial/{\partial x_j})$\index{$D_j$, differential operator}, we differentiate the \emph{Fourier inversion formula}\index{Euclidean Fourier inversion formula}
\begin{eqnarray*}
f(x) = \int \widehat{f}(\xi)  e^{ix\cdot\xi} \, d\xi
\end{eqnarray*}
and get
\begin{eqnarray*}
D^{\alpha} f(x) = \int \xi^{\alpha} \widehat{f}(\xi) e^{ix\cdot\xi} \, d\xi,
\end{eqnarray*}
where $\alpha\in\nn_0^n$. Hence for a differential operator $p(x,D)=\sum_{|\alpha|\leq k} a_{\alpha}(x)D^{\alpha}$.
\begin{eqnarray*}
p(x,D)f(x) = \int p(x,\xi) \widehat{f}(\xi) e^{ix\cdot\xi} \, d\xi;
\end{eqnarray*}
The function
\begin{eqnarray*}
p(x,\xi)=\sum_{\alpha\leq k}a_{\alpha}(x)\xi^{\alpha}
\end{eqnarray*}
is called the \emph{full symbol}\index{full symbol} of the operator $p(x,D)$. These observations lead to the Euclidean version of pseudodifferential operators on the Euclidean space $\rr^n$ (\cite{Tay81}, \cite{Hor3}). As described in the introduction, pseudodifferential operators can be very useful in determining the asymptotic behaviour of the eigenvalues and eigenfunctions of the Laplace operator. In 1986, Steve Zelditch (\cite{Z86}) presented a calculus of pseudodifferential operators that in that case of the unit disk $D$ and a corresponding compact hyperbolic surface $X_{\Gamma}=\Gamma\backslash D$, where $\Gamma\subset PSU(1,1)$ is a cocompact discrete subgroup, is best adapted for this purpose. The main idea is to use Helgason's non-Euclidean Fourier analysis in place of the local Euclidean Fourier analysis in manifolds. An advantage of this calculus lies in its equivariance and invariance properties: $\Gamma$-invariant symbols define $\Gamma$-invariant operators on $T^{*}D$. Other objects of interest in $\Psi$DO-theory, such as lower terms in asymptotic expansions, are invariantly defined in this calculus, too.

In this section we generalize parts of this calculus to symmetric spaces of the noncompact type. Eventually we will have to restrict some results to the case of rank one symmetric spaces.

\subsection{Non-Euclidean Fourier analysis}\label{Non-Euclidean Fourier analysis}
The non-Euclidean Fourier transform\index{Fourier transform} $F$ (\cite{He94}) converts sufficiently regular functions $f$ on $X$ (e.g. $f\in C_c^{\infty}(X)$) into functions $Ff=\tilde{f}$ on $\la^{*}_{\cc} \times K/M$. This integral transform was introduced by S. Helgason in 1965 (\cite{He65}) and shows a lot of analogies with the Euclidean Fourier-transform (\cite{Hor1}). There is an inversion formula, a Plancherel formula, and a non-Euclidean Paley-Wiener theorem. Let $f$ be a complex valued function on $X$. Its \emph{non-Euclidean Fourier transform}\index{non-Euclidean Fourier transform} $Ff=\tilde{f}$ is defined by
\begin{eqnarray}
Ff(\lambda,b) := \tilde{f}(\lambda,b) := \int_{X}f(x)e^{(-i \lambda+\rho)\left\langle x,b \right\rangle}dx
\end{eqnarray}
for all $\lambda\in \la^{*}_{\cc}$, $b\in B$, for which the integral exists.

\begin{prop}\label{rapidly decreasing}
Let $u\in C_c^{\infty}(X)$. Then $\tilde{u}(\lambda,b)$ is rapidly decreasing in $\lambda$.
\end{prop}
\begin{proof}
We use \eqref{character of the Laplacian} and iterate integration by parts via the Laplace operator $L_X$ (see Section \ref{Invariant differential operators}):
\begin{eqnarray*}
\tilde{u}(\lambda,b) &=& \int_X e^{(-i\lambda+\rho)A(z,b)}u(z)dz \\
&=& \int_X  \left(\frac{-1}{\langle\lambda,\lambda\rangle + \langle\rho,\rho\rangle}\right) L_X e^{(-i\lambda+\rho)A(z,b)} u(z)dz \\
&=& \int_X  \left(\frac{-1}{\langle\lambda,\lambda\rangle + \langle\rho,\rho\rangle}\right) e^{(-i\lambda+\rho)A(z,b)} L_X u(z)dz \\
&=& \int_X  \left(\frac{-1}{\langle\lambda,\lambda\rangle + \langle\rho,\rho\rangle}\right)^k e^{(-i\lambda+\rho)A(z,b)} L_X^k u(z)dz.
\end{eqnarray*}
This proves the proposition.
\end{proof}

As usual, we denote \emph{Harish-Chandra's} $c$\emph{-function} by $c(\lambda)$. Explicit formulas for the \emph{Plancherel density} $|c(\lambda)|^{-2}\in C^{\infty}(\la)$ can be found in Section \ref{Special functions and the Plancherel density}. We introduce the notation
\begin{eqnarray}
\dbar\lambda=|c(\lambda)|^{-2}d\lambda.
\end{eqnarray}

Let $w=|W|$ denote the order of the Weyl group. In analogy with the inversion formula for the Euclidean Fourier transform we have (\cite{He94}, pp. 225-226):
\begin{thm}[Fourier inversion formula]
For each $f\in\mathcal{D}(X)$ the Fourier transform is inverted by the formula
\begin{eqnarray}\label{Fourier inversion formula}
f(x) = w^{-1}\int_{\mathfrak{a}^*}\int_{B} e^{(i\lambda+\rho)\left\langle x,b\right\rangle} \tilde{f}(\lambda,b)\dbar\lambda db, \hspace{4mm} x\in X.
\end{eqnarray}
\end{thm}

Let $B(\cdot,\cdot)$ denote the restriction to $\mathfrak{a}$ of the Killing form of $\mathfrak{g}$. Given $\lambda\in \mathfrak{a}^*$, we denote by $H_{\lambda}\in\mathfrak{a}$ the uniquely determined element such that
\begin{eqnarray}
B(H_{\lambda},H) = \lambda(H) \,\, \forall \, H\in \mathfrak{a}
\end{eqnarray}
Recall that we denote the \emph{dual positive Weyl chamber}\index{dual positive Weyl chamber}, that is the \emph{preimage} (under the mapping $\lambda\mapsto H_{\lambda}$) \emph{of the positive Weyl chamber}\index{preimage of the positive Weyl chamber} $\la^+$, by
\begin{eqnarray}\label{preimage}
\mathfrak{a}_{+}^* = \left\{ \lambda\in \mathfrak{a}^*: H_{\lambda}\in \mathfrak{a}^+ \right\}.
\end{eqnarray}

The following theorem (\cite{He94}, p. 227) is the symmetric space analog of the Plancherel formula\index{Plancherel formula} for the Euclidean Fourier transform.
\begin{thm}[Plancherel formula]\label{Plancherel formula}\index{Plancherel formula}
The Fourier transform $f(x)\mapsto\tilde{f}(\lambda,b)$ extends to an isometry of $L^2(X)$ onto $L^2(\mathfrak{a}_{+}^*\times B, {\left|c(\lambda)\right|}^{-2}\,d\lambda\,db)$\index{${\left|c(\lambda)\right|}^{-2}d\lambda db$, Plancherel measure}.
For $f\in L^2(X)$, the Plancherel formula\index{Plancherel formula} reads
\begin{eqnarray}
\int_X f_1(x)\overline{f_2(x)}dx = w^{-1} \int_{\mathfrak{a}^* \times B} \tilde{f_1}(\lambda,b) \overline{\tilde{f_2}(\lambda,b)} |c(\lambda)|^{-2}d\lambda db
\end{eqnarray}
\end{thm}

Given $\lambda\in\mathfrak{a}_{\cc}^{*}$, we can find $\xi,\mu \in \mathfrak{a}^{*}$ such that $\lambda = \xi + i\mu$, where $i=\sqrt{-1}$. We employ the notation $\Im\lambda=\mu$ and $|\lambda|=(|\xi|^2+|\mu|^2)^{1/2}$. A $C^{\infty}$-function $\psi(\lambda,b)$ on $\mathfrak{a}_{\cc}^{*}\times B$, holomorphic in $\lambda$, is called a holomorphic function of \emph{uniform exponential type}\index{uniform exponential type} if there exists a constant $R\geq 0$ such that for each $N\in \nn$
\begin{eqnarray}\label{Paley-Wiener}\index{Paley-Wiener estimates}
\sup_{\lambda\in\mathfrak{a}_{\cc}^{*},\, b\in B}e^{-R|\Im \lambda|}(1+|\lambda|)^N|\psi{(\lambda,b)|<\infty}.
\end{eqnarray}
We denote the space of $\psi$ satisfying \eqref{Paley-Wiener} by $\mathcal{H}^{R}(\mathfrak{a}_{\cc}^{*}\times B)$\index{$\mathcal{H}^{R}(\mathfrak{a}_{\cc}^{*}\times B)$, space of functions} and define\index{$\mathcal{H}(\mathfrak{a}_{\cc}^{*}\times B)$, space of functions}
\begin{eqnarray}
\mathcal{H}(\mathfrak{a}_{\cc}^{*}\times B) := \bigcup_{R>0}\mathcal{H}^{R}(\mathfrak{a}_{\cc}^{*}\times B).
\end{eqnarray}
By $\mathcal{H}(\mathfrak{a}_{\cc}^{*}\times B)_{W}$\index{$\mathcal{H}(\mathfrak{a}_{\cc}^{*}\times B)_{W}$, space of functions} we denote the space of functions $\psi\in \mathcal{H}(\mathfrak{a}_{\cc}^{*}\times B)$ satisfying
\begin{eqnarray}\label{Weyl}
\int_B e^{(is\lambda + \rho)(A(x,b))}\psi(s\lambda,b)db = \int_B e^{(i\lambda + \rho)(A(x,b))}\psi(\lambda,b)db
\end{eqnarray}
for all $s\in W$, $\lambda\in \mathfrak{a}_{\cc}^{*}$ and $x\in X$.

The following theorems (\cite{He94}, Ch. III, Theorem 5.1 and \cite{He94}, Ch. III, Corollary 5.9) are the symmetric space versions of the Paley-Wiener theorems for the Fourier transform and answers the questions concerning the range of the Fourier transform.

\begin{thm}
The Fourier transform $f(x)\mapsto\tilde{f}(\lambda,b)$ is a bijection of $\mathcal{D}(X)$ onto $\mathcal{H}(\mathfrak{a}_{\cc}^{*}\times B)_{W}$.
\end{thm}

A $C^{\infty}$-function $\psi$ on $\mathfrak{a}_{\cc}^{*}\times B$, holomorphic in $\lambda$, is called a \emph{holomorphic function of uniform exponential type and slow growth} if there exist constants $R,C\geq 0$ and $N\in\nn$ such that
\begin{eqnarray}\label{slow}
|\psi(\lambda,b)|\leq C(1+|\lambda|)^N e^{R|Im\lambda|}
\end{eqnarray}
for all $\lambda \in \mathfrak{a}_{\cc}^{*}$ and $b\in B$. Given $R\geq 0$, let $\mathcal{K}^{R}(\mathfrak{a}^{*}_{\cc}\times B)$\index{$\mathcal{K}^{R}(\mathfrak{a}^{*}_{\cc}\times B)$, space of functions} denote the space of these $\psi$ satisfying \eqref{slow} for some $N$ and $C$. We then define\index{$\mathcal{K}(\mathfrak{a}_{\cc}^{*}\times B)$, space of functions}
\begin{eqnarray}
\mathcal{K}(\mathfrak{a}_{\cc}^{*}\times B) := \bigcup_{R\geq 0}\mathcal{K}^{R}(\mathfrak{a}_{\cc}^{*}\times B).
\end{eqnarray}
Let $\mathcal{K}(\mathfrak{a}_{\cc}^*\times B)_W$\index{$\mathcal{K}(\mathfrak{a}_{\cc}^*\times B)_W$, space of functions} denote the space of functions in $\mathcal{K}(\mathfrak{a}_{\cc}^{*}\times B)$ satisfying \ref{Weyl}.

\begin{thm}\label{Paley-Wiener for distributions}
The distributional Fourier transform $T\mapsto\widetilde{T}$, where
\[\widetilde{T}(\lambda,b) = \int_X e^{(-i\lambda+\rho)\left\langle x,b\right\rangle} \, dT(x),\]
is a bijection of $\mathcal{E}'(X)$ onto the space $\mathcal{K}(\mathfrak{a}_{\cc}^*\times B)_W$.
\end{thm}

\subsection{Invariance and equivariance properties}\label{Invariance properties}
In this section we describe important invariance properties of operators defined using the non-Euclidean Fourier transform.

The group action of $G$ on $X$ induces a translation of functions on $X$: Given $g\in G$ and a function $f$ on $X$, we denote by $T_g f$\index{$T_g$, translation} the function $T_gf(z)=f(gz)$. A function $a(z,\lambda,b)$ on $X\times\mathfrak{a}\times B$ is called \emph{invariant} under translation (on $X\times B$) by\index{invariant function} $g$ if and only if
\begin{eqnarray}\label{invariant function}
a(gz,\lambda,gb)=a(z,\lambda,b) \textnormal{ for all } (z,\lambda,b).
\end{eqnarray}
Functions on $X\times B$ are identified with functions on $G/M$ and we call a function $a$ on $G/M$ \emph{invariant under translation} by $g$ if and only if $a(ghM)=a(hM)$ for all $g,h\in G$. Let $f$ be a function on $X\times X$. For $g,h\in G$ we define $T_{g,h}f$\index{$T_{g,h}$, translation} by $(T_{g,h}f)(z,w):=f(gz,hw)$. A function $f$ on $X\times X$ is called \emph{invariant} under $g\in G$ if and only if $T_{g,g}f=f$.

Let for a moment $\langle\cdot,\cdot\rangle$ denote the duality bracket of $C_c^{\infty}(X)$. Given a distribution $u$ on $X$ we define the distribution $T_g u$ on $X$ via duality by
\begin{eqnarray}
\langle T_g u, v \rangle := \langle u, T_{g^{-1}} v \rangle, \hspace{3mm} v \in C_c^{\infty}(X).
\end{eqnarray}
This definition is consistent with the imbedding \eqref{imbedding of functions} $C_c^{\infty}(X) \hookrightarrow \mathcal{E}'(X)$: Given a function $u\in C_c^{\infty}(X)$ one has
\begin{eqnarray}
\langle T_g u, v \rangle = \int_X u(gz)v(z)\, dz = \int_X u(z)v(g^{-1}z) \, dz = \langle u, T_{g^{-1}}v \rangle,
\end{eqnarray}
since $dz$ is $G$-invariant. If $u$ is a distribution on the product space $X\times X$, we define the distribution $T_{g,h}u$ on $X\times X$ via duality on the algebraic tensor product by defining it on the tensor products $\varphi\otimes\psi\in C_c^{\infty}(X\times X)$, where $\varphi,\psi\in C_c^{\infty}(X)$, by
\begin{eqnarray}
\langle T_{g,h}u, \varphi\otimes\psi \rangle := \langle u, T_{g^{-1}}\varphi\otimes T_{h^{-1}}\psi \rangle.
\end{eqnarray}
This definition is again consistent with the imbedding of functions into distributions.

\begin{defn}
\begin{itemize}
\item[(1)] Let $A$ be an operator with Schwartz kernel $k_A$. We say that $k_A$ is \emph{properly supported} if the projections of $X\times X$ to each factor when restricted to the support of the kernel are proper mappings.
\item[(2)] We say that an operator $A$ is \emph{properly supported}\index{properly supported operator} provided $A, A^*: C_c^{\infty}(X)\rightarrow C_c^{\infty}(X)$, hence $A, A^*: C^{\infty}(X)\rightarrow C^{\infty}(X)$, where $A^*$ is the adjoint of $A$ with respect to the $L^2(X)$-inner product. $A$ is properly supported if and only if its kernel is.
\end{itemize}
\end{defn}

\begin{lem}\label{lemma invariant}
Let $A:C^{\infty}(X)\rightarrow C^{\infty}(X)$ denote a linear and continuous operator with properly supported Schwartz kernel $k_A$, viewed as a distribution on $X\times X$. Then $T_g$ commutes with $A$ (i.e. $T_gAu(z)=AT_gu(z)$) if and only if $k_A$ is invariant under the action of $g$.
\end{lem}
\begin{proof}
Let $\langle\cdot,\cdot\rangle$ denote the pairing of distributions and test functions. Then
\begin{eqnarray*}
\langle T_gAu, v \rangle &=& \langle Au, T_{g^{-1}}v \rangle \\
&=& \langle k_A, T_{g^{-1},e}(v\otimes u) \rangle \\
&=& \langle T_{g,e}k_A, v\otimes u \rangle % \left( = \int_X \int_X k_A(gz,w)u(w)v(z) \, dw \, dz \right)
\end{eqnarray*}
and
\begin{eqnarray*}
\langle AT_gu, v \rangle = \langle k_A, v\otimes (T_gu) \rangle % \left( = \int_X \int_X k_A(z,w)u(gw)v(z) \, dw \, dz \right)
.
\end{eqnarray*}
The algebraic tensor product $C_c^{\infty}(X)\otimes C_c^{\infty}(X)$ is dense in the test functions of $X\otimes X$ (\cite{T}, p. 530) and hence we obtain
\begin{eqnarray*}
T_g A=AT_g \,\, \Longleftrightarrow \,\, T_{g,e}k_A = T_{e,g^{-1}}k_A \,\, \Longleftrightarrow \,\, T_{g,g} k_A = k_A.
\end{eqnarray*}
This proves the lemma.
\end{proof}

Recall the notion of non-Euclidean plane waves \eqref{non-Euclidean plane waves}: Given $\lambda\in\la^*$, $b\in B$, the functions $e_{\lambda,b}:X\rightarrow\cc$ are defined by
\begin{eqnarray}
e_{\lambda,b}: x\mapsto e^{(i\lambda+\rho)\langle x,b\rangle}.
\end{eqnarray}

\begin{defn}\label{complete symbol}
Given a linear operator $A:C^{\infty}(X)\rightarrow C^{\infty}(X)$, we define the \emph{complete symbol}\index{complete symbol} (\emph{full symbol}\index{full symbol}) $a(z,\lambda,b)\in C^{\infty}(X\times \mathfrak{a}^*_+ \times B)$ of $A$ by
\begin{eqnarray}\label{definition symbol}
\left(Ae_{\lambda,b}\right)(z)=a(z,\lambda,b)e_{\lambda,b}(z).
\end{eqnarray}
The complete symbol is defined if $A:C^{\infty}(X)\rightarrow C^{\infty}(X)$. We will later see for which classes of operators this condition is satisfied.
\end{defn}

Let $u\in C_c^{\infty}(X)$. We will now use the Fourier inversion formula to represent $Au$ by an integral. The following observations have to be understood formally. We will later define concrete classes of symbols $a(z,\lambda,b)$ for which these computations make sense. We write
\begin{eqnarray*}
Au(z) &=& \int_{\mathfrak{a}^*_+}\int_B e^{(i\lambda+\rho)\langle z,b\rangle}a(z,\lambda,b)\tilde{u}(\lambda,b)\dbar\lambda db \\
&=& \int_X \hspace{1mm} \int_{\mathfrak{a}^*_+} \int_B  e^{(i\lambda+\rho)\langle z,b\rangle} e^{(-i\lambda+\rho)\langle w,b\rangle} a(z,\lambda,b) u(w) \dbar\lambda \, db \, dw.
\end{eqnarray*}
On the level of distributions we then have for the Schwartz kernel\index{distributional Schwartz-kernel of an operator}
\begin{eqnarray}
k_A(z,w) = \int_{\mathfrak{a}^*_+}\int_B e^{(i\lambda+\rho)\langle z,b\rangle}e^{(-i\lambda+\rho)\langle w,b\rangle}a(z,\lambda,b) \dbar\lambda \, db
\end{eqnarray}
in the sense that
\begin{eqnarray}\label{kernel}
\langle Au, v\rangle &=& \langle k_A, v\otimes u\rangle \\
&\hspace{-16mm}=& \hspace{-9mm} \int_X \int_X \int_{\mathfrak{a}^*_+}\int_B e^{(i\lambda+\rho)\langle z,b\rangle}e^{(-i\lambda+\rho)\langle w,b\rangle}a(z,\lambda,b) u(w)v(z)\dbar\lambda \, db \, dw \, dz \nonumber.
\end{eqnarray}

By the Fourier inversion formula, for properly supported kernels $k_A(z,w)$ we can then reconstruct the full symbol of $A$ by using the Helgason-Fourier transform of the kernel:
\begin{eqnarray}
a(z,\lambda,b) = \int_X e^{(i\lambda+\rho)(\langle w,b\rangle-\langle z,b\rangle)} k_A(z,w) \, dw.
\end{eqnarray}
We observe
\begin{eqnarray}\label{the integral becomes}
\langle T_{g,g}k_A, v\otimes u \rangle &=&
\int e^{(i\lambda+\rho)\langle z,b\rangle}e^{(-i\lambda+\rho)\langle w,b\rangle} a(z,\lambda,b) u(g^{-1}w) v(g^{-1}z) \dbar\lambda \,db\,dw\,dz \nonumber\\
&\hspace{-12mm}=& \hspace{-6mm} \int e^{(i\lambda+\rho)\langle gz,b\rangle}e^{(-i\lambda+\rho)\langle gw,b\rangle} a(gz,\lambda,b) u(w) v(z) \dbar\lambda \, db \, dw \, dz.
\end{eqnarray}
Recall the equation $\langle g\cdot z,g\cdot b\rangle = \langle z,b\rangle + \langle g\cdot o,g\cdot b\rangle$ (cf. \eqref{equivariance}), which implies $\langle g\cdot z,b\rangle = \langle z,g^{-1}b\rangle + \langle g\cdot o,b\rangle$. Similarly we obtain $\langle g\cdot w,b\rangle = \langle w,g^{-1}b\rangle + \langle g\cdot o,b\rangle$. Hence the integral \eqref{the integral becomes} becomes
\begin{eqnarray*}
\int e^{(i\lambda+\rho)\langle z,g^{-1}b\rangle}e^{(-i\lambda+\rho)\langle w,g^{-1}b\rangle} a(gz,\lambda,b) u(w) v(z) e^{+2\rho\langle g\cdot o, b\rangle}\dbar\lambda \, db \, dw \, dz.
\end{eqnarray*}
Also recall the formula $\frac{dg\cdot b}{db} = e^{-2\rho\left\langle g\cdot o, g\cdot b \right\rangle}$ from Subsection \ref{Horocycles brackets and the Iwasawa-projection} and change $g^{-1}\cdot b$ into $b$. This yields
\begin{eqnarray}\label{last integral}
\int e^{(i\lambda+\rho)\langle z,b\rangle}e^{(-i\lambda+\rho)\langle w,b\rangle} a(gz,\lambda,g\cdot b) u(w) v(z) \dbar\lambda \, db \, dw \, dz.
\end{eqnarray}

\begin{prop}\label{prop invariant}
Let $A:C^{\infty}(X)\rightarrow C^{\infty}(X)$ have properly supported kernel $k_A$. The following assertions are equivalent:
\begin{itemize}
\item[(1)] $T_g$ commutes with $A$.
\item[(2)] The symbol $a$ of $A$ is invariant under the action of $g$.
\item[(3)] $k_A$ is invariant under the action of $g$.
\end{itemize}
\end{prop}
\begin{proof}
It follows from the equivariance property \eqref{equivariance} for the horocycle bracket that $e_{\lambda,b}(gz) = e_{\lambda,g\cdot b}(g\cdot o) \, e_{\lambda,g^{-1}\cdot b}(z)$. Using this we compute
\begin{eqnarray*}
(T_gA e_{\lambda,b})(z) &=& a(gz,\lambda,b) \, e_{\lambda,b}(gz) \\
&=& a(gz,\lambda,b) \, e_{\lambda,g\cdot b}(g\cdot o) \, e_{\lambda,g^{-1}\cdot b}(z)
\end{eqnarray*}
and
\begin{eqnarray*}
(AT_g e_{\lambda,b})(z) &=& A(e_{\lambda,g\cdot b}(g\cdot o) \cdot e_{\lambda,g^{-1}\cdot b})(z) \\
&=& e_{\lambda,g\cdot b}(g\cdot o) \, A e_{\lambda,g^{-1}\cdot b}(z) \\
&=& e_{\lambda,g\cdot b}(g\cdot o) a(z,\lambda,g^{-1}\cdot b) e_{\lambda,g^{-1}\cdot b}(z).
\end{eqnarray*}
$(1)\Rightarrow(2)$: Assume $T_gA=AT_g$. Then $a$ must be invariant in the sense of \eqref{invariant function}.\\
$(2)\Rightarrow(3)$: Assume $a(gz,\lambda,g\cdot b)=a(z,\lambda,b)$ for all $(z,\lambda,b)$. Then the integral \eqref{last integral} equals $\langle k_A, v\otimes u \rangle$, which proves the invariance of $k_A$.\\
$(1)\Leftrightarrow(3)$: This is proven in Lemma \ref{lemma invariant}.
\end{proof}

\iffalse
Look inside the Schwartzraum: [He], GGA, S. 489. Recall the pullback-rule: $\int_M f \phi^{*}\omega = \int_N (f\circ\phi^{-1})\omega$ (pullback measure). Recall the substitution-diffeomorphism-rule ($\phi: M\rightarrow N$): $\int_N F(q) dq = \int_M F(\phi(p))|det \phi_p| dp$.
\fi

\subsection{Classes of symbols}\label{The calculus}
Let $\Gamma$ denote a cocompact discrete subgroup of $G$ and let $X_{\Gamma}$ denote the corresponding compact quotient $\Gamma\backslash X$. We now define classes of symbols $S_{\cl}^m$ and $S_{\cl,\Gamma}^m$ and establish $C^{\infty}$-continuities for corresponding classes of operators. If $X$ has rank one, the properly supported operators (in $\mathcal{L}_{\cl}^m$ and $\mathcal{L}_{\cl,\Gamma}^m$) are closed under composition and adjoints, and properly supported operators of order $0$ are $L^2$-continuous. In the beginning of this section we let the rank $r:=\dim(A)$ of $X$ be arbitrary.

\iffalse Under $X\times B\cong G/M$, a given function $a(z,\lambda,b)=a(gK,\lambda,g\cdot M)$ on $X\times\mathfrak{a}^*_+\times B$ identifies with the function $a(gM,\lambda)$ on $G/M\times\mathfrak{a}^*_+$. Let $|\cdot|$ denote the norm on $\mathfrak{a}^*_+$ induced by the Killing-form (see \ref{Decomosition of real semisimple Lie groups}). Let $n=\dim G = \dim\mathfrak{g}$ and let $\left\{X_1,...,X_n\right\}$ be a basis consisting of unit vectors of $\mathfrak{g}$ (the elements of $\mathfrak{g}$ are acting on functions on $G$ as left-invariant differential operators).
\fi

\begin{defn}
Let $\overline{\la^*_+}$ denote the closure in $\la^*$ of the positive Weyl chamber. A function $a\in C^{\infty}(X\times\overline{\la^*_+}\times B)$ is a \emph{symbol}\index{symbol of order $m\in\rr$}\emph{ of order} $m\in\rr$ if for all $\beta\in\nn_0^{r}$, for each differential operator $D$ on $X\times B$, and for each compact subset $C\subset\subset X$ it satisfies
\begin{eqnarray}\label{symbol estimates}\index{symbol estimates}
\|\partial_{\lambda}^{\beta} \, D \, a(z,\lambda,b)\| \leq C_{\beta,D}(C)(1+|\lambda|)^{m-|\beta|} \,\,\,\,\, \forall\,z\in C.
\end{eqnarray}
By $S^m$\index{$S^m$, space of symbols} we denote the space of symbols of order $m$.
\iffalse
\begin{itemize}
\item[(x)] A function $a\in C^{\infty}(G\times\la^*_+)$ is a \emph{symbol}\index{symbol of order $m\in\rr$}\emph{ of order} $m\in\rr$ if for all $\beta\in\nn_0^{r}$, $\alpha\in\nn_0^n$ and for each compact subset $C\subset G$ it satisfies
\begin{eqnarray}\label{symbol estimates}\index{symbol estimates}
\|\partial_{\lambda}^{\beta} \, X_1^{\alpha_1}\cdots X_n^{\alpha_n} a(g,\lambda)\| \leq C_{\beta}(C)(1+|\lambda|)^{m-|\beta|} \,\,\,\,\,\,\, \forall\,g\in C.
\end{eqnarray}
\item[(xx)]
A symbol $a\in S^m$ is said to be uniform of order $m$ if the constants $C_{\beta}$ are independent of the compact set $C$. We denote the space of such symbols by $\mathcal{B}S^m$. In this case, \eqref{symbol estimates} is equivalent to
\begin{eqnarray}
\|\partial_{\lambda}^{\beta} \, Da\| = O(|\lambda|^{m-|\beta|})
\end{eqnarray}
for any left-invariant differential operator $D$ on $G$.
\end{itemize}
\fi
\end{defn}

\begin{rem}
Suppose that $X$ has rank one. Then $\mathfrak{a}= \rr H$, where we choose $H$ as a generator of $\mathfrak{a}$ as the unique unit vector in the positive Weyl chamber. Then
$\mathfrak{a}^*_+=\rr\lambda_0$, where $\lambda_0$ is the linear functional $\lambda_0(X)=\langle X, H\rangle$, $X\in\mathfrak{a}$. We always identify $\rr=\la$ and $\rr=\la^*$. It follows that the multi-index $\beta\in\nn_0^{r}$ in \eqref{symbol estimates} is an integer $k\in\nn_0$ and \eqref{symbol estimates} becomes
\begin{eqnarray}
\|\partial_\lambda^k \, D \, a(z,\lambda,b)\| \leq C_{k,D}(C)(1+|\lambda|)^{m-k} \,\,\,\,\,\,\, \forall\,z\in C.
\end{eqnarray}
\end{rem}

\begin{defn}
A symbol $a(z,\lambda,b)$ is \emph{homogeneous}\index{homogeneous symbol} of \emph{degree} $m\in\rr$ if for all $t\geq 1$ and $|\lambda|\geq 1$ it satisfies
\begin{eqnarray}\label{homogeneous}
a(z,t\lambda,b) = t^m a(z,\lambda,b).
\end{eqnarray}
A symbol $a\in S^m$ is \emph{classical}\index{classical symbol} if it has an \emph{asymptotic expansion}\index{asymptotic expansion}
\begin{eqnarray}\label{classical}
a(z,\lambda,b) \sim \sum_{j=0}^{\infty}a_j(z,\lambda,b),
\end{eqnarray}
where the $a_j$ are symbols, homogeneous of degree $s_j$, such that $s_j\rightarrow-\infty$, $s_0=m$. Asymptotics\index{symbol asymptotics} is here denoted by $\sim$ and means that for all $N\geq1$
\begin{eqnarray}
\left(a-\sum_{j=0}^{N-1}a_j\right)\in S^{m-N}.
\end{eqnarray}
The space of classical symbols of order $m$ is denoted by $S_{\textnormal{cl}}^m$\index{$S_{\cl}^m$, space of classical symbols}.
\iffalse
\item[(iii)] If $a\in S^m$ satisfies $(a-\sum_{j=0}^{N-1}a_j)\in\mathcal{B}S^{m-N}$ it is called a \emph{uniform classical symbol}\index{uniform classical symbol} of order $m$.
\fi
The set of symbols which are invariant under the action of $\Gamma$ on $X\times B$ (see \ref{invariant function}) is denoted by $S^m_{\Gamma}$\index{$S^m_{\Gamma}$, space of $\Gamma$-invariant symbols}. By $S^m_{\cl,\Gamma}$\index{$S^m_{\cl,\Gamma}$, space of $\Gamma$-invariant classical symbols} we denote the space of $\Gamma$-invariant classical symbols. We will in most cases replace $a_j(z,\lambda,b)$ by $|\lambda|^{m-j}a_j(z,\lambda/{|\lambda|},b)$. \iffalse Since $\Gamma$ is cocompact, we find $\mathcal{B}S^m_{\Gamma}=S^m_{\Gamma}$ for the uniform symbols of order $m$.\fi
\end{defn}

\begin{prop}\label{equating coefficients}
\begin{itemize}
\item[(i)] Suppose $a(z,\lambda,b)$ is homogeneous of degree $m$ in $\lambda$ and $\phi$ is a smooth cutoff-function such that $\phi(\lambda)=0$ for $|\lambda|\leq C_1$ and $\phi(\lambda)=1$ for $|\lambda|\geq C_2>C_1$, then $\phi(\lambda)a(z,\lambda,b)$ is a symbol of order $m$.
\item[(ii)] If $a(z,\lambda,b)$ is a symbol of order $m$, then a $k$-th order derivative of $a$ with respect to $\lambda$ has order $m-k$.
\item[(iii)] Let $a$ and $b$ be symbols of order $m$ and $k$, respectively. Then the symbol $ab$ defined by $ab=a(z,\lambda,b)b(z,\lambda,b)$ has order $m+k$.
\item[(iv)] Let $a\in S^m$ such that $1/a \leq C(1+|\lambda|)^{-m}$. Then $1/a\in S^{-m}$.
\item[(v)] Let $a\in S_{\Gamma,\textnormal{cl}}^m$ such that $a\sim\sum_{j=0}^{\infty}a_j(z,\lambda,b)$. Then $a_j\in S_{\Gamma}^{m-j}$ for all $j\in\nn_0$.
\end{itemize}
\end{prop}
\begin{proof}
\emph{(i)}-\emph{(iv)} follow from the chain rule (\cite{Tay81}, p. 37). To prove \emph{(v)}, we note that \iffalse we copy the proof given in \cite{Z86}, pp. 78-79. The \fi the terms $a_j$ are uniquely determined by $a$:
\begin{eqnarray*}
a_0(z,\lambda/{|\lambda|},b)=\lim_{\lambda\rightarrow\infty}|\lambda|^{-m}a(z,\lambda,b).
\end{eqnarray*}
The other terms $a_j$ can be successively recovered. Then
\begin{eqnarray*}
\sum_{j=0}^{\infty}|\lambda|^{m-j}a_j(z,\lambda/{|\lambda|},b) \sim a(z,\lambda,b) = a(\gamma z,\lambda,\gamma b) \sim \sum_{j=0}^{\infty}|\lambda|^{m-j}a_j(\gamma z,\lambda/{|\lambda|},\gamma b),
\end{eqnarray*}
so $a_j(\gamma z,\lambda/{|\lambda|},\gamma b)=a_j(z,\lambda/{|\lambda|},b)$ for each $j$ and $\gamma$.
\end{proof}

\begin{defn}\label{class}
Given a symbol $a(z,\lambda,b)$ we define the corresponding pseudodifferential operator $A:=Op(a):=a(z,D)$\index{$Op$, pseudodifferential operator quantization} by
\begin{eqnarray*}
a(z,D)u(z) &=& \int_X\int_{\mathfrak{a^*}}\int_B e^{(i\lambda+\rho)\langle z,b\rangle}e^{(-i\lambda+\rho)\langle w,b\rangle}a(z,\lambda,b)u(w)db \dbar\lambda dw \\
& \hspace{-13mm} = &  \hspace{-6mm} \int_X\int_{\mathfrak{a^*}}\int_B e^{i\lambda(\langle z,b\rangle-\langle w,b\rangle)} e^{\rho(\langle z,b\rangle+\langle w,b\rangle)} a(z,\lambda,b)u(w)db \dbar\lambda dw.
\end{eqnarray*}
Then $A=Op(a)=a(z,D)$ acts on functions $u$ on $X$, for which the integral exists. We write OP$S^m=Op(S^m)$\index{OP$S^m$, pseudodifferential operators in $Op(S^m)$}.
\end{defn}

\begin{thm} Let $a\in S^m$. Then $A=Op(a)=a(z,D)$
\begin{itemize}
\item[(i)] is a continuous operator $A: C_c^{\infty}(X) \rightarrow C^{\infty}(X)$.
\item[(ii)] is a continuous operator $A: \mathcal{E}'(X) \rightarrow \mathcal{D}'(X)$.
\end{itemize}
\end{thm}
\begin{proof}
For \emph{(i)}, let $u\in C_c^{\infty}(X)$. The Fourier transform $\tilde{u}(\lambda,b)$ is rapidly decreasing (Prop. \ref{rapidly decreasing}). Hence $Au(z)$ and all of its derivatives are absolutely and uniformly convergent integrals. For \emph{(ii)}, let $u\in\mathcal{E}'(X)$. Then by Theorem \ref{Paley-Wiener for distributions} we have $|\tilde{u}(\lambda,b)|\leq C(1+|\lambda|)^n$ for some $C>0$ and $n>0$. Then for $v\in\mathcal{D}(X)$
\begin{eqnarray*}
\langle Au,v\rangle &=&\int_{X\times\mathfrak{a}^*_+\times B}e^{(i\lambda+\rho)A(z,b)}\overline{v(z)}a(z,\lambda,b)\tilde{u}(\lambda,b)\dbar\lambda dbdz\\
&=& \int_{\mathfrak{a}^*_+\times B}a_v(\lambda,b)\tilde{u}(\lambda,b)\dbar\lambda db,
\end{eqnarray*}
where (using integration by parts via $L_X$ as in the proof of Prop. \ref{rapidly decreasing})
\begin{eqnarray*}
a_v(\lambda,b) &=& \int_X e^{(i\lambda+\rho)A(z,b)}\overline{v(z)}a(z,\lambda,b)dz \\
&=& \left(\frac{+1}{\langle\lambda,\lambda\rangle + \langle\rho,\rho\rangle}\right)^k \int_X e^{(i\lambda+\rho)A(z,b)}L_X^k(\overline{v}a)dz.
\end{eqnarray*}
Thus $|a_v(\lambda,b)|\leq C_k(v,a)\left(\langle\lambda,\lambda\rangle + \langle\rho,\rho\rangle\right)^{m-k}$ for any $k\in\nn_0$, where $C_k(v,a)$ depends on the $C_{\supp v}^{2k}$-norm of $v$ \eqref{seminorms}. The order of the Plancherel density is $s:=\dim N$. Choose $k$ large enough to finish the proof.
\end{proof}

\begin{defn}\label{properly supported}
\begin{itemize}
\item[(1)] Let $\mathcal{L}^m$\index{$\mathcal{L}^m$, properly supported operators in OPS$^m$}, \iffalse $\mathcal{B}\mathcal{L}^m$\index{$\mathcal{BL}^m$, properly supported operators in OP$\mathcal{BS}^m$},\fi $\mathcal{L}^m_{\Gamma}$\index{$\mathcal{L}^m_{\Gamma}$, properly supported operators in OP$S^m_{\Gamma}$}, denote the properly supported operators with symbols in their respective symbol spaces. \iffalse $S^m$, $\mathcal{BS}^m$, and $S^m_{\Gamma}$.\fi
\item[(2)] Let $d_X(z,w)$ denote the non-Euclidean distance from $z\in X$ to $w\in X$. We say that $A\in\mathcal{L}^m$ is \emph{uniformly properly supported}\index{uniformly properly supported} if there exists a constant $d_0>0$ such that $k_A(z,w)=0$ for all $z$ and $w$ with $d(z,w)\geq d_0$.
\end{itemize}
\end{defn}

\subsection{The Kohn-Nirenberg operator}\label{The Kohn-Nirenberg operator}
For $G=SU(1,1)\cong SL(2,\rr)$, it is proven in \cite{Z86} that the non-Euclidean operator classes (\ref{class}) are contained in the classical space of pseudodifferential operators (\cite{Hor3}). Proofs of these facts are based on the equivalence of phase functions and amplitudes in the definitions of operators. We note that equivalence of phase functions generalizes to arbitrary symmetric spaces (see \cite{Z86} for references, similar results are announced by N. Anantharaman and L. Silberman). The problem is to show, at least in the case of rank one spaces, that the symplectic volume element of $T^*(G/K)$, if expressed in $(z,\lambda,b)$-coordinates, is asymptotically equivalent to the measure $e^{2\rho\langle z,b\rangle}\dbar\lambda db dz$. This is an open problem to me, and I will not go into any more detail at this point. In this section, we build up the analysis of the operator $U: C_c^{\infty}(X\times\mathfrak{a}^*_+\times B)\rightarrow C^{\infty}(X\times\mathfrak{a}^*_+\times B)$,
\begin{eqnarray}
&& \hspace{-10mm}Ua(z,\lambda,b) := \\
&& e^{-(i\lambda+\rho)\langle z,b\rangle}\int_X\int_{\mathfrak{a}^*_+}\int_B e^{(i\mu+\rho)\langle z,b'\rangle}e^{(i\lambda+\rho)\langle w,b\rangle}e^{(-i\mu+\rho)\langle w,b' \rangle} a(w,\mu,b')\dbar\mu dw db'.\nonumber
\end{eqnarray}
In the non-Euclidean calculus of pseudodifferential operators, proofs of many facts are based on the properties of this Kohn-Nirenberg operator\index{$U$, Kohn-Nirenberg operator}, which is the composition of the quantization map $a\mapsto Op(a)$ and the symbol map.

\begin{lem}\label{lemma U isometry}
$U$ is an isometry of $L^2(X\times\mathfrak{a}^*_+\times B,e^{2\rho\langle z,b\rangle}dz \dbar\lambda db)$.
\end{lem}
\begin{proof}
The Fourier inversion formula \eqref{Fourier inversion formula} says that each sufficiently regular function $f$ on $X$ satisfies
\begin{itemize}
\item[(1)] $f(z) = \int_{\mathfrak{a}^*_+}\int_B\int_X e^{(i\lambda+\rho)\langle z,b\rangle}e^{(-i\lambda+\rho)\langle w,b\rangle}f(w)\,dw\,\dbar\lambda\,db$,
\item[(2)] $\tilde{f}(\lambda,b) = \int_X\int_{\mathfrak{a}^*_+}\int_B e^{(-i\lambda+\rho)\langle z,b\rangle}e^{(i\mu+\rho)\langle z,\tilde{b}\rangle} \tilde{f}(\mu,\tilde{b}) \, \dbar\mu\,d\tilde{b}\,dz$.
\end{itemize}
Let $a\in L^2(X\times\mathfrak{a}^*_+\times B,e^{2\rho\langle z,b\rangle}dzdb\dbar\lambda)$ and for the moment, let $\langle\hspace{1mm}|\hspace{1mm}\rangle$ denote the $L^2$-inner product. Let the overline denote complex conjugation. Then $\langle Ua| Ua\rangle$ is the ninefold integral
\begin{eqnarray*}
\langle Ua| Ua\rangle &=& \int e^{(i\mu+\rho)\langle z,b_1\rangle} e^{(i\lambda+\rho)\langle w_1,b\rangle} e^{(-i\mu_1+\rho)\langle w_1,b_1\rangle} a(w_1,\mu_1,b_1) \\
&& \hspace{5mm} e^{(-i\mu+\rho)\langle z,b_2\rangle} e^{(-i\lambda+\rho)\langle w_2,b\rangle} e^{(i\mu_1+\rho)\langle w_2,b_2\rangle} \overline{a}(w_2,\mu_2,b_2) \\
&& \hspace{5mm} \dbar\lambda db dz \dbar\mu_1 db_1 dw_1 d\mu_2 db_2 dw_2,
\end{eqnarray*}
where integration is over $(X\times\mathfrak{a}^*_+\times B)\times(X\times\mathfrak{a}^*_+\times B)\times(X\times\mathfrak{a}^*_+\times B)$.

If we do the $dz\,\dbar\mu_2\,db_2$ integral first, it follows from formula (2) above that
\begin{eqnarray*}
\langle Ua| Ua\rangle &=& \int e^{(-i\mu_1+\rho)\langle w_1,b_1\rangle} e^{(i\lambda+\rho)\langle w_1,b\rangle} e^{(-i\lambda+\rho)\langle w_2,b\rangle} e^{(i\mu_1+\rho)\langle w_2,b_1\rangle} \\
&& \hspace{5mm} a(w_1,b_1,\mu_1) \overline{a}(w_2,b_1,\mu_1) \dbar\lambda db \dbar\mu_1 db_1 dw_1 dw_2.
\end{eqnarray*}
Doing the $dz\,\dbar\lambda\,db$ integral next, formula (1) above yields
\begin{eqnarray*}
\int e^{(-i\mu_1+\rho)\langle w_1,b_1\rangle}e^{(i\mu_1+\rho)\langle w_1,b_1\rangle} |a(w_1,\mu_1,b_1)|^2  dw_1 d\mu_1 db_1 = \langle a|a\rangle,
\end{eqnarray*}
and the lemma is proven.
\end{proof}

\begin{rem}\label{remark above}
Consider the operator
\begin{eqnarray*}
\widetilde{F}a(z,\lambda,b) &=& e^{-\rho\langle z,b\rangle} \int_X\int_{\mathfrak{a}^*_+}\int_B e^{(-i\lambda+\rho)\langle w,b\rangle} e^{-(i\mu+\rho)\langle z,b'\rangle} \\
&& \hspace{30.5mm} \times e^{\rho\langle w,b'\rangle} a(w,\mu,b')\, db'\,\dbar\mu \, dw.
\end{eqnarray*}
Using the Fourier inversion formula as above, one checks that $\widetilde{F}$ is an isometry of $L^2(X\times\mathfrak{a}^*_+\times B,e^{2\rho\langle z,b\rangle}dzdb\dbar\lambda)$. Then, $\widetilde{F}$ is inverted by
\begin{eqnarray*}
\widetilde{G}a(z,\lambda,b) &=& e^{-\rho\langle z,b\rangle} \int_X\int_{\mathfrak{a}^*_+}\int_B e^{(i\lambda+\rho)\langle w,b\rangle} e^{(i\mu+\rho)\langle z,b'\rangle} \\
&& \hspace{30.5mm} \times \, e^{\rho\langle w,b'\rangle} a(w,\mu,b')\, db'\,\dbar\mu \, dw.
\end{eqnarray*}
By definition we have $U = e^{-i\lambda\langle z,b\rangle}\widetilde{G}e^{-i\mu\langle w,b'\rangle}$ and $U^{-1} = e^{i\lambda\langle z,b\rangle}\widetilde{F}e^{i\mu\langle w,b'\rangle}$, which shows that $U$ is invertible.
\end{rem}

\begin{prop}\label{U unitary}
$U$ is a unitary operator on $L^2(X\times\mathfrak{a}^*_+\times B,e^{2\rho\langle z,b\rangle}dz \dbar\lambda db)$ and commutes with each $g\in G$, that is $UT_g=T_g U$ in the notation of \eqref{invariant function}.
\end{prop}
\begin{proof}
$U$ is unitary on $L^2(X\times\mathfrak{a}^*_+\times B,e^{2\rho\langle z,b\rangle}dz \dbar\lambda db)$ by Lemma \ref{lemma U isometry} and Remark \ref{remark above}. For a proof of $U T_g = T_g U$ note that $dw$ is $G$-invariant, so
\begin{eqnarray*}
Ua(gz,\lambda,gb) = \\
&&\hspace{-35mm} e^{-(i\lambda+\rho)\langle gz,gb\rangle}\int_X\int_{\mathfrak{a}^*_+}\int_B e^{(i\mu+\rho)\langle gz,b'\rangle}e^{(i\lambda+\rho)\langle gw,gb\rangle}e^{(-i\mu+\rho)\langle gw,b' \rangle} a(gw,\mu,b')\dbar\mu dw db'.
\end{eqnarray*}
If we substitute $b'\mapsto g\cdot b'$ and use $\langle gz,gb'\rangle=\langle z,b'\rangle+\langle g\cdot o,b'\rangle$ and $\frac{dg\cdot b'}{db'} = e^{-2\rho\left\langle g\cdot o, g\cdot b' \right\rangle}$, the integral becomes
\begin{eqnarray*}
&& e^{-(i\lambda+\rho)\langle z,b\rangle}\int_X\int_{\mathfrak{a}^*_+}\int_B e^{(i\mu+\rho)\langle z,b'\rangle}e^{(i\lambda+\rho)\langle w,b\rangle}e^{(-i\mu+\rho)\langle w,b' \rangle} a(gw,\mu,gb')\dbar\mu dw db' \\
&&\hspace{5mm} = \,\, U(a\circ g)(z,\lambda,b),
\end{eqnarray*}
where $(f\circ g)(z,\lambda,b)=f(gz,\lambda,gb)$. The proposition follows.
\end{proof}

\subsubsection{A convolution formula}
Given two functions $a$ and $b$ on $G$, at least one with compact support, their \emph{convolution product}\index{$\ast$, convolution product}\index{convolution product} $a\ast b$ is defined by
\begin{eqnarray}
(a\ast b)(h) := \int_G a(g)b(g^{-1}h)dg, \,\,\,\,\,\,\,\,\,\,\, h\in G.
\end{eqnarray}
Since $G$ is locally compact and unimodular we may substitute $g\mapsto hg$ and then change $g$ into $g^{-1}$. The unimodularity and the $G$-invariance of $dg$ yield
\begin{eqnarray}\label{convolution formula}
(a\ast b)(h) = \int_G a(hg^{-1})b(g)dg.
\end{eqnarray}
This convolution descends to convolution of $M$-invariant functions on $G$, which we also denote by $\ast$: If $\pi$ denotes the projection $G\rightarrow G/M$ and if $f$ is a function on $G/M$, then, $f\circ\pi$ is an $M$-right-invariant function on $G$. Convolution on $G/M$ is then defined via
\begin{eqnarray*}
(a\ast b)\circ \pi := (a\circ\pi)\ast(b\circ\pi),
\end{eqnarray*}
where $a$ and $b$ denote functions on $G/M$. Written out, this means
\begin{eqnarray}\label{convolution on G/M}
(a\ast b)(hM) = \int_G a(gM)b(g^{-1}hM)dg.
\end{eqnarray}
To see this is well-defined, let $a$ and $b$ be $M$-invariant functions on $G$, such that the convolution integral $a\ast b$ exists. Given $g\in G$, $m\in M$, observe
\begin{eqnarray}\label{convolution well-defined}
(a\ast b)(gm) = \int_G a(h)b(h^{-1}gm)dh = \int_G a(h)b(h^{-1}g)dh = (a\ast b)(g).
\end{eqnarray}
It follows that $a\ast b$ is invariant and thus the convolution product is well-defined.

We identify functions on $X\times B$ and functions on $G/M$. The non-Euclidean Fourier analysis is written in $X\times B$-terms (for example using horocycle bracket), but it is often more convenient to work with the space $G/M$ (and Iwasawa projections) instead. We then observe that under $G/M\cong X\times B$ \eqref{convolution on G/M} corresponds to the convolution on $X\times B$ defined by
\begin{eqnarray}
(a\ast b)(z,b) = \int_G a(g\cdot(o,M))b(g^{-1}\cdot(z,b))dg, \,\,\,\,\,\,\,\,\,\,\, (z,b)\in X\times B,
\end{eqnarray}
where $\cdot$ denotes the action of $G$ on $X\times B$. The integral exists whenever at least one of the functions $a$ and $b$ has compact support.

Given $\mu,\lambda\in\mathfrak{a}^*_+$, we write
\begin{eqnarray}
E_{\mu,\lambda}: X\times B \rightarrow \cc, \,\,\,\,\,\,\, (z,b)\mapsto e^{(i\mu+\rho)\langle z,M\rangle}e^{-(i\lambda+\rho)\langle z,b\rangle}.\index{$E_{\mu,\lambda}$, function}
\end{eqnarray}
In order to rewrite $E_{\mu,\lambda}$ in terms of $G/M$, recall that $(z,b)=(gK,g\cdot M)\in X\times B$ corresponds to $gM\in G/M$. We have $\langle z,b\rangle =-H(g^{-1}k(g)) = H(g)$ and $\langle z,M\rangle =-H(g^{-1})$ by Corollary \ref{corollary bracket}. Hence
\begin{eqnarray}\label{is well-defined}
E_{\mu,\lambda}: G/M \rightarrow \cc, \,\,\,\,\,\,\, gM\mapsto e^{-(i\mu+\rho)H(g^{-1})}e^{-(i\lambda+\rho)H(g)}.
\end{eqnarray}
Note that \eqref{is well-defined} is well-defined since the Iwasawa projection $H$ is $M$-biinvariant.

\begin{prop}\label{U formula one}
Let $a\in C_c^{\infty}(X\times\mathfrak{a}^*_+\times B)$. Then
\begin{eqnarray}
Ua(z,\lambda,b) = \int_{\mathfrak{a}^*_+}(a(\cdot,\mu,\cdot)\ast E_{\mu,\lambda})(z,b)\dbar\mu
\end{eqnarray}
\end{prop}
\begin{proof}
Note that we sometimes write $a(z,b,\lambda)$ instead of $a(z,\lambda,b)$ for simplicity of notation (when a group action is involved). Consider the integral
\begin{eqnarray*}
(a(\cdot,\mu,\cdot)\ast E_{\mu,\lambda})(z,b) &=& \int_G E_{\mu,\lambda}(g^{-1}\cdot (z,b)) a(g\cdot(o,M),\mu) \, dg \\
&=& \int_G e^{(i\mu+\rho)\langle g^{-1}z,M\rangle}e^{-(i\lambda+\rho)\langle g^{-1}z,g^{-1}\cdot b\rangle} a(g\cdot(o,M),\mu) \, dg.
\end{eqnarray*}
We fix $z\in X$, $\lambda,\mu\in\mathfrak{a}^*_+$, $b\in B$ and write
\begin{eqnarray*}
f(g) = e^{(i\mu+\rho)\langle g^{-1}z,M\rangle}e^{-(i\lambda+\rho)\langle g^{-1}z,g^{-1}\cdot b\rangle} a(g\cdot(o,M),\mu).
\end{eqnarray*}
We claim that $f$ is $M$-invariant and hence a function on $G/M$. The action of $m$ on $X\times B$ leaves $(o,M)\in X\times B$ fixed. Recall that $\langle z,b\rangle$ is invariant under the diagonal action of $K$ on $X\times B$. Thus $\langle m^{-1}g^{-1}z,m^{-1}g^{-1}\cdot b\rangle=\langle g^{-1}z,g^{-1}\cdot b\rangle$ and $\langle m^{-1}g^{-1}z,M\rangle=\langle m^{-1}g^{-1}z,m^{-1}M\rangle=\langle g^{-1}z,M\rangle$, and hence
\begin{eqnarray*}
f(gm) &=& e^{(i\mu+\rho)\langle m^{-1}g^{-1}z,M\rangle}e^{-(i\lambda+\rho)\langle m^{-1}g^{-1}z,m^{-1}g^{-1}\cdot b\rangle} a(gm\cdot(o,M),\mu) \\
&=& e^{(i\mu+\rho)\langle g^{-1}z,M\rangle}e^{-(i\lambda+\rho)\langle g^{-1}z,g^{-1}\cdot b\rangle} a(g\cdot(o,M),\mu) = f(g).
\end{eqnarray*}
We have $\langle g^{-1}\cdot z,M\rangle = \langle z,g\cdot M\rangle - \langle g\cdot o, g\cdot M\rangle$ and $\langle g^{-1}z,g^{-1}b\rangle = \langle z,b\rangle - \langle g\cdot o,b\rangle$ by Lemma \ref{formulae for bracket} and thus
\begin{eqnarray*}
f(g) &=& e^{(i\mu+\rho)\langle g^{-1}z,M\rangle}e^{-(i\lambda+\rho)\langle g^{-1}z,g^{-1}\cdot b\rangle} a(g\cdot(o,M),\mu) \\
&=& e^{(i\mu+\rho)(\langle z,g\cdot M\rangle-\langle g\cdot o,g\cdot M\rangle)}e^{-(i\lambda+\rho)(\langle z,b\rangle-\langle g\cdot o,b\rangle)} a(g\cdot o,\mu,g\cdot M).
\end{eqnarray*}
Then by Corollary \ref{m-invariant integral}
\begin{eqnarray*}
(a(\cdot,\mu,\cdot)\ast E_{\mu,\lambda})(z,b) &=& \int_G f(g) \, dg \\
&\hspace{-45mm}=& \hspace{-22mm}\int_X\int_B e^{(i\mu+\rho)(\langle z,b'\rangle - \langle w,b'\rangle)} e^{-(i\lambda+\rho)(\langle z,b\rangle-\langle w,b\rangle)} \, a(w,\mu,b') e^{2\rho\langle w,b'\rangle} \, dw \, db' \\
&\hspace{-45mm}=& \hspace{-22mm} e^{-(i\lambda+\rho)\langle z,b\rangle}\int_X\int_B e^{(i\mu+\rho)\langle z,b'\rangle} e^{(i\lambda+\rho)\langle w,b\rangle} e^{(-i\mu+\rho)\langle w,b' \rangle} a(w,\mu,b') \, dw \, db',
\end{eqnarray*}
and integrating against $\mu\in\mathfrak{a}^*_+$ yields
\begin{eqnarray*}
\int_{\mathfrak{a}^*_+}(a(\cdot,\mu,\cdot)\ast E_{\mu,\lambda})(z,b)\dbar\mu \\
&\hspace{-90mm}=& \hspace{-45mm} e^{-(i\lambda+\rho)\langle z,b\rangle} \int_{\mathfrak{a}^*_+}\int_X\int_B e^{(i\mu+\rho)\langle z,b'\rangle} e^{(i\lambda+\rho)\langle w,b\rangle} e^{(-i\mu+\rho)\langle w,b' \rangle} \, a(w,\mu,b')\dbar\mu \, dw \, db' \\
&\hspace{-90mm}=& \hspace{-45mm} Ua(z,\lambda,b),
\end{eqnarray*}
as desired.
\end{proof}

\subsubsection{Asymptotic expansions in the rank one case}\label{Asymptotic expansions}
Let $X=G/K$ have rank one. We identify $\la$ and $\la^*$ with $\rr$ by means of the Killing form and make no difference between $\la^*_{+}$ and the positive real numbers $\rr^+$: The unit vector $H\in\la^+$ is identified with the real number $1$. Let $a(z,\lambda,b)\in C_c^{\infty}(X\times B\times\la_+^*)$. Recall the definition
\begin{eqnarray*}
Ua(z,\lambda,b) = \\
&\hspace{-40mm}&\hspace{-30mm} \int_{X\times B\times\mathfrak{a}^*_{+}} \hspace{-3mm} e^{-(i\lambda+\rho)\langle z,b\rangle} e^{(i\mu+\rho)\langle z,b'\rangle} e^{(i\lambda+\rho)\langle w,b\rangle} e^{(-i\mu+\rho)\langle w,b'\rangle} a(w,\mu,b') |c(\mu)|^{-2} \, d\mu \, db' \, dw.
\end{eqnarray*}
We collect the $\lambda$-terms and the $\rho$-terms in the integral defining $Ua$, change variables from $\mu$ to $\lambda\mu$ and factor out $\lambda$ from the phase function to find
\begin{eqnarray}\label{coordinate change 0}
Ua(z,\lambda,b) &=& \int_{X\times B\times\mathfrak{a}^*_{+}} e^{i\lambda[\langle w,b\rangle-\langle z,b\rangle+\mu(\langle z,b'\rangle-\langle w,b'\rangle)]} \, e^{\rho[\langle w,b\rangle+\langle w,b'\rangle+\langle z,b'\rangle-\langle z,b\rangle]} \nonumber \\
&& \hspace{20mm} \times \,\,\,\, a(w,\lambda\mu,b') \, \frac{\lambda}{|c(\lambda\mu)|^{2}} \, dw \, db' \, d\mu.
\end{eqnarray}
Hence we have an oscillatory integral $Ua = \int e^{i\lambda\psi} \alpha dx$ with phase function 
\begin{eqnarray}\label{as given in}
\psi_{z,b}(w,\mu,b') = \langle w,b\rangle-\langle z,b\rangle+\mu(\langle z,b'\rangle-\langle w,b'\rangle).
\end{eqnarray}

Let $(z,b)=(g\cdot o,g\cdot M), \, (w,b')=(h\cdot o,h\cdot M)\in X\times B$ correspond to $gM\in G/M$ and $hM\in G/M$, respectively. Then by Corollary \ref{corollary bracket}
\begin{itemize}
\item[(1)] $\langle z,b\rangle = H(g)$,
\item[(2)] $\langle z,b'\rangle = -H(g^{-1}k(h)) = -H(g^{-1}h) + H(h)$,
\item[(3)] $\langle w,b\rangle = -H(h^{-1}k(g)) = -H(h^{-1}g) + H(g)$,
\item[(4)] $\langle w,b'\rangle = H(h)$.
\end{itemize}
It follows that in terms of $G/M$ the function $\psi_{z,b}$ has the form
\begin{eqnarray}\label{in this form}
\psi_{g}(h,\mu) = -H(h^{-1}g) - \mu(H(g^{-1}h)).
\end{eqnarray}
Note that $\psi_{g}(h,\mu)$ is right-$M$-invariant in both $g$ and $h$. Writing $h=nak$, we get for $gM=eM$
\begin{eqnarray*}
\psi_{eM}(n,a,k,\mu) = \log(a) - \mu H(nak).
\end{eqnarray*}
Writing $h^{-1}=nak$, we get for $gM=eM$
\begin{eqnarray*}
\psi_{eM}(n,a,k,\mu) = \mu\log(a) - H(nak).
\end{eqnarray*}
These functions are defined on $N\times A\times K/M\times\rr^+$. As proven in Subsection \ref{Critical sets and Hessian forms}, the unique critical point of $\psi_{eM}$ is $(n,a,kM,\mu)=(e,e,eM,1)$ (and the Hessian form at the critical point is non-degenerate). Under the natural diffeomorphisms
\begin{eqnarray*}
N\times A\times K/M \cong X\times B \cong G/M,
\end{eqnarray*}
the critical point corresponds to $(hM,\mu)=(eM,1)$ in $G/M\times\rr^+$, so if $\psi_{eM}=\psi$ is as in \eqref{in this form} and $gM=eM$, the critical point is $(hM,\mu)=(eM,1)$. But $\psi_{gM}(h,\mu)=\psi_{eM}(g^{-1}h,\mu)$ has the critical set $\left\{g^{-1}h\in M\right\}$, so the critical point of $\psi_{gM}(hM,\mu)$ is $(hM,\mu)=(gM,1)$ and corresponds to $(z,1,b)$ in $X\times\rr^+\times B$. This proves

\begin{lem}\label{lemma phase}
If $\psi_{z,b}(w,\mu,b')$ is as in \eqref{as given in}, then its unique critical point is $(w,\mu,b')=(z,1,b)$ and the Hessian form at this point is non-degenerate.
\end{lem}

\begin{thm}\label{Rearrangement}
Let $a(z,\lambda,b)\in S^m_{\textnormal{cl}}$ be compactly supported in $z$ (uniformly in the other variables). Then there exist $\widetilde{a}_k(z,\lambda,b)$, homogeneous of order $m-k$ for $\lambda\geq 1$, such that
\begin{eqnarray}\label{asymp00}
\left| Ua - \sum_{k=0}^{N-1}\widetilde{a}_k \right| \leq C_N(1+\lambda)^{m-N}.
\end{eqnarray}
\end{thm}
\begin{proof}
Let $\nabla=\nabla_w$ denote the gradient taken w.r.t. $w\in X$. Then the vector $\nabla \langle w,b\rangle$ has norm one for all $b\in B$, since it is a unit vector pointing along a geodesic orthogonal to level sets of $\langle w, b \rangle$ towards $b$ (\cite{E96}, Prop. 1.10.2). It follows that $\nabla\psi\neq 0$ for all $\mu>1$. We choose a cutoff $\chi(\mu)\in C_c^{\infty}(\rr^+)$ such that $\chi(\mu)=1$ in $[0,2]$ and $\chi(\mu)=0$ in $[3,\infty)$, and write
\begin{eqnarray}\label{as here}
Ua(z,\lambda,b)=Ia(z,\lambda,b)+IIa(z,\lambda,b)
\end{eqnarray}
corresponding to $1=\chi(\mu) + [1-\chi(\mu)]$. Then
\begin{eqnarray*}
IIa(z,\lambda,b) &=& \int_{X\times B\times\mathfrak{a}^*_{+}}[1-\chi(\mu)]\,e^{i\lambda[\langle w,b\rangle-\langle z,b\rangle+\mu(\langle z,b'\rangle-\langle w,b'\rangle)]}\\
&& \hspace{10mm} \times \,\, e^{\rho[\langle w,b\rangle+\langle w,b'\rangle+\langle z,b'\rangle-\langle z,b\rangle]} \, a(w,\lambda\mu,b') \, \frac{\lambda}{|c(\lambda\mu)|^{2}} \, dw \, db' \, d\mu.
\end{eqnarray*}
Since $\nabla\psi\neq0$ for $\mu>1$, the operator $L:=\frac{1}{i\lambda}(|\nabla\psi|^2)^{-1}\nabla\psi\cdot\nabla$ is defined in the support of the integrand. Then $L e^{i\lambda\psi}=e^{i\lambda\psi}$, so we can apply the transpose $L^t$ of $L$ to the amplitude. The order of the Plancherel density is $s:=\dim(N)$. Each $(|\nabla\psi|^2)^{-1}$ is at least $\mathcal{O}(\mu^{-1})$. Thus $|(L^t)^k(\alpha)|\leq C_k(\lambda\mu)^{m+1+s-k}$ at each point. Since $\alpha$ is compactly supported in $X\times B$, we have absolute and uniform convergence of $IIa(z,\lambda,b)$. Thus $IIa(z,\lambda,b)=\mathcal{O}(\lambda^{-\infty})$.

Recall that $g\cdot(z,b)=(g\cdot z,g\cdot b)$ denotes the diagonal action of $G$ on $X\times B$. A function $f(z,b)$ on $X\times B$ is pulled-back to an $M$-invariant function on $G$ via $f(g)=f(g\cdot o,g\cdot M)$. We denote by $f\circ g$ the function $(z,b)\mapsto f(g\cdot z,g\cdot b)$. Recall that $U$ commutes with translation by elements $g\in G$, that is $U(a\circ g)=(Ua)\circ g$. We write $(z,b)=g\cdot(o,M)$ corresponding to $X\times B\cong G/M$. The equivariance still holds if we insert $\chi(\mu)$ into the $\rr^+$-integral:
%\begin{eqnarray*}
%Ua(g,\lambda) &:=& Ua(z,\lambda,b) = U(a\circ g)(o,\lambda,M) \\
%&& \hspace{-25mm} \int_{X\times B\times\mathfrak{a}^*_{+}} \hspace{-6mm} \chi(\mu) \, e^{i\lambda[\langle w,M\rangle-\mu\langle w,b'\rangle]} \, e^{\rho[\langle %w,M\rangle+\langle w,b'\rangle]} \, \lambda \, a(g\cdot(w,b'),\lambda\mu) \, \frac{\lambda}{|c(\lambda\mu)|^{2}} \, dw \, db' \, d\mu \nonumber.
%\end{eqnarray*}
\begin{eqnarray}\label{B integral}
Ia(g,\lambda) &=& Ia(z,\lambda,b) = I(a\circ g)(o,\lambda,M) \\
&& \hspace{-25mm} \int_{X\times B\times\mathfrak{a}^*_{+}} \hspace{-6mm} \chi(\mu) \, e^{i\lambda[\langle w,M\rangle-\mu\langle w,b'\rangle]} \, e^{\rho[\langle w,M\rangle+\langle w,b'\rangle]} \, \lambda \, a(g\cdot(w,b'),\lambda\mu) \, \frac{\lambda}{|c(\lambda\mu)|^{2}} \, dw \, db' \, d\mu \nonumber.
\end{eqnarray}
The phase function $\psi_{o,M}(w,\mu,b') = \langle w,M\rangle-\mu\langle w,b'\rangle$ is non-degenerate at its critical point $(w,\mu,b')=(o,1,M)$. We can further assume (by using another cutoff around the critical point) that the integrand is supported in a coordinate patch around the critical point. All remainder integrals will again be $\mathcal{O}(\lambda^{-\infty})$, which follows from the standard principle of non-stationary phase for compactly supported amplitudes.

We use coordinates $x=(x_1,\ldots,x_{d},\mu)$, where $d:=\dim(X\times B)$, around the critical point $(w,b,\mu)=(o,M,1)$. In these coordinates, $(0,1)\in\rr_{(w,b)}^{d}\times\rr^+_{\mu}$ corresponds to $(o,M,1)$. Let $D=(\partial_{x_1},\ldots,\partial_{x_d},\partial_{\mu})$ and let $H_0^D$ denote the Hessian operator of $\psi=\psi_{o,M}$ at this point. The Taylor expansion of $\psi$ around $(0,1)$ is $\psi(x,\mu) = Q(x,\mu) + h(x,\mu)$, where $h$ vanishes up to order $3$ in $(0,1)$ and $Q(x,\mu)=\frac{1}{2}\langle H_0^D(x,\mu)|(x,\mu)\rangle$ (the customary inner product on $\rr^{d+1}$). Then
\begin{eqnarray}\label{expand0}
Ia(g,\lambda) = \int e^{i\lambda Q}\left\{ \chi \alpha e^{i\lambda h} \right\} \, dx \, d\mu + \mathcal{O}(\lambda^{-\infty}).
\end{eqnarray}
Set $s:=\dim(N)=\dim(B)$. Since $\tanh\sim 1$ and $\coth\sim 1$ to all orders, the Plancherel density is asymptotically a polynomial of degree $s$ (cf. \eqref{c-function}). For the asymptotics we can hence replace $|c(\nu)|^{-2}$ by a polynomial $p(\nu)=\sum_{j=1}^{s}c_j\nu^j$ of degree $s$ (without constant term). We split $Ua = \sum_j U_ja$, $Ia = \sum_j I_ja$ and $\alpha = \sum_j \alpha_j$ into the corresponding $s$ summands.

We start by assuming that $a$ is homogeneous of degree $m$. Then $a(z,\lambda,b)\sim\lambda^m a(z,1,b)$ (for $\lambda\rightarrow\infty$), so we can assume that the amplitude $\alpha_j$ in each $I_j(a)$ is homogeneous.

By \eqref{Hessian form} we can choose coordinates such that $\sign(H_0^D)=0$. We thus set $C_0=\frac{(2\pi)^{s+1}}{\sqrt{\det H_0^D}}$ and $R=\left( \frac{1}{2}(H_0^D)^{-1}D,D\right)$. The method of stationary phase yields
\begin{eqnarray}\label{expand}
U_ja(g,\lambda) \sim \frac{C_0}{\lambda^{s+1}}\sum_{k=0}^{\infty}\left(\frac{i}{\lambda}\right)^k\frac{1}{k!} R^k(\alpha_j \, e^{i\lambda h})|_{(x,\mu)=(0,1)}
\end{eqnarray}
in the sense that $|U_j(a)-\sum_{j=0}^{N-1}\widetilde{\alpha}_{j,k}|\leq C_{j}\lambda^{-N}\|\alpha_j\|$, where $\|\alpha_j\|$ is a seminorm of the amplitude (and still influences the order in $\lambda$, see \cite{Hor3}, Sec. 7.7.). We rearrange \eqref{expand} to provide a classical asymptotic expansion: Differentiations in $\mu$ preserve the order in $\lambda$, hence all differentiations do. Thus only derivatives of $e^{i\lambda h}$ affect the order of a term. But $h$ vanishes of order $3$, and so one needs three derivatives to bring down one $\lambda$. It follows that the $k$-th term has an order $\leq m - [k/3] - (s-j)$. Hence for each $l$ there are only finitely many $k$ such that the $k$-th term has an order $\geq l$. After rearrangement into homogeneous terms,
\begin{eqnarray}\label{fact}
U_ja(z,\lambda,b)\sim\sum_{k=0}^{\infty}(i/\lambda)^k \widetilde{\Lambda}_{j,k}(a)(g,o,\lambda,M),
\end{eqnarray}
where $\widetilde{\Lambda}_{j,k}(a)(g,o,\lambda,M)$ is homogeneous of degree $m$, and where $\widetilde{\Lambda}_{j,k}$ is a differential operator on $X\times B\times\rr^+$ of order $2k$ with coefficients in $g$. The $\widetilde{\Lambda}_{j,k}(a)(g,o,\lambda,M)$ are left-$G$-invariant and right-$M$-invariant in $g$, since each $U_j$ is invariant. They also decrease supports, so by Peetre's theorem they define differential operators on $G/M\times[1,\infty)\cong X\times B\times[1,\infty)$. Hence
\begin{eqnarray}
U_ja(z,\lambda,b)\sim\sum_{k=0}^{\infty}(i/\lambda)^{k}\lambda^m \Lambda_{j,k}(a)(z,1,b),
\end{eqnarray}
where $\Lambda_{j,k}$ is a left-invariant differential operator of order $2k$. Summing up we find $U(a)(z,\lambda,b) \sim \sum_k \Lambda_k(a)(z,\lambda,b)$, where $\Lambda_k$ is a differential operator of order $2k$, and the sum can be rearranged into homogeneous summands. If $a$ is not homogeneous, then $a\sim\sum a_k$, and we again rearrange to provide $Ua\sim\sum \widetilde{a}_{k}$, where the order of $\widetilde{a}_k$ is $m-k$, and $\left| Ua - \sum_{k=0}^{N-1}\widetilde{a}_k \right| \leq C_N(1+\lambda)^{m-N}$.
\end{proof}

\begin{rem}\label{subsequent remark 0}
\iffalse 
Summing up we get
\begin{eqnarray}
\left| U(a) - \sum_{k=0}^{N-1} \left(\frac{i}{\lambda}\right)^k \Lambda_{k}(a) \right| \leq C \lambda^{-N} \sum_{|\beta|\leq 2N}\sup|D^{\beta}(a)|
\end{eqnarray}
\fi
The expansion
\begin{eqnarray}\label{directly}
U(a) \sim \sum_{k=0}^{\infty} \left(\frac{i}{\lambda}\right)^k \Lambda_{k}(a)
\end{eqnarray}
can be obtained directly from the method of stationary phase with parameters: We write $Ua(z;z,\lambda,b) = \int e^{i\lambda\psi_{z,b}} \alpha \, dx$ as in \eqref{coordinate change 0}, where the phase function is
\begin{eqnarray*}
\psi_{z,b}(w,\mu,b') = \langle w,b\rangle-\langle z,b\rangle+\mu(\langle z,b'\rangle-\langle w,b'\rangle).
\end{eqnarray*}
The critical point $(z,1,b)$ of $\psi_{z,b}$ is given by the parameter $(z,b)$ (Lemma \ref{lemma phase}). Then \eqref{directly} follows from the method of stationary phase. If $a$ is a classical symbol of order $m$, the sum \eqref{directly} can be rearranged as above in homogeneous summands $Ua\sim\sum\widetilde{a}_j$, where the order of $\widetilde{a}_j$ is $m-j$. The $\Lambda_k$ in the expansion are left-$G$-invariant, since $U$ is left-invariant.
\end{rem}

\begin{rem}\label{subsequent remark}
\begin{itemize}
\item[(1)] Given a function $a(z,w,\lambda,b)$ of two spatial variables we sometimes write $a(z;w,\lambda,b)$ to emphasize the special role of $z$. The operator $U$ still operates in $X\times\rr^+\times B$ and we write
\begin{eqnarray*}
\hspace{-10mm} Ua(z;z,\lambda,b) &=& \\
&\hspace{-49mm}&\hspace{-29mm} \int_{X\times B\times\mathfrak{a}^*_{+}} \hspace{-10mm} e^{-(i\lambda+\rho)\langle z,b\rangle} e^{(i\mu+\rho)\langle z,b'\rangle} e^{(i\lambda+\rho)\langle w,b\rangle} e^{(-i\mu+\rho)\langle w,b'\rangle} a(z;w,\mu,b') |c(\mu)|^{-2} \, dw \, db' \, d\mu.
\end{eqnarray*}
\item[(2)] Let $m<<0$ be so small such that for $a\in S^m_{\textnormal{cl}}$ the integral $U(a)$ makes sense. We can write $Ua\sim\sum_{k=0}^{\infty}(i/\lambda)^k\Lambda_{k}(a)$ and expand out $U^*U(a)=a$. In particular, the principal symbol of $U(a)$ equals $c\cdot\sigma_{a}$, where $\sigma_a$ denotes the principal symbol of $a$, and where $c$ is a constant with $|c|=1$. Since $c>0$ by the MSP-formula, we find that principal symbol of $Ua$ equals the principal symbol of $a$.
\end{itemize}
\end{rem}

\begin{defn}\label{Definition z-w-symbols}
\begin{itemize}
\item[(1)] Let $\mathcal{L}^m_{1,0,0}$\index{$\mathcal{L}^m_{1,0,0}$, properly supported operators in OP$S^m_{1,0,0}$} denote the space of properly supported operators in OP$S^m_{1,0,0}=Op(S^m_{1,0,0})$, where $S^m_{1,0,0}\subset C^{\infty}(X\times X\times\rr^+\times B)$\index{$S^m_{1,0,0}$, space of symbols of order $m$} is the space of functions $a(z,w,\lambda,b)$ satisfying
\begin{eqnarray}
\left|  (\partial/{\partial\lambda})^k \, D \, a(z,w,\lambda,b) \right| \leq C_{D,k}(C)(1+|\lambda|)^{m-k} \,\,\,\,\,\, \forall\, (z,w)\in C,
\end{eqnarray}
for all $k\in\nn_0$, all compact subsets $C$ of $X\times X$, and all differential operators on $X\times X\times B$. \iffalse A symbol is called \emph{uniform}\index{uniform symbol}, if the constants $C_{D,k}(C)$ are independent of $C$.\fi We call $a\in S^m_{1,0,0}$ \emph{classical} ($a\in S^m_{1,0,0,\cl}$)\index{$S^m_{1,0,0}$, space of classical symbols} if for all $N\in\nn_0$
\begin{eqnarray}
a(z,w,\lambda,b) \sim \sum_{j=0}^{\infty}\lambda^{m-j}a_j(z,w,b)  \,\,\,\,\,\,\,\, (\lambda\rightarrow\infty).
\end{eqnarray}
Asymptotics here means $a-\sum_{j=0}^{N-1} a_j \in S^{m-N}_{1,0,0}$ for all $N\in\nn_0$.  By translation on $X\times X\times B$ we mean $g\cdot (z,w,b) = (g\cdot z,g\cdot w,g\cdot b)$. Let $S^m_{1,0,0,\Gamma}$\index{$S^m_{1,0,0,\Gamma}$, space of $\Gamma$-invariant symbols} denote the set of symbols, which are $\Gamma$-invariant:
\begin{eqnarray}
a(\gamma z,\gamma w,\lambda,\gamma \cdot b) = a(z,w,\lambda,b) \,\,\,\,\,\,\,\,\, \forall \, z,w\in X, \lambda\in\rr^+, b\in B, \gamma\in\Gamma.
\end{eqnarray}
\iffalse A symbol on $a(z,w,\lambda,b)$ on $X\times X\times \rr^+\times B$ is called uniformly properly supported\index{uniformly properly supported symbol}, if $a(z,w,\lambda,b)=0$, whenever $d_{X}(z,w)\geq\sigma_0$ (the distance function on $X=G/K$)\index{$d_X$, distance function on $X$}, where $\sigma_0$ is independent of $(z,w)\in X\times X$.\fi
\item[(2)] An operator $Op(a)=a(z,z,D)\in\textnormal{OP}S^m_{1,0,0}$\index{$Op$, pseudodifrerential operator quantization} operates according to the formula
\begin{eqnarray*}
a(z,z,D)u(z) &:=& \\
&& \hspace{-20mm} \int_X \int_B \int_{\rr^+} e^{(i\lambda+\rho)\left(\langle z,b\rangle-\langle w,b\rangle\right)} a(z,w,\lambda,b) u(w) e^{2\rho\langle w,b\rangle} \, dw \, db \, \dbar\lambda.
\end{eqnarray*}
\end{itemize}
\end{defn}

\begin{cor}\label{by corollary}
$Ua\sim\sum_{k=0}^{\infty}(i/\lambda)^k\Lambda_{k}(a)$ in symbol norm asymptotics.
\end{cor}
\begin{proof}
First note that an expansion for $Ua(z;z,\lambda,b)$ is obtained by the method of stationary phase with parameters as above, where the parameter is $(z,b)$. For the $\widetilde{a}_j$ and $\widetilde{a}$ that arise in Remark \ref{subsequent remark 0}, we need that
\begin{eqnarray}\label{need that}
\left| \left( \frac{\partial}{\partial\lambda}\right)^k \, D \, ( Ua - \widetilde{a} )(z,w,\lambda,b) \right| \leq C_{k,D,N}(C) (1+\lambda)^{-N}
\end{eqnarray}
for each $N\in\nn$, any compact subset $C$ of $X\times X$, all $(z,w)\in C$, and each differential operator $D$ on $X\times X\times B$. Recall that a sequence
$\left\{a_j\right\}$ of symbols, where the order of $a_j$ is $m-j$, can be asymptotically summed by setting $\widetilde{a}=\sum_{j=0}^{\infty}\phi(\varepsilon_j\lambda)a_j(z,w,\lambda,b)$, where $\phi=0$ for $\lambda\leq 1/2$, $\phi=1$ for $\lambda\geq 1$ and where the $\varepsilon_j$ are chosen appropriately (\cite{Tay81}, p. 41). If the $a_j$ are $\Gamma$-invariant, then so is $\widetilde{a}$. We claim that this holds if derivatives of $Ua$ have at most polynomial growth, that is
\begin{eqnarray}\label{claim}
|(\partial/{\partial\lambda})^k \, D \, U(a) |\leq C_{k,D}(C)(1+\lambda)^{\sigma},
\end{eqnarray}
for $(z,w)\in C$ and $\sigma=\sigma(k,\alpha)$.
It suffices to prove \eqref{need that} for an open coordinate patch $V$ of $X\times X\times B$ and for $D=\frac{\partial}{\partial x_j}$. By \cite{Tay81}, Prop. 3.3,
\begin{eqnarray}
\left\| D a \right\|^2_{\infty} \leq 4 \, \left\| a \right\|_{\infty} \, \left\| D^2 a \right\|_{\infty}.
\end{eqnarray}
So if $p-q$ is rapidly decreasing in $\lambda$ and if $D^2(p-q)$ has at most polynomial growth,
\begin{eqnarray}\label{inequality}
\sup |D(p-q)|^2 \leq c_1 \sup |p-q| \sup|D^2(p-q)| \leq C_{k,N}(1+\lambda)^{-N}
\end{eqnarray}
for any $N$, since the first factor is rapidly decreasing and the second is at most polynomially growing. Here $\sup$ means $\sup_V$ and $c_1$ is a constant. $\lambda$-derivatives can be handled similarly (loc. cit., p. 41).
\iffalse
To handle $\lambda$-derivatives, put $\widetilde{V}=V\times[0,1)$. Then
\begin{eqnarray*}
\sup_{V} \left| \frac{\partial}{\partial\lambda}(p-q)\right| &\leq& \sup_{\widetilde{V}}\left| \frac{\partial}{\partial\lambda}(p(z,b,\lambda+\mu) - q(z,b,\lambda+\mu))\right|^2 \\
&\leq& c_1 \sup_{\widetilde{V}} |p(z,b,\lambda+\mu) - q(z,b,\lambda+\mu)| \\
&& \hspace{5mm} \times \,\, \sup_{\widetilde{V}} \left| \frac{\partial^2}{\partial\lambda^2}(p(z,b,\lambda+\mu)-q(z,b,\lambda+\mu))  \right| \\
&\leq& C_N(1+\lambda)^{-N},
\end{eqnarray*}
since the first factor is rapidly decreasing and the second is at most polynomially growing.
\fi
We thus have $Ua\sim\sum_{k=0}^{\infty}(i/\lambda)^k\Lambda_{k}(a)$ in symbol norm asymptotics if derivatives of $U(a)$ have polynomial growth. Therefore write $U(a)$ in the form $I(a)+II(a)$ as before, integrate by parts in $II(a)$ as before, pass derivatives under the integral and see that the result is $\mathcal{O}(\lambda^{-\infty})$ uniformly in compact subsets. $I(a)$ is a compactly supported integral, and derivatives can be estimated by a constant times a suitable symbol norm of $a$ times a convenient power of $1+|\lambda|$.
\end{proof}

\iffalse
$|(\partial/{\partial\lambda})^k X^{\alpha} I(z,\lambda) |\leq C_{k,\alpha}(C)(1+\lambda)^{\sigma}$ by a constant times
\begin{eqnarray}
\sum_{\beta\leq\alpha,l\leq k} \left\| a \right\|_{\beta}^l (1-\lambda)^{m+l},
\end{eqnarray}
where $\left\| a \right\|_{\beta}^l$ is the symbol norm corresponding to $(\partial/{\partial\lambda^l})X^{\beta}$, where as usual $X^{\beta}=X^{\beta_1}\cdots X^{\beta_n}$, and the norm is taken over the compact subset (corresponding to \eqref{amplitude})
\begin{eqnarray}
\supp(\rho a) \cap \left\{ ghM \times [0,R+2] \right\},
\end{eqnarray}
as $gM$ ranges over $G/M$, $h\in G$.

\begin{prop}
$\Lambda_k^*=\Lambda_k$.
\end{prop}
\begin{proof}
First, $U^*(a) = U^{-1}(a) = \int_{\la^*_+} a(\cdot,\mu)\ast E_{-\mu,-\lambda} \dbar\mu = \int_{\la^*_+} a(\cdot,\mu) \ast \overline{E_{\mu,\lambda}} \dbar\mu$ (by \ref{remark above}). Then $U^*(a)\sim\sum_{k=0}^{\infty}\overline{(i/\lambda)^k\Lambda_k}a=\sum_{k=0}^{\infty}(-1)^k(i/\lambda)^k\overline{\Lambda_k}a$. But $\Lambda_k$ has real coefficients, so $U^*(a) \sim \sum_{k=0}^{\infty} (-1)^k \left(\frac{i}{\lambda}\right)^k\Lambda_k a$ and $U^*(a) \sim \sum_{k=0}^{\infty} \overline{\left(\frac{i}{\lambda}\right)^k}\Lambda_k^* a$, hence $\Lambda_k^* = \Lambda_k$.
\end{proof}
\fi

\begin{prop}\label{Proposition full symbol}
Let $A\in\mathcal{L}^m_{1,0,0,\textnormal{cl}}$. Then $A\in\mathcal{L}^m_{\textnormal{cl}}$ and a complete symbol of $A$ is given by
\begin{eqnarray*}
a(z,\lambda,b) = Ua(z;z,\lambda,b) \sim \sum_{k=0}^{\infty} (i/\lambda)^k \Lambda_k a(z;w,\lambda,b)_{|w=z}.
\end{eqnarray*}
\end{prop}
\begin{proof}
Since $A$ is properly supported, $a(z,\lambda,b):=e^{-(i\lambda+\rho)\langle z,b\rangle}Ae^{(i\lambda+\rho)\langle z,b\rangle}$ is well-defined and yields a complete symbol for $A$. Written out,
\begin{eqnarray*}
a(z,\lambda,b) &:=& \\
&& \hspace{-2cm} e^{(-i\lambda+\rho)\langle z,b\rangle} \int_{X\times\rr^+\times B} \hspace{-2mm} e^{(i\mu+\rho)\langle z,b'\rangle} e^{(-i\mu+\rho)\langle w,b'\rangle} a(z;w,\mu,b') e^{(i\lambda+\rho)\langle w,b\rangle} \, \dbar\mu \, dw \, db' \\
&& \hspace{-2cm} = Ua(z;\cdot,\cdot,\cdot)_{|(z,\lambda,b)} \sim \sum_{k=0}^{\infty} (i/\lambda)^k \Lambda_k a(z;w,\lambda,b)_{|w=z},
\end{eqnarray*}
where the $\Lambda_k$ operate in the variables $(w,\lambda,b)$. The $\sim$ holds in the sense of Corollary \ref{by corollary}.
\end{proof}

If $a$ is classical, we write $U(a)\sim\sum a_j$ in homogeneous terms and expand out $U^*U=\id$. In particular, the principal symbol of $U(a)$ equals the principal symbol of $a$. The proofs given in \cite{Z86} for the adjoint of properly supported operators (Theorem 2.8, loc. cit.) is formal enough to cover the case of all rank one spaces: It is proven there that if $a(z,\lambda,b)$ is an amplitude of order $0$, then the adjoint $Op(a)^*$ has amplitude $\overline{a(w,\lambda,b)}$, and the principal symbol of $Op(a)^*$ is $\overline{a_0(z,1,b)}$, so that the principal symbol of $Op(a)^*Op(a)$ is $|a(z,b)|^2$ (Thm. 2.9 loc. cit.). In particular, it is shown in Theorem 2. 11 loc. cit. that if $\Gamma$ is cocompact, then properly supported zero order pseudodifferential operator are continuous on $L^2(\Gamma\backslash X)$.

\begin{rem}
The computations of the critical set, the Hessian form, and the application of the method of stationary phase generalize to higher rank spaces, if the spectral parameters are assumed to be regular (cf. Subsection \ref{Critical sets and Hessian forms}). More integral formulas for the integral operator $U(a)(z,\lambda,b)$ are given in the following subsection. It seems reasonable to believe that the proofs given here and in \cite{Z86} can be generalized to higher rank spaces with only slight modifications.
\end{rem}

\subsubsection{Some integral formulas for the Kohn-Nirenberg operator}
We list a few possibilities to write $Ua$ as an oscillatory integral. These representations of the Kohn-Nirenberg operator may be useful to approach a formula for the Hessian operator in the asymptotic expansion for $Ua(z,\lambda,b)$, which would yield a commutator formula for the non-Euclidean calculus of pseudodifferential operators. First, let $G/K$ have rank one. Write $h=h(w,b')$ corresponding to $G/M\cong X\times B$: By \ref{haar measure corollary} we have
\begin{eqnarray*}
Ua(g,\lambda) = Ua(z,\lambda,b) \\
\iffalse
&\hspace{-75mm}=& \hspace{-35mm} \int_{X\times B\times\mathfrak{a}^*_{+}} e^{-(i\lambda+\rho)\langle z,b\rangle} e^{(i\mu+\rho)\langle z,b'\rangle} e^{(i\lambda+\rho)\langle w,b\rangle} e^{(-i\mu+\rho)\langle w,b'\rangle} a(w,\mu,b') \, |c(\mu)|^{-2} \, dw \, db' \, d\mu \\
&\hspace{-75mm}=& \hspace{-35mm} \int_{\mathfrak{a}^*_{+}}\int_G  e^{-(i\lambda+\rho)H(g)} e^{(i\mu+\rho)(-H(g^{-1}h) + H(h))} e^{(i\lambda+\rho)(-H(h^{-1}g) + H(g))} \\
&& \hspace{-23mm} \times \,\,\, e^{(-i\mu+\rho)H(h)} e^{-2\rho H(h)} a(h,\mu) \, |c(\mu)|^{-2} \, dh \, d\mu \\
\fi
&\hspace{-75mm}=& \hspace{-35mm} \int_{\mathfrak{a}^*_{+}}\int_G  e^{i\lambda[-H(h^{-1}g)-\mu H(g^{-1}h)]} e^{-\rho (H(g^{-1}h)+H(h^{-1}g))} a(h,\lambda\mu) \, \lambda|c(\lambda\mu)|^{-2} \, dh \, d\mu.
\end{eqnarray*}
The integral is actually taken over $G/M$, since all terms in the integrand are $M$-invariant. From now on we will work in $G$. First, substitute $h\mapsto gh$. Then
\begin{eqnarray*}
Ua(g,\lambda) = \int_{\mathfrak{a}^*_{+}}\int_G  e^{i\lambda[-H(h^{-1})-\mu H(h)]} e^{-\rho (H(h)+H(h^{-1}))} a(gh,\lambda\mu) \, \lambda|c(\lambda\mu)|^{-2} \, dh \, d\mu.
\end{eqnarray*}
Now by \eqref{integral formula G} the integral equals
\begin{eqnarray}\label{apply MSP 1}
\int_{\mathfrak{a}^*_+}\int_{NAK} e^{i\lambda[\log(a)-\mu H(nak)]} \, e^{\rho[\log(a)-H(nak)]} \, a(gnak) \, \frac{\lambda}{|c(\lambda\mu)|^{2}} \, dn \, da \, dk \, d\mu.
\end{eqnarray}
We could have also changed $h$ to $h^{-1}$ ($G$ is unimodular). Then
\begin{eqnarray}
&& \hspace{-3mm}Ua(g,\lambda) \\
&\hspace{-3mm}=& \hspace{-3mm} \int_{\mathfrak{a}^*_{+}}\int_G  e^{i\lambda[-H(h)-\mu H(h^{-1})]} e^{-\rho (H(h^{-1})+H(h))} a(gh^{-1},\lambda\mu) \, \lambda|c(\lambda\mu)|^{-2} \, dh \, d\mu \nonumber\\
&\hspace{-3mm}=& \hspace{-3mm} \int_{\mathfrak{a}^*_+}\int_{NAK} \hspace{-2mm} e^{i\lambda(\mu\log(a)-H(nak))} e^{-\rho(\log(a)+H(nak))} a(g(nak)^{-1},\lambda\mu) \frac{\lambda}{|c(\lambda\mu)|^{2}} \, dn \, da \, dk \, d\mu\nonumber.
\end{eqnarray}

In higher rank, the same computations are possible: Given $0\neq\lambda\in\mathfrak{a}^*_+$, write $\lambda=\tau\lambda_0$, where $|\lambda_0|=1$ (the norm on $\la^*$ induced by the Killing form). Set $Ua(g,\tau):=Ua(g,\tau\lambda_0)$. Then
\begin{eqnarray*}
Ua(g,\tau) &=& \int_{\mathfrak{a}^*_+}\int_{NAK} e^{i(\mu\log(a) - \tau\lambda_0 H(nak))} e^{-\rho(\log(a)+H(nak))} \\
&& \hspace{18mm} \times \,\,\, a(g(nak)^{-1},\mu) \, dn \, da \, dk \, \dbar\mu \\
&=& \int_{\mathfrak{a}^*_+}\int_{NAK} e^{i\tau(\mu\log(a) - \lambda_0 H(nak))} e^{-\rho(\log(a)+H(nak))} \\
&& \hspace{18mm} \times \,\,\, a(g(nak)^{-1},\tau\mu) \, \frac{\tau^{\dim(A)}}{|c(\tau\mu)|^{2}}\frac{1}{w} \, dn \, da \, dk \, d\mu.
\end{eqnarray*}
where we factored out $\tau$ from the phase function and substituted $\mu\mapsto\mu/\tau$.

Recall that $U$ commutes with translation by $g\in G$. We rewrite $Ua(z,\lambda,b)$ as given in \eqref{coordinate change 0} corresponding to $X\times B\cong AN\times K/M$, evaluate the integral at $(o,\lambda,M)$, and finally choose $g\in G$ such that $g\cdot(o,M)=(z,b)$ to insert $g\in G$ in the amplitude. Then with $|\lambda_0|=1$ and $\lambda>0$,
\begin{eqnarray}\label{coordinate change 2}
Ua(z,\lambda\lambda_0,b) &=& U(a\circ g)(o,\lambda\lambda_0,M) \\
&\hspace{-40mm}=& \hspace{-20mm} \int_{AN\times K/M\times\rr^{+}} e^{i\lambda\lambda_0[\langle an\cdot o,M\rangle-\mu(\langle an\cdot o,kM\rangle)]} \, e^{\rho[\langle an\cdot o,M\rangle+\langle an\cdot o,kM\rangle]} \\
&& \hspace{20mm} \times \,\,\,\, a(g\cdot an\cdot o,\lambda\mu,g\cdot kM) \, \frac{\lambda^{\dim(A)}}{|c(\lambda\mu)|^{2}} \, dw \, db' \, d\mu \nonumber \\
&\hspace{-40mm}=&\hspace{-20mm} \int_{AN\times K/M\times\rr^{+}} e^{i\lambda\lambda_0[-\log(a^{-1})+\mu(H(n^{-1}a^{-1}k))]} \, e^{\rho[-\log(a^{-1})-H(n^{-1}a^{-1}k)]} \\
&& \hspace{20mm} \times \,\,\,\, a(g\cdot an\cdot o,\lambda\mu,g\cdot kM) \, \frac{\lambda^{\dim(A)}}{|c(\lambda\mu)|^{2}} \, dw \, db' \, d\mu \nonumber \\
&\hspace{-40mm}=&\hspace{-20mm} \int_{AN\times K/M\times\rr^{+}} e^{i\lambda\lambda_0[-\log(a)+\mu(H(nak))]} \, e^{-\rho[\log(a)+H(nak)]} \\
&& \hspace{20mm} \times \,\,\,\, a(g\cdot a^{-1}n^{-1}\cdot o,\lambda\mu,g\cdot kM) \, \frac{\lambda^{\dim(A)}}{|c(\lambda\mu)|^{2}} \, dw \, db' \, d\mu \nonumber.
\end{eqnarray}
The phase function $\mu(H(nak))-\lambda_0(\log(a))$ has the critical point $(\mu,n,a,kM)=(\lambda_0,e,e,eM)$ and the Hessian at this point is non-degenerate. The method of stationary phase can be applied to all these integral exactly as before.

\subsection{Conjugation by a wave group-type operator}
Let $X$ have rank one. We identify $\la=\rr$ by means of the Killing form: $\lambda_0$ denotes the functional on $\la$ given by $\lambda_0(X)=\langle X,H\rangle$, where $H$ is the unit vector in $\la^+$. Then $\lambda_0\in\la^+$. We identify $\lambda\in\la$ with the real number $\overline{\lambda}$ such that $\lambda=\overline{\lambda}\lambda_0$.

We denote by $G^t$ the geodesic flow on $X\times B$. The latter space identifies with $G/M$ and hence the geodesic flow on $X\times B$ reads by right-translations with elements $a\in A$, that is $G^t(g\cdot o,g\cdot M) = (g a_t\cdot o,g a_t\cdot M) = (ga_t\cdot o,b)$ for $g\in G$, $a_t=\exp(tH)\in A$. Right-translation on $X\times B$ is well-defined, since $M$ and $A$ commute elementwise. The point $b\in B$ is not moved under $G^t$. Recall that if $\Gamma$ is a cocompact subgroup of $G$, the geodesic flow on $SX_{\Gamma}=\Gamma\backslash G/M$ also reads by right-$A$-translation.

Let $A=a(x,D)\in\mathcal{L}_{\textnormal{cl}}^m$. We denote by $\sigma_A$ the \emph{principal symbol}\index{principal symbol} of $A$, that is the highest order term in the asymptotic sum \eqref{classical}. Let $\lambda\in\la^*$ and $b\in B$. Recall the character of the Laplace operator (cf. \eqref{character of the Laplacian}):
\begin{eqnarray*}
\Delta e^{(i\lambda+\rho)\langle z,b\rangle} = -(\langle\lambda,\lambda\rangle+\langle\rho,\rho\rangle)e^{(i\lambda+\rho)\langle z,b\rangle}.
\end{eqnarray*}
Using functional calculus (the spectral theorem), we define $R:=\sqrt{-(\Delta+|\rho|^2)}$ and the group of operators $e^{itR}$ by its action on the non-Euclidean plance waves $e_{\lambda,b}(z)=e^{(i\lambda+\rho)\langle z,b\rangle}$, that is
\begin{eqnarray*}
e^{itR}e^{(i\lambda+\rho)\langle z,b\rangle} = e^{it\lambda}e^{(i\lambda+\rho)\langle z,b\rangle}.
\end{eqnarray*}

Given $t\in\rr$, we write
\begin{eqnarray}\label{as in}
A_t := e^{itR}Ae^{-itR}.
\end{eqnarray}

\subsubsection{The complete symbol after conjugation}
The complete symbol of $A_t$ is
\begin{eqnarray}\label{symbol a_t}
U^t(a) := a_t(z,\lambda,b) &=& e^{-(i\lambda+\rho)\langle z,b\rangle} e^{itR} A e^{-itR} e^{(i\lambda+\rho)\langle z,b\rangle} \nonumber \\
&=& e^{-(i\lambda+\rho)\langle z,b\rangle} e^{itR} A e^{-it\lambda} e^{(i\lambda+\rho)\langle z,b\rangle} \nonumber  \\
&=& e^{-it\lambda} e^{-(i\lambda+\rho)\langle z,b\rangle} e^{itR} \left( a(z,\lambda,b) e^{(i\lambda+\rho)\langle z,b\rangle} \right).
\end{eqnarray}

Recall the Fourier inversion formula \eqref{Fourier inversion formula}, which states that each sufficiently regular function $f$ on $X$ satisfies
\begin{eqnarray*}
f(z) = \int_{\mathfrak{a}^*_+}\int_B\int_X e^{(i\lambda+\rho)\langle z,b\rangle}e^{(-i\lambda+\rho)\langle w,b\rangle}f(w)\,dw\,\dbar\lambda\,db.
\end{eqnarray*}

It follows that
\begin{eqnarray*}
a(z,\lambda,b) e^{(i\lambda+\rho)\langle z,b\rangle} = \int e^{(i\mu+\rho)\langle z,b'\rangle} e^{(-i\mu+\rho)\langle w,b'\rangle} e^{(i\lambda+\rho)\langle w,b\rangle} a(w,\lambda,b) \, dw \, \dbar\mu \, db',
\end{eqnarray*}
and hence
\begin{eqnarray*}
&& \hspace{-1cm}a_t(z,\lambda,b) \\
&=& \int e^{-it\lambda} e^{-(i\lambda+\rho)\langle z,b\rangle} e^{itR} e^{(i\mu+\rho)\langle z,b'\rangle} e^{(-i\mu+\rho)\langle w,b'\rangle} e^{(i\lambda+\rho)\langle w,b\rangle} a(w,\lambda,b) dw \, \dbar\mu \, db' \\
&=& \int e^{it(\mu-\lambda)} e^{-(i\lambda+\rho)(\langle z,b\rangle-\langle w,b\rangle)} e^{(i\mu+\rho)(\langle z,b'\rangle-\langle w,b'\rangle)} e^{2\rho\langle w,b'\rangle} a(w,\lambda,b) dw \, \dbar\mu \, db' \\
&=& \int e^{it(\mu-\lambda)} e^{-(i\lambda+\rho)(\langle z,b\rangle-\langle w,b\rangle)} e^{(i\mu+\rho)(\langle z,b'\rangle-\langle w,b'\rangle)} a(w,\lambda,b)
\, e^{2\rho\langle w,b'\rangle} \, dw \, db \, \dbar\mu.
\end{eqnarray*}

\begin{cor}\label{recall cor}
For sufficiently regular functions $a(w,\mu,b')$ on $X\times\la^*_+\times B$,
\begin{eqnarray*}
U^t(a) = \int e^{it(\mu-\lambda)} e^{-(i\lambda+\rho)(\langle z,b\rangle-\langle w,b\rangle)} e^{(i\mu+\rho)(\langle z,b'\rangle-\langle w,b'\rangle)} a(w,\lambda,b) \, e^{2\rho\langle w,b'\rangle} \, dw \, db \, \dbar\mu.
\end{eqnarray*}
\end{cor}

\begin{prop}
The $U^t$ are a one-parameter group of unitary operators on $L^2(G/M\times\rr^+,dg\,\dbar\mu)=L^2(X\times B\times\la^+, e^{2\rho\langle w,b'\rangle} \, dw \, db \,\dbar\mu)$.
\end{prop}
\begin{proof}
Let $\langle\cdot,\cdot\rangle$ denote the $L^2$-inner product. Then
\begin{eqnarray*}
\langle U^ta,U^ta\rangle = \int e^{it\mu_1} e^{-it\mu_2} e^{(i\mu_1+\rho)(\langle z,b_1\rangle-\langle w_1,b_1\rangle}) \\
&\hspace{-11cm}\times&\hspace{-5.3cm} e^{-(i\lambda+\rho)(\langle z,b\rangle-\langle w_1,b\rangle)} a(w_1,b,\lambda) e^{(-i\mu_2+\rho)(\langle z,b_2\rangle-\langle w_2,b_2\rangle)}  \\
&\hspace{-11cm}\times&\hspace{-5.3cm} e^{(i\lambda-\rho)(\langle z,b\rangle-\langle w_2,b\rangle} \overline{a}(w_2,b,\lambda) e^{2\rho\langle w_2,b_2\rangle} e^{2\rho\langle w_1,b_1\rangle} \\
&\hspace{-11cm}\times&\hspace{-5.3cm} e^{2\rho\langle z,b\rangle} \, \dbar\lambda \, db \, dz \, \dbar\mu_1 \, db_1 \, dw_1 \, \dbar\mu_2 \, db_2 \, dw_2.
\end{eqnarray*}
The Fourier inversion formula \eqref{Fourier inversion formula} states for sufficiently regular $f: X\rightarrow\cc$
\begin{itemize}
\item[(1)] $f(z) = \int_{\mathfrak{a}^*_+}\int_B\int_X e^{(i\lambda+\rho)\langle z,b\rangle}e^{(-i\lambda+\rho)\langle w,b\rangle}f(w)\,dw\,\dbar\lambda\,db$,
\item[(2)] $\tilde{f}(\lambda,b) = \int_X\int_{\mathfrak{a}^*_+}\int_B e^{(-i\lambda+\rho)\langle z,b\rangle}e^{(i\mu+\rho)\langle z,\tilde{b}\rangle} \tilde{f}(\mu,\tilde{b}) \, \dbar\mu\,d\tilde{b}\,dz$.
\end{itemize}
As in the proof of Proposition \ref{U unitary} we use these formulae to find $\langle U^ta,U^ta\rangle = \langle a,a\rangle$. (The $SU(1,1)$-proof given in \cite{Z86}, p. 100, generalizes completely).
\end{proof}

Recall $X\times B\cong G/M$. We change variables to $w=g\cdot o$, $b'=g\cdot M$ and prove exactly as in Subsection \ref{The Kohn-Nirenberg operator} that
\begin{eqnarray*}
U^t(a)(z,\lambda,b) &=& \int_{\rr^+} \int_G e^{it(\mu-\lambda)\langle g^{-1}z,M\rangle} e^{(i\mu+\rho)\langle g^{-1}z,g^{-1}b\rangle} e^{-(i\lambda+\rho)} a(g\cdot o,b,\lambda) \, dg \, \dbar\mu \\
&=& \int_{\rr^+} \left( a(\cdot,b,\lambda) \ast E_{\mu,\lambda}(z,b) \right) e^{it(\mu-\lambda)} \, \dbar\mu.
\end{eqnarray*}

\begin{lem}
$U^t$ commutes with translations $T_g$ by elements $g\in G$.
\end{lem}
\begin{proof}
Let $g\in G$. Then by \eqref{symbol a_t} we have
\begin{eqnarray*}
(T_g U^t)(a)(z,\lambda,b) &=& a_t(gz,\lambda,gb) \\
&=& e^{-it\lambda} e^{-(i\lambda+\rho)\langle gz,gb\rangle} e^{itR} \left( a(gz,\lambda,gb) e^{(i\lambda+\rho)\langle gz,gb\rangle} \right).
\end{eqnarray*}
Since $\langle gz,gb\rangle=\langle z,b\rangle + \langle go,gb\rangle$ this equals
\begin{eqnarray*}
e^{-it\lambda} e^{-(i\lambda+\rho)\langle z,b\rangle} e^{itR} \left( a(gz,\lambda,gb) e^{(i\lambda+\rho)\langle z,b\rangle} \right) \\
&\hspace{-12.3cm}=&\hspace{-6cm} (a\circ g)_t(z,\lambda,b) \\
&\hspace{-12.3cm}=&\hspace{-6cm} (U^tT_g)(a)(z,\lambda,b),
\end{eqnarray*}
where $(a\circ g)(z,\lambda,b)=a(gz,\lambda,gb)$. This proves $U^tT_g=T_gU^t$.
\end{proof}

Recall from Corollary \ref{recall cor} that
\begin{eqnarray*}
U^t(a) = \int e^{it(\mu-\lambda)} e^{-(i\lambda+\rho)(\langle z,b\rangle-\langle w,b\rangle)} e^{(i\mu+\rho)(\langle z,b'\rangle-\langle w,b'\rangle)} a(w,\lambda,b) \, e^{2\rho\langle w,b'\rangle} \, dw \, db \, \dbar\mu,
\end{eqnarray*}
where the integration space is $X\times B\times \rr^+$. We factor out $\lambda$ from the phase, change variables to $\widetilde{\mu}=\mu/\lambda$, and drop the tilde. Then
\begin{eqnarray*}
U^t(a)(z,\lambda,b) \hspace{-2mm} &=& \hspace{-2mm} \int e^{i\lambda[t(\mu-1) + \langle w,b\rangle - \langle z,b\rangle + \mu(\langle z,b'\rangle - \langle w,b'\rangle)]} \\
&& \, \times \, e^{\rho[\langle w,b\rangle - \langle z,b\rangle + \langle z,b'\rangle - \langle w,b'\rangle]} \, \lambda \, a(w,\lambda,b) \, e^{2\rho\langle w,b'\rangle} |c(\lambda\mu)|^{-2} \, dw \, db \, d\mu.
\end{eqnarray*}
Writing $(z,\lambda,b)=(g\cdot o,\lambda,g\cdot M)$ and using $Ua(z,\lambda,b)=U(a\circ g)(o,\lambda,M)$, we find (also note that $g\cdot M=b\in B$)
\begin{eqnarray}\label{also note}
U^ta(g,\lambda) \hspace{-2mm} &=& \hspace{-2mm} U^t(a)(z,\lambda,b) \\
&\hspace{-35mm}=& \hspace{-20mm} \int e^{i\lambda[t(\mu-1) + \langle w,M\rangle - \mu\langle w,b'\rangle]} \, e^{\rho[\langle w,M\rangle - \langle w,b'\rangle]} \, \lambda \, a(g\cdot w,\lambda,b) \, e^{2\rho\langle w,b'\rangle} |c(\lambda\mu)|^{-2} \, dw \, db \, d\mu \nonumber.
\end{eqnarray}
We change variables $(w,b')=h\cdot(o,M)$ corresponding to $X\times B\cong G/M$. Then $\langle w,M\rangle = -H(h^{-1})$ and $\langle w,b'\rangle=H(h)$, so by \ref{m-invariant integral}
\begin{eqnarray*}
U^ta(g,\lambda) \hspace{-2mm} &=& \hspace{-2mm} U^t(a)(z,\lambda,b) \\
&\hspace{-35mm}=& \hspace{-20mm} \int e^{i\lambda[t(\mu-1) - H(h^{-1}) - \mu H(h)]} \, e^{-\rho[H(h^{-1}) + H(h)]} \, \lambda \, a(gh\cdot o,\lambda,b) \, |c(\lambda\mu)|^{-2} \, dh \, d\mu.
\end{eqnarray*}
Next, write $h=a^{-1}n^{-1}k$ corresponding to $G=ANK$. Then by \eqref{integral formula G} and since $A$ and $N$ are unimodular we obtain
\begin{eqnarray*}
U^ta(g,\lambda) \hspace{-2mm} &=& \hspace{-2mm} U^t(a)(z,\lambda,b) \\
&\hspace{-35mm}=& \hspace{-20mm} \int e^{i\lambda[t(\mu-1) - \log(a) - \mu H(a^{-1}n^{-1}k)]} \\
&& \hspace{-14mm} \times \,\, e^{-\rho[\log(a)+H(a^{-1}n^{-1}k)]} \, \lambda \, a(g(na)^{-1}\cdot o,\lambda,b) \, |c(\lambda\mu)|^{-2} \, da \, dn \, dk \, d\mu.
\end{eqnarray*}

\subsubsection{An Egorov-type formula}
The classical Egorov theorem states that conjugation by the wave group defines an order preserving automorphism on the space of pseudodifferential operators. We will now be able to prove the following version:

\begin{thm}\label{Egorov Theorem}
Let $a(z,\lambda,b)\in S^m_{\textnormal{cl}}$ be compactly supported in $z$ (uniformly in the other variables). Write $A=Op(a)$ and $A_t := e^{itR}Ae^{-itR}$. Then $A_t$ has complete symbol $U^t(a)\in S^{m}_{\textnormal{cl}}$ and  $\sigma_{A_t}=c_t\sigma_{A}(G^{t}(z,b),\lambda)$, where $c_t$ is a constant.
\end{thm}
\begin{proof}
The phase function of the symbol $U^t(a)(z,\lambda,b)$ on $A\times N\times K/M\times\rr^+$ is given by
\begin{eqnarray*}
\psi_t(\mu,n,a,kM) = t(\mu-1) - \log(a) - \mu H(a^{-1}n^{-1}k).
\end{eqnarray*}
As proven in Subs. \ref{Critical sets and Hessian forms}, the phase function $\psi_t$ has the critical point $(\mu,n,a,kM) = (1,e,a_{-t},eM)$, and the Hessian form of $\psi_t$ at the critical point is non-degenerate. Under $X\times B\cong G/M \cong A^{-1}N^{-1} \times K/M$ the critical point $(a_{-t},e,eM)$ corresponds to $(w,b')=(a_t\cdot o,M)=G^t(o,M)$. Given $(z,b)=(g\cdot o,g\cdot M)$ we then have $(g\cdot a^{-1}n^{-1}\cdot o,\lambda,b)_{a=a_{-t},n=e}=(G^t(z,b),\lambda)$. As before, the principle of non-stationary phase yields that $U^t(a)$ is uniquely determined modulo $S^{-\infty}:=\cap_{m}S^m$\index{$S^{-\infty}=\cap_{m}S^m$, smoothing symbols} by a compactly supported cutoff of the integrand. The method of stationary phase is applied to this cutoff of $U^t(a)(z,\lambda,b)$ exactly as in Subs. \ref{Asymptotic expansions}, only the critical point of the phase function is different. The MSP-formula yields an expansion for $U^t(a)(z,\lambda,b)$ which can be rearranged in homogeneous terms, so $U^t(a) \in S^{m}_{\textnormal{cl}}$. In particular, the principal symbol of $U^t(a)$ is given by a constant times an evaluation at the critical point of the principal symbol of $a$ (all other terms in the MSP-formula have lower order). It follows that $\sigma_{A_t}=c_t\sigma_{A}(G^{t}(z,b),\lambda)$, so the theorem is proven.
\end{proof}

\begin{rem}
\begin{itemize}
\item[(1)] It seems reasonable to conjecture that $c_t=1$ for all $t$. In fact, the operators in the MSP-formula are left-$G$-invariant, so the theorem descends to a compact quotient $X_{\Gamma}$. Write $1(z,\lambda,b)$ for the constant function $f(z,\lambda,b)=1$ on $X_{\Gamma}\times\rr^+\times B$. Then $1(G^t(z,b),\lambda)=1$ and we can use diagonal matrix elements $\rho_{\lambda_j}(Op(a))$ as in the introduction to see that $1=\rho_{\lambda_j}(Op(1))\sim c_t$ (when $j\rightarrow\infty$).
\item[(2)] One should caution that conjugation $e^{itR}Ae^{-itR}$ is not equivalent to conjugation by the wave group. If one uses \cite{BO}, Lemma 2.2, to compute (for the standard quantization) the infinitesimal action of the wave group on a symbol of a pseudodifferential operator, one finds that shifting the Laplace operator influences the velocity of the (geodesic) flow defining the symbol in an Egorov theorem.
\end{itemize}
\end{rem}

\newpage
\pagebreak
\thispagestyle{empty} 

\section{Helgason boundary values}\label{Section Helgason boundary values}

We start by recalling some fundamental relations first proven by Helgason (\cite{He70}) between joint eigenfunctions of the algebra of invariant differential operators with hyperfunctions and distributions on the real flag manifold of a symmetric space. These relations are described by means of the Poisson transform. In Subsection \ref{Regularity} we prove a regularity statement (Lemma \ref{Lemma Olbrich}) for boundary values of certain eigenfunctions, which seems to be a new result.

Recall that $\mathbb{D}(G/K)$ denotes the algebra of differential operators on $X=G/K$, which are invariant under left-translations by elements of $G$. Given a homomorphism $\chi:\mathbb{D}(G/K)\rightarrow\cc$, let $\chi(D)$ ($D\in\mathbb{D}(G/K)$) denote the corresponding system of eigenvalues. The space
\begin{eqnarray*}
\mathcal{E}_{\chi}(X) = \left\{ f\in\mathcal{E}(X): Df=\chi(D)f \,\,\, \textnormal{ for all } D\in\mathbb{D}(G/K)\right\}
\end{eqnarray*}
\index{$\mathcal{E}_{\chi}(X)$, joint eigenspace}is called a \emph{joint eigenspace}\index{joint eigenspace} of $\mathbb{D}(G/K)$. We also know that the homomorphisms $\chi$ as above can be parameterized by the orbits of the Weyl group in $\la^*$, that is each $\chi$ is of the form $\chi_{\lambda}$, where $\lambda\in\la^*$. As in Section \ref{Joint eigenfunctions and joint eigenspaces} we write
\begin{eqnarray*}
\mathcal{E}_{\lambda}(X) = \left\{f\in\mathcal{E}(X): Df=\Gamma(D)(i\lambda)f \textnormal{ for all } D\in\dd(X)  \right\}.
\end{eqnarray*}
\index{$\mathcal{E}_{\lambda}(X)$, joint eigenspace}A smooth function $f\in\mathcal{E}(G/K)$ is called \emph{joint eigenfunction}\index{joint eigenfunction} if it belongs to one of the spaces $\mathcal{E}_{\lambda}(X)$.

\begin{defn}
Let $L$ denote the Laplace-Beltrami operator of $B$. Let $\mathcal{A}(B)$ denote the vector space of analytic functions on $B=K/M$. For $T>0$ put
\begin{eqnarray*}
\left|F\right|_T = \sup_{k\in\zz^+} \left( \frac{1}{2k!} T^k \left\| L^k F \right\| \right),
\end{eqnarray*}
where $\|\cdot\|$ is the $L^2$-norm on $B$, and\index{$\mathcal{A}_T(B)$, Banach space}
\begin{eqnarray*}
\mathcal{A}_T(B) = \left\{ F\in\mathcal{E}(B) : |F|_T < \infty \right\}.
\end{eqnarray*}
Then $\mathcal{A}_T(B)$ is a Banach space, $\mathcal{A}(B)$ is the union of the spaces $\mathcal{A}_T(B)$ and is accordingly given the inductive limit topology. The \emph{analytic functionals (hyperfunctions)}\index{analytic functionals (hyperfunctions)} are the functionals in the dual space $\mathcal{A}'(B)$ of $\mathcal{A}(B)$ (cf. \cite{LM}).

We use the integral notation for distributions or hyperfunctions and test functions: For any space $Y$ we denote the pairing of distributions $u$ and test functions $\varphi$ on $Y$ by $\int_Y \varphi(y)u(dy)=\langle \varphi, u\rangle_Y$.
\end{defn}

The Poisson kernel $P(x,b)=e^{2\rho\langle x,b\rangle}$\index{$P(x,b)=e^{2\rho\langle x,b\rangle}$, Poisson kernel} and its powers $e_{\lambda,b}(x)=e^{(i\lambda+\rho)\langle x,b\rangle}$, where $\lambda\in\la^*_{\cc}$, are analytic functions (\cite{He94}, p. 119).

\begin{defn}
Given a function, distribution or hyperfunction $T$ on $B$ and $\lambda\in\la^*_{\cc}$ we define the \emph{Poisson transform}\index{Poisson transform} $P_{\lambda}: \mathcal{A}'(B)\rightarrow \mathcal{E}_{\lambda}(X)$ by
\begin{eqnarray}\label{Poisson transform}
P_{\lambda}(T)(z) := \int_B e^{(i\lambda+\rho)\langle z,b\rangle} T(db).
\end{eqnarray}
As a consequence of \cite{Rou}, p. 167, the function $P_{\lambda}(T)(z)$ is analytic and its derivatives can be computed under the integral sign. It follows from \eqref{character of the Laplacian} that $z\mapsto P_{\lambda}(T)(z)$ is a joint eigenfunction and belongs to $\mathcal{E}_{\lambda}(X)$.
\end{defn}

If the functional $T$ above is actually a function $f$ on $B$, then $T(db)=f(b)db$. Now suppose that $\psi$ is a function on $\la^*\times B$. Writing $\psi_{\lambda}(b)=\psi(\lambda,b)$ we see that \eqref{Weyl} can be written in the form
\begin{eqnarray}
P_{s\lambda}(\psi_{s\lambda}) = P_{\lambda}(\psi_{\lambda}), \hspace{2mm} s\in W.
\end{eqnarray}

The following fundamental theorem (\cite{He94}, p. 507, Theorem 6.5) relates eigenfunctions with hyperfunctions:
\begin{thm}\label{Helgason boundary values}
The joint eigenfunctions of $\mathbb{D}(G/K)$\index{$\mathbb{D}(G/K)$, algebra of translation invariant-differential operators on $G/K$} are the functions
\begin{eqnarray*}
f(x)=\int_B e^{(i\lambda+\rho)\langle z,b\rangle}dT(b),
\end{eqnarray*}
where $\lambda\in\la^*_{\cc}$ and $T\in\mathcal{A}'(B)$.
\end{thm}

Given a joint eigenfunction $\phi$ of $\mathbb{D}(G/K)$, we call the unique functional $T=T_{\phi}$ given by Theorem \ref{Helgason boundary values} the \emph{boundary values}\index{boundary values} (\emph{Helgason boundary values}\index{Helgason boundary values}) of $\phi$. We will consider the following special class of eigenfunctions: Let $d$ denote the distance function on $X$. We define the subspace $\mathcal{E}^*(X)$ of $\mathcal{E}(X)$ of functions of \emph{exponential growth}\index{exponential growth} by
\begin{eqnarray}\label{exponential growth}
\mathcal{E}^*(X) = \left\{    f\in\mathcal{E}(X): \exists C>0: |f(x)|\leq C e^{C d_X(o,x)} \forall x\in X  \right\}
\end{eqnarray}\index{$\mathcal{E}^*(X)$, subspace of $\mathcal{E}(X)$}
and we put\index{$\mathcal{E}_{\lambda}^*(X)$, joint eigenspace} $\mathcal{E}_{\lambda}^*(X) = \mathcal{E}^*(X)\cap \mathcal{E}_{\lambda}(X)$\index{$\mathcal{E}^*_{\lambda}(X)$, subspace of $\mathcal{E}_{\lambda}(X)$}. Denote by $w$ the longest Weyl group element and recall Harish-Chandra's $e$-functions (Subsection \ref{Special functions and the Plancherel density}). It turns out that eigenfunctions with exponential growth have distributional boundary values (cf. \cite{He94}, p. 508):
\begin{thm}\label{distribution}
Let $\lambda\in\mathfrak{a}^*_{\cc}$ be such that $e_w(\lambda)\neq 0$. Then $P_{\lambda}(\mathcal{D}'(B))=\mathcal{E}_{\lambda}^*(X)$.
\end{thm}

We will always consider eigenfunctions with unique and distributional boundary values as in Theorem \ref{distribution}.

Fix any subgroup $\Gamma$ of $G$ and let $\phi\in\mathcal{E}_{\lambda}^*(X)$ ($\lambda\in\la^*_{\cc}$) denote a $\Gamma$-invariant eigenfunction with unique and distributional boundary values $T_{\lambda}$. Then
\begin{eqnarray*}
\phi(z) &=& \int_B e^{(i\lambda+\rho)\langle z ,b\rangle} T_{\lambda}(db).
\end{eqnarray*}
The group $G$ acts on the boundary $B$ of $X$ (cf. Section \ref{boundary}). Hence $G$ acts on $\mathcal{D}'(B)$ by push-forward: Given a distribution $T$ on $B$, a test function $\phi\in\mathcal{E}(B)=\mathcal{D}(B)$ and $g\in G$, this action is given by
\begin{eqnarray}
(gT)(\phi):=T(\phi\circ g^{-1}).
\end{eqnarray}
When we denote the pairing of distributions and functions by an integral, we also write $T(dgb)$ instead of $(gT)(db)$ ($g\in G$). 
One might expect that $T$ as well is invariant under the pull-back action of $\Gamma$. But in fact, since $\phi(\gamma z)=\phi(z)$ for all $\gamma\in\Gamma$ and $z\in X$, we observe (recall \eqref{equivariance}, that is $\left\langle g\cdot x,g\cdot b \right\rangle = \left\langle x,b \right\rangle + \left\langle g\cdot o,g\cdot b \right\rangle$)
\begin{eqnarray*}
\phi(z) &=& \phi(\gamma z) \,\,  = \,\, \int_B e^{(i\lambda+\rho)\langle \gamma z ,b\rangle} T_{\lambda}(db) \\
&=& \int_B e^{(i\lambda+\rho)\langle \gamma z,\gamma b\rangle} T_{\lambda}(d\gamma b) \,\, = \,\, \int_B e^{(i\lambda+\rho)\langle z,b\rangle} e^{(i\lambda+\rho)\langle\gamma o,\gamma b\rangle} T_{\lambda}(d\gamma b).
\end{eqnarray*}
By uniqueness of the boundary values (Theorem \ref{Helgason boundary values}) this implies
\begin{eqnarray*}
T_{\lambda}(db) = e^{(i\lambda+\rho)\langle\gamma o,\gamma b\rangle} T_{\lambda}(d\gamma b),
\end{eqnarray*}
or equivalently
\begin{eqnarray}\label{boundary values equivariance 1}
T_{\lambda}(d\gamma b) = e^{-(i\lambda+\rho)\langle\gamma o,\gamma b\rangle} T_{\lambda}(db).
\end{eqnarray}

\begin{defn}
Let $\phi$ and $T$ be as above. We define $e_{\lambda}\in\mathcal{D}'(X\times B)$ as the distribution on $X\times B=G/M$ given by
\begin{eqnarray}
\langle f,e_{\lambda}\rangle := \int_{X\times B} e^{(i\lambda+\rho)\langle z,b\rangle}f(z,b) \, T(db) \, dz, \,\,\,\,\,\,\,\,\,\, f \in \mathcal{D}(X\times B).
\end{eqnarray}
\end{defn}

The action of $G$ on distributions on $X\times B$ is defined by pulling back the action of $G$ on $X\times B$: Given a distribution $u$ and a test function $f$ on $X\times B$, we write $(g\cdot u)(f):=u(f\circ g^{-1})$. Let $\gamma\in\Gamma$. Then by the invariance of $dz$, by \eqref{equivariance} and \eqref{boundary values equivariance 1} we obtain
\begin{eqnarray*}
\langle f, e_{\lambda}\rangle &=& \int_{X\times B} e^{(i\lambda+\rho)\langle \gamma z,\gamma b\rangle}f(\gamma z,\gamma b) \, T(d\gamma b) \, dz \\
&=& \int_{X\times B} e^{(i\lambda+\rho)\langle z, b\rangle}f(\gamma z,\gamma b) \, T(d b) \, dz \\
&=& \langle f\circ\gamma, e_{\lambda}\rangle = \langle f, \gamma^{-1}\cdot e_{\lambda}\rangle.
\end{eqnarray*}

\begin{cor}
$e_{\lambda}$ is a $\Gamma$-invariant distribution on $X\times B$.
\end{cor}

\subsection{Poisson transform and principal series representations}
We recall some facts concerning the \emph{principal series}\index{principal series} representations of $G$. We follow \cite{He94} and \cite{Wil}. Let $\lambda\in\la$ and consider the representation 
\begin{eqnarray*}
\sigma_{\lambda}(man)=e^{(i\lambda+\rho)\log(a)}
\end{eqnarray*}
of $P=MAN$ on $\cc$. We denote the \emph{induced representation}\label{induced representation} on $G$ by $\pi_{\lambda}=\Ind_P^G(\sigma_{\lambda})$. The \emph{induced picture}\label{induced picture} of this representation is constructed as follows: A dense subspace of the representation space is\index{$H_{\lambda}^{\infty}$, representation space}
\begin{eqnarray*}
H_{\lambda}^{\infty} := \left\{ f\in C^{\infty}(G): f(gman)=e^{-(i\lambda+\rho)\log(a)}f(g)\right\}.
\end{eqnarray*}
We define an inner product on $H_{\lambda}^{\infty}$ by
\begin{eqnarray*}
(f_1,f_2) = \int_{K/M}f_1(k)\overline{f_2(k)}\,dk = \langle {f_1}_{|K},{f_2}_{|K}\rangle_{L^2(K/M)}
\end{eqnarray*}
and denote the corresponding norm by $\|f\|^2 = \int_{K/M}|f(k)|^2\,dk$. The group action of $G$ is given by
\begin{eqnarray*}
(\pi_{\lambda}(g)f)(x)=f(g^{-1}x).
\end{eqnarray*}
The actual Hilbert space, which we denote by $H_{\lambda}$, and the representation on $H_{\lambda}$, which we also denote by $\pi_{\lambda}$, is obtained by completion (cf. \cite{Wil}, Ch. 9).

The representations $\pi_{\lambda}$ ($\lambda\in\la$) form the \emph{spherical principal series} of $G$. The representation $(\pi_{\lambda},H_{\lambda})$ is a unitary (\cite{He94}, p. 528) and irreducible (loc. cit. p. 530) Hilbert space representation.

Given $f\in C^{\infty}(K/M)$, we extend it to a function on $G$ by putting
\begin{eqnarray}\label{f tilde}
\widetilde{f}(g)=e^{-(i\lambda+\rho)H(g)}f(k(g)),
\end{eqnarray}
where $g=k(g)\exp H(g)\,n(g)$ according to the Iwasawa decomposition.

\begin{prop}\label{spherical principal series identification}
\begin{itemize}
\item[(i)] For $f\in C^{\infty}(K/M)$ let $\widetilde{f}$ as in \eqref{f tilde}. Then $\widetilde{f}\in H^{\infty}_{\lambda}$.
\item[(ii)] Let $\widetilde{f}\in H^{\infty}_{\lambda}$ and denote restriction to $K$ by $\widetilde{f}_{|K}$. Then $\widetilde{f}_{|K}\in C^{\infty}(K/M)$
\item[(iii)] Let $f\in C^{\infty}(K/M)$ and $\widetilde{f}$ as in \eqref{f tilde}. Then $\widetilde{f}_{|K}=f$.
\item[(iv)] The mapping $f\mapsto\widetilde{f}$ is isometric with respect to the $L^2(K/M)$-norm. It intertwines the representation $\pi_{\lambda}$ and the representation (which we also denote by $\pi_{\lambda}$) on $C^{\infty}(K/M)$ defined by
\begin{eqnarray}\label{compact model action}
(\pi_{\lambda}(g)f)(kM) = f(k(g^{-1}k)M) e^{-(i\lambda+\rho)H(g^{-1}k)}.
\end{eqnarray}
\end{itemize}
\end{prop}
\begin{proof}
All assertions are clear.
\end{proof}

In view of Proposition \ref{spherical principal series identification} we identify $C^{\infty}(K/M)\cong H_{\lambda}^{\infty}$. The advantage of $C^{\infty}(K/M)$ is that the representation space is independent of $\lambda$. The representations \eqref{compact model action} are called the \emph{compact picture} (\emph{compact realization}) of the (spherical) principal series. Notice that for $g\in K$ the group action \eqref{compact model action} simplifies to the left-regular representation of the compact group $K$ on $K/M$.

Let $\lambda\in\la$ and denote by $1$ the constant function $k\mapsto 1$ on $K/M$. In the compact picture we observe
\begin{eqnarray}\label{Poisson follows}
(\pi_{\lambda}(g)1)(k)=e^{-(i\lambda+\rho)H(g^{-1}k)}=e^{(i\lambda+\rho)\langle gK,kM\rangle},
\end{eqnarray}
and it follows that the Poisson transform
\begin{eqnarray}
P_{\lambda}(T): G/K\rightarrow \cc
\end{eqnarray}
of $T\in\mathcal{D}'(B)$ is given by
\begin{eqnarray}\label{Poisson intertwines}
P_{\lambda}(T)(gK) = T(\pi_{\lambda}(g)\cdot 1).
\end{eqnarray}

It follows that the Poisson transform $P_{\lambda}$ intertwines the dual spherical principal series representation $\widetilde{\pi}_{\lambda}$ and the translation on $G/K$. Now suppose that $\phi\in\mathcal{E}_{\lambda}(X)$ is a $\Gamma$-invariant joint eigenfunction of $\mathbb{D}(G/K)$ with boundary values $T_{\phi}\in\mathcal{D}'(B)$ such that $\phi=P_{\lambda}(T_{\phi})$. Since $\phi$ is invariant, it follows from \eqref{Poisson intertwines} and the uniqueness of the boundary values that $T_{\phi}$ is invariant under all $\widetilde{\pi}_{\lambda}(\gamma)$, for $\gamma\in\Gamma$. Vice versa, if $T$ is a $\Gamma$-invariant distribution, then $P_{\lambda}(T)$ is a $\Gamma$-invariant eigenfunction.

\subsection{Regularity of distributional boundary values}\label{Regularity}
Before we start with our investigation on the regularity of distributional boundary values for eigenfunctions in general symmetric spaces, we motivate this section by recalling some results proven by Otal for compact hyperbolic surfaces. We use the notation of \cite{Otal} and \cite{AZ}.
\begin{defn}
For $0\leq\delta\leq1$ we say that a $2\pi$-periodic function $F:\rr\rightarrow\cc$ is $\delta$-H\"{o}lder if there exists $C\geq 0$ such that $|F(x)-F(y)|\leq C|x-y|^{\delta}$. The smallest constant $C$ is denoted by $\|F\|_{\delta}$. The Banach space of $\delta$-H\"{o}lder functions with norm $\|F\|_{\delta}$ is denoted by $\Lambda^{\delta}$.
\end{defn}

\begin{thm}[\cite{Otal}, Proposition 4]
Suppose that $s=\frac{1}{2}+ir$ with $Re(s)\geq 0$, and that $\phi$ is an eigenfunction of the Laplace operator of $\mathbb{H}_{\Gamma}$ satisfying $\|\nabla\phi\|_{L^{\infty}}<0$. Then its Helgason boundary value $T_{\phi}$ is the derivative of a $Re(s)$-H\"{o}lder function.
\end{thm}

Since outside a finite number of small eigenvalues $s$ of $\mathbb{H}_{\Gamma}$ belonging to the complementary series we always have $Re(s)=\frac{1}{2}$ (for eigenvalues $s=\frac{1}{2}+ir$ of the Laplacian on $\mathbb{H}_{\Gamma}$), it follows that almost all boundary values associated to eigenfunctions and eigenvalues belonging to the discrete spectrum of $\mathbb{H}_{\Gamma}$, are derivatives of certain $\frac{1}{2}$-continuous H\"{o}lder functions. To be more precise, the boundary values are not literally the derivative of a periodic function, but the derivative of a function $F$ on $\rr$ satisfying $F(x+2\pi) = F(x) + C$ for all $x\in\rr$.

As described in \cite{AZ}, it follows from Otal's regularity statement, that given an eigenfunction $\phi$ to the eigenvalue $s=\frac{1}{2}+ir$ with $r\in\rr$, then the H\"{o}lder norm of the corresponding boundary values $T_{\phi,r}$ is bounded by a power of $r$.

As noted in \cite{GO}, there seems to be no straightforward generalization of these concepts, not even in the case of the real hyperbolic spaces. However, related approaches can be found, for example, in \cite{GO}.

In this subsection we give a representation theoretic approach to describe the regularity of distributional boundary values and its dependence on the spectral parameter $\lambda$ and we prove a regularity statement for the boundary values corresponding to joint eigenfunctions with real eigenvalue parameter $\lambda\in\la^*$ on a compact quotient $X_{\Gamma}$. These estimates may not be the sharpest possible, but they are sufficient for our purposes.

Given $\lambda\in\la^*_{\cc}$, let $\mathcal{D}'(B)_{\Gamma}$ denote the space of distribuions on $B$ which are invariant under all actions $\widetilde{\pi}_{\lambda}(\gamma)$ ($\gamma\in\Gamma$). As described in the preceding subsection, if $T\in\mathcal{D}'(B)_{\Gamma}$, then the Poisson transform $P_{\lambda}(T)$ is a function on the quotient $X_{\Gamma}$. We may hence also define\index{$\mathcal{D}'(B)_{\Gamma}^{(1)}$, space of distributions}
\begin{eqnarray}\label{my space}
\mathcal{D}'(B)_{\Gamma}^{(1)} := \left\{ T\in\mathcal{D}'(B)_{\Gamma}: \left\| P_{\lambda}(T)\right\|_{L^2(X_{\Gamma})} = 1  \right\}.
\end{eqnarray}\index{$\mathcal{D}'(B)_{\Gamma}^{(1)}$, subspace of $\mathcal{D}'(B)$}
Now fix $\lambda\in\la^*$ and a $\Gamma$-invariant joint eigenfunction $\phi\in\mathcal{E}_{\lambda}(X)$ of $\mathbb{D}(G/K)$ (it has automatically exponential growth, since it is $\Gamma$-invariant). We also assume that $\phi$ is normalized with respect to the customary $L^2(X_{\Gamma})$-norm. Let $T_{\phi}\in\mathcal{D}'(B)_{\Gamma}^{(1)}$ denote be the (unique) preimage (under the Poisson transform) of $\phi$. 

Under the identification $H_{\lambda}^{\infty}\cong C^{\infty}(K/M)$ we view $T_{\phi}$ as a functional on $H_{\lambda}^{\infty}$: For $f\in H_{\lambda}^{\infty}$ let $T_{\phi}(f)$ be defined by $T_{\phi}(f_{|K})$. Then $T_{\phi}$ is a continuous linear functional on $H_{\lambda}^{\infty}$, invariant under $\widetilde{\pi}_{\lambda}(\gamma)$. As proven in \cite{CG}, Theorem A.1.4, if $f$ is a smooth vector for the principal series representation, then $f\in H_{\lambda}^{\infty}$ is a smooth function on $G$. We consider the mapping
\begin{eqnarray*}
\Phi_{\phi}: H_{\lambda}^{\infty} \rightarrow
C^{\infty}(\Gamma\backslash G), \,\,\,\,\,\,\,\, \Phi_{\phi}(f)(\Gamma g)
= T_{\phi}(\pi_{\lambda}(g)f).
\end{eqnarray*}

\begin{lem}
$\Phi_{\phi}$ is an isometry w.r.t. the norms of $L^2(K/M)$ and $L^2(\Gamma\backslash G)$.
\end{lem}

\begin{proof}
The operator $\Phi_{\phi}$ is equivariant with respect to the
actions $\pi_{\lambda}$ on $H^{\infty}_{\lambda}$ and the right
regular representation of $G$ on $L^2(\Gamma\backslash G)$. We
pull-back the $L^2(\Gamma\backslash G)$ inner product onto the
$(\g,K)$-module $H^{\infty}_{\lambda,K}$ of $K$-finite and smooth
vectors (which is dense in $H^{\infty}_{\lambda}$, \cite{Wal2}, p.
81):
\begin{eqnarray*}
\langle f_1,f_2\rangle_{2} := \langle \Phi_{\phi}(f_1),\Phi_{\phi}(f_2)\rangle_{L^2(\Gamma\backslash G)}.
\end{eqnarray*}
Let $f_1\in H^{\infty}_{\lambda,K}$. Then
\begin{eqnarray*}
A_{f_1}: H^{\infty}_{\lambda,K} \rightarrow \cc, \,\,\,\,\,\,\,\, f_2 \mapsto \langle f_1,f_2\rangle_{2}
\end{eqnarray*}
is a conjugate-linear, $K$-finite functional on the $(\g,K)$-module $H^{\infty}_{\lambda,K}$. This module is irreducible and admissible, since $H_{\lambda}$ is unitary and
irreducible (\cite{Wal2}, Thm. 3.4.10, Thm. 3.4.11). As $A_{f_1}$ is $K$-finite it is nonzero on at most finitely many $K$-isotypic components. It follows that there is a linear map $A:H^{\infty}_{\lambda,K}\rightarrow H^{\infty}_{\lambda,K}$ such that for each $f_1\in H^{\infty}_{\lambda,K}$ the functional $A_{f_1}$ equals $f_2\mapsto \langle A f_1,f_2\rangle_{L^2(K/M)}$. The equivariance of $\Phi_{\phi}$ and the unitarity of $\pi_{\lambda}$ imply that $A$ is $(\g,K)$-equivariant. Using Schur's lemma for irreducible $(\g,K)$-modules (\cite{Wal2}, p. 80), we deduce that $A$ is a constant multiple of the identity and hence $\langle \cdot,\cdot\rangle_{2}$ is a constant multiple of the original $L^2(K/M)$-inner product on $H_{\lambda,K}^{\infty}$. This constant is $1$: First, $\Phi_{\phi}(1) = P_{\lambda}(T_{\phi}) = \phi$ is the $K$-invariant lift of $\phi$ to $L^2(\Gamma\backslash G)$. Then $\|\Phi_{\phi}(1)\|_{L^2(\Gamma\backslash G)}=1=\|1\|_{L^2(K/M)}$.
\end{proof}

Let $(y_j)$ and $(x_j)$ be bases for $\lk$ and $\lp$, respectively,
such that $\langle y_j,y_i\rangle = -\delta_{ij}$, $\langle
x_j,x_i\rangle = \delta_{ij}$, where $\langle\,,\,\rangle$ denotes
the Killing form. The Casimir operator of $\lk$ is
$\Omega_{\lk}=\sum_i y_i^2$ and the Casimir operator of $\g$ is
\begin{eqnarray*}
\Omega_{\g} = -\sum_{j}x_j^2 + \Omega_{\lk} \in \mathcal{Z}(\g),
\end{eqnarray*}
where $\mathcal{Z}(\g)$ is the center of the universal enveloping algebra $\mathcal{U}(\g)$ of $\g$.

It follows from $T_{\phi}(f) = \Phi_{\phi}(f)(\Gamma e)$ that
\begin{eqnarray}\label{first estimate}
|T_{\phi}(f)|\leq \|\Phi_{\phi}(f)\|_{\infty}.
\end{eqnarray}
We may now estimate this by a convenient Sobolev norm on
$L^2(\Gamma\backslash G)$. Let $\widetilde{\Delta}$ denote the
Laplace operator of $\Gamma\backslash G$. Then we have
\begin{eqnarray*}
\widetilde{\Delta} = - \Omega_{\mathfrak{g}} + 2\Omega_{\mathfrak{k}},
\end{eqnarray*}
where $\Omega_{\mathfrak{g}}$ and $\Omega_{\mathfrak{k}}$ are the
Casimir operators on $G$ and $K$, respectively.

\begin{defn}
Let $s\in \rr$. The Sobolev space $W^{2,s}(\Gamma\backslash G)$ is
(cf. \cite{Tay81}, p. 22) the space of functions $f$ on
$\Gamma\backslash G$ satisfying $(1+\widetilde{\Delta})^{s/2}(f)\in
L^2(\Gamma\backslash G)$ with norm
\begin{eqnarray*}
\| f \|_{W^{2,s}(\Gamma\backslash G)} = \| (1+\widetilde{\Delta})^{s/2}(f) \|_{L^2(\Gamma\backslash G)}.
\end{eqnarray*}
\end{defn}

Let $m=\dim(\Gamma\backslash G)=\dim(G)$, and let $s>m/2$. The
Sobolev imbedding theorem for the compact space $\Gamma\backslash G$
(\cite{Tay81}, p. 19) states that the identity
$W^{2,s}(\Gamma\backslash G) \hookrightarrow C^0(\Gamma\backslash
G)$ is a continuous inclusion ($C^0(\Gamma\backslash G)$ is equipped
with the usual sup-norm $\|\cdot\|_{\infty}$). It follows that there
exists a $C>0$ such that
\begin{eqnarray}\label{this is 0}
\| \Phi_{\phi}(f)\|_{\infty} \leq C \| \Phi_{\phi}(f)\|_{W^{2,s}(\Gamma\backslash G)} \,\,\,\,\,\,\,\,
\forall f\in C^{\infty}(K/M)  .
\end{eqnarray}

Now we derive the announced regularity estimate for the boundary
values: First, by increasing the Sobolev order, we may assume
$s/2\in\nn$, so
\begin{eqnarray*}
(1+\widetilde{\Delta})^{s/2} =
(1-\Omega_{\mathfrak{g}}+2\Omega_{\mathfrak{k}})^{s/2} \in
\mathcal{U}(\mathfrak{g}).
\end{eqnarray*}
Hence $(1+\widetilde{\Delta})^{s/2}$ commutes with each
$G$-equivariant mapping. Let $f\in H_{\lambda}^{\infty}$. Then
\begin{eqnarray}\label{this is 1}
\left\|\Phi_{\phi}(f)\right\|_{W^{2,s}(\Gamma\backslash G)}
&=& \left\|(1+\widetilde{\Delta})^{s/2}\Phi_{\phi}(f)\right\|_{L^2(\Gamma\backslash G)} \nonumber \\
&=& \left\|\Phi_{\phi}((1-\Omega_{\g}+2\Omega_{\lk})^{s/2}(f))\right\|_{L^2(\Gamma\backslash G)} \nonumber \\
&=& \left\| (1-\Omega_{\g}+2\Omega_{\lk})^{s/2}(f) \right\|_{L^2(K/M)}.
\end{eqnarray}
Recall $\pi_{\lambda}(\Omega_{\mathfrak{k}})=\Delta_{K/M}$ and
$\Omega_{\g}\in \mathcal{Z}(\g)$. Then \eqref{this is 1} equals
\begin{eqnarray}\label{this is 2}
&&\left\| \sum_{k=0}^{s/2} \binom{s/2}{k} (1+2\Delta_{K/M})^{k}
  (-\Omega_{\g})^{s/2-k}(f)\right\|_{L^2(K/M)} \nonumber \\
&& \hspace{3mm} \leq \, \sum_{k=0}^{s/2} \binom{s/2}{k} \left\|
   (1+2\Delta_{K/M})^{k} (-\Omega_{\g})^{s/2-k}(f)\right\|_{L^2(K/M)}.
\end{eqnarray}
Assume $f\in H_{\lambda.K}^{\infty}$ and recall that $\Omega_{\g}$
acts on the irreducible $\mathcal{U}(\g)$-module $H_{\lambda,K}^{\infty}$ by
multiplication with the scalar $-(\langle\lambda,\lambda\rangle +
\langle\rho,\rho\rangle)$ (cf. \cite{Wil}, p. 163), that is
\begin{eqnarray*}
{\Omega_{\g}}_{|H_{\lambda,K}^{\infty}} = -\left(\langle\lambda,\lambda\rangle
+ \langle\rho,\rho\rangle\right) \id_{H_{\lambda,K}^{\infty}}.
\end{eqnarray*}
Then \eqref{this is 2} equals
\begin{eqnarray}\label{this is 3}
\sum_{k=0}^{s/2} \binom{s/2}{k} \left\| (1+2\Delta_{K/M})^k
(|\lambda|^2+|\rho|^2)^{s/2-k}(f) \right\|_{L^2(K/M)}.
\end{eqnarray}
But $\left(|\lambda|^2+|\rho|^2\right)^{-k}\leq 1 + |\rho|^{-s} =:
C'$ ($0\leq k\leq s/2$), so \eqref{this is 3} is bounded
by
\begin{eqnarray}\label{this is 4}
C'\left(|\lambda|^2+|\rho|^2\right)^{s/2} \sum_{k=0}^{s/2}
\binom{s/2}{k} \left\| (1+2\Delta_{K/M})^{k}(f) \right\|_{L^2(K/M)}.
\end{eqnarray}
Since $H_{\lambda.K}^{\infty}$ is dense in $H_{\lambda}^{\infty}$,
this bound holds for all $f\in H_{\lambda}^{\infty}$. Using
\eqref{first estimate}-\eqref{this is 4} we get
\begin{eqnarray}\label{estimate}
|T_{\phi}(f)| \leq C'\left(|\lambda|^2+|\rho|^2\right)^{s/2}
\sum_{k=0}^{s/2} \binom{s/2}{k} \left\| (1+2\Delta_{K/M})^{k}(f)
\right\|_{L^2(K/M)}.
\end{eqnarray}
for all $f\in H_{\lambda}^{\infty}$ and hence for all $f\in
C^{\infty}(K/M)$. We estimate the sum in \eqref{estimate} by the continuous $C^{\infty}(K/M)$-seminorm (recall that $K/M$ has normalized volume)
\begin{eqnarray}\label{seminorm}
\|f\|' := \sum_{k=0}^{s/2} \binom{s/2}{k} \sup_{K/M} |(1+2\Delta_{K/M})^{k}(f)|, \,\,\,\,\,\,\,\,\,\, f\in\mathcal{E}(B),
\end{eqnarray}
(where $2s>\dim(G)$ is arbitrary, but fixed) and define
\begin{eqnarray}\label{spectral parameter seminorm}
\mathcal{D}'(B)_{\lambda} := \left\{ T\in\mathcal{D}'(B): |T(f)|\leq (1+|\lambda|)^s \|f\|' \,\,\,\, \forall \, f\in C^{\infty}(K/M) \right\}.
\end{eqnarray}\index{$\mathcal{D}'(B)_{\lambda}$, subspace of $\mathcal{D}'(B)$}
(Note that $\mathcal{D}'(B)_{\lambda}$ depends on the number $s>\dim(G)/2$). We summarize these obervations as follows:
\begin{lem}\label{Lemma Olbrich}
$\mathcal{D}'(B)_{\Gamma}^{(1)}\subseteq\mathcal{D}'(B)_{\lambda}$.
\end{lem}

\subsection{Tensor products of distributional boundary values}
We need to recall some background concerning tensor products of distributions, which is naturally based on the tensor product of the underlying test function spaces and their completions. We assume that the reader is familiar with the definitions of the customary algebraic tensor product of general vector spaces. We are mainly interested in the compatibility of the tensor product for distributions with the embedding $f\mapsto I_f$ \eqref{imbedding of functions} of functions into distributions and the tensor product for functions. The material is taken from \cite{T} and \cite{BB}.

If $\Omega_j$ are non-empty open subsets of $\rr^{n_j}$ and $\phi_j\in\mathcal{D}(\Omega_j)$ are test functions, their tensor product is the function $\phi_1\otimes\phi_2\in\mathcal{D}(\Omega_1\times\Omega_2)$ defined by
\begin{eqnarray*}
\phi_1\otimes\phi_2(x_1,x_2)=\phi(x_1)\phi(x_2) \,\,\,\,\,\,\,\,\,\,\, (x_j\in\Omega_j).
\end{eqnarray*}
The vector space spanned by all these tensors is denoted by $\mathcal{D}(\Omega_1)\otimes\mathcal{D}(\Omega_2)$. A general element in $\mathcal{D}(\Omega_1)\otimes\mathcal{D}(\Omega_2)$ is a finite sum $\sum_j\phi_j\otimes\psi_j$, where $\phi_j\in\mathcal{D}(\Omega_1)$, $\psi_j\in\mathcal{D}(\Omega_2)$. Then
\begin{eqnarray*}
\mathcal{D}(\Omega_1)\otimes\mathcal{D}(\Omega_2)\subset\mathcal{D}(\Omega_1\times\Omega_2).
\end{eqnarray*}
This tensor product space is dense in the test function space $\mathcal{D}(\Omega_1\times\Omega_2)$.

On the algebraic tensor product $E\otimes F$ of two Hausdorff locally convex topological spaces $E$ and $F$ over the same field one can define the \emph{projective tensor product} as follows. Let $\mathcal{P}$ and $\mathcal{Q}$ denote the respective filtering systems of seminorms defining the topology of the respective spaces $E$ and $F$. A general element $\chi\in\mathcal{E\otimes F}$ is of the form
\begin{eqnarray*}
\chi = \sum_{j=1}^{m} e_j\otimes f_j \,\,\,\,\,\,\,\,\,\,\, (e_j\in E, f_j\in F).
\end{eqnarray*}
This representation as a finite sum is not unique. Given semi-norms $p\in\mathcal{P}, q\in\mathcal{Q}$ we define the projective tensor product by
\begin{eqnarray*}
p\otimes_{\pi}q := \inf\left\{\sum_{j=1}^{m}p(e_j)q(f_j): \chi=\sum_{j=1}^{m}e_j\otimes f_j\right\}.
\end{eqnarray*}
Then $p\otimes_{\pi}q$ defines a semi-norm on $E\otimes F$, and the system
\begin{eqnarray*}
\mathcal{P}\otimes_{\pi}\mathcal{Q} := \left\{ p\otimes_{\pi} q: p\in\mathcal{P},q\in\mathcal{Q} \right\}
\end{eqnarray*}
is a filtering and thus defines a locally convex topology on $E\otimes F$, called the \emph{projective tensor product topology}. The vector space equipped with this topology is denoted by
\begin{eqnarray*}
E\otimes_{\pi} F
\end{eqnarray*}
and is called the \emph{projective tensor product of the spaces $E$ and $F$}·

In particular, if $X,Y\subset\rr^n$ are both open or compact, then the completion $\mathcal{D}(X)\widehat{\otimes}_{\pi}\mathcal{D}(Y)$ of the projective tensor product $\mathcal{D}(X)\otimes_{\pi}\mathcal{D}(Y)$ is equal to the test function space $\mathcal{D}(X\times Y)$ over the product $X\times Y$ (\cite{T}, p. 530):
\begin{eqnarray*}
\mathcal{D}(X)\widehat{\otimes}_{\pi}\mathcal{D}(Y) = \mathcal{D}(X\times Y).
\end{eqnarray*}

Let $\Omega_j$ be as above. For $\phi\in\mathcal{D}(\Omega_1\otimes\Omega_2)$ and $T\in\mathcal{D}'(\Omega_1)$ we define a function $\psi$ on $\Omega_2$ by $\psi(y)=\langle T,\phi_y\rangle$, where $\phi(y)(x):=\phi(x,y)$. Then $\psi\in\mathcal{D}(\Omega_2)$, and $F(T,\phi):=\psi$ defines a bilinear map $F:\mathcal{D}'(\Omega_1)\times\mathcal{D}(\Omega_1\times\Omega_2)\rightarrow\mathcal{D}(\Omega_2)$. This yields the existence of the tensor product for distributions:

For $T_j\in\mathcal{D}'(\Omega_j)$ there is exactly one distribution $T\in\mathcal{D}'(\Omega_1\times\Omega_2)$, called the tensor product of $T_1$ and $T_2$, such that (\cite{BB}, Ch. 6.2)
\begin{eqnarray*}
\left\langle T,\phi_1\otimes\phi_2 \right\rangle = \left\langle T_1,\phi_1\right\rangle\left\langle T_2,\phi_2\right\rangle.
\end{eqnarray*}

Recall the embedding of functions into distributions as given in \eqref{imbedding of functions}. If $f,g\in L^1_{\textnormal{loc}}(\Omega)$, a direct computation shows (loc. cit.)
\begin{eqnarray*}
\left\langle I_f\otimes I_g,\phi\otimes\psi\right\rangle = \left\langle I_f,\phi\right\rangle\left\langle I_g,\psi\right\rangle,
\end{eqnarray*}
so the tensor product of distributions is consistent with the tensor product of functions.

For convenience, if $T$ is a distribution and $f$ a test function on a space $Y$, we sometimes write $\left\langle T(y),f(y)\right\rangle$ for the pairing between $T$ and $f$ instead of $\left\langle T,f\right\rangle$ to point out the active variables.

The tensor product of $T_1$ and $T_2$ is a continuous linear functional on $\Omega_1\times\Omega_2$ and it satisfies Fubini's theorem for distributions: For every $T_j\in\mathcal{D}(\Omega_j)$ and for every $\chi\in\mathcal{D}(\Omega_1\times\Omega_2)$ one has (loc. cit.)
\begin{eqnarray*}
\left\langle T_1\otimes T_2,\chi\right\rangle &=& \left\langle(T_1\otimes T_2)(x,y),\chi(x,y)\right\rangle \\
&=& \left\langle T_1(x),\left\langle T_2(y)\chi(x,y)\right\rangle\right\rangle \\
&=& \left\langle T_1(y),\left\langle T_2(x)\chi(x,y)\right\rangle\right\rangle.
\end{eqnarray*}

We can now apply the definitions given above to tensor products of distributional boundary values. As usual, let $B=K/M$ denote the real flag manifold of belonging to the Riemannian symmetric space $X=G/K$ of noncompact type.

In the notation of Section \ref{Regularity}, there is a continuous seminorm $\left\|\cdot\right\|'$ on $C^{\infty}(B)$ and a constant $K$ such that for all distributional boundary values $T_{\phi,\lambda}$ corresponding under the Poisson transform to a $\Gamma$-invariant joint eigenfunction $\phi\in\mathcal{E}^*_{\lambda}(X)$ with $\|\phi\|_{L^2(X_{\Gamma})}=1$ we have
\begin{eqnarray*}
|T(f)| \leq (1+|\lambda|)^K\|f\|' \,\,\,\,\, \forall \, f\in C^{\infty}(B).
\end{eqnarray*}
Each $f\in C^{\infty}(B)\otimes C^{\infty}(B)$ has the form
\begin{eqnarray*}
f = \sum_{i,j}c_{i,j}f_i\otimes f_j.
\end{eqnarray*}
We define a cross-(semi-)norm $\|\cdot\|''$ on the customary algebraic tensor product
$C^{\infty}(B)\otimes C^{\infty}(B)$ by
\begin{eqnarray*}
\|f\|'' = \inf\left\{ \sum_{i,j}|c_{i,j}|\|f_i\|'\|f_j\|' : f = \sum_{i,j}c_{i,j}f_i\otimes f_j \right\}.
\end{eqnarray*}
Then by \cite{T}, p. 435, this norm induces a continuous seminorm on the projective tensor product $C^{\infty}(B)\widehat{\otimes}_{\pi} C^{\infty}(B)$.

Let $\phi\in\mathcal{E}^*_{\lambda}(X)$ and $\psi\in\mathcal{E}^*_{\mu}(X)$ denote $\Gamma$-invariant and $L^2(X_{\Gamma})$-normalized eigenfunction with distributional boundary values $T_{\phi},T_{\psi}\in\mathcal{D}'(B)$ and eigenvalue parameter $\mu\in\la^*$. Given
\begin{eqnarray*}
f = \sum_{i,j}c_{i,j}f_i\otimes f_j \in C^{\infty}(B)\otimes C^{\infty}(B)
\end{eqnarray*}
we obtain
\begin{eqnarray*}
|(T_{\phi}\otimes T_{\psi})(f)|
&\leq& \sum_{i,j}|c_{i,j}| \cdot |T_{\phi}(f_i)| \cdot |T_{\psi}(f_j)| \nonumber \\
&\leq& (1+|\lambda|)^{s}(1+|\mu|)^{s} \sum_{i,j}|c_{i,j}|\cdot\|f_i\|'\cdot\|f_j\|',
\end{eqnarray*}
which implies (by taking the infimum)
\begin{eqnarray}\label{BtimesB}
|(T_{\phi}\otimes T_{\psi})(f)| \leq (1+|\lambda|)^{s}(1+|\mu|)^{s} \|f\|''
\end{eqnarray}
for all $f\in C^{\infty}(B)\otimes C^{\infty}(B)$. But
\begin{eqnarray*}
C^{\infty}(B\times B) \cong C^{\infty}(B) \widehat{\otimes}_{\pi} C^{\infty}(B)
\end{eqnarray*}
(\cite{T}, p. 530) implies that \eqref{BtimesB} holds for all $f\in C^{\infty}(B\times B)$.

\newpage
\pagebreak
\thispagestyle{empty} 

\section{Patterson-Sullivan distributions}\label{Patterson-Sullivan distributions}

In this Section, we introduce \emph{Patterson-Sullivan distributions} for symmetric spaces of the noncompact type and establish a couple of invariance properties. It will then turn out how these phase space distributions are related to the questions of quantum ergodicity.

We carry over the notation from the preceding chapters. $G$ denotes a noncompact semisimple Lie group with finite center and Iwasawa decomposition $G=KAN$. By $X=G/K$ we denote the corresponding symmetric space of the noncompact type. By $B=K/M$ we denote the (Fürstenberg) boundary of $X$. Given a cocompact torsion free discrete subgroup $\Gamma$ of $G$, we denote by $X_{\Gamma}$ the corresponding locally symmetric compact manifold of nonpositive sectional curvature. At this point, we make no restriction on the rank of $X$. In general, a (diagonal) Patterson-Sullivan distribution $ps_{\lambda}=ps_{\phi,\lambda}$ will be associated to a joint eigenfunction $\phi\in\mathcal{E}^*_{\lambda}(X)$, where $\lambda\in\la^*_{\cc}$.

In Subsection \ref{Intermediate values} we build up a concept of functions, which we call \emph{intermediate values}. The intermediate values depend on the spectral parameter $\lambda$. Invariance properties of the Patterson-Sullivan distributions arise from equivariance properties of the intermediate values. Tensoring the $ps_{\lambda}$-distributions with an appropriate \emph{Radon transform}, one obtains $A$-invariant distributions $PS_{\phi,\lambda}$. In Subsection \ref{Diagonal Patterson-Sullivan distributions} we generalize the constructions given in \cite{AZ} to symmetric spaces of the noncompact type. We will explain that these special constructions are only possible for eigenvalue-parameters that satisfy a certain condition (see Lemma \ref{minus identity}). It is not possible to generalize these definitions to a larger class of eigenfunctions and eigenvalues. Eigenvalues of the Laplacian of a rank one space satisfy this condition. In Subsection \ref{Off-diagonal Patterson-Sullivan distributions} we introduce off-diagonal Patterson-Sullivan distributions $PS_{\phi,\lambda,\psi,\mu}$, which are associated to two eigenfunctions $\phi\in\mathcal{E}^*_{\lambda}(X)$ and $\psi\in\mathcal{E}^*_{\mu}(X)$. These distributions exist for all symmetric spaces of the noncompact type. If $\phi=\psi$, they coincide with the $PS_{\phi,\lambda}$ for the special cases considered in Subsection \ref{Diagonal Patterson-Sullivan distributions}.

\subsection{Intermediate values}\label{Intermediate values}
Let $\mathbb{H}^n$ be the real hyperbolic space of dimension $n$, that is, the complete and simply connected Riemannian manifold of constant curvature $-1$. Using the Poincaré model, we identify $\mathbb{H}^n$ with the unit ball of $\rr^n$ and its (geodesic) boundary at infinity $\partial\mathbb{H}^n$ with the unit sphere $S^{n-1}$ of $\rr^n$.

For $z\in\mathbb{H}^n$, let $\gamma$ be an isometry of $\mathbb{H}^n$ such that $z=\gamma^{-1}\cdot 0$, where $0\in\rr^n$ is the origin of $\mathbb{H}^n$. Then $|\gamma'(\xi)| = P(z,\xi)$, where $P$ is the Poisson kernel of $\mathbb{H}^n$ and where $|\gamma'(\xi)|$ is the conformal factor of the derivative of $\gamma$ at the point $\xi\in S^{n-1}$. 

Given two points $\xi,\xi'\in S^{n-1}$, we denote their chordal (Euclidean) distance by $|\xi-\xi'| = 2\sin(\theta/2)$, where $\theta$ is the spherical distance between $\xi$ and $\xi'$. One has the \emph{intermediate value formula} (cf. \cite{Sul})
\begin{eqnarray}\label{intermediate value formula}
|\gamma(\xi')-\gamma(\xi)|^2 = |\gamma'(\xi')||\gamma(\xi)||\xi'-\xi|^2.
\end{eqnarray}

The derivatives in \eqref{intermediate value formula} are (cf. \eqref{dgb / db}) given by $\frac{d(\gamma\cdot b)}{db}=e^{-2\rho\left\langle\gamma\cdot o,\gamma\cdot b \right\rangle}$. Suppose that $G=SU(1,1)$ and
\begin{eqnarray*}
K=\left\{\begin{pmatrix} e^{i\theta} & 0 \\ 0 & e^{-i\theta} \end{pmatrix}, \,\, \theta\in\rr  \right\}.
\end{eqnarray*}
Then the non-Euclidean disk $\mathbb{D}$ identifies with the symmetric space $G/K$. Writing $\rho=\frac{1}{2}$ we find
\begin{eqnarray}\label{raising to the power}
|\gamma b - \gamma b'| = e^{-(\langle\gamma\cdot o,\gamma\cdot b\rangle+\langle\gamma\cdot o,\gamma\cdot b'\rangle)}|b-b'|
\end{eqnarray}
for $b,b'\in\partial D$. Caution that the horocycle bracket $\langle z,b\rangle$ we use is written $\frac{1}{2}\langle z,b\rangle$ in \cite{AZ}, \cite{PJN} etc., because the hyperbolic metric on $\mathbb{D}$ is often defined to be a multiple of the metric used in \cite{He01}, \cite{He00}, \cite{He94}. (Sometimes the abelian subgroup $A=a_t$ of $G$ is parameterized by $t/2$ instead of $t$, that is $a_t=\textnormal{diag}(e^{t/2},e^{-t/2})\in G$.) Raising \eqref{raising to the power} to the power $\frac{1}{2} + ir$ we obtain
\begin{eqnarray}\label{idea}
|\gamma b - \gamma b'|^{\frac{1}{2} + ir} = e^{-(\frac{1}{2} + ir)\cdot(\langle\gamma\cdot o,\gamma\cdot b\rangle+\langle\gamma\cdot o,\gamma\cdot b'\rangle)}|b-b'|^{\frac{1}{2} + ir}.
\end{eqnarray}
In this setting it is standard (\cite{AZ}) to parameterize the eigenvalue parameters corresponding to the eigenvalues of $\Delta$ on compact hyperbolic surfaces by $\lambda_j = \frac{1}{2} + i r_j$. In the disk model we have $(b_{\infty},b_{-\infty})=(M,wM)\in B\times B$ is $(1,-1)\in\partial\mathbb{D}\times\partial\mathbb{D}$. Writing $(b,b')=(\gamma\cdot M,\gamma\cdot wM)$, then \eqref{idea} yields
\begin{eqnarray}
|b - b'|^{\frac{1}{2} + ir} = 2^{\frac{1}{2} + ir} e^{-(\frac{1}{2} + ir)\cdot(\langle\gamma\cdot o,\gamma\cdot 1\rangle+\langle\gamma\cdot o,\gamma\cdot (-1)\rangle)}.
\end{eqnarray}

For a general symmetric space $X=G/K$ with real flag manifold $B=K/M$ we can neither use a Poincaré ball model nor Euclidean distances. We will now see how to generalize equation \eqref{idea} in group-theoretical terms.

\subsubsection{Generalized intermediate values}
As usual, let $H$ denote the Iwasawa projection $KAN\rightarrow\mathfrak{a}$. We denote the longest Weyl group element and (by abuse of notation) a representative of it in $M'$ by $w$, where $M'$ is the normalizer of $A$ in $K$. Let $\lambda,\mu\in\mathfrak{a}^*_{\cc}$. We introduce the function $d_{\lambda,\mu}: G\rightarrow\cc$,
\begin{eqnarray}\label{dlambdamu on G/M}
d_{\lambda,\mu}(g) = e^{(i\lambda+\rho)H(g)}e^{(i\mu+\rho)H(gw)}.
\end{eqnarray}

\begin{defn}
We call the functions $d_{\lambda,\mu}$ \emph{off-diagonal intermediate values}. In the case when $\lambda=\mu$ we define \emph{diagonal intermediate values} $d_{\lambda}:=d_{\lambda,\lambda}$.
\end{defn}

Recall that the action of $W$ on $\la^*$ is defined via duality by
\begin{eqnarray*}
(s\cdot\nu)(X):=\nu(s^{-1}\cdot X),
\end{eqnarray*}
where $s\in W$, $\nu\in\la^*$, $X\in\la$, and where $\cdot$ denotes the adjoint action. We have $s\cdot X\in\la^*_{\cc}$, since $M'$ (hence $s\in W$) normalizes $A$ and $\la^*_{\cc}$. The action is extended to $\la^*_{\cc}$ by complex linearity.

\begin{lem}
Let $g\in G$, $m\in M$, $a\in A$. Then
\begin{eqnarray}\label{observation}
d_{\lambda,\mu}(gam) = d_{\lambda,\mu}(g) e^{i(\lambda+w\cdot\mu)\log(a)}.
\end{eqnarray}
\end{lem}
\begin{proof}
Recall that the Iwasawa-projection $H$ is $M$-invariant and that $M'$ (hence $w\in W$) normalizes $M$. Then
\begin{eqnarray*}
d_{\lambda,\mu}(gam) &=& e^{(i\lambda+\rho)H(gam)} e^{(i\mu+\rho)H(gamw)} \\
&=& e^{(i\lambda+\rho)H(ga)} e^{(i\mu+\rho)H(gww^{-1}aw)} \\
&=& e^{(i\lambda+\rho)(H(g)+\log(a))} e^{(i\mu+\rho)(H(gw)+\log(w^{-1}aw))} \\
&=& e^{(i\lambda+\rho)H(g)} e^{(i\mu+\rho)H(gw)} e^{(i\mu+\rho)(\log(w^{-1}aw))} e^{(i\lambda+\rho)\log(a)}\\
&=& d_{\lambda,\mu}(g) e^{(i\mu+\rho)(\log(w^{-1}aw))} e^{(i\lambda+\rho)\log(a)}.
\end{eqnarray*}
Also recall $\log(w^{-1}aw)=w^{-1}\cdot\log(a)$, since $\exp$ and $\log$ intertwine $\Ad_G(w)$ with the conjugation by $w$ on $A$. It follows from $w\cdot\rho=-\rho$ that the last line equals
\begin{eqnarray*}
d_{\lambda,\mu}(g) e^{i(\lambda+w\cdot\mu)\log(a)},
\end{eqnarray*}
and the lemma is proven.
\end{proof}

\begin{rem}
\begin{itemize}
\item[(1)] The functions $d_{\lambda,\mu}$ are right-$M$-invariant. Thus
\begin{eqnarray*}
d_{\lambda,\mu}:G/M\rightarrow\cc.
\end{eqnarray*}
\item[(2)] Suppose that $w\cdot\lambda=-\lambda$. This is satisfied if the longest Weyl group element satisfies $\Ad_G(w)=-\id_{\la^*}$, which is for example true if $G/K$ has rank one. Then the diagonal intermediate values function $d_{\lambda}$ is invariant under right-translation by elements $a\in A$ and hence a function on $G/MA$. In all other cases, $d_{\lambda}$ is not a function on $G/MA$. We will see in \eqref{repair} how to circumvent this problem.
\item[(3)] Let $G/K$ have rank one. If $m'\in M'$, then $d_{\lambda}(gm')=d_{\lambda}(g)$, so $d_{\lambda}$ is a function on $G/M'A$.
\end{itemize}
\end{rem}

Recall (from \ref{The general case}) that $B^{(2)}\cong G/MA$: The group $G$ acts transitively on $B^{(2)}$. The closed subgroup of $G$ fixing $(M,wM)\in B^{(2)}$ is $MA$. Thus each pair of distinct boundary points $(b,b')$ may be written in the form $g\cdot(M,wM)$, where $g(b,b')MA=gMA\in G/MA$ is unique.

\begin{defn}
Time reversal refers to the involution on the unit cosphere bundle defined by $\iota(x,\xi)=(x,-\xi)$. Suppose that $G/K$ has rank one. Under $\Gamma\backslash G/M= S^* X_{\Gamma}$ the time reversal map takes the form $\Gamma g\mapsto \Gamma w g$. We say that a distribution $T$ is time-reversible if $\iota^*T=T$. Let $(b,b')=(g\cdot M,g\cdot wM)\in B^{(2)}$, where $g\in G$ and $gMA\in G/MA$ is unique. Recall $w^2\in M$. Then time reversal means
\begin{eqnarray*}
(b,b') = (g\cdot M,g\cdot wM) \mapsto (gw\cdot M,g\cdot w^2M) = (b',b),
\end{eqnarray*}
that is the interchanging $(b,b')\leftrightarrow(b',b)$. We call a function or distribution on $B^2$ time reversal invariant if it is invariant under $(b,b')\leftrightarrow(b',b)$.
\end{defn}

\begin{cor}\label{time reversal}
Let $G/K$ have rank one. The functions $d_{\lambda}$ are time reversal invariant. 
\end{cor}

For the rest of this subsection suppose that $\Ad_G(w)=-\id_{\la^*}$. Under the identification $B^{(2)}\cong G/MA$ the function $d_{\lambda}$ corresponds to a function on $B^{(2)}$ which we also denote by $d_{\lambda}$. If $g=g(b,b')$, then $d_{\lambda}: B^{(2)} \rightarrow \cc$,
\begin{eqnarray*}
d_{\lambda}(b,b') = d_{\lambda}(g\cdot M, g\cdot wM) = e^{(i\lambda+\rho)(H(g)+H(gw))}.
\end{eqnarray*}

Recall the horocycle bracket $\langle\cdot,\cdot\rangle$ on $G/K\times K/M$. Let $g\in G$. We have shown in Lemma \ref{Iwasawa and bracket} that $\langle g\cdot o, g\cdot M\rangle = H(g)$ and $\langle g\cdot o, g\cdot wM\rangle = H(gw)$.

\begin{cor}
Let $\langle\cdot,\cdot\rangle$ denote the horocycle-bracket. Then
\begin{eqnarray}\label{dlambda}
d_{\lambda}(g\cdot M, g\cdot wM) = e^{(i\lambda+\rho)(\langle g\cdot o, g\cdot M\rangle+\langle g\cdot o,g\cdot wM\rangle)}.
\end{eqnarray}
\end{cor}
\begin{proof}
This follows from $\langle g\cdot o, g\cdot M\rangle = H(g)$ and $\langle g\cdot o, g\cdot wM\rangle = H(gw)$.
\end{proof}

\begin{lem}\label{equivariance property off-diag}
Let $\gamma,g\in G$. Then
\begin{eqnarray}\label{equivariance property off-diag2}
d_{\lambda,\mu}(\gamma g) = e^{(i\lambda+\rho)\langle\gamma\cdot o, \gamma g\cdot M\rangle} e^{(i\mu+\rho)\langle\gamma\cdot o, \gamma g\cdot wM\rangle} d_{\lambda,\mu}(g).
\end{eqnarray}
\end{lem}
\begin{proof}
Let $z=g\cdot o\in G/K$. By \eqref{equivariance} and by Lemma \ref{Iwasawa and bracket} we find
\begin{eqnarray*}
H(\gamma g) &=& \langle \gamma g \cdot o, \gamma g \cdot M\rangle \\
&=& \langle\gamma\cdot z,\gamma g\cdot M\rangle \\
&=& \langle z,g\cdot M\rangle + \langle \gamma\cdot o, \gamma g\cdot M\rangle \\
&=& H(g) + \langle\gamma\cdot o,\gamma g\cdot M\rangle.
\end{eqnarray*}
Similarly we compute
\begin{eqnarray*}
H(\gamma gw) &=& \langle \gamma g \cdot o, \gamma g \cdot wM\rangle \\
&=& \langle\gamma\cdot z,\gamma g \cdot wM\rangle \\
&=& \langle z,g \cdot wM\rangle + \langle\gamma\cdot o,\gamma g\cdot wM\rangle \\
&=& H(gw) + \langle\gamma\cdot o,\gamma g\cdot wM\rangle.
\end{eqnarray*}
Summing up we obtain the assertion.
\end{proof}

\begin{cor}\label{equivariance property}
$d_{\lambda}(\gamma g) = e^{(i\lambda+\rho)(\langle\gamma\cdot o, \gamma g\cdot M\rangle + \langle\gamma\cdot o, \gamma g\cdot wM\rangle)} d_{\lambda,\mu}(g)$.
\end{cor}

\begin{lem}
Let $(b,b')\in B^{(2)}$ and $\gamma\in G$. Then
\begin{eqnarray}\label{equivariance dlambda}
(d_{\lambda}\circ\gamma)(b,b') = d_{\lambda}(\gamma\cdot b, \gamma\cdot b') = e^{(i\lambda+\rho)(\langle\gamma\cdot o, \gamma\cdot b\rangle+\langle\gamma\cdot o, \gamma\cdot b'\rangle)}d_{\lambda}(b, b').
\end{eqnarray}
\end{lem}
\begin{proof}
Let $g\in G$ such that $(b,b')=(g\cdot M,g\cdot wM)$. Then $d_{\lambda}(\gamma\cdot b,\gamma\cdot b')=d_{\lambda}(\gamma g)$, so the Lemma follows from Corollary \ref{equivariance property}.
\end{proof}

\subsubsection{An equivariance property}\label{An equivariance property}
Recall from Section \ref{Section Helgason boundary values} that in case of $\Gamma$-invariant joint eigenfunctions $\phi_{\lambda}$ the corresponding distribution boundary values satisfy
\begin{eqnarray*}
T_{\lambda}(d\gamma b)\otimes T_{\lambda}(d\gamma b') = e^{-(i\lambda+\rho)(\langle\gamma o,\gamma b\rangle + \langle\gamma o,\gamma b'\rangle)} \,\, T_{\lambda}(db)\otimes T_{\lambda}(db').
\end{eqnarray*}
To obtain $\Gamma$-invariant distributions we multiply with so-called \emph{intermediate values} $d_{\lambda}(b,b')$ which satisfy the inverse equivariance property
\begin{eqnarray}\label{fullfill this}
d_{\lambda}(\gamma\cdot b,\gamma\cdot b')=e^{(i\lambda+\rho)(\langle\gamma\cdot o,\gamma\cdot b\rangle+\langle\gamma\cdot o,\gamma\cdot b'\rangle)}\, d_{\lambda}(b,b').
\end{eqnarray}
The result of this subsection is very interesting: We prove in the following that the existence of a non-trivial function satisfying \eqref{fullfill this} is equivalent to a certain condition on the eigenvalue parameter.

The idea is that the function $d_{\lambda}$ is independent of the concrete subgroup $\Gamma$, so we suppose \eqref{fullfill this} to be satisfied for all $g,\gamma\in G$. Let $w\in W$ denote the longest Weyl group element. We identify $w$ with a representative in $M'$.

\begin{lem}\label{minus identity}
Suppose that there exists a function $d_{\lambda}: G/MA\rightarrow \cc$ satisfying \eqref{fullfill this} for all $\gamma\in G$ and all $(b,b')\in B^{(2)}$. Then $w\cdot\lambda=-\lambda$.
\end{lem}
\begin{proof}
Given $(b,b')\in B^{(2)}$, there is $g\in G$ such that under $G/MA\cong B^{(2)}$ we can write $b=g\cdot M$ and $b'=g\cdot wM$. Then \eqref{fullfill this} for a function on $B^{(2)}$ is equivalent to the existence of a function $d_{\lambda}$ on $G/MA$ satisfying
\begin{eqnarray}\label{fullfill this 2}
d_{\lambda}(\gamma g) = e^{(i\lambda+\rho)(\langle\gamma\cdot o,\gamma\cdot g\cdot M\rangle+\langle\gamma\cdot o,\gamma\cdot g\cdot wM\rangle)}\, d_{\lambda}(g) \,\,\,\,\,\, \forall \, \gamma,g\in G.
\end{eqnarray}
Let $a\in A, n\in N$. We first have
\begin{eqnarray}\label{compare null}
d_{\lambda}(n) = e^{(i\lambda+\rho)(\langle n\cdot o,M\rangle + \langle n\cdot o,nw\cdot M\rangle}d_{\lambda}(e) = e^{(i\lambda+\rho)H(nw)}d_{\lambda}(e).
\end{eqnarray}
Since $ana^{-1}\in N$ the assumed $MA$-invariance then yields
\begin{eqnarray}\label{compare one}
d_{\lambda}(an) = d_{\lambda}(ana^{-1}) = e^{(i\lambda+\rho)H(ana^{-1}w)}d_{\lambda}(e).
\end{eqnarray}
Combining \eqref{fullfill this 2} and \eqref{compare null} we also find
\begin{eqnarray}\label{compare two}
d_{\lambda}(an) &=& e^{(i\lambda+\rho)(\langle a\cdot o,a\cdot n\cdot M\rangle+\langle a\cdot o,a\cdot n\cdot wM\rangle)}\, d_{\lambda}(n) \nonumber\\
&=& e^{(i\lambda+\rho)(\log(a) + \langle a\cdot o,an\cdot wM\rangle + H(nw))}d_{\lambda}(e).
\end{eqnarray}

Comparing \eqref{compare one} with \eqref{compare two} and assuming $d_{\lambda}(e)\neq0$ (otherwise $d_{\lambda}=0$ everywhere by the transitivity of the $G$-action on $G/MA$) we get
\begin{eqnarray}\label{finally consider}
(i\lambda+\rho)H(ana^{-1}w) = (i\lambda+\rho)[\log(a) + \langle a\cdot o,an\cdot wM\rangle + H(nw)].
\end{eqnarray}

On the the left hand side of \eqref{finally consider} we have
\begin{eqnarray}\label{so becomes}
H(ana^{-1}w) = H(anww^{-1}a^{-1}w).
\end{eqnarray}
Note that $w^{-1}a^{-1}w\in A$, since $W$ normalizes $A$. Thus \eqref{so becomes} equals
\begin{eqnarray}\label{finally one}
H(anw) + \log(w^{-1}a^{-1}w).
\end{eqnarray}
For the right hand side of \eqref{finally consider} recall that
\begin{eqnarray*}
\langle a\cdot o,an\cdot wM\rangle = -H(a^{-1}k(anw)).
\end{eqnarray*}
If $anw=\tilde{k}\tilde{a}\tilde{n}$, then $a^{-1}k(anw)=nw\tilde{n}^{-1}\tilde{a}^{-1}$, so
\begin{eqnarray*}
\langle a\cdot o,an\cdot wM\rangle &=& -H(a^{-1}k(anw)) \\
&=& -H(nw\tilde{n}^{-1}\tilde{a}^{-1}) \\
&=& -H(nw) + \log(\tilde{a}).
\end{eqnarray*}
Thus on the right hand side of \eqref{finally consider} we have
\begin{eqnarray}\label{finally two}
\log(a) + \langle a\cdot o,an\cdot wM\rangle + H(nw) &=& \log(a) - H(nw) + \log(\tilde{a}) + H(nw) \nonumber \\
&=& \log(a) + \log(\tilde{a}) \nonumber \\
&=& \log(a) + H(anw).
\end{eqnarray}
If we now compare \eqref{finally one} with \eqref{finally two} we see that \eqref{fullfill this} implies
\begin{eqnarray*}
(i\lambda+\rho)\log(a) = (i\lambda+\rho)\log(w^{-1}a^{-1}w)
\end{eqnarray*}
for all $a\in A$. But 
\begin{eqnarray*}
\rho(\log(w^{-1}a^{-1}w)) = (w\cdot\rho)\log(a^{-1}) = -\rho\log(a^{-1}) = \rho\log(a),
\end{eqnarray*}
since $w\cdot\rho=-\rho$, since $w$ maps positive roots into negative roots. Moreover,
\begin{eqnarray*}
\lambda\log(w^{-1}a^{-1}w)=\lambda(w^{-1}\cdot\log(a^{-1})))=(w\cdot\lambda)(-\log(a)),
\end{eqnarray*}
so our final condition is $w\cdot\lambda = -\lambda$, as desired.
\end{proof}

\begin{rem}
Note that equation \eqref{fullfill this 2} can be satisfied by a function $d_{\lambda}$ defined on $G/M$. We will later see how to circumvent the problem of missing $A$-invariance.
\end{rem}

\subsection{Definitions and invariance properties}\label{Definitions and invariance properties}
We now build up the theory of Patterson-Sullivan distributions. We start by generalizing the definitions given in \cite{AZ}, which is possible if $\Ad(w)=-\id_{\la}$ (recall that by $w$ we denote the longest Weyl group element). Later we see how to define Patterson-Sullivan distributions for general symmetric spaces. We also study interesting invariance properties of these distributions.

\subsubsection{Diagonal Patterson-Sullivan distributions}\label{Diagonal Patterson-Sullivan distributions}
In this Section we fix $\lambda\in\la^*_{\cc}$ and suppose that $w\cdot\lambda=-\lambda$. We fix an eigenfunction $\phi\in\mathcal{E}^*_{\lambda}(X)$. At this point, we do not assume that $\phi$ is real-valued. Let $T_{\phi}$ denote the boundary values of $\phi$. The assumption on $\lambda$ is satisfied if the longest Weyl group element $w$ satisfies $\Ad_G(w^*)_{|\la}=-\id_{\la}$. This is the case for all rank one spaces. Recall the concept of \emph{intermediate values} (Section \ref{Intermediate values})
\begin{eqnarray*}
d_{\lambda}(b,b') = d_{\lambda}(g\cdot M, g\cdot wM) = e^{(i\lambda+\rho)(H(g)+H(gw))},
\end{eqnarray*}
where $g=g(b,b')$ corresponding to $B^{(2)} \cong G/MA$. We have proven in Subsection \ref{An equivariance property} that this function exists if and only if $w\cdot\lambda=-\lambda$.

\begin{defn}
The \emph{Patterson-Sullivan distribution $ps_{\phi,\lambda}(db,db')$ associated to}\index{Patterson-Sullivan distribution} $\phi\in\mathcal{E}_{\lambda}$ is the distribution on $C_c^{\infty}(B^{(2)})$ defined by
\begin{eqnarray}\label{ps define}
ps_{\phi,\lambda}(db,db') := d_{\lambda}(b,b') \cdot T_{\phi}(db) \otimes T_{\phi}(db').
\end{eqnarray}
The same definition \eqref{ps define} extends $ps_{\lambda}$ to a linear functional on the larger space $d_{\lambda}(b,b')^{-1}\cdot C^{\infty}(B\times B)$. If $\phi\in\mathcal{E}^*_{\lambda}(X)$ is fixed we write for simplicity $ps_{\lambda}$ instead of $ps_{\phi,\lambda}$. Moreover, we often write $T_{\phi}(db) T_{\phi}(db')$ instead of $T_{\phi}(db) \otimes T_{\phi}(db')$.
\end{defn}

\begin{prop}\label{ps invariant}
Let $\phi\in\mathcal{E}^*_{\lambda}(X)$ be a $\Gamma$-invariant eigenfunction of $\mathbb{D}(G/K)$. Let $T_{\phi}$ denote its boundary values. Then $ps_{\lambda}(db,db')$ is $\Gamma$-invariant. 
\end{prop}
\begin{proof}
Given a test function $f\in C_c^{\infty}(B^{(2)})$ and $\gamma\in\Gamma$, we observe
\begin{eqnarray*}
ps_{\lambda}(f\circ\gamma^{-1}) = (T_{\phi}\otimes T_{\phi})(d_{\lambda}\cdot (f\circ\gamma^{-1})) = (\gamma T_{\phi}\otimes \gamma T_{\phi})((d_{\lambda}\circ\gamma) \cdot f).
\end{eqnarray*}
It follows from \eqref{boundary values equivariance 1} that
\begin{eqnarray*}
T_{\phi}(d\gamma b)T_{\phi}(d\gamma b') = e^{-(i\lambda+\rho)\langle\gamma\cdot o,\gamma\cdot b\rangle}e^{-(i\lambda+\rho)\langle\gamma\cdot o,\gamma\cdot b'\rangle}T_{\phi}(db)T_{\phi}(db').
\end{eqnarray*}
By \eqref{equivariance dlambda}, the $d_{\lambda}(b,b')$ have the inverse equivariance property, so multiplying with \eqref{equivariance dlambda} yields $(T_{\phi}\otimes T_{\phi})(d_{\lambda}\cdot f) = ps_{\lambda}(f)$ and completes the proof of $\Gamma$-invariance.
\end{proof}

Recall the time reversal map $b \leftrightarrow b'$. Then by Corollary \ref{time reversal}:
\begin{prop}
Suppose that $\phi\in\mathcal{E}^*_{\lambda}(X)$ is $\Gamma$-invariant. Then the distribution $ps_{\phi,\lambda}(db,db')$ is time reversal invariant.
\end{prop}

We now construct $A$-invariant distributions. Recall that under the identification $G/MA\cong B^{(2)}$ we write $g(b,b')\in G$ if $g(b,b')\cdot (M,wM)=(b,b')\in B^{(2)}$. The element $g(b,b')$ is uniquely determined modulo $MA$.

\begin{defn}
For functions $f$ on $G/M$, the \emph{Radon transform}\index{$\mathcal{R}$, Radon transform}\index{Radon transform} $\mathcal{R}$ on $G/M$ is given by
\begin{eqnarray}\label{Radon define}
\mathcal{R}f(b,b') = \int_A f(g(b,b')aM) da,
\end{eqnarray}
whenever this integral exists. Then $\mathcal{R}f(b,b')$ is a function on $B^{(2)}$. By unimodularity of $A$ we find that \eqref{Radon define} does not depend on the choice of $g(b,b')$.
\end{defn}

\begin{lem}\label{application}
The Radon transform maps $\mathcal{R}: C_c(G/M)\rightarrow C_c(B^{(2)})$.
\end{lem}
\begin{proof}
Recall $B^{(2)}\cong G/MA$ as homogeneous spaces. Given $f\in C_c(G/M)$ we define $\tilde{f}\in C_c(G)$ by $\tilde{f}(g)=f(gM)$. Then
\begin{eqnarray*}
\mathcal{R}f(gMA) = \int_A \tilde{f}(ga)da = \int_{MA} \tilde{f}(gam)dadm.
\end{eqnarray*}
It follows from \eqref{tilde} and its subsequent remark applied to $MA$ that $\mathcal{R}f$ has compact support.
\end{proof}

\subsubsection{Patterson-Sullivan distributions on the compact quotient}
We keep the assumption that $w\cdot\lambda=-\lambda$. ($w\in W$ is the longest Weyl group element, $\lambda\in\la^*_{\cc}$).

\begin{defn}
Let $\mathcal{F}$ denote a bounded fundamental domain for $\Gamma$ in $X$. Following \cite{AZ}, pp. 380-381, we say that $\chi\in C_c^{\infty}(X)$ is a \emph{smooth fundamental domain cutoff function} if it satisfies \begin{eqnarray}\label{smooth fundamental domain cutoff}
\sum_{\gamma\in\Gamma}\chi(\gamma z) = 1 \,\,\,\,\, \forall z\in X.
\end{eqnarray}
Such a function can for example be constructed by taking $\nu\in C_c^{\infty}(X)$, $\nu=1$ on $\mathcal{F}$, and putting $\chi(z)=\nu(z)\cdot(\sum_{\gamma\in\Gamma}\nu(\gamma z))^{-1}$. If $\chi$ satisfies \eqref{smooth fundamental domain cutoff}, then
\begin{eqnarray}
\int_{\mathcal{F}} f \,  dz = \int_X \chi f \, dz, \,\,\,\,\, f\in C(X_{\Gamma}).
\end{eqnarray}
\end{defn}

Since $B$ is compact, we can (by using partition of unity) also choose a cutoff $\chi\in C_c^{\infty}(X\times B)$ such that $\sum_{\gamma\in\Gamma}\chi(\gamma\cdot(z,b))=1$. Let $T\in\mathcal{D}'(X\times B)$ be a $\Gamma$-invariant distribution and $a$ a $\Gamma$-invariant smooth function on $X\times B$. Suppose there is $a_1\in\mathcal{D}(X\times B)$ such that $\sum_{\gamma\in\Gamma}a_1(\gamma\cdot(z,b))=a(z,b)$. Then
\begin{eqnarray*}
\langle a_1,T\rangle_{X\times B} &=& \int_{X\times B} \left\{ \sum_{\gamma\in\Gamma} \chi(\gamma\cdot(z,b)) \right\} a_1(z,b) \, T(dz, db) \\
&=& \int_{X\times B} \sum_{\gamma\in\Gamma} \chi(z,b) \, a_1(\gamma\cdot(z,b)) \, T(dz, db).
\end{eqnarray*}
By the invariance of $T$ this equals $\int_{X\times B} \chi(z,b) a(z,b) \, T(dz, db)$. We thus have

\begin{prop}\label{independent}
Let $T\in\mathcal{D}'(X\times B)$ be a $\Gamma$-invariant distribution. Let $a$ be a $\Gamma$-invariant smooth function on $X\times B$. Then for any $a_1,a_2\in\mathcal{D}(X\times B)$ such that $\sum_{\gamma\in\Gamma}a_j(\gamma\cdot(z,b))=a(z,b)$ ($j=1,2$) we have $\langle a_1,T\rangle=\langle a_2,T\rangle$.
\end{prop}

Given $T$ and $a$ as in Proposition \ref{independent} and if moreover $\chi_j$ ($j=1,2$) are smooth fundamental domain cutoffs, then $a_j=\chi_j a$ satisfy the assumptions of the proposition. Hence $\langle a,T\rangle_{\Gamma\backslash G/M}:=\langle \chi a,T\rangle_{G/M}$ defines a distribution on the quotient $\Gamma\backslash G/M$ and this definition is independent of the choice of $\chi$.

\begin{defn}\label{Definition Patterson Sullivan}
Let $\lambda\in\la^*_{\cc}$ and $\phi\in\mathcal{E}^*_{\lambda}(X)$ denote a $\Gamma$-invariant joint eigenfunction. The Patterson-Sullivan distribution $PS_{\lambda}=PS_{\phi,\lambda}$ \emph{associated to $\phi$} is defined by
\begin{eqnarray}
\langle a,PS_{\lambda}\rangle_{G/M} = \int_{B^{(2)}}(\mathcal{R}a)(b,b') \, ps_{\phi,\lambda}(db,db').
\end{eqnarray}
On the quotient $\Gamma\backslash G/M$, we define the Patterson-Sullivan distributions by
\begin{eqnarray}
\langle a,PS_{\lambda}\rangle_{\Gamma\backslash G/M} := \langle\chi a,PS_{\lambda}\rangle_{G/M}.
\end{eqnarray}
We define \emph{normalized Patterson-Sullivan distributions} by
\begin{eqnarray}\label{normalized}
\widehat{PS}_{\lambda} = \frac{1}{\langle 1,PS_{\lambda}\rangle_{\Gamma\backslash G/M}}PS_{\lambda}.
\end{eqnarray}
In view of Proposition \ref{independent} these definitions do not depend on $\chi$.
\end{defn}

We look at the expression
\begin{eqnarray}
\langle a,PS_{\lambda}\rangle = \int_{B^{(2)}} \, d_{\lambda}(b,b') \, \mathcal{R}(a)(b,b') \, T_{\phi}(db)\, T_{\phi}(db').
\end{eqnarray}
It follows that $PS_{\lambda}(a)$ is well-defined if $(d_{\lambda}\cdot\mathcal{R}a)(b,b')\in C^{\infty}(B\times B)$, which is the case for $a\in C_c^{\infty}(G/M)$: In fact, then $\mathcal{R}a\in C_c^{\infty}(B^{(2)})$, so
\begin{eqnarray*}
d_{\lambda}(b,b')\mathcal{R}(a)(b,b')\in C_c^{\infty}(B^{(2)}) \subset C_c^{\infty}(B\times B) = C^{\infty}(B\times B).
\end{eqnarray*}

As a consequence of Proposition \ref{independent} we obtain (recall that $w\cdot\lambda=-\lambda$):
\begin{prop}
$PS_{\phi,\lambda}$ is an $A$-invariant and $\Gamma$-invariant distribution on $G/M$. On the quotient $\Gamma\backslash G/M$, the distribution $PS_{\phi,\lambda}$ is still $A$-invariant.
\end{prop}

\subsubsection{Off-diagonal Patterson-Sullivan distributions}\label{Off-diagonal Patterson-Sullivan distributions}
In this Subsection, we drop the assumption that $w_{\la}=-\id$. Let $\lambda,\mu\in\la^*_{\cc}$ and fix $\phi\in\mathcal{E}^*_{\lambda}(X)$ and $\psi\in\mathcal{E}^*_{\mu}(X)$. At this point, we do not assume that these eigenfunctions are real-valued. Let $T_{\phi}$ and $T_{\psi}$ denote the respective boundary values. Recall the \emph{off-diagonal intermediate values} (Section \ref{Intermediate values})
\begin{eqnarray*}
d_{\lambda,\mu}(g) = e^{(i\lambda+\rho)H(g)}e^{(i\mu+\rho)H(gw)}.
\end{eqnarray*}

\begin{defn}
For functions $f$ on $G/M$, the \emph{weighted Radon transform}\index{$\mathcal{R}$, weighted Radon transform}\index{weighted Radon transform} $\mathcal{R}_{\lambda,\mu}$ on $G/M$ is by definition the Radon transform \eqref{Radon define} of $d_{\lambda,\mu}f$, that is
\begin{eqnarray}\label{repair}
\mathcal{R}_{\lambda,\mu}f(g) := \int_A d_{\lambda,\mu}(ga)f(ga)\,da,
\end{eqnarray}
whenever this integral exists.
\end{defn}

It is clear that $\mathcal{R_{\lambda,\mu}}(f)$ is an $A$-invariant function on $G/M$ (right-$A$-invariant), that is a function on $G/MA\cong B^{(2)}$. Note that by integrating $d_{\lambda,\mu}$ with respect to $a\in A$ we circumvent the problem that $d_{\lambda,\mu}$ alone is not a function on $G/MA$ (see \eqref{observation} and its subsequent remark).

Exactly as in Lemma \ref{application} we find
\begin{lem}
Let $f\in C_c^{\infty}(G/M)$. Then $\mathcal{R_{\lambda,\mu}}(f)\in C_c^{\infty}(G/MA)$. 
\end{lem}

\begin{defn}
As usual, let $g(b,b')\in G$ be a representative for the element $g(b,b')MA\in G/MA$ that corresponds to $(b,b')\in B^{(2)}$. Let $f\in C_c^{\infty}(G/M)$. We pull-back the Radon transform \eqref{repair} to $B^{(2)}$ and define \begin{eqnarray*}
\mathcal{R}_{\lambda,\mu}f(b,b') = \mathcal{R}_{\lambda,\mu}f(g(b,b')).
\end{eqnarray*}
Then $\mathcal{R}_{\lambda,\mu}f \in C_c^{\infty}(B^{(2)})$. This definition is independent of the choice of representative $g(b,b')$, since $\mathcal{R}_{\lambda,\mu}(f)$ is invariant.
\end{defn}

Let $f\in C_c^{\infty}(B^{(2)})\subset C_c^{\infty}(B\times B)\subset C^{\infty}(B\times B)$. We interpret $\mathcal{R}_{\lambda,\mu}f$ as a function on $B\times B$ with compact support contained in $B^{(2)}$.

\begin{defn}
Let $\phi\in\mathcal{E}^*_{\lambda}(X)$ and $\psi\in\mathcal{E}^*_{\mu}(X)$ have boundary values $T_{\phi}$ and $T_{\psi}$. The \emph{off-diagonal Patterson-Sullivan distribution $PS_{\lambda,\mu}$ associated to $\phi$ and $\psi$} on $G/M$ is defined by
\begin{eqnarray}\label{Off-Diagonal Patterson-Sullivan}
\langle f,PS_{\lambda,\mu}\rangle = \int_{B^{(2)}} \, \mathcal{R}_{\lambda,\mu}f(b,b') \, T_{\phi}(db) \, T_{\psi}(db').
\end{eqnarray}
\end{defn}

It follows that $PS_{\lambda,\mu}(f)$ is well-defined if $\mathcal{R}_{\lambda,\mu}f(b,b') \in C^{\infty}(B\times B)$. A simple case is when $f\in C_c^{\infty}(G/M)$: Then $\mathcal{R}_{\lambda,\mu}\in C_c^{\infty}(B^{(2)})$, so
\begin{eqnarray*}
\mathcal{R}_{\lambda,\mu}(f)(b,b')\in C_c^{\infty}(B^{(2)}) \subset C_c^{\infty}(B\times B) = C^{\infty}(B\times B).
\end{eqnarray*}

\begin{prop}
Suppose that $\phi\in\mathcal{E}^*_{\lambda}(X)$ and $\phi\in\mathcal{E}^*_{\mu}(X)$ are $\Gamma$-invariant eigenfunctions. Then the distribution $PS_{\lambda,\mu}$ on $G/M$ is $\Gamma$-invariant.
\end{prop}
\begin{proof}
Let $f\in C_c^{\infty}(G/M)$ and let $f_{\gamma}$ denote the translation $f\circ\gamma^{-1}$. Then
\begin{eqnarray*}
\langle f_{\gamma}, PS_{\lambda,\mu}\rangle = \int_{B^{(2)}} \int_A \, d_{\lambda,\mu}(g(b,b')a) \, f(\gamma^{-1}g(b,b')a) \, da \, T_{\lambda}(db) \, T_{\mu}(db'),
\end{eqnarray*}
where $(b,b') = (g\cdot M,g\cdot wM)$ for $g=g(b,b')$. By \eqref{boundary values equivariance 1} this equals
\begin{eqnarray*}
&& \int_{B^{(2)}}\int_A d_{\lambda,\mu}(g(\gamma\cdot(b,b'))a) \, f(\gamma^{-1}g(\gamma(b,b'))a) \\
&& \hspace{13mm} \times \,\, e^{-(i\lambda+\rho)\langle\gamma\cdot o,\gamma\cdot b\rangle} \, e^{-(i\mu+\rho)\langle\gamma\cdot o,\gamma\cdot b'\rangle} \, da \, T_{\lambda}(db) \, T_{\mu}(db').
\end{eqnarray*}
Recall that $a\in A$ acts trivially on $(M,wM)$. Using this and \eqref{equivariance property off-diag2} we observe
\begin{eqnarray*}
d_{\lambda,\mu}(\gamma ga) = e^{(i\lambda+\rho)\langle\gamma\cdot o,\gamma\cdot  b\rangle} e^{(i\mu+\rho)\langle\gamma\cdot o,\gamma\cdot  b'\rangle} d_{\lambda,\mu}(ga).
\end{eqnarray*}
We also have $g(\gamma\cdot(b,b')) = \gamma g(b,b'))$, since $(b,b')\mapsto g(b,b')\in G/MA$ is $G$-equivariant. Hence $\gamma^{-1}g(\gamma\cdot(b,b'))=g(b,b')$. Thus we have
\begin{eqnarray*}
\langle f_{\gamma}, PS_{\lambda,\mu}\rangle &=& \int_{B^{(2)}} \int_A d_{\lambda,\mu}(g(b,b')a) f(g(b,b')a) \, da \, T_{\lambda}(db) T_{\mu}(db') \\
&=& \int_{B^{(2)}} \mathcal{R}_{\lambda,\mu}f(b,b') \, T_{\lambda}(db) T_{\mu}(db') = \langle f, PS_{\lambda,\mu}\rangle,
\end{eqnarray*}
and the proposition follows.
\end{proof}

\begin{rem}\label{general radon special}
Let $(b,b')\in B^{(2)}$, $g=g(b,b')$ and suppose $w\cdot\lambda=-\lambda$. Then
\begin{eqnarray}\label{general radon special 2}
\mathcal{R}_{\lambda,\lambda}(f)(g) = \int_A d_{\lambda,\lambda}(ga)f(ga)\,da
= d_{\lambda}(g(b,b'))(\mathcal{R}f)(b,b').
\end{eqnarray}
Let $\phi\in\mathcal{E}^*_{\lambda}(X)$ and consider the distributions $PS_{\lambda,\lambda}$ and $PS_{\phi,\lambda}$ associated to $\phi$. By \eqref{general radon special 2} we have $PS_{\lambda,\lambda}=PS_{\lambda}$. If $\phi=\psi$ and $\lambda=\mu$, it follows as in Subsection \ref{Diagonal Patterson-Sullivan distributions} that the $PS_{\lambda,\lambda}$ are invariant under time-reversal and right-translation by $A$. Vice versa, if $T_{\phi}\neq T_{\psi}$, then $PS_{\phi,\psi}$ needs not to be invariant under $b \leftrightarrow b'$.
\end{rem}

\begin{rem}\label{A-eigendistributions}
Given $\tilde{a}\in A$ we write $f_{\tilde{a}}:=f\circ\tilde{a}^{-1}$. Then
\begin{eqnarray}
\mathcal{R}_{\lambda,\mu}(f_{\tilde{a}})(g) = \int_A d_{\lambda,\mu}(ga\tilde{a})f(ga)\,da = e^{i(\lambda+w\cdot\mu)\log(\tilde{a})} \, \mathcal{R}_{\lambda,\mu}(f)(g),
\end{eqnarray}
which follows from
\begin{eqnarray}
d_{\lambda,\mu}(g\tilde{a}) = e^{i(\lambda+w\cdot\mu)\log(\tilde{a})}d_{\lambda,\mu}(g)
\end{eqnarray}
(cf. \eqref{observation}). Given eigenfunctions $\phi,\psi$ we thus have
\begin{eqnarray}\label{eigendistributions}
\langle f_{\tilde{a}}, PS_{\lambda,\mu}\rangle = e^{i(\lambda+w\cdot\mu)\log(\tilde{a})} \langle f, PS_{\lambda,\mu}\rangle.
\end{eqnarray}
In other words, the $PS_{\lambda,\mu}$ are eigendistributions for the action of $A$ on $G/M$ (given by right-translation). In particular, if $\lambda+w\cdot\mu=0$, then the associated Patterson-Sullivan distribution is invariant under right-translation by $A$. This is for example the case when $\phi=\psi$, $\lambda=\mu$, and $w\cdot\lambda=-\lambda$.
\end{rem}

\begin{defn}\label{Definition Patterson Sullivan off diag}
Suppose that $\phi\in\mathcal{E}^*_{\lambda}(X)$ and $\phi\in\mathcal{E}^*_{\mu}(X)$ are $\Gamma$-invariant joint eigenfunctions. Since $PS_{\lambda,\mu}$ is a $\Gamma$-invariant distribution on $G/M$, the definition descends to the quotient $\Gamma\backslash G/M$ via
\begin{eqnarray}
\langle a,PS_{\lambda,\mu}\rangle_{\Gamma\backslash G/M} := \langle\chi a,PS_{\lambda,\mu}\rangle_{G/M},
\end{eqnarray}
where $\chi$ is a smooth fundamental domain cutoff. We normalize these distributions by setting
\begin{eqnarray}\label{normalized off diag}
\widehat{PS}_{\lambda,\mu} := \frac{1}{\langle 1,PS_{\mu,\mu}\rangle_{\Gamma\backslash G/M}}PS_{\lambda,\mu}.
\end{eqnarray}
In view of Proposition \ref{independent} these definitions do not depend on $\chi$.
\end{defn}

\subsection{The Knapp-Stein intertwining operators}
In this Section we introduce the Knapp-Stein intertwiners. We will later see how these operators yield an explicit relation between the Patterson-Sullivan distributions and the Wigner distributions (\ref{Wigner distributions}). For background on similar intertwining operators see \cite{Knapp}. Let $\lambda\in\mathfrak{a}^*_{\cc}$ and define\index{$L_{\lambda}$, Knapp-Stein intertwiner}
\begin{eqnarray}\label{Intertwiner define}
L_{\lambda}a(g) := \int_N e^{-(i\lambda+\rho)(H(n^{-1}w))} \, a(gn) \, dn, \,\,\,\,\,\,\,\,\, a\in C(G),
\end{eqnarray}
whenever the integral exists. The integrals $L_{\lambda}a(g)$ may be viewed as a \emph{weighted horocyclic Radon transform}.

\begin{rem}\label{preserve remark}
Each $Ad_G(\tilde{m})$, $\tilde{m}\in M$, fixes the elements of $\mathfrak{a}$ and hence the root subspaces. Thus $M$ normalizes $N$, that is $\tilde{m}N=N\tilde{m}$ for all $\tilde{m}\in M$. Hence $n\mapsto \tilde{m}^{-1}n\tilde{m}$ defines an automporphism of $N$ which by uniqueness of Haar-measures maps $dn$ into a multiple of $dn$. Since $M$ is compact, $dn$ is preserved.
\end{rem}

It is a basic remark that $L_{\lambda}$ preserves $M$-invariance:

\begin{lem}
$L_{\lambda}: C_c^{\infty}(G/M) \rightarrow C^{\infty}(G/M)$.
\end{lem}
\begin{proof}
Suppose that $a\in C_c^{\infty}(G/M)$ and let $g\in G$, $n\in N$, $m\in M$. Then $a(gmn)=a(gmnm^{-1})$ and by \ref{preserve remark} we know that $n\mapsto\tilde{n}:=mnm^{-1}\in N$ preserves $dn$. Moreover, $H(n^{-1}w)=H(mn^{-1}m^{-1}w)$ by invariance of the Iwasawa projection and since $w$ normalizes $M$. Thus
\begin{eqnarray*}
L_{\lambda}a(gm) &=& \int_N e^{-(i\lambda+\rho)(H(n^{-1}w))} \, a(gmn) \, dn = \int_N e^{-(i\lambda+\rho)(H(\tilde{n}^{-1}w))} \, a(g\tilde{n}) \, dn \\
&=& \int_N e^{-(i\lambda+\rho)(H(n^{-1}w))} \, a(gn) \, dn = L_{\lambda}a(g).
\end{eqnarray*}
\end{proof}

\subsubsection{Harish-Chandra's phase function}
We absorb the term $e^{-\rho H(n^{-1}w)}$ in \eqref{Intertwiner define} into the amplitude, so that the phase function is
\begin{eqnarray*}
\psi(n) = -H(n^{-1}w).
\end{eqnarray*}
By uniqueness of the longest element of a Coxeter group, we have $w^{-1}=w\in W$. Thus $w^{-1}=wm$ ($m\in M)$ as elements in $M'$, so $H(n^{-1}w)=H(wn^{-1}w^{-1})$ by invariance of $H(kan)=\log(a)$. We write
\begin{eqnarray}\label{coordinate change N}
\widetilde{\theta}: N\rightarrow\overline{N}, \,\,\,\,\,\,\,\,\,\,\,\,  n\mapsto wnw^{-1}.
\end{eqnarray}
Then $\widetilde{\theta}(dn)=d\overline{n}$ (cf. Subsection \ref{Measure theoretic preliminaries}), since $M'$ is compact, so since $\overline{N}$ is unimodular
\begin{eqnarray*}
L_{\mu}(a)(g) &=& \int_{\overline{N}} e^{-i\mu(H(\overline{n}))}     \,   e^{-\rho(H(\overline{n}))} \, a(gw^{-1}\overline{n}^{-1}w) \, d\overline{n}.
\end{eqnarray*}
Given $0\neq\mu\in\la^*_{\cc}$, we identify the ray $\rr^+\mu\subset\la^*_{\cc}$ with $\rr^+$ by means of the Killing form: First, we denote by $H_{\mu}$ the unique element in $\la_{\cc}$ such that $\mu(X)=\langle X,H_{\mu}\rangle$ for all $X\in\la^*_{\cc}$. Then
\begin{eqnarray}\label{real mu}
\mu(X) = |\mu| \langle X,H_{\mu/{|\mu|}}\rangle, \,\,\,\,\,\,\,\,\,\,\,\,\,\,\, X\in\la^*_{\cc}, \,\, |\mu|\in\rr^+.
\end{eqnarray}
We can now fix $\mu\in\la^*$ and $H:=H_{\mu/{|\mu|}}\in\la^*$. Using these identifications we make from now on no difference between $|\mu|$ and $\mu$. We rewrite the integrals \eqref{Intertwiner define} in the form (note that $\rho$ ``remains'' an element of $\la^*$)
\begin{eqnarray*}
L_{\mu}(a)(g) = \int_{\overline{N}} e^{-i\mu\langle H(\overline{n}), H\rangle} \, e^{-\rho H(\overline{n})} \, a(gw\overline{n}^{-1}w^{-1}) \, d\overline{n}, \,\,\,\,\,\,\,\,\,\,\, \mu\in\rr.
\end{eqnarray*}
We choose a smooth fundamental domain cutoff function $\chi$. Then $L_{\mu}(\chi a)(g)$ is an oscillatory integral with real-valued phase function
\begin{eqnarray}\label{define psiH}
\psi_H: \overline{N} \rightarrow \rr, \,\,\,\,\,\,\,\,\,\,\,\,\,\,\,\,\,\,\, \overline{n} \mapsto \langle H(\overline{n}), H\rangle.
\end{eqnarray}

We would be able to compute the critical points and the Hessian form of $n\rightarrow H(n^{-1}w)$ as we did for the other phase functions in Subsection \ref{Critical sets and Hessian forms}. However, we do not have to: The point is that $\psi_H$ is the phase in the integral
\begin{eqnarray}\label{to the integrals}
c(\lambda) = \int_{\overline{N}} e^{-(i\lambda+\rho)H(\overline{n})} \, d\overline{n}, \,\,\,\,\,\,\,\,\,\,\,\,\,\,\,\,\,\ Re(i\lambda)\in\la^*_{+},
\end{eqnarray}
defining Harish-Chandra's $c$-function. The calculations concerning the critical points and Hessians of the $\psi_H$ were for example carried out in \cite{Cohn}, \S19. The following proposition taken from \cite{DKV}, Section 7, gives the complete description of facts concerning $\psi_H$. Recall that $\overline{N}_H$ denotes the centralizer of $H\in\la$ in $\overline{N}$. For a root $\beta$, let $R_{\beta}$ denote the orthogonal projection $\g\rightarrow\g_{\beta}$. If $g\in G$ is decomposed $g=kan$ corresponding to the Iwasawa decomposition, then we denote its triangular part\index{triangular part of the Iwasawa decomposition} by $t(g)=an\in AN$. Writing, as usual, $\langle\cdot,\cdot\rangle$ for the Killing form, we denote in the next Proposition by $(\cdot,\cdot)$ the inner product $Z,Z'\mapsto -\langle Z, \theta Z'\rangle$ on $\g\times \g$.

\begin{prop}\label{phase function c function}
Let $H\in\la$. The critical set of $\psi_H$ is equal to $\overline{N}_H$. For the Hessian of $\psi_H$ at the critical points we have the formula
\begin{eqnarray*}
\textnormal{Hess}_{\overline{n}}(\overline{Y},\overline{Y}') = - \sum_{\alpha\in\Delta^+}\alpha(H)(\theta R_{\alpha}(\overline{Y}^{t(\overline{n})}) - R_{-\alpha}(\overline{Y}^{t(\overline{n})}), R_{-\alpha}(\overline{Y}'^{t(\overline{n})}),
\end{eqnarray*}
valid for $\overline{n}\in\overline{N}_H$ and $\overline{Y},\overline{Y}'\in\overline{\lnn}$. The index of the Hessian Hess$_{\overline{\lnn}}$ at any point of $\overline{N}_H$ is
\begin{eqnarray*}
\sum_{\alpha\in\Delta^+, \hspace{1mm} \alpha(H)<0} \dim(\g_{\alpha}).
\end{eqnarray*}
Let $\overline{\lnn}_H$ denote the Lie algebra of the closed subgroup $\overline{N}_H$ of $\overline{N}$. Write $\overline{\lnn}_{\lambda}$ for the eigenspace of $\ad(H)$ in $\overline{\lnn}$ for the eigenvalue $\lambda\in\rr$. Then, with respect to the Lie algebra decomposition $\overline{\lnn}=\overline{\lnn}_H \oplus \oplus_{\lambda\neq0}\overline{\lnn}_{\lambda}$ (cf. \cite{DKV} Corollary 7.3), the matrix Hess$_{\overline{\lnn}}$ is diagonal and $\psi_H$ is clean.
\end{prop}

\begin{rem}
It is clear that Proposition \ref{phase function c function} still holds if $H\in\la^*_{\cc}$: The case of complex $H$ is dealt by passing to the real and imaginary part of $\psi_H$, since by uniqueness of real and imaginary parts a point is critical for $\psi_H$ if and only if it is critical for both $\psi_{Re(H)}$ and $\psi_{Im(H)}$. In this way we could also handle complex $\mu$ in \eqref{real mu} with no extra work. However, in view of our results of Section \ref{Section Helgason boundary values}, we only consider real eigenvalue parameters. Anyway, the mehod of stationary phase only applies for phase functions with non-negative imaginary part.
\end{rem}

\subsubsection{Asymptotic expansions for the Knapp-Stein intertwiner}
Recall that an element $X\in\lp$ is called regular, if $Z(X)\cap\lp$ is a maximal abelian subspace of $\lp$. We call an element $\mu\in\la^*$ regular, if $H_{\mu}$ is regular, where $H_{\mu}$ is the vector in $\la$ such that $\mu(X)=\langle X,H_{\mu}\rangle$ (Killing form) for all $X\in\la$. The centralizer of a regular element $X\in\la$ in $N$ (resp. $\overline{N}$) is the trivial subgroup $\left\{e\right\}$ of $G$. If $G/K$ has rank one, then all nonzero elements of $\la$ (resp. $\la^*$) are regular.

We fix a regular $\mu\in\la^*$ and write $H=H_{\mu}\in\la$ and $\psi=\psi_{H_{\mu}}$. Then $\psi(e)=0$ and for the amplitude $\alpha(\overline{n}) = e^{-\rho H(\overline{n})}\, \chi a(gw\overline{n}^{-1}w^{-1})$ we have $\alpha(e)=\chi a(g)$. Let $s=\dim(N)=\dim(\overline{N})$. It follows from Proposition \ref{phase function c function} that after the coordinate change \eqref{coordinate change N} the function $n\mapsto \langle H(n^{-1}w),H_{\mu}\rangle$ has the unique critical point $n=e$ and its Hessian form at $n=e$ is non-degenerate. The method of stationary phase (\cite{Hor1}) yields
\begin{eqnarray}\label{integrate this}
L_{\mu}(\chi a)(g) \sim C \cdot (2\pi/\mu)^{s/2} \sum_{n=0}^{\infty}\mu^{-n} R_{2n}(\chi a)(g),
\end{eqnarray}
where $R_{2n}$ is a differential operator on $G$ of order $2n$ and $R_0$ is the identity. If $Q$ denotes the Hessian matrix at the critical point, then
\begin{eqnarray}
C=|\det Q|^{-1/2} e^{\pi i\sign(Q)/4}.
\end{eqnarray}
One could also show $C \cdot (2\pi/\mu)^{s/2}\sim c(\mu)$ for the factor in \eqref{integrate this} by applying the method of stationary phase to the integrals \eqref{to the integrals} and using Proposition \eqref{phase function c function}.

\begin{lem}
For each $n\in N$, the operator $R_{2n}$ arising in the expansion \eqref{integrate this} is a left-invariant differential operator on $G/M$.
\end{lem}
\begin{proof}
We can replace $\chi a$ in \eqref{integrate this} by an arbitrary $a\in C_c^{\infty}(G/M)$. The coefficients $R_{2n}(a)(g)$ are independent of $\mu$ and hence uniquely determined. Since $L_{\mu}(a)$ is $M$-invariant, it follows that $R_{2n}(a)(g)=R_{2n}(a)(gm)$ for all $n\in\nn$, $g\in G$, $m\in M$, $a\in C_c^{\infty}(G/M)$. Hence
\begin{eqnarray*}
R_{2n}:C_c^{\infty}(G/M)\rightarrow C_c^{\infty}(G/M)
\end{eqnarray*}
is a linear operator. To see that $R_{2n}$ is a local operator, take $a\in C_c^{\infty}(G/M)$. Then $K:=\supp(a)\subset\subset G/M$ is compact. Write $\pi:G\rightarrow G/M$ and set $V=\pi^{-1}(K)$. Then $\supp_G(R_{2n}(a))\subset V$, since $R_{2n}$ is a differential operator on $G$. Thus $\supp R_{2n}(a)\subset K$. It follows that $R_{2n}:C_c^{\infty}(G/M)\rightarrow C_c^{\infty}(G/M)$ decreases supports, so by Peetre's theorem it is a differential operator on $G/M$. The same reasoning shows that the $R_{2n}$ are left-invariant.
\end{proof}

\subsection{An integral formula}\label{An integral formula}
In this subsection we prove an important integral formula involving the Radon transform, intermediate values and the intertwining operators.

\begin{lem}\label{formula}
Let $a\in C_c^{\infty}(G/M)=C_c^{\infty}(X\times B)$ and $(b,b')\in B^{(2)}$. Then
\begin{eqnarray}\label{intertwining}
\mathcal{R}_{\lambda,\mu}(L_{\mu}a)(b,b') = \int_X a(z,b)e^{(i\lambda+\rho)\langle z,b\rangle}e^{(i\mu+\rho)\langle z,b'\rangle}dz.
\end{eqnarray}
\end{lem}
\begin{proof}
Let $g\in G$ such that $(b,b')=(g\cdot M,g\cdot wM)$. We manipulate the right side of \eqref{intertwining}: First note that since $dz$ is $G$-invariant we obtain
\begin{eqnarray}\label{this equals}
\int_X a(z,b)e^{(i\lambda+\rho)\langle z,b\rangle} e^{(i\mu+\rho)\langle z,b'\rangle}dz = \int_X a(g\cdot z,b)e^{(i\lambda+\rho)\langle g\cdot z,b\rangle}e^{(i\mu+\rho)\langle g\cdot z,b'\rangle}dz.
\end{eqnarray}
We consider $a$ as a function on $G/M\cong X\times B$. Then since $b=g\cdot o$ it follows that $a(gan\cdot o,g\cdot M)=a(gan\cdot o,gan\cdot M)=a(ganM)$. Recall that $P=MAN$ fixes $b_{\infty}=M\in K/M$. By \eqref{integral formula AN and G/K} we find that \eqref{this equals} equals
\begin{eqnarray}\label{the integral becomes 2}
\int_{AN} a(ganM) e^{(i\lambda+\rho)\langle gan\cdot o,g\cdot M\rangle} e^{(i\mu+\rho)\langle gan\cdot o,g\cdot wM\rangle} \, dn \, da.
\end{eqnarray}
We first have
\begin{eqnarray}
\langle gan\cdot o,g\cdot M\rangle = \langle gan\cdot o,gan\cdot M\rangle = H(gan) = H(ga).
\end{eqnarray}
Next, by \eqref{equivariance}, by the definition of $\langle z,b\rangle$ and since $a\cdot wM=wM$ for all $a\in A$,
\begin{eqnarray}
\langle gan\cdot o,g\cdot wM\rangle &=& \langle ga\cdot n\cdot o,ga\cdot wM\rangle \nonumber \\
&=& \langle n\cdot o,wM \rangle + \langle ga\cdot o, ga\cdot wM\rangle \nonumber \\
&=& -H(n^{-1}w) + H(gaw).
\end{eqnarray}
It follows that \eqref{the integral becomes 2} equals
\begin{eqnarray}\label{consider}
&& \hspace{-2cm} \int_{AN} a(ganM) e^{(i\lambda+\rho)H(ga)} e^{(i\mu+\rho)H(gaw)} e^{-(i\mu+\rho)H(n^{-1}w)} \, dn \, da \nonumber \\
&=& \int_A d_{\lambda,\mu}(gaM) \int_N a(ganM) e^{-(i\mu+\rho)H(n^{-1}w)} \, dn \, da \\
&=& \int_A d_{\lambda,\mu}(gaM) L_{\mu}a(gaM) \, da \nonumber \\
&=& \mathcal{R}_{\lambda,\mu}(L_{\mu}a)(b,b'). \nonumber
\end{eqnarray}
Note that $\mathcal{R}_{\lambda,\mu}(L_{\mu}a)(b,b') = \mathcal{R}_{\lambda,\mu}(L_{\mu}a)(g)$ is defined if $a$ has compact support. This follows from the Fubini theorem and the often used formula given in \eqref{integral formula AN and G/K}. The lemma is proven.
\end{proof}

\begin{rem}
If $g\tilde{m}\tilde{a}$ is another representative of $gMA\in G/MA$, then in \eqref{consider}
\begin{eqnarray}\label{consider 2}
\int_A d_{\lambda,\mu}(g\tilde{m}\tilde{a}aM) \int_N a(g\tilde{m}\tilde{a}anM) e^{-(i\mu+\rho)H(n^{-1}w)} dn da.
\end{eqnarray}
Since $A$ is unimodular we get rid of $\tilde{a}$. Moreover, $H(n^{-1}w)$ is preserved under $n^{-1}\mapsto\tilde{m}^{-1}n^{-1}\tilde{m}$, since $H(kan)=\log(a)$ is $M$-bi-invariant and $w$ normalizes $M$. Then by Remark \ref{preserve remark} we find that \eqref{consider} and \eqref{consider 2} coincide. Hence the proof of Lemma \ref{formula} does not depend on the choice of representative of $g(b,b')MA$.
\end{rem}

\begin{rem}\label{Some remarks}
\begin{itemize}
\item[(1)] If $X$ has rank one can show that $H(n^{-1}w)=H(nw)$ for all $n\in N$. This follows from \cite{He94}, Ch. II, \S6, Thm. 6.1). Hence in these cases we obtain a slight simplification of the formulae above. In general, the formula $H(n^{-1}w)=H(nw)$ is not correct. It is easy to find counterexamples for example in $SL(3,\rr)$, where $\Ad_G(w)_{|\la}\neq -\id_{\la}$ (see Section \ref{The special linear groups}).
\item[(2)] In the notation of \cite{AZ}, we identify $i\lambda+\rho = \frac{1}{2}+ir$. Then $d_{\lambda}(b,b')$ and $|b - b'|^{-\frac{1}{2} - ir}$ satisfy the same equivariance property. By the transitivity of the $G$-action on $B^{(2)}$ these function are constant multiples of each other. This explains the factor $2^{\frac{1}{2} + ir}$ in \cite{AZ}. It appears because $|1-(-1)|=2$ in the disk model, whereas we defined $d_{\lambda}$ such that $d_{\lambda}(M,wM)=1$.
\item[(2)] 
The intertwining operator
\begin{eqnarray*}
L_{r}a(g) = \int_{\mathbb{R}} a(gn_u) (1+u^2)^{-(\frac{1}{2}+ir)}du
\end{eqnarray*}
introduced in \cite{AZ} is generalized by our intertwiner
\begin{eqnarray*}
L_{\lambda}a(g) = \int_N a(gn) e^{-(i\lambda+\rho)(H(n^{-1}w))}dn.
\end{eqnarray*}
In the notation of \cite{AZ}, we always identify
$i\lambda+\rho=\frac{1}{2}+ir$. The group $PSL(2,\mathbb{R})$ has the following Iwasawa decomposition components:
\begin{eqnarray*}
k_{\alpha} &=& \begin{pmatrix} \cos(\alpha) & -\sin(\alpha) \\ \sin(\alpha) & \cos(\alpha) \end{pmatrix} \hspace{9mm} a_t = \begin{pmatrix} e^{t/2} & 0 \\ 0 & e^{-t/2} \end{pmatrix}\\
n_u &=& \begin{pmatrix} 1 & u \\ 0 & 1 \end{pmatrix} \hspace{30.8mm} w \hspace{0.3mm} = \begin{pmatrix} 0 & 1 \\ -1 & 0 \end{pmatrix}
\end{eqnarray*}
It suffices to prove $H(n_u^{-1}w)=\ln(1+u^2)$ for all $u\in\rr$. Writing out $n_uw$ gives
\begin{eqnarray*}
n_uw = \begin{pmatrix} -u & 1 \\ -1 & 0 \end{pmatrix}.
\end{eqnarray*}
We have the following cases: $u=0$, $u<0$ and $u>0$.
\begin{itemize}
\item[(i)] $u=0$. Then $n_uw=w\in K$, so the formula is obvious.
\item[(ii)] $u<0$. Let $t=\ln(1+u^2)$, $\alpha=-\arcsin(\frac{1}{e^{t/2}})$. Then let $n_v$ be the element $n_v := a_t^{-1}k_{\alpha}^{-1}n_uw$.
Multiplying out shows that $n_v$ is of the form
\begin{eqnarray*}
n_v = \begin{pmatrix} 1 & v \\ 0 & 1 \end{pmatrix}.
\end{eqnarray*}
Then $k_{\alpha}a_t n_v = n_u^{-1}w$, so $H(n_u^{-1}w)=t=\ln(1+u^2)$.
\item[(iii)] $u>0$. This case is very similar to the preceding case (ii). The formula also follows from $H(n_uw)=H(n_u^{-1}w)=H(n_{-u}w)$, since in this example $G/K$ has rank one.
\end{itemize}
\end{itemize}
\end{rem}

\subsection{Eigenfunctions on a compact quotient}\label{Eigenfunctions on a compact quotient}
As before, let $X=G/K$ denote a symmetric space of the noncompact type with Laplace-Beltrami operator $L_X$. Let $\Gamma$ denote a cocompact, discrete and torsion free subgroup $\Gamma$ of $G$ and let $X_{\Gamma}:=\Gamma\backslash G/K$ be given the quotient metric. Then $X_{\Gamma}$ is a compact hyperbolic manifold and a locally symmetric space. We write $\Delta$ for the Laplace operator of $X_{\Gamma}$.

Let $0=c_0<c_1<c_2<\ldots$ denote the discrete spectrum of $-\Delta$ on $X_{\Gamma}$ (cf. Subsection \ref{The Laplacian}). We choose a corresponding complete Hilbert space basis $(\phi_{j})$ of $L^2(X_{\Gamma})$ consisting of normalized (with respect to the $L^2$-norm of $X_{\Gamma}$) eigenfunctions of $\Delta$. Then
\begin{eqnarray}\label{eigenvalue problem}
\Delta\phi_{j} = -c_j \phi_j \,\,\,\,\, \textnormal{ for all } j\in\nn_0.
\end{eqnarray}

Let $\pi$ denote the natural projection of $X$ onto $X_{\Gamma}$. Then $\pi$ is a local isometry and since the Laplace operator is isometry-invariant, $\pi$ intertwines the Laplace operators $L_X$ of $X$ and $\Delta$ of $X_{\Gamma}$. It follows that an eigenfunction on $X_{\Gamma}$ (for the Laplacian of $X_{\Gamma}$) is a $\Gamma$-invariant eigenfunction on $X$ (for the Laplacian of $X$).

A $\Gamma$-invariant eigenfunction of $L_X$ is called an \emph{automorphic eigenfunction}. Thus, \eqref{eigenvalue problem} corresponds to the automorphic eigenvalue problem
\begin{eqnarray*}%\label{automorphic eigenvalues problem}
L_X\phi &=& - c \, \phi, \\
\phi(\gamma z) &=& \phi(z) \textnormal{ for all } \gamma\in\Gamma \textnormal{ and for all } z\in G/K.
\end{eqnarray*}

The rank of an algebra is defined as the maximal number of pairwise commuting generators of the algebra. The rank of the algebra $\mathbb{D}(G/K)$ of translation invariant differential operators equals the real rank of $G/K$, that is the number $\dim(A)$, where $G=KAN$ is an Iwasawa decomposition, or equivalently the dimension of a maximal flat subspace of $G/K$. It follows that if $X$ has higher rank $\geq 2$, the $\phi_{j}$ chosen above may not necessarily be joint eigenfunctions of $\mathbb{D}(G/K)$. However, if $X$ has rank one, then this is true (Remark \ref{polynomials}). In particular, if $X$ has rank one, the joint eigenspaces are given by ($\langle\cdot,\cdot\rangle$ denotes the extension of the Killing form to $\la^*_{\cc}$)
\begin{eqnarray*}
\mathcal{E}_{\lambda}(X) = \left\{f\in\mathcal{E}(X): L_X f=-(\langle\lambda,\lambda\rangle+\langle\rho,\rho\rangle)f \right\}.
\end{eqnarray*}

Suppose that $\phi\in\mathcal{E}^*_{\lambda}$, where $\lambda\in\la^*_{\cc}$, is a $\Gamma$-invariant joint eigenfunction of $G/K$. Then ($|\cdot|$ denotes the norm on $\la^*$ induced by the Killing form of $\g$)
\begin{eqnarray*}
D\phi = -(\langle\lambda,\lambda\rangle+|\rho|^2) \,\,\,\,\, \textnormal{ for all } D\in\mathbb{D}(G/K).
\end{eqnarray*}

\subsubsection{The rank one case}\label{The rank one case}
Recall the situation when the symmetric space has rank one: We only consider joint eigenfunctions with exponential growth. Given such a $\phi_{j}$, it follows that there is $\lambda_j\in\la^*_{\cc}$ such that $c_j=-(\langle\lambda_j,\lambda_j\rangle+\langle\rho,\rho\rangle)$. Then
\begin{eqnarray}
\Delta\phi_{j} = -(\langle\lambda_j,\lambda_j\rangle+\langle\rho,\rho\rangle) \phi_{j}.
\end{eqnarray}
We can then fix the eigenvalue parameters $\lambda_j$ corresponding to the spectrum
\begin{eqnarray*}
0=c_0<c_1<c_2<\ldots
\end{eqnarray*}
It follows from $\langle\rho,\rho\rangle\in\rr$ that $\langle\lambda_j,\lambda_j\rangle\rightarrow\infty$ ($j\rightarrow\infty$). Suppose that $X$ has rank one. Then for all $j\in\nn_0$ we must have $\lambda_j\in\la^*\cup i\la^*$, where $i=\sqrt{-1}$. We can hence for at most finitely many $j$ have $\lambda_j\in i\la^*$, that is only finitely many $\lambda_j$ are contained in the so-called \emph{complementary series}. All remaining $\lambda_j$ are contained in the unitary principal series, which we have studied in Section \ref{Section Helgason boundary values}. We will in this context sometimes also write $\lambda\rightarrow\infty$, which means $\lambda(H)\rightarrow\infty$ for each $H$ in the positive Weyl chamber $\la^+$.

\subsubsection{Wigner distributions}\label{Wigner distributions}
Given a joint eigenfunction $\phi\in\mathcal{E}^*_{\lambda}(X)$, we denote the corresponding (uniquely determined) distributional boundary values by $T_{\phi}\in\mathcal{D}(B)$ (Theorem \ref{Helgason boundary values}). Then
\begin{eqnarray*}
\phi(z) = \int_B e^{(i\lambda+\rho)\langle z,b\rangle} \, T_{\phi}(db),  \,\,\,\,\,\,\, z\in X.
\end{eqnarray*}
Recall that given $\lambda\in\mathfrak{a}^*_{\cc}$ and $b\in B$, the functions
\begin{eqnarray*}
e_{\lambda,b}: X\rightarrow\cc, \hspace{2mm} z\mapsto e^{(i\lambda+\rho)\langle z,b\rangle}.
\end{eqnarray*}
are called non-Euclidean plane waves. The symmetric space calculus of pseudodifferential operators (Chapter \ref{Pseudodifferential analysis on symmetric spaces}) is defined by
\begin{eqnarray}
\left(Op(a)e_{\lambda,b}\right)(z)=a(z,\lambda,b)e_{\lambda,b}(z).
\end{eqnarray}
Non-Euclidean Fourier analysis extends this definition to $C_c^{\infty}(X)$. We always assume that the symbol $a:X\times B\times\mathfrak{a}\rightarrow\cc$ of $Op(a)$ is a polyhomogeneous function in $\lambda$ in the classical sense defined in \eqref{homogeneous}. We know from Section \ref{Invariance properties} that $Op(a)$ commutes with the action of $\gamma\in\Gamma$ if and only if $a$ is invariant under the diagonal action of $\Gamma$ on $X\times B=G/M$. We will from now on always assume that $Op(a)$ is properly supported. In the non-Euclidean calculus we then have
\begin{eqnarray}
Op(a)\phi(z) = \int_B a(z,\lambda,b) e^{(i\lambda+\rho)\langle z,b\rangle} \, T_{\phi}(db).
\end{eqnarray}

\begin{defn}
Let $\lambda,\mu\in\la^*_{\cc}$ and suppose that $\phi\in\mathcal{E}^*_{\lambda}(X)$ and $\psi\in\mathcal{E}^*_{\mu}(X)$ are $L^2(X_{\Gamma})$-normalized and $\Gamma$-invariant joint eigenfunctions of $\mathbb{D}(G/K)$. We define the Wigner distributions $W_{\phi,\psi}$ \emph{associated to $\phi$ and $\psi$} on $C^{\infty}(\Gamma\backslash G/M)$ by
\begin{eqnarray}
W_{\phi,\psi}(a) := \langle Op(a)\phi, \psi \rangle_{L^2(X_{\Gamma})}.
\end{eqnarray}
\end{defn}

We view $a\in C^{\infty}(\Gamma\backslash G/M)$ as a symbol $a\in S^0$, which is is independent of $\lambda$. Note that $W_{\phi,\psi}$ is a well-defined distribution: Using the boundary values, we express (as we will do in \eqref{express}) the $L^2$-inner product by means of the Poisson transform and obtain the distribution
\begin{eqnarray}
W_{\phi,\psi} = e^{(i\lambda+\rho)\langle z,b\rangle} \, e^{(i\mu+\rho)\langle z,b'\rangle} \, dz \, T_{\phi}(db) \, T_{\psi}(db').
\end{eqnarray}
Hence $W_{\phi,\psi}(a)$ is bounded by a continuous $C^{\infty}(\Gamma\backslash G/M)$-seminorm of $a$. In the special case when $\phi=\psi$ we write $W_{\phi}:=W_{\phi,\phi}$.

Let $X$ have rank one. Recall from \ref{The rank one case} the fixed basis $(\phi_j)$ of eigenfunctions of $\Delta$. We denote the corressponding boundary values by $T_j$. Then $\phi_j=P_{\lambda_j}(T_j)$ by means of the Poisson-Helgason transform, where $\lambda_j$ is as in Subsection \ref{Eigenfunctions on a compact quotient}. We will then write $W_{j,k}:=W_{\phi_j,\phi_k}$.

\begin{rem}
Let $\phi\in\mathcal{E}^*_{\lambda}(X)$ and $\lambda\in\la$ be real valued. The distributions $W_{\phi}$ are \emph{quantum time reversible} in the following sense: Let $Cf=\overline{f}$ denote complex conjugation and write $\mathcal{C}a(z,\lambda,b)=\overline{a}(z,-\lambda,b)$. We have $COp(a)C=Op(\mathcal{C}a)$ by a direct computation. Hence $\langle COp(a)C\phi,\phi\rangle=\langle Op(a)\phi,\phi\rangle$, so $\mathcal{C}^{*}W_{\lambda}=W_{\lambda}$.
\end{rem}

\subsubsection{An intertwining formula}
Asymptotic properties of Wigner distributions only concern principal symbols. We hence assume symbols $a(z,\lambda,b)$ of order $0$ to be independent of $\lambda$. Recall that if $\chi$ is a smooth fundamental domain cutoff function, then $W_{\phi,\psi}(a) = \langle Op(\chi a)\phi, \psi \rangle_{L^2(X)}$.

\begin{rem}
In what follows we need a certain amount of regularity for the boundary values we work with. From now on, we will always work with distributional boundary values which are actually functions, that is $T_{\phi}\in L^1(B)$, the space of integrable functions on $B$\index{$L^1(B)$, space of integrable functions on $B$}.
\end{rem}

\begin{thm}\label{Intertwining Formula}
Let $\phi\in\mathcal{E}^*_{\lambda}(X)$ and $\psi\in\mathcal{E}^*_{\mu}(X)$ be $\Gamma$-invariant joint eigenfunctions with respective boundary values $T_{\phi}\in L^1(B)$ and $T_{\psi}\in L^1(B)$. Let $\psi$ be real-valued. Then for $a\in C^{\infty}(\Gamma\backslash G/M)$ we have
\begin{eqnarray}
W_{\phi,\psi}(a) = \langle L_{\mu}(\chi a), PS_{\lambda,\mu}\rangle.
\end{eqnarray}
\end{thm}
\begin{proof}
We express this $L^2(X)$-inner product by means of the Poisson-Helgason transform formula \eqref{Poisson transform}:
\begin{eqnarray}\label{express}
\langle Op(\chi a)\phi, \psi\rangle_{L^2(X)} &=& \int_X (Op(\chi a)\phi)(z)\psi(z) \, dz \nonumber \\
&\hspace{-60mm}=& \hspace{-30mm}\int_{B\times B}\left(\int_X (\chi a)(z,b) e^{(i\lambda+\rho)\langle z,b\rangle} e^{(i\mu+\rho)\langle z,b'\rangle} \, dz \right) T_{\phi}(db) \, T_{\psi}(db').
\end{eqnarray}
It follows from Lemma \ref{formula} that $\mathcal{R}_{\lambda,\mu}(L_{\mu}\chi a)(b,b')$ extends to a smooth function on $B\times B$, which is given by the inner $X$-integral above. Then \eqref{express} equals
\begin{eqnarray}\label{also shows}
\langle \mathcal{R}_{\lambda,\mu}(L_{\mu}\chi a), T_{\phi}\otimes T_{\psi}\rangle_{B\times B}  = \langle L_{\mu}(\chi a), PS_{\lambda,\mu} \rangle,
\end{eqnarray}
and the theorem is proven.
\end{proof}

\begin{rem}
If $\phi=\psi$ and $\lambda=\mu$, then \eqref{also shows} shows
\begin{eqnarray*}
W_{\phi}(a) &=& \langle d_{\lambda}\mathcal{R}(L_{\lambda}\chi a), T_{\phi}\otimes T_{\phi}\rangle_{B\times B} \\
&=& \langle \mathcal{R}(L_{\lambda}\chi a), ps_{\lambda} \rangle_{G/M} \\
&=& \langle L_{\lambda}(\chi a), PS_{\lambda} \rangle_{G/M}.
\end{eqnarray*}
\end{rem}

\subsection{The spectral order principle}
 Let $X=G/K$ have rank one. As usual, we identify $\la$ and $\la^*$ with $\rr$ by means of the Killing form $\langle\cdot,\cdot\rangle$: The unit vector (w.r.t. the Killing form) $H\in\la_+$ and the linear functional $\lambda_0\in\la^*$ given by $\lambda_0(X)=\langle X,H\rangle$ are identified with the real number $1$.

In this section we introduce an idea which we call the \emph{spectral order principle}\index{spectral order principle}. This principle is geared to explain asymptotic relations between phase space distributions and Wigner-distributions. To be as general as possible, we let $Op: C^{\infty}(SX_{\Gamma})\rightarrow B(L^2(SX_{\Gamma}))$ denote an arbitrary operator convention.

Let $\left\{\phi_{\lambda}\right\}$ denote a family of $\Gamma$-invariant joint eigenfunctions $\phi_{\lambda}\in\mathcal{E}^*_{\lambda}$ to spectral parameters $\lambda\in\la^*_{\cc}$, which are all normalized w.r.t. the norm of $L^2(X_{\Gamma})$. Recall that $\Gamma$-invariant distributions on $SX$ descend to distributions on $SX_{\Gamma}$ by using smooth fundamental domain cutoff functions. We fix a smooth fundamental domain cutoff function $\chi$.

\begin{defn}[Intertwining operator]\label{Intertwining operator}
We say a family $\left\{T_{\lambda,\mu}\right\}\subset\mathcal{D}'(SX)$ of $\Gamma$-invariant distributions is \emph{intertwined} with the Wigner distributions $W_{\phi_{\lambda},\phi_{\mu}}$ if for each $\mu$ there is a linear operator $L_{\mu}:C_c^{\infty}(SX)\rightarrow C_c^{\infty}(SX)$ such that
\begin{eqnarray}\label{comparing with}
W_{\phi_{\lambda},\phi_{\mu}}(a) = T_{\lambda,\mu}(L_{\mu}(\chi a)) \,\,\,\,\,\,\,\, \forall \, a\in C_c^{\infty}(SX_{\Gamma}).
\end{eqnarray}
The operators $L_{\mu}$ are called \emph{intertwining operators}\label{intertwining operators}.
\end{defn}

\begin{defn}[Spectral order of a distribution]
Let $\left\{T_{\lambda,\mu}\right\}\subset\mathcal{D}'(SX)$ denote a family of distributions. We say that $\left\{T_{\lambda,\mu}\right\}\subset\mathcal{D}'(SX)$ has \emph{spectral order} $K\in\rr$ if there is a continuous seminorm $\|\cdot\|$ on $C_c^{\infty}(SX)$ such that for all $\lambda,\mu$
\begin{eqnarray}\label{spectral estimate}
|T_{\lambda,\mu}(f)| \leq (1+|\lambda|)^K (1+|\mu|)^K \cdot \|f\| \,\,\,\,\,\,\,\, \forall \, f \in C_c^{\infty}(SX).
\end{eqnarray}
\end{defn}

\begin{defn}[Left-invariant asymptotic expansion]
Let $L_{\mu}: C_c^{\infty}(SX)\rightarrow C_c^{\infty}(SX)$ be a family of intertwining operators (in the sense of \ref{Intertwining operator}. Suppose that there is an aymptotic expansion
\begin{eqnarray}\label{left-invariant asymptotic expansion}
L_{\mu}(a)(gM) \sim \sum_{j=0}^{\infty} \mu^{-j-s/2} R_j(a)(gM)
\end{eqnarray}
in the sense that $|L_{\mu}(a) - \sum_{j=0}^{N-1} \mu^{-j-s} R_j(a)|\leq C_N (1+|\mu|)^{-N}$, where $s\in\rr$ is a constant and where the $R_j: C_c^{\infty}(SX)\rightarrow C_c^{\infty}(SX)$ are differential operators on $SX$. We say that \eqref{left-invariant asymptotic expansion} is a \emph{left-invariant asymptotic expansion}\index{left-invariant asymptotic expansion}, if the $R_j$ are left-invariant differential operators.
\end{defn}

Suppose that $T_{\lambda,\mu}\in\mathcal{D}'(SX_{\Gamma})$ is a distribution depending on two spectral parameters, with $T_{\mu,\mu}(1)\neq 0$. We denote by $\widehat{T}_{\lambda,\mu}\in\mathcal{D}'(SX_{\Gamma})$ the normalized distribution
\begin{eqnarray}
\langle \widehat{T}_{\lambda,\mu},f\rangle := \frac{\langle T,f\rangle}{\langle T_{\mu,\mu},1\rangle}.
\end{eqnarray}

\begin{thm}\label{spectral order principle}
Suppose that $\left\{T_{\lambda,\mu}\right\}$ is a family of distributions of spectral order $K$ which is intertwined with the Wigner distributions $W_{\lambda}$ by the uniformly continuous (in $\mu$) intertwining operators $L_{\mu}$. Let the $L_{\mu}$ have an asymptotic expansion with left-invariant coefficients. Suppose $\mathcal{O}(|\lambda|^{-1}) = \mathcal{O}(|\mu|^{-1})$. Let $a\in C^{\infty}(SX_{\Gamma})$. Then we have the asymptotic equivalence
\begin{eqnarray}
W_{\lambda,\mu}(a) = \widehat{T}_{\lambda,\mu}(a) + \mathcal{O}(\mu^{-1}).
\end{eqnarray}
The constant in the $\mathcal{O}$-term is a $C^{\infty}(SX_{\Gamma})$-seminorm of $a$.
\end{thm}
\begin{proof}
We copy the asymptotic argument given in \cite{AZ}. First, integrating \eqref{left-invariant asymptotic expansion} with respect to $T_{\lambda,\mu}$ and comparing with \eqref{comparing with} we get an asymptotic expansion (in the sense of \eqref{in this sense})
\begin{eqnarray*}
\langle Op(a)\phi_{\lambda},\phi_{\mu}\rangle_{SX_{\Gamma}} \sim \sum_{n\leq0} \mu^{-n-s/2} \langle R_{n}(\chi a),T_{\lambda,\mu} \rangle_{SX}.
\end{eqnarray*}
Note that the coefficients of this expansion depend on the spectral parameters. By left-invariance, each distribution
\begin{eqnarray*}
f\mapsto \langle R_{n}(f), T_{\lambda,\mu} \rangle_{SX}
\end{eqnarray*}
is $\Gamma$-invariant, so by Proposition \ref{independent}, the functional
\begin{eqnarray*}
a\mapsto \langle R_{n}(\chi a), T_{\lambda,\mu} \rangle_{SX}
\end{eqnarray*}
defines a distribution on $SX_{\Gamma}$ and the first term (for $n=0$) is $T_{\lambda,\mu}$. Then
\begin{eqnarray}\label{in this sense}
&& \langle Op(a)\phi_{\lambda},\phi_{\mu}\rangle_{SX} = \langle L_{\mu}(\chi a),T_{\lambda,\mu}\rangle_{SX} \nonumber \\
&& \,\,\,\,\, = \sum_{n=0}^{N} \mu^{-n-s/2} \langle R_{n}(\chi a),T_{\lambda,\mu} \rangle + \mathcal{O}(\mu^{-N-1+2K}).
\end{eqnarray}
We choose $N>2K$. Since $R_0$ is the identity, the operator $L_{\mu}^{(N)}=\sum_n^{N}\mu^{-n}R_{n}$ can be inverted up to $\mathcal{O}(\mu^{-N-1})$, i.e. one finds differential operators
$M_{\mu}^{(N)}=\sum_{n=0}^N \mu^{-n}M_{n}$, where $M_0=\id$, and $R_{\mu}^{(N)}$, such that
\begin{eqnarray*}
L_{\mu}^{(N)} M_{\mu}^{(N)} = \id + \mu^{-N-1}R_{\mu}^{(N)}.
\end{eqnarray*}
We apply \ref{comparing with} to $M_{\mu}^{(N)}(a)$ and find
\begin{eqnarray*}
\langle Op(M_{\mu}^{(N)}a)\phi_{\lambda},\phi_{\mu}\rangle_{SX_{\Gamma}}
&=& \langle L_{\mu}^{(N)}\chi M_{\mu}^{(N)}a,T_{\lambda,\mu}\rangle_{SX}
    + \mathcal{O}(\mu^{-N-1+2K}) \\
&=& \langle L_{\mu}^{(N)}M_{\mu}^{(N)}\chi a,T_{\lambda,\mu}\rangle_{SX}
    + \mathcal{O}(\mu^{-N-1+2K}) \\
&=& \langle a,T_{\lambda,\mu}\rangle_{SX_{\Gamma}} + \mathcal{O}(\mu^{-N-1+2K}).
\end{eqnarray*}
The second line is a consequence of Proposition \ref{independent}. But
\begin{eqnarray*}
M_{\mu}^{(N)}(a) = a + \mu^{-1}\left( M_1 + \ldots + \mu^{-N+1}M_{2}\right)(a),
\end{eqnarray*}
so by the $L^2$-continuity of zero-order pseudodifferential operators,
\begin{eqnarray*}
\langle
Op(M_{\mu}^{(N)}(a))\phi_{\lambda},\phi_{\mu}\rangle_{L^2(X_{\Gamma})}
= \langle
Op(a)\phi_{\lambda},\phi_{\mu}\rangle_{L^2(X_{\Gamma})} +
\mathcal{O}(1/\mu).
\end{eqnarray*}
This proves
\begin{eqnarray}\label{left side 3}
\langle a,T_{\lambda,\mu}\rangle_{SX_{\Gamma}} = \langle Op(a)\phi_{\lambda},\phi_{\mu}\rangle_{SX_{\Gamma}} +
\mathcal{O}(1/{\mu}).
\end{eqnarray}
Putting $\langle a,T_{\lambda,\mu}\rangle = \langle 1,T_{\mu,\mu}\rangle \langle a,\widehat{T}_{\lambda,\mu}\rangle$ into \eqref{left side 3} we obtain
\begin{eqnarray}\label{left side 4}
\langle 1,T_{\mu,\mu}\rangle \cdot \langle a,\widehat{T}_{\lambda,\mu}\rangle = \langle a,W_{\lambda,\mu}\rangle + \mathcal{O}(1/{\mu}).
\end{eqnarray}
In particular, for $a=1$, we get
\begin{eqnarray*}
\langle 1,T_{\mu,\mu}\rangle_{SX_{\Gamma}} = 1 + \mathcal{O}(1/{\mu}).
\end{eqnarray*}
Together with \eqref{left side 4} this yields
\begin{eqnarray}\label{left side 5}
\left( 1 + \mathcal{O}(1/{\mu}) \right) \cdot \langle a,\widehat{T}_{\lambda,\mu}\rangle = \langle a,W_{\lambda,\mu}\rangle + \mathcal{O}(1/{\mu}).
\end{eqnarray}
The Wigner distributions and hence by \eqref{left side 5} the $\langle a,\widehat{T}_{\lambda,\mu}\rangle$ are uniformly bounded. It follows that the left side of \eqref{left side 5} is asymptotically the same as $\langle a,\widehat{T}_{\lambda,\mu}\rangle$.
\end{proof}

\begin{rem}
One can weaken some assumptions of the above principle (Theorem \ref{spectral order principle}). For example, it is not really neccessary to claim $\mathcal{O}(|\lambda|^{-1}) = \mathcal{O}(|\mu|^{-1})$. The condition $\mathcal{O}(|\lambda|^{-1})\leq\mathcal{O}(|\mu|^{-L})$ for an $L\geq 1$ will still be sufficient: We can then choose $N>2LK$ in the above asymptotic expansions. Moreover, the condition that the intertwiners $L_{\mu}$ preserve compact supports is not neccessary, if the expression $T_{\lambda,\mu}(L_{\mu}(\chi a))$ still makes sense for $a\in C^{\infty}(SX_{\Gamma})$, and if for $f=L_{\mu}(\chi a)$ the spectral estimate \eqref{spectral estimate} is still satisfied.
\end{rem}

The problem is to show that the spectral order principle (or a version with weaker assumptions) can be applied to the intertwining formula \ref{Intertwining Formula} for the non-Euclidean Wigner distributions, the Patterson-Sullivan distributions, and the Knapp-Stein intertwiners. I will now describe what the concrete problems are and restrict these considerations to the case of diagonal elements ($\phi=\psi$, $\lambda=\mu$). Let $f\in C_c^{\infty}(G/M)$. The values $|d_{\lambda}(b,b')|$ are independent of $\lambda$ and all derivatives of $d_{\lambda}$ have polynomial growth in $\lambda$. It follows that given a continuous seminorm $\|\cdot\|_1$ on $C^{\infty}(B\times B)$ there exist $K_1>0$ and a continuous seminorm $\|\cdot\|_2$ on $C_c^{\infty}(G/M)$ such that
\begin{eqnarray}\label{Radon estimate 0}
\|d_{\lambda}(b,b')\mathcal{R}(f)(b,b')\|_1 \leq (1+|\lambda|)^{K_1}\|f\|_2.
\end{eqnarray}
Note that $\|\cdot\|_2$ may depend on the support of $f$. Assume $d_{\lambda}(b,b')\mathcal{R}(f)(b,b') \in C^{\infty}(B\times B)$. Then $PS_{\lambda}(f)$ is well-defined. A simple example is when $f\in C_c^{\infty}(SX)=C_c^{\infty}(SX)$. Let $\lambda\in\la^*$. In this case, it follows from \eqref{BtimesB}, \eqref{Off-Diagonal Patterson-Sullivan} and \eqref{Radon estimate 0} that there exist $K>0$ and a continuous seminorm $\|\cdot\|_2$ on $C_c^{\infty}(SX)$ (possibly depending on the support of $f$) such that
\begin{eqnarray}\label{define K 0}
|PS_{\lambda}(f)| \leq (1+|\lambda|)^K \|f\|_2.
\end{eqnarray}

It is stated in \cite{AZ} (equation (3.14) there) that there is a seminorm independent of the function $f$. I cannot find such an estimate. However, even if we would have this equation for $f\in C_c^{\infty}(SX)$, another problem would occur in the well-definedness of the intertwining formula from Theorem \ref{Intertwining Formula}: The Knapp-Stein intertwiners do not preserve compact supports, so the intertwining formula can only be understood formally in the sense of continuation from $B^{(2)}$ to $B\times B$ (Lemma \ref{formula}). The problem is that for the $ps_{\lambda}$-distributions there is no spectral order estimate in the sense of \eqref{spectral estimate} for the enlarged domain $d_{\lambda}(b,b')^{-1}\cdot C^{\infty}(B\times B)$. For the $PS$-distributions, the constant $K$ and the seminorm $\|\cdot\|_2$ cannot be used in a proof of \ref{spectral order principle}, since the remainder terms in the asymptotic expansion \eqref{integrate this} are not compactly supported.

For $a\in C^{\infty}(\Gamma\backslash G/M)$, let $f_{a,\lambda,\mu}(b,b')\in C^{\infty}(B\times B)$ denote the inner $X$-integral in \eqref{express}. The intertwining formula in Theorem \ref{Intertwining Formula} is understood in the sense of $\langle L_{\mu}(\chi a), PS_{\lambda,\mu}\rangle_{G/M} = \langle f_{a,\lambda,\mu},T_{\lambda}\otimes T_{\lambda}\rangle_{B\times B}$. In this sense, \eqref{BtimesB} yields $|\langle L_{\mu}(\chi a), PS_{\lambda,\mu}\rangle|\leq(1+|\lambda|)^K(1+|\mu|)^K\|\chi a\|$, where $\|\cdot\|$ is a seminorm on $C^{\infty}(G/M)$ and only depends on the support of $\chi$.

\subsubsection{Further remarks and some open questions}
\begin{itemize}
\item[(1)] Recall that the intertwining formula is the same in each case. One could conjecture that the asymptotic argument given in the proof of the spectral order principle can be generalized to all symmetric spaces of the noncompact type. It should be conjectured that most limits of Wigner distributions are $A$-invariant (see \cite{SV}. Similar results are announced by L. Silberman and N. Anantharaman). In view of Remark \ref{A-eigendistributions}, we see that limits of Patterson-Sullivan distributions, as defined via $B^{(2)}$ will not always be $A$-invariant.
\item[(2)] It is in some cases possible to modify the definitions and to obtain off-diagonal $ps_{\lambda,\mu}$-distributions: For simplicity, let $G/K$ have rank one, so that the function $d_{\lambda}(b,b')$ exists. Recall
\begin{eqnarray}
\mathcal{R}_{\lambda,\mu}f(b,b') = \int_A d_{\lambda,\mu}(g(b,b')a)f(g(b,b')a)\,da.
\end{eqnarray}
The choice of $g=g(b,b')$ was immaterial (modulo $MA$), so if we assume $H(g)=0$, then $d_{\lambda,\mu}(g)=e^{(i\lambda+\rho)H(g)}e^{(i\mu+\rho)H(gw)}=d_{\mu}(b,b')$. One can then define the distributions $ps_{\lambda,\mu}(db,db')=d_{\mu}(b,b')T_{\lambda}(db)T_{\mu}(db')$ and $\widetilde{PS}_{\lambda,\mu}(f)=ps_{\lambda,\mu}(\mathcal{R}(f))$. However, $ps_{\lambda,\mu}$ is not $\Gamma$-invariant in the off-diagonal case.
\item[(3)] It is possible to express the normalized version of the Patterson-Sullivan distributions by means of Harish-Chandra's $c$-function. Therefore, a generalization of Lemma 6.4 in \cite{AZ} is needed, which does not a priori make sense in $G/M$, since there is no horocycle flow on $G/M$. However, some of the formulas given in Theorem 1.2 of \cite{AZ}, in particular the one for the normalization of the $PS$-distributions, generalize to arbitrary symmetric spaces.
\item[(4)] It is still an open question if there is a purely classical dynamical interpretation of the Patterson-Sullivan distributions in terms of closed geodesics (see \cite{AZ}).
\end{itemize}
Details concerning these open questions are in progress and will eventually appear later.

\newpage
\pagebreak
\thispagestyle{empty} 

\addcontentsline{toc}{section}{Bibliography}

%\begin{thebibliography}{ddddddidd}

\newpage

\pagebreak

\thispagestyle{empty} 

\addcontentsline{toc}{section}{Index}

\thispagestyle{empty}

\printindex

\thispagestyle{empty}

\end{document}